%% file: main.tex
\documentclass[12pt,twoside,reqno,a4paper]{report}
\usepackage{amsmath}
\usepackage{amssymb,amsfonts,amsthm,amscd}
\usepackage{exscale}
\usepackage{pictex,dcpic,epic,eepic}
\usepackage{mathrsfs}
\usepackage{color}
\definecolor{gray}{gray}{.75}
\definecolor{gray2}{gray}{.50}

\usepackage{local}

\renewcommand{\thechapter}{\Roman{chapter}}
\renewcommand{\thesection}{\arabic{section}}
\numberwithin{equation}{chapter}
\renewcommand{\thetheorem}{(\Roman{chapter},~\arabic{section}.\arabic{theorem})}
\renewcommand{\theequation}{\Roman{chapter},~\arabic{equation}}

\begin{document}

\title{An introduction to Seiberg-Witten theory on closed 3-manifolds}
\author{Michael Bohn\\[12pt] \small\it Department of Mathematics, Bonn
University\\ \small\it Beringstr.1, D-53115 Bonn, Germany\\ \small
\textrm{e-mail: mbohn@math.uni-bonn.de}}
\date{}

\maketitle

\input{abstract}

\pagestyle{myheadings}\pagenumbering{roman}
\markboth{\textsc{Contents}}{\textsc{Contents}} \tableofcontents
\cleardoublepage\pagenumbering{arabic}

\markboth{\textsc{Introduction}}{\textsc{Introduction}}

\input{intro}

\cleardoublepage
\renewcommand{\chaptermark}[1]%
   {\markboth{\textsc{Chapter \thechapter.\ #1}}{}}
\renewcommand{\sectionmark}[1]%
   {\markright{\textsc{\thesection.\ #1}}}

\input{chap_I}

\input{chap_II}
\input{chap_III}

\cleardoublepage
\renewcommand{\chaptermark}[1]%
   {\markboth{\textsc{Appendix \thechapter.\ #1}}{}}
\renewcommand{\sectionmark}[1]%
   {\markright{\textsc{\thesection.\ #1}}}
\renewcommand{\thetheorem}{(\Alph{chapter},~\arabic{section}.\arabic{theorem})}
\renewcommand{\theequation}{\Alph{chapter},~\arabic{equation}}
\addtocontents{toc}{\newpage}
\addtocontents{toc}{\contentsline{chapter}{Appendix}{}}

\appendix

\input{app_A}
\input{app_B}
\input{app_C}
\input{app_D}

\cleardoublepage
\markboth{\textsc{Bibliography}}{\textsc{Bibliography}}
\addcontentsline{toc}{chapter}{Bibliography}
\bibliography{books,papers}
\bibliographystyle{amsplain}

\end{document}

%% file: abstract.tex
\abstract

This is a version of the author's diploma thesis written at the
University of Cologne in 2002/03. The topic is the construction
of Seiberg-Witten invariants of closed 3-manifolds. In analogy to
the four dimensional case, the structure of the moduli space is
investigated. The Seiberg-Witten invariants are defined and their
behaviour under deformation of the Riemannian metric is analyzed.

Since it is essentially an exposition of results which were
already known during the time of writing, the thesis has not been
published. In particular, the author does not claim any
originality concerning the results. Moreover, new developments of
the theory are not included. However, the detailed
account---together with the appendices on the required functional
analytic and geometric background---might be of interest for
people starting to work in the area of gauge field theory.

%% file: intro.tex
\chapter*{Introduction}
\addcontentsline{toc}{chapter}{Introduction}

\section{Historical remarks and background}

Seiberg-Witten theory is a gauge theoretical approach to low
dimensional topology. By making use of Riemannian geometry and
the analysis of partial differential equations, one finds smooth
invariants of three or four dimensional manifolds, thus entering
the area of differential topology.

\subsection{Low dimensional topology}

There has been a long standing interest in understanding the
structure of manifolds. In this process, the dimensions three and
four have always played a prominent role since an easy
classification as in the two dimensional case is not possible. On
the other hand, the very efficient topological tools which have
proved useful in the theory of higher dimensional manifolds
cannot be applied in the dimensions three and four. Among many
other results alluding to these peculiarities are, for example,
the discovery of exotic structures on $\R^4$.

At the beginning of the 1980s, the work of M. Freedman gave a new
insight in the topological classification of simply connected
compact 4-manifolds via their intersection forms. About the same
time, S.K. Donaldson succeeded in establishing criteria how the
intersection form can prevent a topological 4-manifold from
being smoothable.\\

\noindent\textbf{Mathematical gauge theory.} The main idea of
Donaldson's work is to study solutions of the anti self-dual
instanton equations---a set of partial differential equations
arising from Yang-Mills theory, which describes elementary
particles in physics. By making use of gauge symmetries, one
defines the so-called {\em moduli space} whose structure reflects
much of the underlying manifold's topology. It turns out that
generically, the moduli space is a finite dimensional manifold
with boundary except at a finite number of singular points
occurring at solutions having too much gauge symmetry. By means of
establishing a relation between these singularities and the
intersection form, Donaldson's famous Theorem A excludes certain
4-manifolds from admitting a differentiable structure.

Subsequent work---part of which culminated in the definition of
Donaldson's polynomial invariants---has emphasized the
fruitfulness of the gauge theoretical approach to low dimensional
topology. As self-contained introductions to this
development---written by some of the major participants---the
books of Donaldson \& Kronheimer \cite{DK} and Freed \& Uhlenbeck
\cite{FU} are highly
recommended. \\

\noindent\textbf{Gauge theory on 3-manifolds.} The basic idea of
gauge theory is not restricted to the four dimensional setting.
It thus was soon applied to 3-manifolds as well. However, since
the anti self-dual equations cannot be formulated on a manifold
of this dimension, the viewpoint had to be altered a little. The
critical points of the so-called {\em Chern-Simons function} were
found to be a promising substitute.\footnote{At least on homology
spheres.} An important concept of this work---particularly due to
C.H. Taubes \cite{Tau:CI} and A. Floer \cite{Flo:Inst}---is to
interpret the Chern-Simons function as a Morse function on the
space of all gauge fields modulo the action of the group of gauge
transformations. An invariant is then defined in the same manner
as in finite dimensional Morse theory: There, the Euler
characteristic of a manifold can be computed as the signed count
of Morse indices. The major problems connected with this idea,
namely the question of how to deal with degenerate critical
points and how to generalize the notion of the Morse index to an
infinite dimensional setting, has successfully been solved by
Taubes and Floer. The invariant obtained by using their approach
turned out to equal the topologically defined Casson invariant.

In fact, Floer's work went much beyond that. He generalized the
concept of a Morse complex and constructed cohomology groups
associated to the Chern-Simons function which yield even more
refined invariants. A recent monograph by Donaldson \cite{D}
gives a detailed exposition of the so-called {\em Floer homology
groups} in the gauge theoretical context.\\

\noindent\textbf{Seiberg-Witten theory.} While the anti self-dual
instanton equations of pure Yang-Mills theory are easily written
down, the involved calculations are complicated due to the
non-abelian nature of the symmetry group. In 1994, new impulses in
mathematical gauge theory came from E. Witten's famous article
\cite{Wit:Mon}. He announced that a system of partial differential
equations---the {\em monopole equations} which arose in his joint
work with N. Seiberg---should in some sense be equivalent to the
anti self-dual instanton equations. However, the Seiberg-Witten
equations have an abelian gauge symmetry and are therefore easier
to be dealt with from an analytical point of view. For example,
Witten proved that the corresponding moduli space is always
compact so that smooth invariants can be extracted in a much
easier way than in instanton theory.

In the subsequent months, many of the results obtained via
Donaldson theory could be reproved by making use of what was soon
called {\em Seiberg-Witten theory}. For example, P.B. Kronheimer
and T. Mrowka established a remarkably simple proof of the Thom
conjecture \cite{KroMro:Gen}. Moreover---as Witten pointed out in
his original paper---the structure of the Seiberg-Witten
equations simplify considerably when formulated on K\"ahler
surfaces so that the new theory had an instantaneous impact on
complex geometry. Soon, C.H. Taubes managed in a series of
papers---starting with \cite{Tau:Sym}---to establish deep
relations to invariants of symplectic 4-manifolds. Beginning with
the work of C. LeBrun \cite{LeB:Einst}, Seiberg-Witten theory was
also found promising in answering unsolved questions in
Riemannian geometry.

Although the pace in which new results were obtained decreased
after a while, Seiberg-Witten theory has become an important tool
for studying 4-manifolds. Nowadays there are not only many survey
articles reviewing the dawn of Seiberg-Witten theory (e.g.
Donaldson \cite{Don:SW}, Kronheimer \cite{Kro:ES} and recently K.
Iga \cite{Iga:Top}) but also monographs giving a more detailed
exposition of the theory (e.g. J.W. Morgan \cite{Mo} and J.D.
Moore \cite{Moo}). The book of L.I. Nicolaescu \cite{Nic:SW} is
perhaps the most extensive introduction to the four dimensional
theory which has appeared until now. M. Marcolli's textbook
\cite{Ma} provides a remarkable selection of excellent references
for any aspect of Seiberg-Witten theory---including the physical
background.

\subsection{Seiberg-Witten theory on 3-manifolds}

It was soon realized by P.B. Kronheimer and T. Mrowka in
\cite{KroMro:Gen} that the four dimensional theory can be carried
over to 3-manifolds if the Seiberg-Witten equations are studied
on a manifold $M\times S^1$, where $M$ is a compact 3-manifold.
Subsequently, much concentration was focussed on studying the new
theory from a three dimensional point of view as well.

Kronheimer and Mrowka found out that as in instanton theory, the
partial differential equations obtained for 3-manifolds have a
natural interpretation as the gradient flow equations of a
Chern-Simons-like function. Moreover, the Seiberg-Witten moduli
space of a Riemannian 3-manifold is always compact. Due to the
variational aspects of the theory, it turns out that---up to a
generic perturbation---the moduli space consists of isolated
points. Applying Taubes' and Floer's ideas from instanton theory
it is thus possible to define the signed count of monopoles,
which is again reminiscent of expressing a finite dimensional
manifold's Euler characteristic in terms of a Morse function's
critical points. The number obtained in this way is then expected
to be independent of the chosen metric and the perturbation term.

It was proved by Y. Lim in \cite{Lim:SW}, for example, that this
is indeed true for manifolds with first Betti number $b_1>1$,
whereas in the case $b_1=1$ the number depends on a certain
cohomological datum encoded in the perturbation term. In
\cite{MenTau:SW}, G. Meng and C.H. Taubes exposed a relationship
between a version of the Seiberg-Witten invariant and the Milnor
torsion invariant for manifolds with $b_1\ge 1$.

For rational homology spheres, however, one finds a severe
dependence on the underlying Riemannian structure and the
perturbation term. Nevertheless, there is the possibility of
adding a counter term---a certain combination of
$\eta$-invariants---to the signed count of monopoles so that the
sum obtained in this way has the desired invariance properties. It
was conjectured by Kronheimer and later independently proved by
W. Chen \cite{Che:CI} and Y.Lim \cite{Lim:CI} that in the case of
an integer homology sphere, the number obtained in this way
equals the Casson invariant. Moreover, for rational homology
spheres, there is a relation to the so-called Casson-Walker
invariant (cf. M. Marcolli \& B.L. Wang \cite{MarWan:SW} and L.I.
Nicolaescu \cite{Nic:Rat}).\\

\noindent\textbf{Seiberg-Witten-Floer homology.} Very soon after
the appearance of the new theory, Donaldson pointed out in
\cite{Don:SW} that Floer's construction of a Morse complex
associated to the Chern-Simons function should carry over to
Seiberg-Witten theory as well. A derivation of three dimensional
theory from the four dimensional case, pointing out the physical
background and the connection to topological quantum field theory,
was performed by A.L. Carey et al. in \cite{CarWan:SW3} and made
establishing a Seiberg-Witten Floer homology even more demanding.

For manifolds with non-vanishing first Betti number, it was soon
accomplished by M. Marcolli in \cite{Mar:SWF} to build-up the
Morse complex and prove its topological invariance. About the
same time, B.L. Wang \cite{Wan:SWF} exposed a severe dependence
on the metric in the case of homology spheres. Subsequently, many
authors began approaching the problem of defining a unified
Seiberg-Witten-Floer homology for all 3-manifolds and much of
this task seems to be solved nowadays (cf. Marcolli \& Wang
\cite{MarWan:SWF} and K. Iga \cite{Iga:SWF}).\newpage

\section{Organization of this thesis}
The goal of this thesis is to present a detailed and largely
self-contained construction of Seiberg-Witten invariants on
closed 3-manifolds. We take a purely gauge theoretical point of
view and shall not attempt to expose the relations to other
topological invariants to which we have alluded above. In this
sense, we restrict ourselves to only one---though the
major---aspect of three dimensional Seiberg-Witten theory. With a
view towards Seiberg-Witten-Floer theory, we emphasize the Morse
theoretical aspects of the constructions but again, a more
detailed integration of this far reaching subject is beyond the
scope of this thesis.\\

\noindent The organization of the chapters is as follows:
\begin{itemize}
\item Chapter \ref{SW:mon} establishes the gauge theoretical
set-up in which the three dimensional Seiberg-Witten equations
are formulated.

\item Chapter \ref{moduli} investigates the topological structure
of the moduli space in analogy to the four dimensional case.
Understanding the local structure of the moduli space will then
make it possible to define the signed count of monopoles in the
same way as it is performed in Taubes' work on instanton theory.

\item Chapter \ref{inv} is devoted to the analysis of how the
signed count of monopoles depends on the metric. Following the
work of Lim \cite{Lim:SW}, Chen \cite{Che:CI} and Nicolaescu
\cite{Nic:SW3}, we shall establish the main theorems of this
thesis, which prove invariance for manifolds with $b_1>1$, provide
a ``wall-crossing" formula in the case $b_1=1$, and exhibit the
severe dependence on the metric in the case of rational homology
spheres.
\end{itemize}
Since gauge theory requires nontrivial geometrical and functional
analytic constructions, we append short summaries of the material
we need:
\begin{itemize}
\item Appendix \ref{app:ell} contains a survey of the functional
analytic aspects of nonlinear elliptic equations on compact
manifolds.

\item In Appendix \ref{app:det}, we present a version of the
determinant line bundle over the space of Fredholm operators
which is needed in gauge theory to equip moduli spaces with an
orientation.

\item In Appendix \ref{app:SF&OT}, the notion of spectral flow is
recalled, which we shall need to exhibit a geometrical
interpretation of the orientation of the moduli space as the
signed count of critical points.

\item The material needed to understand the geometrical set-up of
Seiberg-Witten theory is presented in Appendix \ref{app:spinc}.
\end{itemize}

Even though familiarity with most of these constructions is
assumed, the reader is advised to browse through the appendices
since it is there, where most notations are fixed.

\section{Acknowledgements}
Many thanks are due to my thesis advisor Matthias Lesch for
offering me the opportunity to work on this beautiful subject. I
am deeply grateful for his patience and advice. Special thanks
also go to Christian Frey for answering a lot of my questions
concerning functional analysis. Many helpful discussions with him
made the preparation of these pages a lot easier to handle.
Moreover, I would like to thank my family and my friends for
their encouragement while I was working on this thesis. Finally,
the support of the Studienstiftung des deutschen Volkes during my
studies is gratefully acknowledged.




\cleardoublepage

%% file: chap_I.tex
\chapter{Seiberg-Witten Monopoles}\label{SW:mon}

The Seiberg-Witten equations are formulated within a framework
arising from spin geometry. They are a set of partial differential
equations involving a spinor field---the ``matter field"---and a
connection on a certain Hermitian line bundle---the ``gauge
field".

In this chapter we describe the special set-up arising in the
three dimensional context. We follow the notation of Appendix
\ref{app:spinc} where an exposition of \spinc manifolds is given.

To describe the interrelation between the curvature of the gauge
field---the ``field strength"---and the spinor, we have to
perform some purely linear algebraic constructions. This is the
content of Section \ref{alg:prem}. Having done so, we shall
formulate the Seiberg-Witten equations in Section \ref{SW:eqn}.
With a view towards the Morse theoretical approach to three
dimensional Seiberg-Witten theory, we then interpret solutions to
these equations as the critical points of a Chern-Simons-like
functional. We shall also see how the Seiberg-Witten equations
fit into the context of elliptic equations.

\section{Algebraic preliminaries}\label{alg:prem}

\textbf{Spin representation in dimension three.} Let $(V,g)$ be
an oriented three dimensional Euclidean vector space. The complex
Clifford algebra $\clc(V)$ is isomorphic to $M_2(\C)\oplus
M_2(\C)$, where $M_2(\C)$ denotes the ring of $(2\times
2)$-matrices. If $(e_1,e_2,e_3)$ is an oriented orthonormal basis
of $V$, then this isomorphism has an explicit description, which
is given by its action on this basis via
\[
\clc(V) \to M_2(\C)\oplus M_2(\C),\quad e_j\mapsto
\begin{pmatrix} i\gs_j &0\\ 0 & -i\gs_j \end{pmatrix}.
\]
Here, $\gs_j$ denote the Pauli matrices
\[
\gs_1 = \begin{pmatrix} 0 &1\\ 1 &0 \end{pmatrix},\quad
\gs_2 = \begin{pmatrix} 0 &-i\\ i &0 \end{pmatrix},\quad \gs_3 =
\begin{pmatrix} 1 &0\\ 0 &-1 \end{pmatrix}.
\]
The two non-isomorphic irreducible representations of $\clc(V)$
on $\gD:=\C^2$ are given by $e_j\mapsto i\gs_j$ and
$e_j\mapsto -i\gs_j$ respectively. We require that the complex
volume element $\go^c=-e_1e_2e_3$ (cf. \eqref{compl:vol}) acts as
the identity on $\gD$, i.e., we fix the latter
representation\footnote{There is some ambiguity in the literature
but most authors consider this representation as the standard
one.}. In particular, Clifford multiplication takes the following
form:
\[
c:V \to \End(\gD),\quad c(e_j)=-i\gs_j.
\]
It is skew Hermitian with respect to the standard metric
$\scalar{\ }{\ }$ on $\gD$.\\

\noindent\textbf{The quadratic map.} Let $\psi\in \gD$ be a
spinor. We define a linear map $q(\psi): V\to \C$ by letting
$q(\psi)(v):=-\lfrac{1}{2}\Scalar{c(v)\psi}{\psi}$. With respect
to an orthonormal basis:
\begin{equation}\label{q}\index{>@$q(\psi)$,
quadratic map} \fbox{$\displaystyle q(\psi)=
-\lfrac{1}{2}\Scalar{c(e_j)\psi}{\psi}e^j$}\;,
\end{equation}
where $(e^1,e^2,e^3)$ denotes the dual basis and we take the
sum\footnote{When using coordinates we shall always use the
Einstein convention, i.e., we sum over all indices appearing
twice.} over all $j$. Since $c(e_j)$ is skew Hermitian, $q(\psi)$
is a purely imaginary valued co-vector. We thus obtain a
quadratic map
\[
q:\gD\to V^*\otimes_\R i\R=:iV^*.
\]
Polarization gives the associated $\R$-bilinear map
\begin{equation}\label{b^q}
q(\psi,\gf)=\lfrac{1}{4}\big(q(\psi+\gf)-q(\psi-\gf)\big)=
-\lfrac{1}{2}i \Im \Scalar{c(e_j)\psi}{\gf}e^j.
\end{equation}

Clifford multiplication extends to $iV^*$ via action on the
co-vector part, i.e.,
\[
c(i\ga)\psi:= ic(v_\ga)\psi,
\]
where $\psi\in\gD$, $\ga\in V^*$, and $v_\ga$ denotes the metric
dual of $\ga$. Observe that Clifford multiplication with
imaginary valued co-vectors is Hermitian and trace-free. In fact,
we have
\begin{lemma}
Clifford multiplication is an isomorphism of $\R$ vector spaces
\[
c:iV^*\to\setdef{T\in\End(\gD)}{T\text{ Hermitian}, \Tr T=0}.
\]
\end{lemma}
\begin{proof}
The map is injective because $c(i\ga)^2=- c(v_\ga)^2 = |v_\ga|^2
\id$. Since the $\R$ vector space of trace-free and Hermitian
endomorphisms on $\gD$ is three dimensional, the result follows.
\end{proof}
\begin{remark*}
The trace-free and Hermitian endomorphism given by Clifford
multiplication with $q(\psi)$ is given by
\begin{equation}\label{q:alt}
\fbox{$\displaystyle c\big(q(\psi)\big)=\psi^*\otimes \psi
-\lfrac{1}{2}|\psi|^2\id$}\;,
\end{equation}
which means\footnote{Observe that we use the convention that
Hermitian metrics are complex linear in the first entry.} that
$c(q(\psi))\gf=\scalar{\gf}{\psi}\psi-\lfrac{1}{2}|\psi|^2 \gf$
for any spinor $\gf$. With respect to any unitary basis of $\gD$,
this endomorphism has the matrix description
\[
c\big(q(\psi)\big) =  \frac{1}{2}\begin{pmatrix}
|\ga|^2-|\gb|^2 & 2\ga\Bar \gb \\
2\Bar \ga \gb & |\gb|^2 -|\ga|^2
\end{pmatrix},\qquad  \psi=\begin{pmatrix}\ga \\ \gb
\end{pmatrix}.
\]
As we shall not use this description, we skip the proof of
\eqref{q:alt}. With the explicit representation given by the
Pauli matrices the involved computations are rather simple. The
description \eqref{q:alt} is more satisfactory than definition
\eqref{q} because it is invariant of any explicit representation
and can easily be carried over to other dimensions. However, the
definition we gave is much more convenient for explicit
calculations.
\end{remark*}

We endow the $\R$ vector space $iV^*$ with the scalar product
induced by $g$, i.e., we let
\[
\scalar{i\ga}{i\gb}_g:=g(v_\ga,v_\gb),
\]
where $\ga,\gb\in V^*$ with metric duals $v_\ga,v_\gb$. Later on,
we shall drop the subscript $g$ for notational convenience.
However, for the time being, we keep it to distinguish $\scalar{\
}{\ }_g$ from the Hermitian scalar product $\scalar{\ }{\ }$ on
$\gD$.

\begin{prop}\label{q:prop}
For all spinors $\psi$ and $\gf$ the following holds:
\begin{enumerate}
\item $\Scalar{a}{q(\psi,\gf)}_g=\lfrac{1}{2}\Re\Scalar{c(a)\psi}{\gf}$
for any $a\in iV^*$.
\item $\big|q(\psi,\gf)\big|_g^2=
\lfrac{1}{4}\Big(|\psi|^2|\gf|^2-\big(\Re\scalar{i\psi}{\gf}\big)^2\Big)$.
In particular, $|q(\psi)|_g = \lfrac{1}{2}|\psi|^2$.
\item If $\psi\neq 0$, then $\gf\in\ker q(\psi,.)$
if and only if $\gf$ is a multiple of $\psi$ by an imaginary
number.
\end{enumerate}
\end{prop}

\begin{proof} Let $(e_1,e_2,e_3)$ be an oriented orthonormal basis of
$(V,g)$.
\begin{enumerate}
\item  We write $a=i \ga_je^j$ with $\ga_j\in\R$. Then
\begin{align*}
\Scalar{a}{q(\psi,\gf)}_g &= -\lfrac{1}{2}\SCalar{i\ga_je^j}{i
\Im\Scalar{c(e_k)\psi}{\gf}e^k}_g
\\
&=-\lfrac{1}{2}\ \ga_j \Im\Scalar{c(e_k)\psi}{\gf}\gd^{jk}
= - \lfrac{1}{2}\Im\Scalar{c(\ga_je^j)\psi}{\gf} \\
&=-\lfrac{1}{2}\Im\Scalar{-i c(i\ga_je^j)\psi}{\gf}
=\lfrac{1}{2}\Re\Scalar{c(a)\psi}{\gf}.
\end{align*}
\item Without loss of generality, we may assume that $\psi\neq 0$.
If we interpret $\gD$ as a 4-dimensional Euclidean vector space
with respect to the real scalar product $\Re \scalar{.}{.}$, then
the elements $i\psi$, $c(e_1)i\psi$, $c(e_2)i\psi$, $c(e_3)i\psi$
form an orthogonal basis of $\gD$. We thus have
\[
|\psi|^2|\gf|^2=\big(\Re\scalar{i\psi}{\gf}\big)^2+\sum_{j=1}^3
\big(\Re\scalar{c(e_j)i\psi}{\gf}\big)^2.
\]
This in mind, we conclude:
\begin{align*}
\big|q(\psi,\gf)\big|_g^2 &=\lfrac{1}{4} \sum_{j=1}^3
\big(\Im\Scalar{c(e_j)\psi}{\gf}\big)^2 = \lfrac{1}{4}\sum_{j=1}^3
\big(\Re\Scalar{c(e_j)i\psi}{\gf}\big)^2 \\
&=\lfrac{1}{4}\Big(|\gf|^2|\psi|^2 -
\big(\Re\scalar{i\psi}{\gf}\big)^2\Big).
\end{align*}
\item According to part (ii), we have
\[
q(\psi,\gf)=0\quad\Longleftrightarrow\quad |\psi||\gf| =
|\Re\scalar{i\psi}{\gf}|.
\]
It thus follows from the Cauchy-Schwarz inequality that
$q(\psi,\gf)=0$ if and only if
$\gf$ is a real multiple of $i\psi$.\qedhere\\
\end{enumerate}
\end{proof}

\noindent\textbf{Hodge-star-operator and wedge product.} We
recall that there is an isometry on $\gL^\bullet V^*$,
\[
*:\gL^k V^*\to \gL^{3-k}V^*,\quad k\in\{0,\ldots,3\},
\]
uniquely characterized by the property that
\[
\ga\wedge*\gb=\scalar{\ga}{\gb}dv_g,\quad \ga\in\gL^k
V^*,\gb\in\gL^{3-k}V^*.
\]
Here, $dv_g$\index{d@$dv_g$, oriented volume element} denotes the
oriented volume element of $V$. The fact that $V$ is three
dimensional implies that
\[
*^2=\id_{\gL^\bullet V^*}\,.
\]
We shall also need the Hodge-star-operator on $i\gL^\bullet V^*$.
This is because $i\R$ is the Lie algebra of $\U_1$ and in gauge
theory, Lie algebra valued forms play a decisive role. We adapt
the standard convention and extend $*$ complex linearly, i.e., we
let $* (i\ga) = i* \ga$, where $\ga\in\gL^\bullet V^*$. This is
not to be confused with regarding $i\gL^\bullet V^*$ as a subspace
of $\gL^\bullet V^*\otimes\C$ endowed with the complex anti-linear
Hodge-star-operator of complex differential geometry.

Interpreting $i\R$ as a Lie algebra, there is a canonical way of
defining a wedge product on $i\gL^\bullet V^*$. As $i\R$ is an
abelian Lie algebra, the so defined product would, however,
vanish. In contrast to the above, we will therefore use the wedge
product of $\gL^\bullet V^*\otimes\C$. This gives the possibility
to form the wedge product of an imaginary valued co-vector with a
real valued one. Note that for $\ga,\gb\in\gL^\bullet V^*$
\begin{equation}\label{Hodge:wedge}
i\ga\wedge i\gb = - \ga\wedge\gb.
\end{equation}
As a consequence of our conventions, $a\in i\gL^\bullet V^*$
satisfies
\[
a\wedge * a = -\scalar{a}{a}dv_g.
\]

\section{The Seiberg-Witten equations}\label{SW:eqn}

Let $(M,g)$ be a closed,\footnote{We use the convention that a
closed manifold is a compact and connected manifold without
boundary. Although the assumption about connectedness is usually
not standard, we include it here for simplicity of notation.}
oriented Riemannian 3-manifold. According to Proposition
\ref{3mfd=spinc}, $M$ admits a \spinc structure. Fixing
$\gs\in\text{spin}^c(M)$, we let $L(\gs)$ denote the associated
Hermitian line bundle, and let $S=S(\gs)$ be the fundamental
spinor bundle over $M$ associated to the $\Spinc$-bundle
$P_{\Spinc}(\gs)$ via the representation chosen in Section
\ref{alg:prem}. Then $S(\gs)$ is a Hermitian vector bundle of
rank 2 over $M$.

The quadratic map $q$ extends to a morphism $C^\infty(M,S)\to
i\gO^1(M)$. For later purposes we establish a necessary condition
for $q(\psi)$ to be co-closed.
\begin{prop}\label{dastq}
Let $A$ be an arbitrary Hermitian connection on $L(\gs)$, $\cD_A$
the associated \spinc Dirac operator. Then we have the
(pointwise) identity
\[
d^*q(\psi)=i\Im \scalar{\cD_A\psi}{\psi}.
\]
In particular, $q$ is co-closed whenever $\psi$ is a harmonic
spinor.
\end{prop}
\begin{proof}
At an arbitrary point $x_0$, we consider a normal\footnote{Recall
that this means $(\nabla_i e_j)(x_0)=0$.} frame $(e_1,e_2,e_3)$
with dual co-frame $(e^1,e^2,e^3)$. This implies that
\[
\nabla^A_j \big(c(e_i)\psi\big)(x_0)=
\big(c(e_i)\nabla^A_j\psi\big)(x_0).
\]
Using that in the case at hand, $d^*=- *d*$, that $A$ is
Hermitian and that $c(e_i)$ is skew Hermitian, we find that at
the point $x_0$
\begin{align*}
d^*q(\psi) &= -\lfrac{1}{2}d^*\scalar{c(e_i)\psi}{\psi}e^i
=\lfrac{1}{2}* d* \scalar{c(e_i)\psi}{\psi}e^i \\
&=\lfrac{1}{2}* \Big(\scalar{c(e_i)\nabla^A_j\psi}{\psi} +
\scalar{c(e_i)\psi}{\nabla^A_j\psi}\Big)e^j\wedge* e^i\\
&=\lfrac{1}{2} \Big(\scalar{c(e_i)\nabla^A_i\psi}{\psi} -
\scalar{\psi}{c(e_i)\nabla^A_i\psi}\Big) * (e^i\wedge* e^i) \\
&= \lfrac{1}{2} \Big(\scalar{\cD_A\psi}{\psi} -
\scalar{\psi}{\cD_A\psi}\Big)* dv_g \\
&=i\Im \scalar{\cD_A\psi}{\psi}.
\end{align*}
Note that we have also employed the local description of $\cD_A$
(cf. \eqref{dirac:def}) and the fact that $e^i\wedge*
e^j=\gd^{ij}dv_g$.
\end{proof}

Let $\cA(\gs)$ denote the affine space of Hermitian connections
on $L(\gs)$. We define the {\em configuration space}
as\index{>@$\cC(\gs)$, configuration space}\index{configuration
space}
\[
\cC=\cC(\gs):=C^\infty(M,S(\gs))\times\cA(\gs).
\]
Since $\cA(\gs)$ is an affine space modelled on
$C^\infty(M,iT^*M)$, the configuration space is also an affine
space which is modelled on $C^\infty\big(M,S\oplus iT^*M\big)$. To
make formul{\ae} clearer and notation shorter we define
\[
E=E(\gs):=S(\gs)\oplus iT^*M.
\]
The {\em group of gauge transformations} of a \spinc structure is
(cf. Definition \ref{gg})
\[
\cG:=C^\infty(M,U_1).
\]
Its natural operations on $C^\infty(M,S)$ and $\cA$ (cf.
\eqref{gg:on:spinor} and \eqref{gg:on:A}) induce an action on the
configuration space $\cC$, given by
\begin{equation}\label{gg:on:conf}
\begin{split}
\cG\times\cC(\gs)&\longrightarrow\cC(\gs), \\
\big(\gamma,(\psi,A)\big)&\longmapsto
\big(\gamma^{-1}\psi,A+2\gamma^{-1}d\gamma\big).
\end{split}
\end{equation}
We define the quotient of the configuration space with respect to
the $\cG$ action by\index{>@$\cB(\gs)$, quotient of $\cC$ by
$\cG$}
\[
\cB(\gs):=\cC(\gs)/\cG.
\]
The action of $\cG$ on $\cC$ lifts naturally to $C^\infty(M,E)$,
the tangent space of $\cC$ at an arbitrary point $(\psi,A)$, via
\begin{align}\label{gg:on:tangent}
\gamma\cdot(\gf,a):= \lfrac{d}{dt}\big|_{t=0}
\gamma\cdot(\psi+t\gf, A+ta)= (\gamma^{-1}\gf,a).
\end{align}
\begin{dfn}\index{configuration space!reducible configuration}
\index{configuration space!irreducible configuration}
$(\psi,A)\in\cC(\gs)$ is called {\em irreducible} if the
stabilizer $\cG_{(\psi,A)}$ of the $\cG$ action at the point
$(\psi,A)$ is trivial. Otherwise, it is called {\em reducible}.
The subset of irreducible configurations is denoted by
$\cC^*(\gs)$. and its quotient with respect to the $\cG$ action by
$\cB^*(\gs)$.
\end{dfn}

Let $\gamma\neq 1$ lie in the stabilizer $\cG_{(\psi,A)}$, i.e.,
$\gamma\cdot(\psi,A)=(\psi,A)$. Then $\gamma^{-1}\psi=\psi$ and
$2\gamma^{-1}d\gamma=0$. As $M$ is connected, we deduce the
following.
\begin{lemma}
A configuration $(\psi,A)$ is reducible if and only if
$\psi\equiv 0$. In this case, the stabilizer $\cG_{(\psi,A)}$
consists of the constant maps $M\to\U_1$.\\
\end{lemma}

\noindent\textbf{The moduli space.} We can now formulate the
Seiberg-Witten equations.
\begin{dfn} Let $M$ be a closed, oriented Riemannian
3-manifold with \spinc structure $\gs$. For each gauge field $A$,
we let $\cD_A$ denote the \spinc Dirac operator associated to
$A$, and $F_A\in i\gO^2(M)$ the curvature 2-form of $A$. Then a
configuration $(\psi,A)$ is called a {\em Seiberg-Witten
monopole} if it solves the equations
\begin{equation}\label{eqn}\index{Seiberg-Witten
equations}\index{configuration space!Seiberg-Witten monopole}
\fbox{$\begin{array}{rcl}
\cD_A\psi&=&0 \\
* F_A &= &\lfrac{1}{2}q(\psi)
\end{array}$}\;.\end{equation}
\end{dfn}
We can interpret Seiberg-Witten monopoles as the zeros of a
vector field on $\cC$. We define the {\em Seiberg-Witten
map}\index{>@$\SW$, Seiberg-Witten map}
\[
\SW:\cC(\gs)\to C^\infty(M,E(\gs)),\quad (\psi,A)\mapsto
(\cD_A\psi,\lfrac{1}{2}q(\psi)-* F_A).
\]
\begin{lemma}\label{SW:equiv}
The map $\SW:\cC\to C^\infty(M,E)$ is equivariant with respect to
the $\cG$-actions \eqref{gg:on:conf} and \eqref{gg:on:tangent} on
$\cC$ and $C^\infty(M,E)$ respectively.
\end{lemma}
\begin{proof}
Let $(\psi,A)$ be an arbitrary configuration, and let
$\gamma\in\cG$. Then Lemma \ref{A+a} implies
\begin{align*}
\cD_{(A+2\gamma^{-1}d\gamma)}(\gamma^{-1}\psi) &=
\cD_A(\gamma^{-1}\psi)+c(\gamma^{-1}d\gamma)\gamma^{-1}\psi\\
&=\gamma^{-1}\cD_A\psi+c(d\gamma^{-1})\psi+c(\gamma^{-2}d\gamma)\psi=
\gamma^{-1}\cD_A\psi
\end{align*}
because $d\gamma^{-1}=-\gamma^{-2}d\gamma$. Furthermore, since
$\gamma^{-1}d\gamma$ is closed,
\[
\lfrac{1}{2}q(\gamma^{-1}\psi) - * F_{A+2\gamma^{-1}d\gamma}
=\lfrac{1}{2}q(\gamma^{-1}\psi)-* (F_A+2d(\gamma^{-1}d\gamma))
=\lfrac{1}{2}q(\psi)-* F_A\,.
\]
Note that in the last equality we have used that
$q(\gl\psi)=q(\psi)$ for every $\gl\in\C$ with $|\gl|=1$.
\end{proof}

In particular, the set of Seiberg-Witten monopoles is a
$\cG$-invariant subset of the configuration space $\cC$ and thus,
the set
\[
\cM(\gs):=\SW^{-1}(0)/\cG \subset\cB(\gs)
\]
is a well-defined subset of the set of gauge equivalence classes.
$\cM$ is called the {\em Seiberg-Witten moduli
space}\index{>@$\cM(\gs)$, moduli space}\index{moduli space}
associated to the \spinc structure $\gs$ on the Riemannian
manifold $(M,g)$. The subset of $\cM(\gs)$ given by gauge
equivalence classes of irreducible monopoles is denoted by
$\cM^*(\gs)\subset\cB^*(\gs)$.
\begin{remark*}\index{>@$\cC^*,\cB^*,\cM^*$, irreducible parts}
The structure of the moduli space $\cM(\gs)$ depends heavily on
the particular choice of $g$. Whenever we want to stress this
dependence on the metric, we shall write $\cM(\gs;g)$ instead of
$\cM(\gs)$.\\
\end{remark*}

\noindent\textbf{Variational aspects of the Seiberg-Witten
equations.} We will now present a special feature of three
dimensional Seiberg-Witten theory that does not carry over to the
four dimensional case: We can interpret Seiberg-Witten monopoles
as critical points of a Chern-Simons like functional. This was
firstly observed by Kronheimer and Mrowka in \cite{KroMro:Gen}.

\begin{dfn}\label{chern-simons}\index{>@$\csd$,
Chern-Simons-Dirac functional}\index{Chern-Simons-Dirac
functional} For every configuration $(\psi,A)$, we define the {\em
Chern-Simons-Dirac functional}
\[
\fbox{$\displaystyle
\csd(\psi,A):=\frac{1}{2}\int_M\Big(\scalar{\psi}{\cD_A\psi}dv_g
+ (A-A_0)\wedge (F_A+F_{A_0})\Big)$}\;,
\]
where $A_0$ is an arbitrary fixed connection on $L(\gs)$. Observe
that $\csd$ is $\R$-valued since $\cD_A$ is formally self-adjoint.
\end{dfn}

\begin{remark*}
Note that the definition of the Chern-Simons-Dirac functional
depends on the choice of $A_0$ so that we should write
$\csd_{A_0}$. However, if $A_1$ is another fixed connection on
$L(\gs)$, then a short calculation shows that
\[
\csd_{A_1}-\csd_{A_0} = -\int_M F_{A_1}\wedge \big(A_1-A_0\big),
\]
i.e., the value of $\csd$ is well-defined up to an additive
constant. Since we are only interested in critical points of
$\csd$ we shall not bother stressing the dependence on $A_0$.
\end{remark*}

We endow the $\R$ vector space $C^\infty(M,E)$ with the pre
Hilbert scalar product induced by the scalar products $\scalar{\
}{\ }$ on $S$ and $\scalar{\ }{\ }_g$ on $iT^*M$, i.e., for
$(\gf,a)$, $(\gf',a')\in C^\infty(M,E)$ we let
\begin{align*}
\LScalar{(\gf,a)}{(\gf',a')}&:=
\Re\Lscalar{\gf}{\gf'}+\Lscalar{a}{a'}\\
&=\int_M\Re\scalar{\gf}{\gf'}dv_g+\int_M\scalar{a}{a'}_g dv_g.
\end{align*}

\begin{prop}\label{F:prop}
The map $\SW:\cC\to C^\infty(M,E)$ is the gradient of the
Chern-Simons-Dirac functional with respect to the $L^2$ scalar
product, i.e., for $(\psi,A)\in\cC$ and $(\gf,a)\in
C^\infty(M,E)$:
\[
\lfrac{d}{dt}\big|_{t=0}\csd(\psi+t\gf,A+ta)=
\LScalar{\SW(\psi,A)}{(\gf,a)}.
\]
The Hessian of $\csd$ at $(\psi,A)$ is the formally self-adjoint
first-order differential operator
\[
F_{(\psi,A)}:=D_{(\psi,A)}\SW:C^\infty(M,E)\to C^\infty(M,E),
\]
given by
\begin{equation}\label{F}\index{>@$F_{(\psi,A)}$, differential of $\SW$}
F_{(\psi,A)}(\gf,a)=\big(\cD_A\gf+\lfrac{1}{2}c(a)\psi,q(\psi,\gf)-*
da\big).
\end{equation}
\end{prop}

\begin{proof}
Let $(\gf,a)\in C^\infty(M,E)$. Then
\begin{align*}
\lfrac{d}{dt}\big|_{t=0}&\csd(\psi+t\gf,A+ta) \\
=&\lfrac{d}{dt}\big|_{t=0}\frac{1}{2}
\int_M\Big(\Re\Scalar{\psi+t\gf}{\cD_A(\psi+t\gf)+
\lfrac{1}{2}c(ta)(\psi+t\gf)}dv_g\\
&\qquad\qquad\quad +(A+ta-A_0)\wedge(F_A + tda +
F_{A_0})\Big)\\
=&\frac{1}{2}
\int_M\Big(\Re\Scalar{\gf}{\cD_A\psi}dv_g+\Re\Scalar{\psi}{\cD_A\gf+
\lfrac{1}{2}c(a)\psi}dv_g\\
&\qquad\qquad\quad +a\wedge(F_A+ F_{A_0}) +(A-A_0)\wedge da\Big)\\
=&\int_M\Re\Scalar{\gf}{\cD_A\psi}dv_g + \frac12
\int_M\Re\Scalar{\psi}{\lfrac{1}{2}c(a)\psi}dv_g\\
&\qquad\qquad\quad -\frac12 \int_M\Big( \Scalar{a}{*(F_A+
F_{A_0})}+\Scalar{A-A_0}{*da}\Big)dv_g, \intertext{where we have
employed formal self-adjointness of $\cD_A$ and formula
\eqref{Hodge:wedge} in the last line. Formal self-adjointness of
$*d$ on 1-forms and Proposition \ref{q:prop} show that this
equals} \Re\LScalar{\cD_A\psi&}{\gf}
+\LScalar{\lfrac{1}{2}q(\psi)}{a}
-\frac12\LSCalar{*(F_A+F_{A_0})+*d(A-A_0)}{a}\\
=&\Re\LScalar{\cD_A\psi}{\gf} +
\LScalar{\lfrac{1}{2}q(\psi)-*F_A}{a}\\
=&\LScalar{\SW(\psi,A)}{(\gf,a)}.
\end{align*}
The Hessian $F_{(\psi,A)}$ at $(\psi,A)$ is computed as follows:
\begin{align*}
D_{(\psi,A)}\SW(\gf,a)&=\lfrac{d}{dt}\big|_{t=0}\SW(\psi+t\gf,A+ta)\\
&=\lfrac{d}{dt}\big|_{t=0}\big(\cD_{A+ta}(\psi+t\gf),\ \lfrac12
q(\psi+t\gf)-*F_{A+ta}\big)\\
&=\big(\cD_A\gf+\lfrac12 c(a)\psi,\ q(\psi,\gf)-*da\big).
\end{align*}
Observe that we have used that the differential of $q$ at the
point $\psi$ is given by
\[
D_\psi q(\gf)= 2 q(\psi,\gf).
\]
As it is the Hessian of $\csd$, formal self-adjointness of
$F_{(\psi,A)}$ is immediate.
\end{proof}

Our main interest lies in gauge equivalence classes of critical
points of the Chern-Simons-Dirac functional. However, one major
observation is that $\csd:\cC\to \R$ is \emph{not} gauge
invariant. In fact,
\begin{equation}
\csd(\gamma\cdot(\psi,A))=\csd(\psi,A)-8\pi^2
\int_M\big[\lfrac{1}{2\pi i}\gamma^{-1}d\gamma\big]\wedge c(\gs),
\end{equation}
where $c(\gs)$ denotes the canonical class of the \spinc
structure $\gs$ (cf. Definition \ref{canline}).
\begin{proof}
Since $d(\gamma^{-1}d\gamma)=0$ and
$d\gamma^{-1}=-\gamma^{-2}d\gamma$ we have
\begin{align*}
\csd(\gamma\cdot(\psi,A))=&\frac{1}{2}\int_M\Big(
\SCalar{\gamma^{-1}\psi}{\cD_A(\gamma^{-1}\psi)+
\lfrac{1}{2}c(2\gamma^{-1}d\gamma)(\gamma^{-1}\psi)}dv_g\\
&\qquad\qquad\quad
+(A+2\gamma^{-1}d\gamma-A_0)\wedge(F_A+F_{A_0})\Big)\\
=&\frac{1}{2}\int_M\Big(\scalar{\gamma^{-1}\psi}{
\gamma^{-1}\cD_A\psi}dv_g
+ (A-A_0)\wedge (F_A+F_{A_0})\Big)\\
&\qquad\qquad\quad +\int_M \gamma^{-1}d\gamma\wedge(F_A + F_{A_0})
\intertext{Since $\gamma^{-1}$ acts unitary and
$[F_A]=[F_{A_0}]=2\pi i c(\gs)$,} \csd(\gamma\cdot(\psi,A)) =&
\csd(\psi,A) - 8\pi^2 \int_M\big[\lfrac{1}{2\pi i}
\gamma^{-1}d\gamma\big]\wedge c(\gs).\qedhere
\end{align*}
\end{proof}
According to \eqref{derham}, the class $[\lfrac{1}{2\pi
i}\gamma^{-1}d\gamma]$ belongs to $H^1_{dR}(M;\Z)$ which is the
image of $H^1(M;\Z)$ in $H^1_{dR}(M;\R)$. The same applies to
$c(\gs)$. Therefore, the integral over $[\lfrac{1}{2\pi
i}\gamma^{-1}d\gamma]\wedge c(\gs)$ is integer valued and
vanishes for all $\gamma$ only if $c(\gs)$ is a torsion class.

Hence in the general case, the Chern-Simons-Dirac functional
descends to $\cB$ as a function with values in $\R/(8\pi^2\Z)$.
Whenever we want to refer to this phenomenon, we shall usually
view $\csd$ as a function
\[
\csd:\cB\to S^1
\]

We now want to linearize the action of the group of gauge
transformations.\footnote{In the next chapter we shall see that
if we consider suitable Sobolev completions, the group of gauge
transformations is a Banach Lie group which acts smoothly on the
configuration space. Hence, taking the differential of the action
is meaningful. For the time being, we will perform the involved
calculations only formally.} As the Lie algebra of $\U_1$ is
$i\R$, infinitesimal gauge transformations are smooth maps $M\to
i\R$.

\begin{prop}\label{G:adj}
If $(\psi,A)\in \cC$, the ``derivative" of the action map
\[
\cG\to\cC,\quad \gamma\mapsto \gamma\cdot(\psi,A)
\]
at $\gamma=1$ is the first-order differential operator
\[
G_{(\psi,A)}: C^\infty(M,i\R)\to C^\infty(M,E),
\]
given by
\begin{equation}\label{G}\index{>@$G_{(\psi,A)}$, the action map}
G_{(\psi,a)}(f):=\lfrac{d}{dt}\big|_{t=0}\exp(tf)\cdot(\psi,A)=
(-f\psi,2df).
\end{equation}
The formal adjoint of $G_{(\psi,A)}$ with respect to the $L^2$
scalar products is
\begin{equation}\label{Gt}\index{>@$G_{(\psi,A)}^*$, formal
adjoint of $G_{(\psi,A)}$}
G_{(\psi,A)}^*(\gf,a)=2d^*a-i\Im\scalar{\gf}{\psi}.
\end{equation}
\end{prop}
\begin{proof}
Let $f\in C^\infty(M,i\R)$. Then
\begin{align*}
\lfrac{d}{dt}\big|_{t=0}\exp(tf)\cdot(\psi,A)&=
\lfrac{d}{dt}\big|_{t=0}\big(\exp(-tf)\psi,\
A+2\exp(-tf)d\exp(tf)\big)\\
&=\big(-f\psi,\ 2df\big).
\end{align*}
To calculate the formal adjoint of $G_{(\psi,A)}$, we now let
$(\gf,a)\in C^\infty(M,E)$. Then, recalling that $f$ is imaginary
valued, we find
\begin{align*}
\LScalar{(\gf,a)}{G_{(\psi,A)}(f)}&=  \Re\Lscalar{\gf}{-f\psi}+
\Lscalar{a}{2df} \\
&=\LScalar{2d^*a-i\Im\scalar{\gf}{\psi}}{f}.\qedhere\\
\end{align*}
\end{proof}

\noindent\textbf{The elliptic complex.} If $(\psi,A)$ is a
monopole, then $\SW(\gamma\cdot(\psi,A))=0$ for every
$\gamma\in\cG$. Taking derivatives at $\gamma=1$ yields
$F_{(\psi,A)}\circ G_{(\psi,A)}=0$, and we obtain a complex
\begin{equation}\label{complex}
0\to C^\infty(M,i\R)\overset{G}{\longrightarrow}
C^\infty(M,E\oplus i\R )\overset{F+
G}{\longrightarrow}C^\infty(M,E)\to 0,
\end{equation}
where we denote the map $f\mapsto (G(f),0)$ a little inaccurately
also by $G$. This is a complex of first-order differential
operators. Associated to \eqref{complex} there is the rolled-up
operator
\begin{equation}\label{T}
\begin{split}
T_{(\psi,A)}:=\big(F_{(\psi,A)}\oplus
G_{(\psi,A)},&G_{(\psi,A)}^*\big):\\ &C^\infty(M,E\oplus i\R)\to
C^\infty(M,E\oplus i\R).
\end{split}
\end{equation}
This operator is well-defined, irrespective of whether $(\psi,A)$
is a SW-monopole or not. Explicitly, it is given by
\begin{equation}\label{T:explicit}\index{>@$T_{(\psi,A)}$}
T_{(\psi,A)}\begin{pmatrix}\gf \\a \\ f\end{pmatrix} =
\begin{pmatrix} \cD_A &0 &0 \\
                0 &-* d &2d \\
                0 &2d^* &0 \end{pmatrix}
\begin{pmatrix}\gf \\a \\ f\end{pmatrix} +
\begin{pmatrix} \lfrac{1}{2}c(a)\psi -f\psi \\
                q(\psi,\gf) \\
                -i\Im\scalar{\gf}{\psi}\end{pmatrix}.
\end{equation}
\begin{prop}\label{T:ell}
For each configuration $(\psi,A)$, the differential operator
$T_{(\psi,A)}$ is elliptic and formally self-adjoint. Hence, if
$(\psi,A)$ is a monopole, the complex \eqref{complex} is
elliptic, i.e., the associated sequence of principal symbols is
exact.
\end{prop}
\begin{proof}
To prove ellipticity, we only have to consider the first-order
term of $T_{(\psi,A)}$. According to the explicit description
\eqref{T:explicit}, this term splits into the elliptic operator
$\cD_A:C^\infty(M,S)\to C^\infty(M,S)$ and the operator
\begin{equation}\label{odd:signature}
\begin{pmatrix} -* d &2d \\ 2d^* &0 \end{pmatrix}:
C^\infty(M,iT^*M\oplus i\R)\to C^\infty(M,iT^*M\oplus i\R).
\end{equation}
To prove ellipticity of the latter operator we use the rolled-up
operator of the deRham complex
\[
0\to \gO^0(M)\overset{d}{\longrightarrow}\gO^1(M)
\overset{d}{\longrightarrow}\gO^2(M)
\overset{d}{\longrightarrow}\gO^3(M)\to 0.
\]
This is an elliptic operator given by
\[
\begin{pmatrix} d^* &d \\ d &0 \end{pmatrix}= \begin{pmatrix} *d*
&d \\ d &0 \end{pmatrix}:\gO^2\oplus \gO^0\to \gO^1\oplus \gO^3.
\]
Here, we are using the explicit description of $d^*$ on
3-manifolds. The operator \eqref{odd:signature} is elliptic
because it can be obtained from the above elliptic operator by a
combination with bundle isomorphisms in the following way:
\[
\gO^1\oplus \gO^0\overset{\big(\begin{smallmatrix} * &0 \\ 0
&-2\id
\end{smallmatrix}\big)}{\longrightarrow}\gO^2\oplus\gO^0
\overset{\big(\begin{smallmatrix} *d* &d \\ d &0
\end{smallmatrix}\big)}{\longrightarrow}\gO^1\oplus\gO^3
\overset{\big(\begin{smallmatrix} -\id &0\\ 0 &-2*
\end{smallmatrix}\big)}{\longrightarrow}\gO^1\oplus\gO^0.
\]
Therefore, the first-order term of $T_{(\psi,A)}$ is the direct
sum of two elliptic operators, which shows that $T_{(\psi,A)}$ is
also elliptic. The assertion about formal self-adjointness is an
immediate consequence of formal self-adjointness of
$F_{(\psi,A)}$.
\end{proof}
\begin{remark*}
In four dimensional Seiberg-Witten theory, the geometric origin
of the complex corresponding to \eqref{complex} is more
transparent since one does not have to add the term $G$ to the
operator $F$ in order to have ellipticity. Later, when we are
going to study the local structure of the moduli space in Section
\ref{moduli}.\ref{loc:struc}, the nature of \eqref{complex} will
become clearer.
\end{remark*}

\cleardoublepage

%% file: chap_II.tex
\chapter{The Structure of the Moduli Space}\label{moduli}

To understand the structure of the Seiberg-Witten moduli space, we
have to endow the configuration space and the group of gauge
transformations with suitable topologies. As the objects of our
study are solutions of a system of partial differential
equations, it is natural to do this via Sobolev spaces. We can
then exploit the powerful machinery provided by the theory of
elliptic equations on compact manifolds to prove remarkable
topological properties. The material we need is summarized in
Appendix \ref{app:ell}.

The organization of this chapter is as follows. Section
\ref{funct:setup} introduces the functional analytic setting. It
turns out that $\cG$ can be made a Banach Lie group acting
smoothly on $\cC$. In Section \ref{top:quot}, we analyse the
action of $\cG$ on $\cC$ using the situation occurring for proper
actions of finite dimensional Lie groups as a guideline. We
establish the Hausdorff property of the quotient and a slice
theorem which shall provide the irreducible part of $\cB$ with
the structure of a Banach manifold.

Moreover, we will see in Section \ref{compact} that the moduli
space $\cM$ is a sequentially compact subset of $\cB$. This
contrasts the corresponding result in instanton theory, where the
moduli space has to be compactified through a complicated
procedure (cf. Donaldson \& Kronheimer \cite{DK} or Freed \&
Uhlenbeck \cite{FU}). This is one of the reasons why
Seiberg-Witten theory is considered as a simplification of
Donaldson theory. Making use of the implicit function theorem we
will then observe in Section \ref{loc:struc} that the irreducible
part of the moduli space is usually expected to be a submanifold
of dimension zero. Therefore, in the absence of reducible
monopoles, $\cM$ consists solely of finitely many points.
Finally, we shall use this observation to define an orientation of
the moduli space in Section \ref{count}. There is a general
procedure, introduced by Donaldson in \cite{Don:Orient}, to endow
gauge theoretical moduli spaces with an orientation. This applies
to Seiberg-Witten theory as well. Using the Chern-Simons-Dirac
functional, we shall then interpret the signed count of monopoles
obtained in this way as an Euler characteristic associated to the
irreducible part of the quotient $\cB$.

The presentation of the topological aspects we are giving is an
imitation of the corresponding results in four dimensional
Seiberg-Witten theory as they are presented, for example, in the
monographs by J. Morgan \cite{Mo} and L.I. Nicolaescu
\cite{Nic:SW}. The discussion of the local structure and the
orientation of the moduli space follows the work of C.H. Taubes
\cite{Tau:CI} and W. Chen \cite{Che:CI}.

\section{Functional analytic set-up}\label{funct:setup}

Suppose that $M$ is a closed, oriented Riemannian 3-manifold with
\spinc structure $\gs$. Let $L_1^2(M,S)$ denote the sections of
the spinor bundle $S(\gs)$ which are of Sobolev class 1, and let
$\cA_1(\gs)$ be the affine Hilbert space of $L_1^2$-connections
on $L(\gs)$, i.e.,\index{=@$\cA(\gs)$, space of gauge fields}
\[
\cA_1(\gs):=\setdef{A_0+a}{a\in L^2_1(M,iT^*M)},
\]
where $A_0$ is a fixed $C^\infty$ gauge field. Because of the
affine structure of $\cA$, the definition of $\cA_1(\gs)$ is
independent of the particular choice of $A_0$. Then the
configuration space\index{>@$\cC(\gs)$, configuration
space}\index{configuration space}
\[
\cC_1(\gs):=L^2_1(M,S(\gs))\times \cA_1(\gs)
\]
is a real affine Hilbert space modelled on $L^2_1(M,S\oplus iT^*
M)$. As the group of gauge transformations we now take
\index{>@$\cG$, group of gauge transformations}
\[
\cG_2:=L^2_2(M,\U_1),
\]
which is the set of functions $\gamma:M\to \C$ of Sobolev class 2
that take values in $\U_1$. Note that this definition makes sense
since on 3-manifolds, $L^2_2$ embeds in $C^0$ (cf. Theorem
\ref{sobolev}). The moduli space $\cM$ carries the topology
induced by the quotient topology on $\cB_1=\cC_1/\cG_2$.

\begin{remark*}
As we shall see below, the Sobolev orders we are choosing are
motivated by the consideration in Example \ref{sob:mult:n=3}: It
must be guaranteed that there are continuous Sobolev
multiplications
\[
L^2_k\times L^2_k\to L^2_k\quad\text{ and }\quad L^2_k\times
L^2_l\to L^2_l,
\]
where $k$ and $l$ are the Sobolev orders associated to gauge
transformations and configurations respectively. Since this
depends on the dimension of the underlying manifold, one has to
choose $k$ and $l$ differently in four dimensional Seiberg-Witten
theory. However, some authors assume higher Sobolev orders in the
three dimensional case as well which simplifies some proofs.
Although we will see in Section \ref{compact} that the structure
of the moduli space does not depend on the particular choice, we
do not avoid the slightly bigger effort of working with the lowest
possible Sobolev orders as this will allow some of the involved
differential operators to be defined on their natural domains.\\
\end{remark*}

\noindent\textbf{Differentiability properties.} We now want to
establish the basic set-up for performing calculus in the given
framework.

\begin{lemma}\label{q:diff}
The quadratic map $\psi\mapsto q(\psi)$ induces a smooth map
\[
q:L^2_1(M,S)\to L^2(M,iT^*M).
\]
If $k\ge 2$, we obtain a smooth map $q:L^2_k(M,S)\to
L^2_k(M,iT^*M)$.
\end{lemma}
\begin{proof}
Since $M$ is three dimensional, Proposition \ref{sob:mult}
guarantees that there is a bounded Sobolev multiplication
$L^2_1\times L^2_1\to L^2$. The first assertion then immediately
follows because $q(\psi)$ is a quadratic expression in $\psi$.
For $k\ge 2$ we deduce from Example \ref{sob:mult:n=3} that there
is a bounded Sobolev multiplication $L^2_k \times L^2_k\to L^2_k$
associated to any bilinear map. This yields smoothness of
$q:L^2_k(M,S)\to L^2_k(M,iT^*M)$ in this case as well.
\end{proof}

\begin{prop}\label{cs:diff}
The Chern-Simons-Dirac functional $\csd:\cC_1\to\R$ is a smooth
map. Its $L^2$-gradient $\SW$ extends naturally to a smooth map
$\cC_1\to L^2(M,E)$. For any $(\psi,A)\in\cC_1$, the Hessian
$F_{(\psi,A)}$ defines a symmetric operator in $L^2(M,E)$ with
domain $L^2_1(M,E)$.
\end{prop}

\begin{proof}
Since $\int_M: L^1(M,\gL^3T^*M)\to \R$ is smooth, we have to
establish that the integrand of the Chern-Simons-Dirac functional
is a smooth map $\cC_1\to L^1(M,\gL^3T^*M)$. For this, we have to
prove first that
\[
\cC_1\to L^1(M,\R),\;(\psi,A)\mapsto \Re\scalar{\psi}{\cD_A\psi}
\]
is smooth. Smoothness of the multiplication $L^2\times L^2\to
L^1$ shows that it suffices to establish that for a fixed
$C^\infty$ gauge field $A_0$, the map
\[
L^2_1(M,E)\to L^2(M,S),\quad (\psi,a)\mapsto
\cD_{A_0}\psi+\lfrac{1}{2}c(a)\psi\,.
\]
is smooth. $\cD_{A_0}$ induces a bounded linear---and
consequently a smooth---map $L^2_1(M,S)\to L^2(M,S)$. As we have
already seen before, Proposition \ref{sob:mult} shows that the
second term also yields a smooth map $L^2_1(M,E)\to L^2(M,E)$.
Secondly, we have to show that
\[
\cA_1\to L^1(M,\gL^3T^*M),\; A\mapsto (A-A_0)\wedge(F_A+F_{A_0}),
\]
is smooth. Since there is a Sobolev multiplication $L_1^2\times
L^2\to L^1$, this easily follows from smoothness of
\[
L^2_1(M,iT^*M)\to L^2(M,\gL^2iT^*M),\quad a\mapsto F_{A_0} + da.
\]
From the above considerations and Lemma \ref{q:diff}, one also
concludes that the $L^2$ gradient $\SW$ extends to a smooth map
$\cC_1\to L^2(M,E)$. Since the computations of Proposition
\ref{F:prop} in the last chapter remain valid for $L^2_1$
configurations, the extension of $\SW$ is the $L^2$ gradient of
$\csd:\cC_1\to\R$.\footnote{However, $\SW$ is not a gradient
vector field with respect to the natural metric on $\cC_1$, given
by the $L^2_1$ scalar product on $L^2_1(M,E)$.} Moreover, the
differential of $\SW$ at $(\psi,A)\in\cC_1$ is again given by
formula \eqref{F}. Since this is a formally self-adjoint
first-order differential operator, it extends to an unbounded
symmetric operator in $L^2(M,E)$ with natural domain
$L_1^2(M,E)$. Notice that the zero-order term of $F_{(\psi,A)}$
is possibly non-smooth. Then, however, it is easy to check that
the multiplication rule $L_1^2\times L_1^2\to L^2$ guarantees that
$F_{(\psi,A)}$ is still well-defined as a bounded operator
$L_1^2(M,E)\to L^2(M,E)$.
\end{proof}

\begin{prop}\label{action:diff}\index{gauge transformations|(}
The group of gauge transformations $\cG_2$ is a Banach Lie group
modelled on $L^2_2(M,i\R)$, and its action on $\cC_1$ is smooth.
\end{prop}
\begin{proof}
Sobolev multiplication \ref{sob:mult:n=3} guarantees that
multiplication of complex functions on $M$ extends to a smooth
bilinear map
\[
L^2_2(M,\C)\times L^2_2(M,\C)\to L^2_2(M,\C).
\]
Like in finite dimensional Lie group theory, the implicit
function theorem---which is also valid in Banach
spaces---guarantees that taking the inverse of an invertible
function $f\in L^2_2(M,\C)$ is a smooth map, defined on the subset
\[
L^2_2(M,\C^*):=\bigsetdef{f\in L^2_2(M,\C)}{\forall_{x\in
M}:\,f(x)\in\C^*}.
\]
Note that this set is well-defined since there is a continuous
embedding of $L^2_2(M,\C)$ in $C^0(M,\C)$. Moreover, with respect
to this embedding, $L^2_2(M,\C^*)$ is the preimage of the open
subset $C^0(M,\C^*)$ of $C^0(M,\C)$ and is therefore open in
$L^2_2(M,\C)$. Hence, $L^2_2(M,\C^*)$ is a Banach Lie group.

We now want to show that it contains $\cG_2=L^2_2(M,\U_1)$ as a
closed Lie subgroup. For this it suffices to establish that
$\cG_2$ is a closed submanifold of $L^2_2(M,\C^*)$. This shall be
done by constructing local charts of $L^2_2(M,\C^*)$ mapping
$\cG_2$ to the real subspace $L^2_2(M,i\R)$ of $L^2_2(M,\C)$.
Note that this subspace is closed---again because of the
embedding of $L^2_2$ in $C^0$.

Let $\C^-:=\C\setminus (-\infty,0]$. Then the set
\[
L^2_2(M,\C^-):=\setdef{f\in L^2_2(M,\C)}{\forall_{x\in
M}:\,f(x)\in\C^-}
\]
is an open subset of $L^2_2(M,\C^*)$. If we let
\[
V:=\R\times (-i\pi,i\pi)\subset\C,
\]
then $\exp|_V:V\to \C^-$ is a diffeomorphism. This in turn induces
a diffeomorphism
\begin{equation*}
\exp:L^2_2(M,V)\to L^2_2(M,\C^-),\quad f\mapsto \exp(f),
\end{equation*}
where $L^2_2(M,V)$ is defined in the same manner as
$L^2_2(M,\C^*)$ and $L^2_2(M,\C^-)$ above. Clearly,
\begin{equation*}
\exp\Big(L^2_2(M,V)\cap L^2_2(M,i\R)\Big)= L^2_2(M,\C^-)\cap
L^2_2(M,\U_1).
\end{equation*}
Multiplication by an element of $L^2_2(M,\U_1)$ induces a
diffeomorphism of $L^2_2(M,\C^*)$. Hence, for arbitrary
$\gamma\in L^2_2(M,\U_1)$, the set $\gamma\cdot L^2_2(M,\C^-)$ is
an open neighbourhood of $\gamma$ in $L^2_2(M,\C^*)$. As a
consequence,
\[
\gamma\cdot\exp:L^2_2(M,V)\to \gamma\cdot L^2_2(M,\C^-),\quad
f\mapsto \gamma\cdot\exp(f),
\]
is a diffeomorphism satisfying
\[
(\gamma\cdot\exp)\big(L^2_2(M,V)\cap L^2_2(M,i\R)\big)=
\big(\gamma\cdot L^2_2(M,\C^-)\big)\cap L^2_2(M,\U_1).
\]
Since $L^2_2(M,\U_1)$ can be covered by sets of the above type, it
is a closed submanifold of $L^2_2(M,\C^*)$ modelled on
$L^2_2(M,i\R)$.

According to Example \ref{sob:mult:n=3}, there is a smooth
multiplication
\[
L_2^2\times L_1^2\to L_1^2.
\]
Hence, the action of $\cG_2$ on $L^2_1(M,S)$, which is given by
$(\gamma,\psi)\mapsto\gamma^{-1}\psi$, is smooth. $\cG_2$ acts on
the space of gauge fields via $(\gamma,A)\mapsto
A+2\gamma^{-1}d\gamma$. As $d:L^2_2(M,\C)\to L^2_1(M,\C)$ is a
bounded linear operator, it follows again from Example
\ref{sob:mult:n=3} that the map
\[
L^2_2(M,\U_1)\to L^2_1(M,iT^*M),\quad
\gamma\mapsto\gamma^{-1}d\gamma,
\]
is smooth. Thus, the Seiberg-Witten configuration space is acted
on smoothly by $\cG_2$.
\end{proof}\index{gauge transformations|)}

We are now in the position to make the considerations in
Proposition \ref{G:adj} more precise. Let $(\psi,A)\in\cC_1$.
Since the action of $\cG_2$ on $\cC_1$ is smooth, we may take the
derivative of the map
\[
\cG_2\to\cC_1,\quad \gamma\mapsto \gamma\cdot(\psi,A).
\]
The formal calculation in loc.~cit. shows that this results in
the bounded linear map
\[
G_{(\psi,A)}:L^2_2(M,i\R)\to L^2_1(M,E),\quad f\mapsto
(-f\psi,2df),
\]
where---as always---we are using the abbreviation
\[
E:=S\oplus iT^*M.
\]
Notice that $G_{(\psi,A)}$ above is not defined on its natural
domain $L^2_1(M,i\R)$. We shall, however, consider $G_{(\psi,A)}$
as a closed operator $L^2(M,i\R)\to L^2(M,E)$ with domain
$L_1^2(M,i\R)$ restricting to $L_2^2(M,i\R)$ whenever it is
necessary. In the same way, $G_{(\psi,A)}^*$ shall always denote
the {\em functional analytic} adjoint of $G_{(\psi,A)}$. When
restricted to $L_1^2(M,E)$ it coincides with the natural
extension of the formal adjoint \eqref{Gt}.

We will now turn to the extension of the elliptic operator
$T_{(\psi,A)}$ which was defined in \eqref{T} for smooth
$(\psi,A)$ as
\[
T_{(\psi,A)}=\big(F_{(\psi,A)}+
G_{(\psi,A)},G_{(\psi,A)}^*\big):C^\infty(M,E\oplus i\R)\to
C^\infty(M,E\oplus i\R).
\]
Then $T_{(\psi,A)}$ is also well-defined as an operator in
$L^2(M,E\oplus i\R)$ with natural domain $L^2_1(M,E\oplus i\R)$.
Our aim is to show that $T_{(\psi,A)}$ is a self-adjoint operator
with compact resolvent. Fixing a smooth gauge field $A_0$ and
writing $A=A_0+a_0$, we let
\[
K_{(\psi,a_0)}:= T_{(\psi,A_0+a_0)}-T_{(0,A_0)}.
\]
Since $\cD_A=\cD_{A_0}+\lfrac{1}{2}c(a_0)$, formula
\eqref{T:explicit} shows that
\[
K_{(\psi,a_0)}
\begin{pmatrix}\gf \\a \\ f\end{pmatrix} =
\begin{pmatrix}\lfrac{1}{2}c(a_0)\gf+\lfrac{1}{2}c(a)\psi -f\psi \\
                q(\psi,\gf) \\
                -i\Im\scalar{\gf}{\psi}\end{pmatrix},
\quad \begin{pmatrix}\gf \\a \\ f\end{pmatrix}\in L^2_1(M,E\oplus
i\R).
\]
Using the considerations at the end of Appendix \ref{app:ell} as a
guideline, we now need to ascertain the following:
\begin{lemma} For any $(\psi,a_0)\in L^2_1(M,E)$, the operator
\[
K_{(\psi,a_0)}:L^2_1(M,E\oplus i\R)\to L^2(M,E\oplus i\R)
\]
is compact and symmetric with respect to the $L^2$ scalar product.
\end{lemma}
\begin{proof}
From the explicit description of $K_{(\psi,a_0)}$ we deduce that
the assertion of the lemma reduces to the claim that
multiplication by an element of $L^2_1$ yields a compact operator
$L^2_1\to L^2$.

Carefully checking the assumptions of Proposition \ref{sob:mult},
we obtain a continuous Sobolev multiplication
\[
L^2_1\times L^2_1 \to L_1^p\,,
\]
provided that $p\in (1,\lfrac{3}{2})$. If in addition
$p>\lfrac65$, we have $1-\lfrac{3}{p}> -\lfrac{3}{2}$, and the
Rellich-Kondrachov Theorem \ref{rellich} implies that $L_1^p$
embeds compactly in $L^2$. Then the desired compactness
property follows. The symmetry with respect to the $L^2$ scalar
product is an immediate consequence of the definition of
$K_{(\psi,a_0)}$ as the difference of two symmetric operators.
\end{proof}

Invoking Theorem \ref{relative:comp} about relative compact
perturbations of operators with compact resolvent, we can now draw
an important conclusion:

\begin{prop}\label{T:prop}
Let $(\psi,A)\in \cC_1$. Then the operator $T_{(\psi,A)}$ induces
a self-adjoint operator in $L^2(M,E\oplus i\R)$ with domain
$L^2_1(M,E\oplus i\R)$, i.e., using the notation of \eqref{Lsa},
\[
T_{(\psi,A)}\in \sL_{sa}\big(L^2_1(M,E\oplus i\R),L^2(M,E\oplus
i\R)\big).
\]
Moreover, $T_{(\psi,A)}$ has compact resolvent and thus discrete
spectrum. In particular, it is a Fredholm operator.
\end{prop}

According to the considerations in App.~\ref{app:SF&OT},
Sec.~\ref{app:SF}, $\sL_{sa}$ is an open subset of the Banach
space $\sL_{sym}$ which is the space of symmetric operators with
the same fixed domain. Hence $\sL_{sa}$ inherits the structure of
a Banach manifold if endowed with the operator norm topology. With
respect to this, we have:

\begin{prop}\label{T:diff}
The assignment $(\psi,A)\mapsto T_{(\psi,A)}$ defines a smooth map
\[
\cC_1(\gs)\to \sL_{sa}\big(L^2_1(M,E\oplus i\R),L^2(M,E\oplus
i\R)\big),
\]
\end{prop}
\begin{proof}
Since $\sL_{sa}$ is an open subset of $\sL_{sym}$, it suffices to
insure that the assignment
\[
L^2_1(M,E)\to \sL_{sym}\big(L^2_1(M,E\oplus i\R),L^2(M,E\oplus
i\R)\big),\quad (\psi,a_0)\mapsto K_{(\psi,a_0)}
\]
is smooth. By linearity of this map, smoothness is equivalent to
continuity.

Using the continuous Sobolev multiplication $L^2_1\times L^2_1\to
L^2$, one straightforwardly obtains
\[
\big\|K_{(\psi,a_0)}(\gf,a,f)\big\|_{L^2}\le \const\cdot
\big\|(\psi,a_0)\big\|_{L^2_1}\cdot\big\|(\gf,a,f)
\big\|_{L^2_1}\,.
\]
Therefore, we can estimate the operator norm by
\begin{align*}
\big\|K_{(\psi,a_0)}\big\|_{L^2_1,L^2}=\sup_{(\gf,a,f)}
\frac{\big\|K_{(\psi,a_0)}
(\gf,a,f)\big\|_{L^2}}{\big\|(\gf,a,f)\big\|_{L^2_1}} \le
\const\cdot \big\|(\psi,a_0)\big\|_{L^2_1}
\end{align*}
which ensures continuity of $(\psi,a_0)\mapsto K_{(\psi,a_0)}$.
\end{proof}

\section{Topology of the quotient space}\label{top:quot}

Our next aim is to investigate the topology of the set of
configurations modulo gauge equivalence, i.e., of the quotient
$\cB_1=\cC_1/\cG_2$. The first observation is that $\cB_1$ is
second countable since its topology is defined as the quotient of
an affine space modelled on a separable Hilbert space.\\

\noindent\textbf{The Hausdorff property.} When studying the
quotient of group actions $G\times X\to X$, the situation
simplifies if $G$ acts \emph{properly} on $X$, i.e., if the map
\[
G\times X\to X\times X,\quad (g,x)\mapsto (x,gx)
\]
is proper.\footnote{This means that the map is closed and that
preimages of points are compact} Recall (e.g. from \cite{toD:TG},
Sec.~I.3) that the quotient $X/G$ of a proper group action is
always a Hausdorff space. The following (simple) criterium is
useful in our context:
\begin{lemma}\label{proper:crit}
Let $G\times X\to X$ be a (topological) group action. Suppose
that all stabilizers are compact and that for all sequences
$(x_n)$ in $X$ and $(g_n)$ in $G$ the following holds:
\begin{quotation}
\noindent If $x_n\to x$ and $g_nx_n\to y$, then there exists a
convergent subsequence of $(g_n)$ whose limit point $g\in G$
satisfies $gx=y$.
\end{quotation}
Then $G$ acts properly on $X$.
\end{lemma}

\begin{prop}\label{proper}
The group $\cG_2$ acts properly on $\cC_1$. In particular, the
quotient $\cB_1$ is a Hausdorff space.
\end{prop}
\begin{proof}
We use the criterium in Lemma \ref{proper:crit}. Let
$\big((\psi_n,A_n)\big)$ and $(\gamma_n)$ be sequences in $\cC_1$
and $\cG_2$ respectively. Suppose that there exists
$(\psi_0,A_0)$ and $(\psi,A)$ in$ \cC_1$ such that
\[
(\psi_n,A_n)\xrightarrow{n\to\infty} (\psi_0,A_0)\quad
\text{and}\quad \gamma_n\cdot(\psi_n,A_n)\xrightarrow{n\to\infty}
(\psi,A).
\]
In particular,
\begin{equation}\label{proper:1}
\big\|\gamma_n^{-1}\psi_n-\psi\big\|_{L^2_1}
\xrightarrow{n\to\infty}0
\end{equation}
and, if we let $a_n:=A_n-A$,
\begin{equation}\label{proper:2}
\big\|A_n+2\gamma_n^{-1}d\gamma_n-A\big\|_{L^2_1}=\big\|a_n+
2\gamma_n^{-1}d\gamma_n\big\|_{L^2_1} \xrightarrow{n\to\infty}0.
\end{equation}

Since multiplication with a gauge transformation does not change
the value of $\|.\|_{L^p}$, we deduce from the continuous
embedding $L^2_1\hookrightarrow L^4$ that
\begin{equation}\label{proper:3}
\begin{split}
\big\|d\gamma_n\big\|_{L^4}&=
\big\|\gamma_n^{-1}d\gamma_n\big\|_{L^4} \le \const\cdot
\big\|\gamma_n^{-1}d\gamma_n\big\|_{L^2_1}\\ &\le \const\cdot\big(
\big\|a_n\big\|_{L^2_1}
+\big\|a_n+2\gamma_n^{-1}d\gamma_n\big\|_{L^2_1}\big),
\end{split}
\end{equation}
where in the last line we have employed the triangle inequality.
As a convergent sequence, $(a_n)$ is bounded in $L_1^2$. We thus
conclude from \eqref{proper:2} that $(d\gamma_n)$ is a bounded
sequence in $L^4$ (and hence also in $L^2$). To obtain an
$L^2_2$-bound on $(\gamma_n)$ it remains to consider the sequence
of the second derivatives. Viewing $(\gamma_n)$ as a sequence in
$L^2_2(M;\C)$, we have
\[
\nabla^2\gamma_n=\nabla (d\gamma_n)\in L^2(M,T^*M^{\otimes
2}\otimes \C).
\]
Note that
\[
\nabla (d\gamma_n) =\gamma_n \nabla (\gamma^{-1}_nd\gamma_n) -
\gamma_n d\gamma_n^{-1}\otimes d\gamma_n
\]
Since $d\gamma_n^{-1} =-\gamma_n^{-2}(d\gamma_n)$, the
(pointwise) norm of the second summand is equal to
$|d\gamma_n|^2$. Therefore,
\begin{equation}\label{proper:4}
\begin{split}
\big\|\nabla(d\gamma_n)\big\|_{L^2} &\le \big\|\nabla
(\gamma^{-1}_nd\gamma_n)  \big\|_{L^2}
+\big\|d\gamma_n^{-1}\otimes d\gamma_n\big\|_{L^2}\\ &\le
\|\gamma_n^{-1}d\gamma_n\|_{L^2_1} + \|d\gamma_n\|_{L^4}^2\, .
\end{split}
\end{equation}
As we have seen in \eqref{proper:3}, the right hand side of
\eqref{proper:4} is bounded so that we obtain the desired
$L^2_2$-bound on $(\gamma_n)$.

The Rellich-Kondrachov Theorem \ref{rellich} now implies that
$L^2_2$ embeds compactly in $L^2_1$. We can thus find a
subsequence of $(\gamma_n)$ which converges in $L^2_1$.
Additionally, according to \eqref{proper:4}, the sequence
$(\nabla(d\gamma_n))$ is bounded in the Hilbert space
$L^2(M,T^*M^{\otimes 2}\otimes\C)$. Hence, $(\gamma_n)$ contains
a subsequence such that the second derivatives are weakly
convergent in $L^2$. Without loss of generality, we may therefore
assume that $(\gamma_n)$ converges strongly in $L^2_1$ to, say,
$\gamma$ and that $(\nabla(d\gamma_n))$ converges weakly in $L^2$
to, say, $\eta$. This implies that $\nabla(d\gamma)=\eta$ weakly
in $L^2$. As $\nabla\circ d$ is injectively elliptic, we have
$\gamma\in L^2_2(M,\U_1)$.

It remains to show that $\gamma\cdot(\psi_0,A_0)=(\psi,A)$, i.e.,
with $a:=A_0-A$, that
\[
\gamma^{-1}\psi_0=\psi\quad\text{ and }\quad 2d\gamma +\gamma a=0.
\]
Equations \eqref{proper:1} and \eqref{proper:2} guarantee that
$(\gamma_n^{-1}\psi_n)\to\psi$ in $L^2$ as well as
$2d\gamma_n+\gamma_n a_n\to 0$ in $L^2$. On the other hand,
continuity of the multiplication $L^2_1\times L^2_1\to L^2$ shows
that the sequence $(\gamma_n^{-1}\psi_n)$ converges to
$\gamma^{-1}\psi_0$ in $L^2$, and $d\gamma_n\to d\gamma$ as well
as $\gamma_na_n\to\gamma a$ in $L^2$. Hence, uniqueness of the
limit points implies the above formul\ae.\qedhere\\
\end{proof}

\noindent\textbf{Local slices for the action.} In this paragraph
we shall establish a slice theorem for the action of $\cG_2$ on
$\cC_1$, analogous to the well-known situation from the theory of
finite dimensional Lie group actions (cf. \cite{toD:TG},
Sec.~I.5): For every $(\psi,A)\in\cC_1$ we are looking for a
subspace
\[
S_{(\psi,A)}\subset T_{(\psi,A)}\cC_1=L_1^2(M,E),
\]
complementary to the tangent space of the orbit $\cG_2\cdot
(\psi,A)$. We wish to model nearby orbits by $\cG_2\times
S_{(\psi,A)}$, making use of the natural map
\begin{equation}\label{pi}
\begin{split}
\pi:\cG_2\times S_{(\psi,A)} \to \cC_1,\quad
\pi\big(&\gamma,(\gf,a)\big)=\gamma\cdot\big((\psi,A)+(\gf,a)\big)
\\
&=(\gamma^{-1}\psi+\gamma^{-1}\gf,A+a+2\gamma^{-1}d\gamma).
\end{split}
\end{equation}
Clearly, if $(\psi,A)$ is a reducible configuration, the map
$\pi$ cannot be injective. Therefore, $S_{(\psi,A)}$ has to be
chosen to be invariant under the natural action of the stabilizer
$\cG_{(\psi,A)}$. We will then have to study
\[
\cG_2\times_{\cG_{(\psi,A)}}S_{(\psi,A)}:=(\cG_2\times
S_{(\psi,A)})/\cG_{(\psi,A)}.
\]
Since the action of the stabilizer on
$T_{(\psi,A)}\cC_1=L_1^2(M,E)$ is orthogonal with respect to the
$L^2$ metric, a natural choice for $S_{(\psi,A)}$ is provided by
taking the orthogonal complement of the tangent space to the
gauge orbit. This tangent space is essentially the image of the
differential of action map, i.e., the image of (cf. \eqref{G})
\[
G_{(\psi,A)}:L^2_2(M,i\R)\to L^2_1(M,E),\quad f\longmapsto
(-f\psi,2df).
\]
As the leading term of $G_{(\psi,A)}$ is injectively elliptic, we
infer from the Hodge decomposition \eqref{Hodge:dec:inj} that
there is an $L^2$-orthogonal splitting
\begin{equation}\label{Hodge:dec:G}
L^2_1(M,E)=\im(G_{(\psi,A)}|_{L_2^2})\oplus \ker(
G_{(\psi,A)}^*|_{L_1^2}).
\end{equation}
Recall from Proposition \ref{G:adj} that for $(\gf,a)\in
L_1^2(M,E)$,
\[
G_{(\psi,A)}^*(\gf,a)=2d^*a-i\Im\scalar{\gf}{\psi}.
\]

\begin{dfn}
For all $(\psi,A)\in\cC_1$ we define the {\em local slice} of the
$\cG_2$-action at the point $(\psi,A)$ as the subspace
\[
S_{(\psi,A)}:=\ker (G_{(\psi,A)}^*|_{L_1^2})\subset L^2_1(M,E).
\]
\end{dfn}

\begin{lemma}
For every $(\psi,A)\in \cC_1$, the local slice $S_{(\psi,A)}$ is
$\cG_{(\psi,A)}$-invariant.
\end{lemma}
\begin{proof}
If $(\psi,A)$ is irreducible, the stabilizer is trivial. For
reducible $(\psi,A)$, i.e., if $\psi=0$, we have
\begin{equation}\label{slice:red}
S_{(\psi,A)}=L_1^2(M,S)\oplus \ker (d^*|_{L_1^2}).
\end{equation}
Recall from \eqref{gg:on:tangent} that $\cG_2$, and hence also
$\cG_{(\psi,A)}$, acts only on the spinor part of $L_1^2(M,E)$.
Therefore, \eqref{slice:red} is invariant under the action of
$\cG_{(\psi,A)}$.
\end{proof}

\noindent The stabilizer acts on $\cG_2\times S_{(\psi,A)}$ via
\[
\gl\cdot\big(\gamma,(\gf,a)\big):=
\big(\gamma\gl^{-1},\gl\cdot(\gf,a)\big)=
\big(\gamma\gl^{-1},(\gl^{-1}\gf,a)\big),\quad
\gl\in\cG_{(\psi,A)}.
\]
One readily checks that the natural map \eqref{pi} is
$\cG_{(\psi,A)}$-invariant thus factoring to a map
\[
\Bar{\pi}:\cG_2\times_{\cG_{(\psi,A)}} S_{(\psi,A)}
\longrightarrow \cC_1.
\]
To lift the $\cG_2$-action on $\cC_1$ to $\cG_2\times
S_{(\psi,A)}$ we let $\cG_2$ act from the left on the first
factor, i.e., for $\big(\gamma,(\gf,a)\big)\in \cG_2\times
S_{(\psi,A)}$ and $\gamma'\in\cG_2$ we let
\[
\gamma'\cdot\big(\gamma,(\gf,a)\big):=
\big(\gamma'\gamma,(\gf,a)\big).
\]
Then $\pi$ is clearly $\cG_2$-equivariant. Moreover, the actions
of $\cG_2$ and $\cG_{(\psi,A)}$ on $\cG_2\times S_{(\psi,A)}$
commute so that the quotient $\cG_2\times_{\cG_{(\psi,A)}}
S_{(\psi,A)}$ inherits a $\cG_2$-action.

\begin{lemma}\label{slice:lem}
Let $(\psi,A)\in\cC_1$. The differential
\[
D_{(1,0)}\pi:L^2_2(M,i\R)\oplus S_{(\psi,A)}\longrightarrow
L^2_1(M,E)
\]
of $\pi$ at the point $(1,0)\in\cG_2\times S_{(\psi,A)}$ is
surjective, and
\[
\ker(D_{(1,0)}\pi)= \ker (G_{(\psi,A)}|_{L_2^2})\oplus\{0\}.
\]
\end{lemma}
\begin{proof}
The differential at the point $(1,0)$ is given by
\begin{equation*}
\begin{split}
D_{(1,0)}\pi(f,\gf,a)&=\lfrac{d}{dt}\big|_{t=0}
\exp(tf)\cdot\big((\psi,A)+t(\gf,a)\big)\\
&=G_{(\psi,A)}(f)+(\gf,a).
\end{split}
\end{equation*}
Hence, surjectivity of $D_{(1,0)}\pi$ is immediate from the
decomposition
\[
L^2_1(M,E)=\im(G_{(\psi,A)}|_{L_2^2})\oplus S_{(\psi,A)},
\]
given in \eqref{Hodge:dec:G}. Moreover, as the above decomposition
is direct,
\[
G_{(\psi,A)}(f)+(\gf,a)=0\quad\Longleftrightarrow\quad
G_{(\psi,A)}(f)=0\quad\text{and}\quad(\gf,a)=0
\]
which proves the second assertion.
\end{proof}

\begin{prop}\label{slice:prop}
For each $(\psi,A)\in\cC_1$ there exists a
$\cG_{(\psi,A)}$-invariant open neighbourhood $V$ of $(1,0)$ in
$\cG_2\times S_{(\psi,A)}$ such that
\begin{enumerate}
\item $\pi|_V$ is a submersion.
\item The fibres of $\pi|_V$ are in 1-1 correspondence with the
$\cG_{(\psi,A)}$-orbits.
\end{enumerate}\end{prop}
\begin{proof}
We have to study two cases:
\begin{Cases}
\item \textit{$(\psi,A)$ is irreducible:}
In this case $\ker(G_{(\psi,A)})=0$ so that according to the
preceding lemma, the differential of $\pi$ at $(1,0)$ is an
isomorphism. Invoking the inverse function theorem for Banach
manifolds we conclude that there exists a neighbourhood $V$ of
$(1,0)$ in $\cG_2\times S_{(\psi,A)}$ such that $\pi|_V$ is a
diffeomorphism onto its image. As $\cG_{(\psi,A)}=\{1\}$, the set
$V$ is clearly $\cG_{(\psi,A)}$-invariant.

\item \textit{$(\psi,A)$ is reducible, i.e., $\psi\equiv 0:$}
Lemma \ref{slice:lem} ensures that the differential of $\pi$ at
the point $(1,0)$ is surjective. Invoking the implicit function
theorem, we deduce that this holds true also on an open
neighbourhood $V$ of $(1,0)$. Obviously, $V$ can be chosen to be
$\cG_{(0,A)}$-invariant since otherwise, we may consider
$\cG_{(0,A)}\cdot V$. Note that the differential of $\pi$ is
still surjective on that set because $\pi$ is
$\cG_{(0,A)}$-invariant.

It remains to prove the second assertion in this case. Suppose we
have $(\gamma_i,\gf_i,a_i)\in V$ such that
$\pi(\gamma_1,\gf_1,a_1)=\pi(\gamma_2,\gf_2,a_2)$. Then, since
$\psi=0$,
\[
\big(\gamma_1^{-1}\gf_1,\;a_1+2\gamma_1^{-1}d\gamma_1\big)=
\big(\gamma_2^{-1}\gf_2,\;a_2+2\gamma_2^{-1}d\gamma_2\big).
\]
Defining $\gamma:=\gamma_2^{-1}\gamma_1$, we can express this
alternatively as
\begin{equation}\label{slice:prop:1}
\gf_1=\gamma^{-1}\gf_2\quad\text{ and }\quad
a_2-a_1=2\gamma^{-1}d\gamma.
\end{equation}
Then part (ii) is established provided that
$\gamma\in\cG_{(0,A)}$, i.e., that $\gamma$ is constant.

Recall from \eqref{derham} that $[\gamma^{-1}d\gamma]\in
H^1_{dR}(M;2\pi i\Z)$, which is a lattice in $H^1_{dR}(M;i\R)$.
If $V$ is chosen small enough, cohomology classes of the form
$[a_2-a_1]$ can be forced to lie in a small neighbourhood of 0 in
$H^1_{dR}(M;i\R)$ hitting $H^1_{dR}(M;2\pi i\Z)$ only in 0.
Hence, without loss of generality, the second part of
\eqref{slice:prop:1} can only be fulfilled if
$[a_2-a_1]=[\gamma^{-1}d\gamma]=0$.

On the other hand, according to \eqref{slice:red}, the 1-forms
$a_2$ and $a_1$ are co-closed which implies that
$2\gamma^{-1}d\gamma=a_2-a_1$ is also co-closed and hence
harmonic. Together with $[\gamma^{-1}d\gamma]=0$, this implies
$\gamma^{-1}d\gamma=0$ and therefore, $\gamma$ is
constant.\qedhere
\end{Cases}
\end{proof}

After these preparations we can now state and prove the slice
theorem. We follow the presentation in J.W. Morgan's book
\cite{Mo}.
\begin{theorem}\label{slice:thm}\index{configuration space!slice
theorem} Let $(\psi,A)$ be an arbitrary configuration. Then there
exists a $\cG_{(\psi,A)}$-invariant open neighbourhood $U$ of
$\;0\in S_{(\psi,A)}$ such that
\[
\pi:\cG_2\times U\longrightarrow \cC_1
\]
induces a homeomorphism of $\cG_2\times_{\cG_{(\psi,A)}}U$ onto a
$\cG_2$-invariant open neighbourhood of $(\psi,A)$ in $\cC_1$.
\end{theorem}

\begin{proof}
Let $V$ be chosen as in Proposition \ref{slice:prop}. Then
$\cG_2\cdot V$ can be written as $\cG_2\times U$ where $U$ is a
$\cG_{(\psi,A)}$-invariant open neighbourhood of $0\in
S_{(\psi,A)}$. More concretely,
\begin{equation}\label{slice:0}
U:=\bigsetdef{(\gf,a)}{\exists_{\gamma\in\cG_2}:\,(\gamma,\gf,a)\in
V}.
\end{equation}
The map $\pi|_{\cG_2\times U}$ is a submersion since $\pi|_V$ is
one and $\pi$ is $\cG_2$-equivariant. Furthermore,
$\pi(\cG_2\times U)$ is a $\cG_2$-invariant open neighbourhood of
$(\psi,A)$ in $\cC_1$.

We now establish the assertion by contradiction. Assume that
possibly making $V$ smaller, we cannot achieve that the induced
map $\Bar{\pi}$ on $\cG_2\times_{\cG_{(\psi,A)}}U$ is injective.
This means that for every $V$ as in Proposition \ref{slice:prop}
and corresponding $U$ of the form \eqref{slice:0} there exists a
point $(\gamma,\gf,a)\in \cG_2\times U$ such that the fibre of
$\pi|_{\cG_2\times U}$ which contains $(\gamma,\gf,a)$ is larger
than the corresponding $\cG_{(\psi,A)}$ orbit. We may thus choose
sequences $(\gf_n,a_n),\;(\gf_n',a_n')$ in $S_{(\psi,A)}$ and
$(\gamma_n)$ in $\cG_2$ such that
\[
(\gf_n,a_n)\xrightarrow{n\to\infty}0,\qquad
(\gf_n',a_n')\xrightarrow{n\to\infty}0,
\]
and
\begin{equation}\label{slice:1}
\gamma_n\cdot\big((\psi,A)+(\gf_n,a_n)\big)=
(\psi,A)+(\gf_n',a_n'),\quad\text{ but }
\gamma_n\notin\cG_{(\psi,A)}.
\end{equation}
From this we conclude that
\[
d\gamma_n=\lfrac{1}{2}\gamma_n(a_n' -a_n).
\]
Since all $\gamma_n$ are maps $M\to \U_1$, the sequence
$(\gamma_n)$ is bounded with respect to $\|.\|_{\infty}$. On the
other hand, $(a_n'-a_n)$ converges to 0 in $L^2_1$ and thus also
in $L^p$ for all $p\le 6$. Therefore, $(d\gamma_n)$ converges to
0 in every $L^p$ for $p\le 6$ which in turn provides an
$L_1^p$-bound on $(\gamma_n)$. If $p>3$, then there is a
continuous multiplication $L_1^p\times L^2_1\to L^2_1$. Invoking
the above equation again shows that $(d\gamma_n)$ converges to 0
in $L^2_1$. Consequently, $(\gamma_n)$ is a bounded sequence in
$\cG_2$.

We may thus deduce from the Rellich-Kondrachov Theorem
\ref{rellich} that---possibly restricting to a subsequence---the
sequence $(\gamma_n)$ converges in $L_1^2(M,\C)$. Let $\gamma$
denote the limit point. Since $d\gamma_n\to 0$ in $L^2$, we
conclude that $d\gamma=0$ weakly in $L^2$. Since $d$ is
injectively elliptic on functions, regularity guarantees that
$\gamma\in C^\infty(M,\C)$ and that $d\gamma=0$. In particular,
$\gamma$ is a constant function and $(\gamma_n)$ converges to
$\gamma$ in $L^2_1$. Actually, this convergence is with respect
to $L_2^2$ because $d\gamma_n\to 0$ in $L_1^2$ and $d\gamma=0$.
By virtue of the embedding of $L^2_2$ in $C^0$, this implies that
$\gamma$ takes values in $\U_1$ since all $\gamma_n$ do so.
Invoking continuity of $\cG_2\times \cC_1\to\cC_1$ we may now
deduce from \eqref{slice:1} that
\[
\gamma\cdot(\psi,A)=(\psi,A)
\]
which shows that $\gamma\in\cG_{(\psi,A)}$.

The remaining part of the proof works as in the finite
dimensional case: Let us consider an open neighbourhood $V$ of
$(1,0)$ in $\cG_2\times S_{(\psi,A)}$ as in Proposition
\ref{slice:prop}. Since $V$ is $\cG_{(\psi,A)}$-invariant, we
also have $(\gamma,0)\in V$. As the sequences
$(\gamma_n,\gf_n,a_n)$ and $(1,\gf_n',a_n')$ converge to
$(\gamma,0)$ and $(1,0)$ respectively, there exists $n\in\N$ such
that
\[
(\gamma_n,\gf_n,a_n)\in V\quad\text{ and }(1,\gf_n',a_n')\in V
\]
since $V$ is an open neighbourhood of both limit points. By means
of \eqref{slice:1},
\[
\pi(\gamma_n,\gf_n,a_n)=\pi(1,\gf_n',a_n').
\]
According to Proposition \ref{slice:prop}, this requires
$\gamma_n\in\cG_{(\psi,A)}$ since the fibres of $\pi|_V$
correspond to the $\cG_{(\psi,A)}$ orbits. However,
$\gamma_n\in\cG_{(\psi,A)}$ contradicts \eqref{slice:1}.

We may thus suppose that the induced map $\Bar{\pi}$ is injective
on $\cG_2\times_{\cG_{(\psi,A)}}U$. Moreover, since $\pi|_V$ is
continuous and an open map, the same holds true for the respective
restriction of $\Bar{\pi}$. Hence, it induces a homeomorphism from
$\cG_2\times_{\cG_{(\psi,A)}}U$ onto the $\cG_2$-invariant open
neighbourhood $\pi(\cG_2\times U)$ of $(\psi,A)$ in $\cC_1$ as is
illustrated in the following diagram:
\[
\begindc[5]
\obj(1,1){$\cG_2\times_{\cG_{(\psi,A)}}U$}[quot]
\obj(1,11){$\cG_2\times U$}[GxU] \obj(17,11){$\pi(\cG_2\times
U)$}[C] \mor{GxU}{quot}{} \mor{GxU}{C}{$\pi$}
\mor(2,3)(15,10){$\Bar{\pi}$}
\enddc\qedhere
\]
\end{proof}

\begin{cor}\label{slice:cor}\quad
\begin{enumerate}
\item Suppose $U\subset S_{(\psi,A)}$ is chosen as in the slice
theorem. Then $U/\cG_{(\psi,A)}$ is homeomorphic to a
neighbourhood of $[\psi,A]$ in $\cB_1=\cC_1/\cG_2$.
\item The irreducible part $\cB_1^*\subset\cB_1$ carries the structure
of a smooth Banach manifold. Its tangent space at a point
$[\psi,A]$ is naturally isomorphic to the local slice
$S_{(\psi,A)}$.
\item The projection $\cC_1^*\to\cB_1^*$ is a principal
$\cG_2$-bundle.
\end{enumerate}
\end{cor}
\begin{proof}
\begin{enumerate}
\item Since $\cG_2$ acts only on the first factor of $\cG_2\times U$,
the quotient $(\cG_2\times U)/\cG_2$ can be identified with $U$.
Therefore, as the actions of $\cG_2$ and $\cG_{(\psi,A)}$ commute,
\[
\big(\cG_2\times_{\cG_{(\psi,A)}}U\big)/\cG_2\cong
U/\cG_{(\psi,A)}.
\]
The map $\Bar\pi$ is a $\cG_2$-equivariant homeomorphism hence
induces a homeomorphism
\[
\big(\cG_2\times_{\cG_{(\psi,A)}}U\big)/\cG_2\cong
\Bar\pi(\cG_2\times_{\cG_{(\psi,A)}} U)/\cG_2.
\]
This establishes part (i) for the right hand side is an open
neighbourhood of $[\psi,A]$ in $\cB_1$.
\item Let $(\psi,A)\in\cC_1^*$ be an irreducible
configuration, $U$ a neighbourhood of $0\in S_{(\psi,A)}$ as in
the slice theorem. Suppose that $V:=\pi(\cG_2\times U)$ is
entirely contained in the irreducible part $\cC_1^*$. We define a
map $\gF:V/\cG_2\to U$ by letting
\[
\gF\big([\gamma(\psi+\gf,A+a)]\big):=(\gf,a).
\]
Note that $\gF$ is well-defined, and that (i) ensures that it
yields a homeomorphism $\cB_1^*\supset V/\cG_2\cong U\subset
S_{(\psi,A)}$. We have to ascertain that the collection of all
such $\gF$ provides $\cB_1^*$ with a differentiable structure.

Suppose that $(\psi',A')$ is another irreducible configuration,
and let $U'\subset S_{(\psi',A')}$ and $V'\subset \cC_1^*$ be
chosen correspondingly. Without loss of generality, we may assume
that $V\cap V'\neq\emptyset$. Since $V$ and $V'$ are
$\cG_2$-invariant, this yields
\[
V/\cG_2\cap V'/\cG_2=(V\cap V')/\cG_2\neq\emptyset.
\]
If $\Tilde{U}:=\gF^{-1}(V/\cG_2\cap V'/\cG_2)$ and
$\Tilde{U}':={\gF'}^{-1}(V/\cG_2\cap V'/\cG_2)$, then the
following diagram commutes.
\[
\begindc
\obj(1,1){$\Tilde{U}$}[lu] \obj(1,3){$\cG_2\times\Tilde{U}$}[lo]
\obj(4,1){$V/\cG_2\cap V'/\cG_2$}[mu] \obj(4,3){$V\cap V'$}[mo]
\obj(7,1){$\Tilde{U}'$}[ru] \obj(7,3){$\cG_2\times\Tilde{U}'$}[ro]
\mor{lu}{lo}{}[\atleft,\injectionarrow] \mor{mo}{mu}{}
\mor{ro}{ru}{} \mor{lo}{mo}{$\pi$} \mor{mo}{ro}{${\pi'}^{-1}$}
\mor{lu}{mu}{$\gF^{-1}$} \mor{mu}{ru}{$\gF'$}
\enddc
\]
The slice theorem shows that $\pi$ and $\pi'$ are diffeomorphisms
so that ${\pi'}^{-1}\circ\pi:\cG_2\times\Tilde{U}\to
\cG_2\times\Tilde{U}'$ is smooth. Thus, the above diagram
establishes that
\[
\gF'\circ\gF^{-1}:\Tilde{U}\subset S_{(\psi,A)}\to S_{(\psi',A')}
\]
is also smooth. Therefore, the maps $\{\gF:U\to S_{(\psi,A)}\}$
define a differentiable atlas of $\cB_1^*$. This shows that
$\cB_1^*$ is indeed a Banach manifold modelled on the isomorphism
class of $S_{(\psi,A)}$.
\item This is an easy consequence of the proof of (ii).\qedhere
\end{enumerate}\end{proof}
\begin{remark*}
The proof of (ii) shows that $\cB_1^*$ is actually a Hilbert
manifold with respect to the induced $L^2_1$ scalar product on the
local slice $S_{(\psi,A)}$. Whenever we use a scalar product on
$S_{(\psi,A)}$, it shall, however, be the induced $L^2$ scalar
product. To stress that the local slice is not complete with
respect to $\Lscalar{.}{.}$ we thus do not use the terminology
\emph{Hilbert} manifold.
\end{remark*}

\section{Compactness}\label{compact}
Our next task is to establish that the moduli space is
sequentially compact. It will turn out that as corollary to the
proof of the compactness theorem, the topology of the moduli
space is independent of the initially chosen Sobolev orders.\\

\noindent\textbf{Gauge fixing.}\index{gauge transformations!gauge
fixing} The first idea leading to the results mentioned above is
to find a suitable representative of a gauge equivalence class of
monopoles. An appropriate method of fixing such a configuration
is the so-called {\em Coulomb gauge}.

\begin{lemma}\label{d^asta=0}
Let $A, A_0\in\cA_1$. Then there exists $\gamma\in\cG_2$ such that
\[
d^*(A-A_0+2\gamma^{-1}d\gamma)=0,
\]
i.e., each connection $A$ is equivalent to a gauge field differing
from a given $A_0$ only by a co-closed, imaginary valued 1-form.
\end{lemma}
\begin{proof}
The Hodge decomposition assures that
\[
A-A_0=\eta+df+d^*\go,
\]
where $f\in L^2_2(M,i\R)$, $\go\in L^2_2(M,\gL^2iT^*M)$, and
$\eta\in i\gO(M)$ is harmonic. Let
\[
\gamma:=\exp\big(-\lfrac{f}{2}\big)\in L^2_2(M,\U_1).
\]
Then $2\gamma^{-1}d\gamma=-df$ and therefore,
\[
d^*(A-A_0+2\gamma^{-1}d\gamma)=d^*(\eta+d^*\go)=0.\qedhere
\]
\end{proof}

\begin{prop}\label{coulomb:gauge}
Let $A_0$ be a fixed gauge field of Sobolev class $k_0$. If
$(\psi,A)\in\cC_1$ is a monopole such that $a:=A-A_0$ is
co-closed, then
\[
(\psi,A)\in L^2_{k_0}(M,S)\times\cA_{k_0}.
\]
Moreover, for all $k\in\N$ and $p\ge 2$ such that $L_{k_0}^2$
embeds in $L_{k+1}^p$, the following inequalities hold.
\begin{equation}\label{coulomb:gauge:ell}
\begin{split}
\|\psi\|_{L_{k+1}^p}&\le
\const\cdot\Big(\|c(a)\psi\|_{L_k^p}+\|\psi\|_{L_k^p}\Big),\\
\|a\|_{L_{k+1}^p}&\le \const\cdot\Big(
\big\|q(\psi)-*F_{A_0}\big\|_{L_k^p}+\|Pa\|_{L_{k+1}^p}\Big),
\end{split}
\end{equation}
where $P:L^2(M,iT^*M)\to L^2(M,iT^*M)$ denotes the
$L^2$-orthogonal projection onto the space of harmonic 1-forms.
\end{prop}

\begin{proof}
The proof is an impressive application of the so-called {\em
elliptic bootstrap} technique. Since $d^*a=0$, we can reformulate
the Seiberg-Witten equations for $(\psi,A)$ in the following way
\begin{equation}\label{coulomb:gauge:1}
\begin{split}
\cD_{A_0}\psi&=-\lfrac{1}{2}c(a)\psi\,, \\
* (d+d^*)a&=\lfrac{1}{2}q(\psi)-* F_{A_0}\,.
\end{split}
\end{equation}
There is a Sobolev embedding of $L^2_1$ in $L^6$. Therefore, $a$
and $\psi$ lie in $L^6$ and the H\"older inequality shows that
$c(a)\psi\in L^3$. As $\cD_{A_0}$ is an elliptic operator,
elliptic regularity for $L^p$ Sobolev spaces applied to the first
line of \eqref{coulomb:gauge:1} guarantees that $\psi\in L_1^3$.
Hence, $\psi\in L^p$ for all $1\le p<\infty$. Employing the
H\"older inequality again, we deduce that $q(\psi)\in L^p$ for
all $p$.

Since $d+d^*:\gO^*\to \gO^*$ is elliptic, we obtain from elliptic
regularity---this time applied to the second equation in
\eqref{coulomb:gauge:1}---that $a\in L_1^p$ whenever $1\le
p<\infty$. Proposition \ref{sob:mult} shows that there is a
Sobolev multiplication $L^p_1\times L^3_1\to L^2_1$ for all $p\ge
2$. Therefore, $c(a)\psi\in L^2_1$.

Again, ellipticity of $\cD_{A_0}$ yields $\psi\in L^2_2$.
Applying Lemma \ref{q:diff}, we deduce that this yields
$q(\psi)\in L^2_2$. Thus, elliptic regularity shows that $a\in
L^2_3$. Using Sobolev multiplication and elliptic regularity in
this manner further on we can prove inductively that $a\in L_k^2$
and $\psi\in L^2_k$ for all $1\le k\le k_0$. Note that $*F_{A_0}$
is in $L^2_{k_0-1}$ since $A_0$ is of Sobolev class $k_0$.
Moreover, $\cD_{A_0}$ can be expressed as an elliptic operator
with smooth coefficients plus a zero order term given by Clifford
multiplication with a 1-form of Sobolev class $k_0$, i.e., a
continuous map $L^2_k(M,S)\to L^2_k(M,S)$ for all $1\le k\le
k_0$. Therefore, the bootstrapping does not cease at an earlier
level.

Let $k\in\N$ and $p\ge 2$ such that $L_{k_0}^2$ embeds in
$L_{k+1}^p$. Combining the elliptic estimate \eqref{ell:est},
i.e.,
\[
\|\psi\|_{L_{k+1}^p}\le
\const\cdot\big(\|\cD_{A_0}\psi\|_{L_k^p}+\|\psi\|_{L_k^p}\big),
\]
with the first equation in \eqref{coulomb:gauge:1}, we readily
obtain the first inequality in \eqref{coulomb:gauge:ell}.
Moreover, the triangle inequality yields
\[
\|a\|_{L_{k+1}^p}\le \|a-Pa\|_{L_{k+1}^p} + \|Pa\|_{L_{k+1}^p}.
\]
The Poincar\'{e} inequality \eqref{poinc} applied to $d+d^*$ shows
that
\[
\big\|a-Pa\big\|_{L^p_{k+1}}\le\const\cdot
\big\|*(d+d^*)(a-Pa)\big\|_{L^p_k}
=\const\cdot\big\|*(d+d^*)a\big\|_{L_k^p}\,
\]
because $(d+d^*)Pa=0$. Inserting the second equation of
\eqref{coulomb:gauge:1} in the above inequalities, we get the
second inequality in \eqref{coulomb:gauge:ell}.
\end{proof}

An immediate consequence of Proposition \ref{coulomb:gauge} and
Lemma \ref{d^asta=0}---when combined with Sobolev embedding---is
the following.
\begin{cor}\label{moduli:smooth}
Every SW-monopole is gauge equivalent to a $C^\infty$-monopole,
i.e., every class in $\cM\subset \cC_1/\cG_2$ has a representative
which is smooth.
\end{cor}

The next task shall be to derive compactness of the moduli space
from the inequalities \eqref{coulomb:gauge:ell}. For this we need
to find a priori estimates for the maximum norms.

\begin{lemma}\label{Pa:bounded}\index{=@$\cH^k(M;g)$, harmonic
forms} Let $\cH^1(M)$ denote the space of harmonic 1-forms and
let $P:L^2(M,iT^*M)\to L^2(M,iT^*M)$ be the $L^2$-orthogonal
projection onto the subspace $i\cH^1(M)$. Then there exists a
constant $C>0$ such that for each $a\in L^2(M,iT^*M)$ we can find
$\gamma\in \cG_2$ such that
\[
\|P(a+2\gamma^{-1}d\gamma)\|_\infty < C.
\]
\end{lemma}
\begin{proof}
The image of $H^1(M;\Z)$ in $H^1(M;\R)$ is a lattice. According to
\eqref{derham}, there is a surjective map
\[
C^\infty(M,\U_1)\to H^1_{dR}(M;2\pi i\Z),\quad \gamma\mapsto
[\gamma^{-1}d\gamma].
\]
Therefore, the image of
\[
C^\infty(M,\U_1)\to i\cH^1(M),\quad \gamma\mapsto
P(2\gamma^{-1}d\gamma),
\]
forms a lattice in the space of imaginary valued harmonic 1-forms.
Given an arbitrary harmonic 1-form $\go$, we can thus find
$\gamma\in C^\infty(M,\U_1)$ such that
\[
\|\go +P(2\gamma^{-1}d\gamma)\|_\infty < C,
\]
where $C$ is a constant depending only on the lattice. Note that
we are using that $\cH^1(M)$ is finite dimensional. This implies
the assertion.\qedhere\\
\end{proof}

\noindent\textbf{The key estimate.} The second ingredient to
establish compactness of the moduli space is an a priori estimate
for the norm of the spinor part of a monopole. We need the
following result.

\begin{lemma}\label{kato:ineq}
Let $\psi$ be a twice continuously differentiable spinor, for
example, $\psi\in L^2_4(M,S)$. Then for every $C^2$-gauge field
$A$,
\[
\gD_g|\psi|^2\le 2\Re\Scalar{(\nabla^A)^*\nabla^A\psi}{\psi},
\]
where $\gD_g:=d^*d$ denotes the Laplacian of the Riemannian
manifold $(M,g)$.
\end{lemma}
\begin{proof} At an arbitrary point $x$, we choose a normal frame
$(e_1,e_2,e_3)$ with dual co-frame $(e^1,e^2,e^3)$. Using the
fact that $\nabla^A$ is compatible with the metric, we deduce that
at the point $x$,
\begin{align*}
\gD_g|\psi|^2 &= -* d* d\scalar{\psi}{\psi} = -2 * d
\Re\scalar{\nabla^A_{e_i}\psi}{\psi}* e^i
\\
&=-2\Re\Big(\scalar{\nabla^A_{e_j}\nabla^A_{e_i}\psi}{\psi} +
\scalar{\nabla^A_{e_i}\psi}{\nabla^A_{e_j}\psi}\Big)*(e^j\wedge
* e^i) \\
&=-2\sum_i\Re\scalar{\nabla^A_{e_i}\nabla^A_{e_i}\psi}{\psi} -
2\sum_i |\nabla^A_{e_i}\psi|^2\,.
\end{align*}
Note that we have employed the relation $e^j\wedge * e^i =
\gd^{ji}dv_g$. Recall that in a normal frame at $x$,
\[
(\nabla^A)^*\nabla^A \psi
=-\sum_i\nabla^A_{e_i}\nabla^A_{e_i}\psi\,.
\]
In combination with the above computations this implies the
assertion.
\end{proof}

\begin{prop}\label{key:prop}
Suppose that $(\psi,A)$ is a monopole which is at least $C^2$.
Then
\begin{equation}\label{key:estimate}
\|\psi\|_\infty^2\le \max\big\{0,\max_{x\in M}-2s_g(x)\big\}
=\max\big\{0,-2\min_{x\in M}s_g(x)\big\},
\end{equation}
where $s_g$ denotes the scalar curvature of the Riemannian
manifold $(M,g)$.
\end{prop}

\begin{proof}
Combining Lemma \ref{kato:ineq} with the Weitzenb\"{o}ck formula
\ref{weitzenboeck}, we obtain
\begin{align*}
\gD_g|\psi|^2 &\le 2\cdot\Re\scalar{\cD_A^2\psi}{\psi}
-\lfrac{1}{2}s_g|\psi|^2 -\Re\scalar{c(F_A)\psi}{\psi} \\
&=-\lfrac{1}{2}s_g|\psi|^2
-\lfrac{1}{2}\Re\scalar{c\big(q(\psi)\big)\psi}{\psi},
\end{align*}
where we have employed the Seiberg-Witten equations in the last
line.\footnote{Note that we have also used that $c(*
F_A)=c(F_A)$. This relation is an immediate consequence of our
agreement to choose the spin representation in such a way that
$c(dv_g)=-\id$. Recall that Clifford multiplication by 2-forms is
defined via the isomorphism of vector spaces $\gL^\bullet V\cong
\cl(V)$.} Using Proposition \ref{q:prop}, we infer that
\[
\gD_g|\psi|^2 \le -\lfrac{1}{2}s_g|\psi|^2 - |q(\psi)|^2 =
-\lfrac{1}{2}s_g|\psi|^2 -\lfrac{1}{4}|\psi|^4.
\]
Let $x_0\in M$ be a point where $|\psi|^2$ achieves its maximum.
Then according to our sign convention,
\[
\gD_g|\psi|^2(x_0)\ge 0.
\]
Together with the above estimate, this yields
\[
\big(-\lfrac{1}{2}s_g|\psi|^2- \lfrac{1}{4}|\psi|^4\big)(x_0)\ge
0.
\]
We therefore obtain that
\[
|\psi(x_0)|^2=0 \quad\text{ or }\quad |\psi(x_0)|^2\le -2s_g(x_0).
\]
Since $|\psi(x_0)|^2$ is maximal, this proves the proposition.
\end{proof}
This result allows two immediate conclusions.

\begin{cor}
Suppose $(M,g)$ is a closed, oriented Riemannian 3-manifold whose
scalar curvature is nonnegative, i.e., $s_g \ge 0$. Then every
monopole $(\psi,A)$ fulfills $\psi\equiv 0$. That is, the
Seiberg-Witten moduli space consists only of equivalence classes
of reducible configurations.
\end{cor}

\begin{remark*}
This is a typical example of how gauge theory can be used to
prove nonexistence results: As it shall turn out that the
existence of irreducible monopoles is in a sense independent of
the chosen Riemannian structure, finding an irreducible monopole
with respect to an arbitrary metric prevents $M$ from admitting a
metric of nonnegative scalar curvature. Corresponding statements
in the four dimensional case have turned out to be very useful.
For a brief discussion of the above result's implications, we
refer to Meng \& Taubes \cite{MenTau:SW}.
\end{remark*}

\begin{cor}\label{fin}
Let $M$ be a closed, oriented Riemannian 3-manifold. Then there
exist only finitely many \spinc structures on $M$ for which the
irreducible part of the moduli space is nonempty.
\end{cor}

\begin{proof}
Let $\gs$ be a \spinc structure on $M$ with canonical class
$c(\gs)$. For any point in the moduli space we can find a
representative $(\psi,A)$ which is at least $C^2$. According to
Proposition \ref{q:prop},
\[
|q(\psi)| = \lfrac12 |\psi|^2,
\]
so that the key estimate \eqref{key:estimate} implies that
\[ \psi=0\quad\text{or}\quad
|q(\psi)|\le -\min_{x\in M}s_g(x).
\]
Since $* F_A=\lfrac{1}{2}q(\psi)$, this establishes a bound on
$|F_A|$. According to the Chern-Weil construction, this implies
that the image of $c(\gs)$ in $H^1_{dR}(M;\Z)$ lies in a bounded
subset.

Therefore, only a finite subset of $H^1_{dR}(M;\Z)$ corresponds to
canonical classes of \spinc structures admitting irreducible
monopoles. This proves the assertion since the number of
canonical classes mapped to the same element in $H^1_{dR}(M;\Z)$
is finite. Here, we are using that $H^2(M;\Z)$ is finitely
generated so that there cannot be infinitely many torsion
elements.
\end{proof}

\begin{theorem}\label{thm:comp}\index{moduli space!compactness}
The SW-moduli space $\cM\subset\cC_1/\cG_2$ is sequentially
compact.
\end{theorem}
\begin{proof}
We have to show that any sequence $\big([\psi_n,A_n]\big)$ in
$\cM$ contains a convergent subsequence. Choosing representatives
$(\psi_n,A_n)\in \cC_1$ and a fixed $C^\infty$-gauge field $A_0$,
we let $a_n:=A_n-A_0$. As in Lemma \ref{Pa:bounded} we may apply
gauge transformations to achieve that $(Pa_n)$ is a bounded
sequence in $i\cH^1(M)$. Possibly gauge transforming again, we
may also assume that $d^* a_n=0$. Observe that the second
property can be achieved using a gauge transformation of the form
$\gamma:=\exp(\lfrac{f}{2})$. Therefore, $2\gamma^{-1}d\gamma
=df$ lies in the kernel of $P$ and the fact that $(Pa_n)$ is
bounded remains unaffected. In consequence of the second
condition, Proposition \ref{coulomb:gauge} shows that all
$(\psi_n,A_n)$ are smooth configurations.

\begin{steps}
\item \textit{For $p\ge 2$, the sequence $(\psi_n,a_n)$ is bounded
in $L^p_1(M,E)$:}\\
Due to the key estimate \eqref{key:estimate}, the sequence
$(\psi_n)$ is bounded with respect to $\|.\|_\infty$. Therefore,
the sequence $\big(q(\psi_n)-* F_{A_0}\big)_{n\ge 1}$ is also
$\|.\|_{\infty}$-bounded and hence with respect to $\|.\|_{L^p}$.
Moreover, $(Pa_n)$ is $L_1^p$-bounded because all norms are
equivalent on the finite dimensional space $\cH^1(M)$. Therefore,
the second equation in \eqref{coulomb:gauge:ell} shows that
$(a_n)$ is a bounded sequence in $L_1^p$. Since $(\psi_n)$ is
bounded with respect to $\|.\|_\infty$, we deduce that
$\big(c(a_n)\psi_n\big)$ is bounded in $L^p$. Therefore, from the
first inequality in \eqref{coulomb:gauge:ell} we may infer that
$(\psi_n)$ is bounded in $L^p_1$.

\item \textit{$(\psi_n,a_n)$ is a bounded sequence in
$L^2_2(M,E)$:}\\
Proposition \ref{sob:mult} implies that if $p$ is large enough
(e.g. $p=5$), there is a continuous multiplication
\[
L^p_1\times L^p_1\to L^2_1.
\]
As $(\psi_n,a_n)$ is bounded in $L^p_1$ for each $p\ge 2$, this
shows that $(c(a_n)\psi_n)$ and $(q(\psi_n))$ are bounded with
respect to $\|.\|_{L_1^2}$. Since $(Pa_n)$ is bounded in $L_2^2$,
the right hand sides of the inequalities
\eqref{coulomb:gauge:ell} are bounded. Therefore, $(a_n)$ and
$(\psi_n)$ are bounded in $L_2^2$.
\end{steps}

The Rellich-Kondrachov Theorem \ref{rellich} shows that
$L^2_2(M,E)$ embeds compactly in $L^2_1(M,E)$. Hence, there
exists a subsequence of $(\psi_n,A_n)$ which converges in $\cC_1$
to, say, $(\psi,A)\in\cC_1$. The Seiberg-Witten map $\SW:\cC_1\to
L^2(M,E)$ is continuous which yields $\SW(\psi,A)=0$. Therefore,
$\big([\psi_n,A_n]\big)$ contains a subsequence which converges
in $\cM$ with respect to the induced topology.
\end{proof}

\begin{remark*}
A simple induction shows that the sequence $(\psi_n,a_n)$ in the
above proof is also bounded with respect to $\|.\|_{L_k^2}$ for
any $k\ge 1$: Assume that $k\ge 2$ and that $(\psi_n,a_n)$ is
bounded in $L_k^2$. Then, since there is a continuous
multiplication $L_k^2 \times L_k^2\to L_k^2$, the sequences
$c(a_n)\psi_n$ and $q(\psi_n)$ are also $L_k^2$-bounded. The
inequalities \eqref{coulomb:gauge:ell} then guarantee that
$(\psi_n,a_n)$ is bounded in $L_{k+1}^2$. We shall need this
remark in the next paragraph.\\
\end{remark*}

\noindent\textbf{Choosing different Sobolev orders.} As was
pointed out before, some authors endow the configuration space
and the group of gauge transformations with different Sobolev
structures. We now want to deduce from the above considerations
that this does not affect the structure of the moduli space.\\

\noindent For any $k\ge 1$ we define
\[
\cC_k:=L_k^2(M,S)\times \cA_k\quad\text{ and }\quad
\cG_{k+1}:=L_{k+1}^2(M,\U_1).
\]
Since $k\ge 1$, we deduce from Example \ref{sob:mult:n=3} that
there are continuous multiplications
\[
L^2_{k+1}\times L^2_{k+1}\to L^2_{k+1}\quad\text{ and }\quad
L^2_{k+1}\times L^2_k\to L^2_k.
\]
This guarantees in the same way as before that $\cG_{k+1}$ is a
Banach Lie group acting smoothly on $\cC_k$. We define
\[
\cB_k:=\cC_k/\cG_{k+1}\quad\text{ and }\quad\cM_k:=
(SW^{-1}(0)\cap \cC_k)/\cG_{k+1}.
\]
\begin{lemma}
Let $(\psi,A)\in\cC_k$ and $\gamma\in\cG_2$ such that
$\gamma\cdot(\psi,A)\in\cC_k$. Then $\gamma\in\cG_{k+1}$.
\end{lemma}
\begin{proof} The proof is another application of the elliptic
bootstrap technique and is established by induction on $k$. If
$k=1$, the assertion is trivial. Hence, let $k\ge 2$ and assume
that we have already proved that $\gamma\in\cG_k$. As
$A+2\gamma^{-1}d\gamma \in\cA_k$, we have
\[
\gamma^{-1}d\gamma\in L_k^2(M,T^*M\otimes\C).
\]
Since $k\ge 2$ there is a multiplication $L^2_k\times L^2_k\to
L^2_k$ and we find
\[
d\gamma=\gamma\cdot(\gamma^{-1}d\gamma)\in L^2_k(M,T^*M\otimes\C)
\]
From elliptic regularity we deduce that $\gamma\in
L_{k+1}^2(M,\C)$.
\end{proof}

We may now interpret $\cB_k$ as a subset of $\cB_1$: Since
$\cC_k\subset \cC_1$, taking the quotient of $\cC_1$ modulo
$\cG_2$ induces a map $\pi:\cC_k\to \cB_1$. From the above lemma
we deduce that
\[
\begin{split}
\pi(\psi_1,A_1)=\pi(\psi_2,A_2)&\Longleftrightarrow
\exists_{\gamma\in\cG_{k+1}}:\
(\psi_2,A_2)=\gamma\cdot(\psi_1,A_1)\\
&\Longleftrightarrow [\psi_1,A_1]=[\psi_2,A_2] \text{ in }\cB_k.
\end{split}
\]
Therefore, $\pi$ induces an injective map
\[
\bar \pi:\cB_k\to \cB_1.
\]
Since $\cC_k\subset \cC_1$ is continuous, the map $\bar\pi$ is
also continuous. In particular, $\cB_k$ is Hausdorff. Moreover,
$\bar\pi$ restricts to an inclusion of the moduli spaces,
\[\cM_k\hookrightarrow\cM_1.\]
Corollary \ref{moduli:smooth} shows that this map is, in fact, a
bijection. From the remark we stated after the proof of Theorem
\ref{thm:comp}, one easily establishes that $\cM_k$ is
sequentially compact as well. Therefore, the inclusion of $\cM_k$
in $\cM_1$ is a continuous bijection defined on a sequentially
compact set. This implies continuity of the inverse as well, and
we have the following result:

\begin{cor}
The topology of the moduli space $\cM$ is independent of the
chosen Sobolev orders, i.e., for any $k\ge 2$ the map
\[
\cM_k \hookrightarrow \cM_1
\]
is a homeomorphism.
\end{cor}

\section{Local structure of the moduli space}\label{loc:struc}

We give a brief motivation for some considerations arising at this
point in four dimensional Seiberg-Witten theory as well as in
Yang-Mills theory. Consider the restriction of the Seiberg-Witten
map to the irreducible part of $\cC_1$, i.e., $\SW:\cC_1^*\to
L^2(M,E)$. We suppose that at some point $(\psi,A)\in \cC_1^*$ the
differential $F_{(\psi,A)}$ of $\SW$ is surjective. Under this
assumption, the implicit function theorem ensures that the set of
monopoles near $(\psi,A)$ is a smooth manifold. Its tangent space
at $(\psi,A)$ is given by $\ker F_{(\psi,A)}$. Dividing out the
group action and invoking the slice theorem shows that in this
case a neighbourhood of $[\psi,A]$ in the moduli space is a smooth
manifold which is modelled on the tangent space
\[
T_{(\psi,A)}\cM=\frac{\ker F_{(\psi,A)}}{\im
(G_{(\psi,A)}|_{L_2^2})}.
\]
This space is the first cohomology group of the following complex:
\[
0\to L^2_2(M,i\R)\xrightarrow{G_{(\psi,A)}}L^2_1(M,E)
\xrightarrow{F_{(\psi,A)}} L^2(M,E)\to 0.
\]
In four dimensional Seiberg-Witten theory, the corresponding
complex---being of a slightly different form than here---is
elliptic so that the expected dimension of the moduli space can
be computed as the index of rolled-up elliptic operator. At a
first glimpse, the three dimensional situation is a bit more
complicated since the above complex is not elliptic but has to be
altered as in \eqref{complex}. Nevertheless, by slightly
reformulating the Seiberg-Witten equations and introducing
``virtual" monopoles, the local analysis of the moduli space can
be carried over exactly as in the four dimensional case (see Lim
\cite{Lim:SW}). However, we shall take another approached since
the arguments involved are more intuative from a geometrical point
of view.\\

\noindent\textbf{The Chern-Simons-Dirac functional
revisited.}\index{Chern-Simons-Dirac functional|(} The nature of
the Seiberg-Witten map as the gradient of the Chern-Simons-Dirac
functional yields another possibility to analyse the local
structure of the moduli space than the one via virtual monopoles
mentioned above. This point of view reproduces the
original\footnote{A very clear explanation of Taubes's ideas is
given by P. Kirk in Sec.~3 of \cite{Kir:GT}. Our approach mimics
the arguments given there.} ideas of Taubes \cite{Tau:CI} so that
similarities with instanton theory on 3-manifolds become more
intriguing. However, the essential ingredients---the slice
theorem combined with the implicit function theorem---are the
same in both approaches.

For the time being, we restrict our attention to the irreducible
part of $\cB_1$ and consider the principal $\cG_2$-bundle
$\cC_1^*\to\cB_1^*$. The assignment
\[
(\psi,A)\mapsto S_{(\psi,A)}=\ker (G_{(\psi,A)}^*|_{L_1^2})
\]
defines a smooth subbundle of the tangent bundle of $\cC_1^*$.
This follows from the fact that $G_{(\psi,A)}$ is injective with
injectively elliptic first order term. Moreover, the bundle is
$\cG_2$-invariant, i.e., for all $\gamma\in\cG_2$,
\[
\gamma\cdot\ker (G_{(\psi,A)}^*|_{L_1^2})= \ker(
G_{\gamma\cdot(\psi,A)}^*|_{L_1^2}).
\]
From part (ii) of Corollary \ref{slice:cor} it becomes clear that
this subbundle is the pullback of the tangent bundle of $\cB_1^*$
to $\cC_1^*$. This in mind, we can relate objects defined on
$T\cB_1^*$ with objects on the tangent bundle of $\cC_1^*$.

\begin{lemma}\label{grad:lem}
The section $\SW$ is the pullback of the Chern-Simons-Dirac
functional's $L^2$-gradient on $\cB_1^*$. In particular, the
irreducible part of the moduli space is exactly the set of
critical points of $\csd:\cB_1^*\to S^1$.
\end{lemma}

\begin{remark*}
As we have already noted, $\SW$ is only a gradient vector field
in a weak sense since it takes values in the $L^2$-completion of
the tangent bundle of $\cC_1^*$. Hence, we have also to consider
the $L^2$-completion of the tangent bundle of $\cB_1^*$.
According to Corollary \ref{L^2:compl}, the $L^2$-completion of
$\ker (G_{(\psi,A)}^*|_{L_1^2})$ coincides with $\ker
G_{(\psi,A)}^*$.
\end{remark*}

\begin{proof}[Proof of Lemma \ref{grad:lem}]
We have to ascertain that $\SW$ is a $\cG_2$-equivariant map with
values in the subbundle $\ker G^*$. The equivariance property has
been proved earlier. For a smooth configuration $(\psi,A)$, we
have
\begin{equation*}
G_{(\psi,A)}^*\big(\SW(\psi,A)\big)=2d^*\big(\lfrac{1}{2}q(\psi)-
* F_A\big)- i\Im\scalar{\cD_A\psi}{\psi} =0
\end{equation*}
since Proposition \ref{dastq} implies that
$d^*q(\psi)=i\Im\scalar{\cD_A\psi}{\psi}$. Since smooth
configurations lie dense in $\cC_1$ the assertion follows from
continuity of $\SW$.
\end{proof}

The Hessian of $\csd:\cC_1^*\to\R$ at a point $(\psi,A)$ is given
by the differential $F_{(\psi,A)}:L^2_1(M,E)\to L^2(M,E)$ of
$\SW$. To obtain the pullback of the Hessian on $\cB_1^*$ we have
to project $F_{(\psi,A)}$ to the subbundle $\ker G^*$ since this
means taking the induced covariant derivative of the gradient
$\cB_1^*\to T\cB_1^*$. Again, we have to account for the fact that
$F_{(\psi,A)}$ maps $L_1^2(M,E)$ to $L^2(M,E)$. Then the pullback
of Hessian of $\csd:\cB_1^*\to S^1$ is the unbounded operator in
$\ker G_{(\psi,A)}^*$ given by
\begin{equation}\label{H}\index{>@$H_{(\psi,A)}$, Hessian of $\csd$}
H_{(\psi,A)}:=\Proj_{\ker G^*}\circ F_{(\psi,A)},\quad
\dom(H_{(\psi,A)}):= \ker ( G_{(\psi,A)}^*|_{L_1^2}).
\end{equation}

In consistency with the terminology in finite dimensional Morse
theory, we now define:

\begin{dfn}
An irreducible Seiberg-Witten monopole is called {\em
non-degenerate} if the Hessian is an invertible operator
\[
H_{(\psi,A)}:\ker (G_{(\psi,A)}^*|_{L_1^2})\to\ker G^*_{(\psi,A)}
\]
Otherwise, it is called {\em degenerate}.
\end{dfn}
Clearly, this definition only depends on the gauge equivalence
class of the monopole $(\psi,A)$. An immediate consequence of the
inverse function theorem and Corollary \ref{slice:cor} is:

\begin{prop}\label{irr:nd}\index{moduli space!irreducible part}
Let $(\psi,A)\in\cC_1^*$ be an irreducible, non-degenerate
monopole. Then its gauge equivalence class $[\psi,A]$ is an
isolated point of the moduli space $\cM(\gs)$.
\end{prop}\index{Chern-Simons-Dirac functional|)}

The Hessian $H_{(\psi,A)}$ is not very tractable since, for
example, the Hilbert space in which it is defined depends on the
point $(\psi,A)$. Moreover, it is not yet clear how to assign a
``Hessian" to reducible configurations. Yet, as we are ultimately
only interested in the spectral properties of $H_{(\psi,A)}$, we
will relate the Hessian to the elliptic operator $T_{(\psi,A)}$ we
considered at the end of Chapter \ref{SW:mon}.

For the time being, we shall drop the reference to the base point
$(\psi,A)$ to simplify notation. We recall that,
\[
T:=(F + G,G^*):L^2_1(M,E\oplus i\R)\to L^2(M,E\oplus i\R).
\]
Hodge decomposition yields that
\begin{equation}\label{Hodge:G}
L^2_1(M,E)=\ker (G^*|_{L_1^2})\oplus\im (G|_{L_2^2})\quad \text{
and }\quad L^2(M,E)=\ker G^*\oplus\im G.
\end{equation}
With respect to this we now extend the Hessian $H$ an operator in
$L^2(M,E\oplus i\R)$, with domain $L^2_1(M,E\oplus i\R)$, by
letting
\begin{equation}\label{Tilde:H}
\Tilde{H}:=\left(
\begin{array}{ccc}
H &0 &0 \\
0 &0 & G \\
0 &G^*&0
\end{array}\right).
\end{equation}

\begin{lemma}\label{H=F}
If $(\psi,A)$ is an irreducible monopole, then $\Tilde
H_{(\psi,A)}$ coincides with the operator $T_{(\psi,A)}$.
Moreover, $(\psi,A)$ is non-degenerate if and only if
$T_{(\psi,A)}$ is invertible.
\end{lemma}
\begin{proof}
According to the definition of $T$ and \eqref{H}, it suffices to
show that the operator $F:L_1^2(M,E)\to L^2(M,E)$ satisfies
\begin{equation}\label{H=F:1}
F=\Proj_{\ker G^*}\circ\, F\circ \Proj_{\ker G^*}.
\end{equation}
As we have noticed before, $F\circ G|_{L_2^2}=0$ whenever
$(\psi,A)$ is a monopole. On the one hand this implies that
$F=F\circ \Proj_{\ker G^*}$ and on the other hand,
\[
\im (G|_{L_2^2}) \subset \ker F\subset \ker F^*
\]
since $F$ is symmetric. Hence, $\im F \subset (\im G)^\perp = \ker
G^*$. Together, we get \eqref{H=F:1}. Then the second assertion
follows from the next lemma.
\end{proof}

\begin{lemma}\label{G:mat}
The eigenvalues of
\[
\Tilde G:= \begin{pmatrix} 0 & G \\ G^* &0\end{pmatrix} : \im
(G|_{L_2^2})\oplus L^2_1(M,i\R)\longrightarrow \im G\oplus
L^2(M,i\R).
\]
form a symmetric subset of $\R$, not containing 0. In particular,
$H$ is invertible if and only if $\Tilde H$ is.
\end{lemma}

\begin{proof}
Clearly, the operator $\Tilde G$ is symmetric with respect to the
induced $L^2$ scalar product. Therefore, the set of eigenvalues
is a subset of $\R$. If $\gl$ is an eigenvalue with corresponding
eigenvector $\big((\gf,a),f\big)$,
\[
\Tilde G\big((\gf,a), - f\big) = -\gl\cdot\big((\gf,a), - f\big) .
\]
Hence, $-\gl$ is also an eigenvalue, and that shows the set
of eigenvalues is symmetric. Moreover, $G$ is injective since
$(\psi,A)$ is irreducible. This implies that $\Tilde G$ is
injective as well. Since $\Tilde H$ is the direct sum of $H$ and
$\Tilde G$, the second assertion follows.
\end{proof}

\begin{remark}\label{Rem:H:SF}
We will see in Corollary \ref{H:prop} that $\Tilde H$ is a
self-adjoint operator depending smoothly on $(\psi,A)$ and having
compact resolvent. Hence, given a $C^1$-path of irreducible
configuration, Definition \ref{SF:def} of the spectral flow
applies to the operators $\Tilde H$ associated to this path. Now,
Lemma \ref{G:mat} shows that the direct summand $\Tilde G$ gives
no contribution to the spectral flow since the corresponding
spectra are bounded away from 0. It is thus intuitively clear
that, the spectral flow of the Hessian $H$ may be represented by
the spectral flow of $\Tilde H$.
\end{remark}

If $(\psi,A)\in\cC_1^*$ is not a monopole, the equality $\Tilde
H_{(\psi,A)}=T_{(\psi,A)}$ does not necessarily hold. However, we
can say the following:

\begin{lemma}\label{F-H:compact}
For each irreducible configuration $(\psi,A)$, the operator
\[
T_{(\psi,A)}-\Tilde{H}_{(\psi,A)}: L^2_1(M,E\oplus i\R)\to
L^2(M,E\oplus i\R)
\]
is compact and symmetric with respect to the $L^2$ scalar product.
\end{lemma}
\begin{proof}
It obviously suffices to show that, with respect to the
decomposition \eqref{Hodge:G}, the operator
\[
F-\begin{pmatrix} H &0\\
0&0\end{pmatrix}:L^2_1(M,E)\to L^2(M,E)
\]
is compact and symmetric. The latter property is obviously
fulfilled since both, $F$ and $H$, are symmetric. According to
the definition of $H$, we are now to show that the operators
\begin{equation}\label{F-H:compact:1}
F|_{\im G} :\im (G|_{L_2^2}) \to L^2(M,E)
\end{equation}
and
\begin{equation}\label{F-H:compact:2}
(F- \Proj_{\ker G^*}\circ F )|_{\ker G^*}:\ker (G^*|_{L_1^2}) \to
L^2(M,E)
\end{equation}
are compact operators. Note that $G:L_2^2(M,i\R)\to \im
(G|_{L_2^2})$ is an isomorphism because it is injective and has
closed range. Hence, compactness of \eqref{F-H:compact:1} is
equivalent to compactness of
\[
F\circ G|_{L_2^2}: L_2^2(M,i\R)\to L^2(M,E).
\]
A short computation using part (iii) of Proposition \ref{q:prop}
shows that at the point $(\psi,A)$, for all $f\in L_2^2(M,i\R)$,
\[
F\circ G(f)= \big( - f\cdot \cD_A\psi, 0\big).
\]
If $p>3$, there is a continuous Sobolev multiplication
$L_1^p\times L^2 \to L^2$. On the other hand, $L_2^2$ embeds
compactly in $L_1^p$ for all $2\le p<6$. This shows that
$f\mapsto f\cdot \cD_A\psi$ is a compact operator $L_2^2\to L^2$.
Thus, compactness of \eqref{F-H:compact:1} is proved. Regarding
\eqref{F-H:compact:2}, we claim that for $(\gf,a)\in
C^{\infty}(M,E)$,
\begin{equation}\label{F-H:compact:3}
(F- \Proj_{\ker G^*}\circ F)(\gf,a) =\Proj_{\im G}\circ F(\gf,a)
\in L_1^2(M, E).
\end{equation}
For this note first, that $F(\gf,a)\in L_1^2(M,E)$ since $(\gf,a)$
is smooth and $F$ is a differential operator with $L_1^2$
coefficients. Second, projecting an element of $L_1^2(M,E)$ to
$\im G$ results in an element of $\im (G|_{L_2^2})$ because $G$
is injectively elliptic, cf. Corollary \ref{L^2:compl}. Together,
this implies that \eqref{F-H:compact:3} holds for smooth
$(\gf,a)$.
As $L_1^2$ embeds compactly in $L^2$, this shows that the operator
\[
\Proj_{\im G}\circ F : L_1^2(M,E)\to L^2(M;E),
\]
restricted to the subspace $C^\infty(M,E)$ is compact. Since
$C^{\infty}$ is dense in $L_1^2$, it follows that the operator is
compact on the whole domain. This clearly implies compactness of
\eqref{F-H:compact:2}.
\end{proof}

The above result shows that $\Tilde H_{(\psi,A)}$ is obtained via
a relative compact perturbation of $T_{(\psi,A)}$. Using the
corresponding properties of $T_{(\psi,A)}$ (cf. Proposition
\ref{T:prop}), Theorem \ref{relative:comp} implies the following:

\begin{cor}\label{H:prop}
Suppose $(\psi,A)$ is an irreducible configuration. Then the
operator $\Tilde H_{(\psi,A)}$ defines a closed, self-adjoint
Fredholm operator in $L^2(M,E\oplus i\R)$. It has compact
resolvent and thus discrete spectrum. Moreover, the assignment
\[
(\psi,A)\longmapsto \Tilde H_{(\psi,A)}
\]
is smooth with respect to the operator norm topology on
$\sL_{sa}$.
\end{cor}

\begin{cor}\label{H:SF}
Let $(\psi_t,A_t):[a,b]\to \cC_1^*$, be a $C^1$-path of
irreducible configurations such that $(\psi_a,A_a)$ and
$(\psi_b,A_b)$ are monopoles. Then
\[
\SF(T_{(\psi_a,A_t)})= \SF(\Tilde H_{(\psi_t,A_t)}).
\]
\end{cor}
\begin{proof}
Consider the homotopy
\[
[a,b]\times [0,1]\to \sL(L_1^2,L^2),\quad (t,s)\mapsto  (1-s)\cdot
T_{(\psi_t,A_t)} + s\cdot\Tilde H_{(\psi_t,A_t)},
\]
which---due to our assumption---leaves the endpoints fixed. It
follows as in Corollary \ref{H:prop} that this homotopy takes
values in the space $\sL_{sa}$ of closed, self-adjoint operators
in $L^2$ with domain $L_1^2$. Thus, according to Proposition
\ref{SF}, the paths $T_{(\psi_t,A_t)}$ and $\Tilde
H_{(\psi_t,A_t)}$ have the same spectral flow.\qedhere\\
\end{proof}

\noindent\textbf{Reducible configurations.} Until now we have
restricted our attention to irreducible configurations for the
above geometrical motivation is only meaningful on the manifold
$\cB_1^*$. However, $T_{(\psi,A)}$ is defined independently of
$\psi$ being zero or not. If $A$ is a reducible configuration,
then the explicit formula is
\begin{equation}\label{T:red:explicit}
T_{(0,A)} = \begin{pmatrix} \cD_A &0 &0 \\
                0 &-* d &2d \\
                0 &2d^* &0 \end{pmatrix}.
\end{equation}
Therefore,
\begin{equation}\label{T:red}
\begin{split}
\ker T_{(0,A)} & = \ker \cD_A \oplus \ker(d\oplus d^*) \oplus
\ker d\\
&= \ker \cD_A\oplus H_{dR}^1(M;i\R)\oplus H_{dR}^0(M;i\R).
\end{split}
\end{equation}
\begin{proof}[Proof of \eqref{T:red}]
From \eqref{T:red:explicit} it is clear that
$T_{(0,A)}(\gf,a,f)=0$ if and only if $\cD_A\gf=0$, $2df=* da$,
and $d^* a=0$. According to the Hodge decomposition of
$L^2(M,T^*M)$, we have $\im d\perp \im d^*$. Therefore, since $*
da = -d^*(*a)$, this implies that $2df = * da$ if and only if
both, $df$ and $da$, vanish.
\end{proof}

\begin{remark*}
We shall see in the next chapter that if $A$ is a reducible
monopole, then the summand $H^1_{dR}(M;i\R)$ of $\ker T_{(0,A)}$
represents the ``tangent space" to the reducible part of the
moduli space whereas $\ker \cD_A$ determines whether $[0,A]$ is an
accumulation point for irreducible elements.\\
\end{remark*}

\section{Counting monopoles}\label{count}\index{moduli
space!orientation}

We shall now equip the points of $\cB_1$ with a sign. With respect
to this, the algebraic count of points in the modulo space will be
the number which lies in the center of our interest. Clearly,
this is only meaningful if the number of gauge equivalence
classes of monopoles is finite. Yet, compactness of the moduli
space and the fact that non-degenerate monopoles lie isolated in
the irreducible part of $\cB_1$ suggest that this might indeed be
true---at least modulo some small perturbation of the
Chern-Simons-Dirac functional. How this perturbation has to be
chosen is the topic of Section \ref{inv}.\ref{pert:SW}. For the
time being we shall have to assume that the irreducible part of
the moduli space, $\cM^*(\gs)$, consists only of non-degenerate
monopoles and is finite.

Another expression of the algebraic count of monopoles, which we
shall derive shortly after its initial definition, is reminiscent
of a kind of Euler characteristic associated to $\cB_1^*$. The
subject of Chapter \ref{inv} is, essentially, to prove that this
number is indeed independent of the chosen Riemannian metric (and
the perturbation term), thus yielding a smooth invariant of the
underlying 3-manifold $M$.

After these preliminary remarks let us now equip each
configuration with a sign. Recall from Proposition \ref{T:diff}
that the assignment $(\psi,A)\mapsto T_{(\psi,A)}$ is a smooth
map from $\cC_1$ to $\sL_{sa}$, the space of self-adjoint
operators in $L^2$ with domain $L_1^2$. If $(\psi_t,A_t):[a,b]\to
\cC_1$ is a continuous path of configurations, then the
associated family $\{T_{(\psi_t,A_t)}\}_{t\in[a,b]}$ also depends
continuously on $t$ and we are in the situation of Definition
\ref{OT:def}.

\begin{dfn}\index{>@$\eps(\psi,A)$}\index{moduli space!orientation}
For each configuration $(\psi,A)$ let
\[
\eps(\psi,A):=\eps(T_{(t\psi,A)};\;{\scriptstyle 0\le t\le 1})
\]
be the orientation transport along the family
$\{T_{(t\psi,A)}\}_{t\in[0,1]}$ assigned to the affine path from
$(0,A)$ to $(\psi,A)$.
\end{dfn}

\begin{lemma}\label{eps:prop}
For every configuration $(\psi,A)$ the following holds:
\begin{enumerate}
\item If $A_0$ is an arbitrary connection, then
\begin{equation*}
\eps(\psi,A)= (-1)^{\SF(T_{(0,A_0)+t(\psi,A-A_0})}.
\end{equation*}
\item For every gauge transformation $\gamma\in\cG$,
\[
\eps(\psi,A)=\eps(\gamma\cdot(\psi,A)).
\]
\end{enumerate}
\end{lemma}

\begin{proof}\quad
(i) The linear path from $(0,A)$ to $(\psi,A)$ is homotopic to the
the path going from $(0,A)$ to $(0,A_0)$ and then to $(\psi,A)$.
Using the homotopy invariance and the additivity property of the
spectral flow, we deduce that
\begin{equation}\label{eps:prop:1}
\SF(T_{(t\psi,A)})=\SF(T_{(0,A) + t(0,A_0-A)})+
\SF(T_{(0,A_0)+t(\psi_0,A-A_0)}).
\end{equation}
According to \eqref{T:red:explicit}, the first path entering the
right hand side is the direct sum of a path of complex linear
operators and a constant path. Since we are regarding $L^2(M,S)$
as an $\R$-Hilbert space, the spectral flow of a path of complex
linear operators is always congruent 0 mod 2. Moreover, the
spectral flow of a constant path is 0 so that the first summand in
\eqref{eps:prop:1} is 0 mod 2. Now, using Theorem \ref{OT=SF} and
inserting the considerations we just made we find that
\[
\eps(\psi,A)=(-1)^{\SF(T_{(t\psi,A)})} =
(-1)^{\SF(T_{(0,A_0)+t(\psi_0,A-A_0)})}.
\]

(ii) From the equivariance of $\SW$ one readily deduces that $T$
is $\cG_2$-equivariant which means that for all $(\gf,a,f)\in
L_1^2(M,E\oplus i\R)$,
\[
T_{\gamma\cdot(\psi,A)}(\gf,a,f)=\gamma\cdot
T_{(\psi,A)}\big(\gamma^{-1}\cdot(\gf,a,f)\big).
\]
Recall for this from \eqref{gg:on:tangent} that $\cG_2$ acts only
on the spinor part of $E\oplus i\R$. Using this, one
straightforwardly concludes that the determinant line bundles of
$T_{(t\psi,A)}$ and $T_{\gamma\cdot(t\psi,A)}$ are canonically
isomorphic via the isomorphism induced by
\[
\ker T_{(t\psi,A)}\to \ker T_{\gamma\cdot(t\psi,A)},\quad
(\gf,a,f)\mapsto \gamma\cdot(\gf,a,f)
\]
Then it is immediate that the orientation transports along
$T_{(t\psi,A)}$ and $T_{\gamma\cdot(t\psi,A)}$ coincide.
\end{proof}

\begin{dfn}
Assume that $\cM^*(\gs)$ consists only of non-degenerate
monopoles and is finite. Then we let
\begin{equation}\label{sw0}\index{>@$\sw_0(\gs)$}
\sw_0(\gs):=\sum_{[\psi,A]\in\cM^*(\gs)}\eps(\psi,A),
\end{equation}
where $(\psi,A)$ is an arbitrary representative of $[\psi,A]$.
\end{dfn}

Note that as a consequence of Lemma \ref{eps:prop}, part (ii),
the number $\sw_0$ is well-defined. Moreover, Corollary \ref{fin}
immediately implies that $\sw_0(\gs)$ vanishes for all but
finitely many \spinc structures $\gs$.\\

\noindent\textbf{Morse theoretical
interpretation.}\index{Chern-Simons-Dirac functional!signed count
of indices|(} From part (i) of Lemma \ref{eps:prop} we now deduce
a geometrical motivation for the definition of $\sw_0(\gs)$. For
this let us briefly recall the ideas we need from finite
dimensional Morse theory:

If $f:M\to \R$ is a Morse function on a compact manifold, then
the Euler characteristic of $M$ can be expressed as the signed
count of critical points. Here, the sign associated to a critical
point is given by the parity of its Morse index, i.e., the number
of negative eigenvalues of the respective Hessian.

However, this expression of the Euler characteristic does not
necessarily require an a priori knowledge of all Morse indices.
If we fix one critical point $x_0$, then the Morse index of
another critical point $x$ can be computed via the Morse index at
the point $x_0$ and the {\em difference} of the Morse index of
$x$ relative to the index of $x_0$. In other words, what we need
is an understanding of how the number of negative eigenvalues of
the Hessian changes from one critical point to another, i.e., we
have to consider the spectral flow of the Hessian along paths
connecting two critical points. In this way, the problem of having
to define an ``index" for operators with unbounded spectrum in
the infinite dimensional setting at hand can be overcome. These
ideas in mind, we now give an alternative expression of
$\sw_0(\gs)$.

\begin{prop}
Assume that $\cM^*(\gs)$ is finite and consists only of gauge
equivalence classes of non-degenerate monopoles. Then for every
fixed $(\psi_0,A_0)\in \cM^*(\gs)$,
\begin{equation}\label{sw:alt}
\sw_0(\gs) = \eps(\psi_0,A_0)\cdot \sum_{[\psi,A]\in\cM^*(\gs)}
(-1)^{\SF(T_{(\psi_0,A_0) + t (\psi-\psi_0,A-A_0)})}\,.
\end{equation}
\end{prop}

\begin{proof}
The affine path connecting $(0,A_0)$ with $(\psi,A)$ is homotopic
to the concatenation of the affine path from $(0,A_0)$ to
$(\psi_0,A_0)$ and the affine path from $(\psi_0,A_0)$ to
$(\psi,A)$. Hence,
\[
\SF(T_{(0,A_0)+t(\psi,A-A_0)})= \SF(T_{(t\psi_0,A_0)})
+\SF(T_{(\psi_0,A_0) + t (\psi-\psi_0,A-A_0)}).
\]
Inserting this in Lemma \ref{eps:prop}, part (i), we immediately
get the assertion.
\end{proof}

\begin{remark*}\quad
\begin{enumerate}
\item Corollary \ref{H:SF} and Remark \ref{Rem:H:SF} show that the
fact that \eqref{sw:alt} involves the operator $T$ rather than the
Hessian $H$ does not collide with the ideas of the Morse
theoretical motivation.
\item The term
$\eps(\psi_0,A_0)$ occurring in formula \eqref{sw:alt}
corresponds to fixing the parity of one particular Morse index.
Without this term only the absolute value of $\sw_0(\gs)$ could
be expected to be an invariant. The sign convention of Chen
\cite{Che:CI} is a bit more complicated but seems to be the same
we have chosen. The description in terms of $\eps(\psi,A)$
exhibits a natural choice for this convention---depending,
however, on the way of how to define the spectral flow of a path
whose endpoints are not invertible.
\end{enumerate}
\end{remark*}

\index{Chern-Simons-Dirac functional!signed count of indices|)}

\cleardoublepage

%% file: chap_III.tex
\chapter[Seiberg-Witten Invariants]{Seiberg-Witten Invariants
of 3-Manifolds}\label{inv}

Until now there is one major problem in the definition of
$\sw_0(\gs)$: We cannot guarantee that all critical points of the
Chern-Simons-Dirac functional are non-degenerate. As in finite
dimensional Morse theory we cannot expect that the signed count
of critical points describes the Euler characteristic if there
also exist degenerate ones. Therefore, we are lead to study
perturbations of the Chern-Simons-Dirac functional in order to
obtain non-degenerateness. This is explicitly carried out in
Section \ref{pert:SW}. Before, we include a section about the
abstract setting which lies behind these ideas.

In the remaining parts of this chapter we will then analyse the
behaviour of $\sw_0(\gs)$ under deformation of the Riemannian
metric and the perturbation. It turns out that for manifolds with
$b_1>1$ we achieve topological invariance of the signed count of
monopoles in this way. If $b_1\le 1$, a topological invariant can
also be obtained---at least with some extra effort.

\section{Regular values and perturbed level sets}\label{reg:val}
\index{Fredholm maps|(}

Suppose that $\gF:X\to Y$ is a smooth Fredholm map between
paracompact Banach manifolds, i.e., $\gF$ is a smooth map such
that for every $x\in X$, the differential $D_x\gF:T_xX\to
T_{\gF(x)}Y$ is a Fredholm operator. If $X$ is connected---what
we will henceforth assume---then the function $x\mapsto\ind
D_x\gF$ is constant on $X$. We thus can define the {\em index} of
$\gF$ as this common value.

Let us assume for a moment that $y_0\in Y$ is a regular value of
$\gF$, i.e., $D_x \gF$ is surjective for every $x\in
M:=\gF^{-1}(y_0)$. Then $M$ is either empty or a smooth
submanifold of $X$ with $\dim M= \ind \gF$. In the applications
we have in mind, this assumption is usually not fulfilled. The
following considerations show how to achieve regularity by
perturbing $\gF$.

Suppose that $P$ is an affine Banach space, modelled on a
separable Banach space $E$. Let $\gF$ extend to a $C^m$-map $\Hat
\gF : X\times P \to Y$ with $\Hat\gF(\cdot,p_0)=\gF$ for some
$p_0\in P$. We will usually refer to $P$ as the {\em perturbation
space} and call $\Hat \gF$ the {\em perturbation map}. In
addition, we require that $\Hat\gF(\cdot,p):X\to Y$ is Fredholm
for each $p\in P$. As $P$ is connected, the index of
$\Hat\gF(\cdot,p)$ remains to be equal to the index of $\gF$ as
$p$ varies. Moreover, a reasonable perturbation has to be chosen
in such a way that $y_0$ is a regular value of $\Hat \gF$, which
we will assume in the following. Then the level set
\[
\widehat M:=\Hat\gF^{-1}(y_0)
\]
is either empty or a (possibly infinite dimensional)
$C^m$-submanifold of $X\times P$.

\begin{prop}\label{pert:levelset}\index{Fredholm maps!perturbed
level sets} The projection map $\pi:\widehat M\to P$ is a
$C^m$-Fredholm map, with index equal to $\ind\gF$. Thus for each
regular value $p$ of $\pi$, the set $M_p:=\pi^{-1}(p)$ defines an
$\ind \gF$-dimensional $C^m$-submanifold of $X$. Moreover, $p$ is
a regular value of $\pi$ if and only if $y_0$ is a regular value
of the map $\Hat\gF(\cdot,p)$.
\end{prop}
\begin{proof}
The asserted differentiability properties are obvious so that we
have only to compute the index of the differential at a point
$(x,p)\in \widehat M$. Consider the maps $\Hat\gF_p:=\Hat
\gF(\cdot,p): X\to Y$ and $\Hat \gF_x:= \Hat \gF(x,\cdot): P\to Y$
obtained by fixing $p$ and $x$ respectively. Clearly, if we prove
that there exist algebraic isomorphisms
\begin{equation}\label{pert:levelset:1}
\ker D_{(x,p)}\pi \cong \ker D_x(\Hat\gF_p)\quad\text{and}\quad
\coker D_{(x,p)}\pi \cong \coker D_x(\Hat\gF_p)\,,
\end{equation}
the Fredholm property of $\pi$ and the assertion about its index
are immediate. To prove \eqref{pert:levelset:1} note first of all
that the tangent space of $\widehat M$ at $(x,p)$ is given by
\begin{equation}\label{pert:levelset:2}
T_{(x,p)}\widehat M = \ker D_{(x,p)}\Hat \gF = \ker
\big(D_p(\Hat\gF_x) + D_x(\Hat\gF_p)\big).
\end{equation}
From this the first equation of \eqref{pert:levelset:1} follows
because
\[
\begin{split}
\ker D_x(\Hat\gF_p) &\cong \bigsetdef{(v,0)\in T_xX\oplus
T_pP}{D_x(\Hat\gF_p)(v)=0}\\
&= \bigsetdef{(v,w)\in T_{(x,p)}\widehat M}{w=0} =\ker
D_{(x,p)}\pi\,.
\end{split}
\]
Next observe that \eqref{pert:levelset:2} implies
\begin{align*}
\im D_{(x,p)} \pi &= \bigsetdef{w\in T_p P}{\exists_{v\in T_xX}:
(v,w)\in T_{(x,p)}\widehat M} \\
&= \bigsetdef{w\in T_p P}{D_p(\Hat \gF_x) (w)\in \im
D_x(\Hat\gF_p)}\,.
\end{align*}
Hence, there is an algebraic isomorphism
\[
\frac{\im D_{(x,p)}\pi}{\ker D_p(\Hat\gF_x)}\cong \im
D_p(\Hat\gF_x) \cap \im D_x(\Hat\gF_p)\,.
\]
Furthermore, there always exists an abstract isomorphism
\[
\frac{T_pP}{\ker D_p(\Hat\gF_x)}\cong \im D_p(\Hat\gF_x)\,.
\]
Together, these isomorphisms imply that
\[
\frac{T_pP}{\im D_{(x,p)}\pi} \cong \frac{\big(\frac{T_pP}{\ker
D_p(\Hat\gF_x)}\big)}{\big(\frac{\im D_{(x,p)}\pi}{\ker
D_p(\Hat\gF_x)}\big)} \cong \frac{\im D_p(\Hat\gF_x)}{\im
D_p(\Hat\gF_x) \cap \im D_x(\Hat\gF_p)}\,.
\]
Invoking surjectivity of $D_{(x,p)}\Hat\gF$, we can decompose
$T_{y_0} Y$ as
\[
T_{y_0} Y = \im D_p(\Hat\gF_x) + \im D_x(\Hat\gF_p) \cong
\frac{\im D_p(\Hat\gF_x)}{\im D_p(\Hat\gF_x) \cap \im
D_x(\Hat\gF_p)} \oplus \im D_x(\Hat\gF_p)\,.
\]
Finally, we get the following chain of isomorphisms
\[
\coker D_{(x,p)}\pi = \frac{T_pP}{\im D_{(x,p)}\pi} \cong
\frac{\im D_p(\Hat\gF_x)}{\im D_p(\Hat\gF_x) \cap \im
D_x(\Hat\gF_p)} \cong \coker D_x(\Hat\gF_p)\,
\]
This proves the second part of \eqref{pert:levelset:1}. Hence,
$\pi:\widehat M\to P$ is a Fredholm map of index $\ind
D_x(\Hat\gF_p)$ at the point $(x,p)$. By our assumptions this
index equals the index of $\gF$.

Suppose now that $p$ is a regular value of $\pi$. Then according
to the above, $\coker D_x(\Hat \gF_{p_0})=0$ for all $x\in M_p$.
Hence, $D_x(\Hat\gF_p)$ is surjective whenever $x\in M_p$. This
shows that $y_0$ is a regular value of $\Hat\gF_p$.
\end{proof}

The importance of the above proposition becomes obvious when we
combine it with an infinite dimensional version of Sard's Theorem
due to Smale. Recall that a subset of a topological space is
called {\em generic} (or equivalently {\em of second category})
if it is the countable intersection of open and dense subsets.

\begin{theorem}\label{sard-smale}{\rm(cf. \cite{Sma:Sard-Smale},
Thm.~1.3)}.\index{Fredholm maps!Sard-Smale Theorem} Let $\pi:M\to
P$ be a $C^m$-Fredholm map between two paracompact Banach
manifolds, and suppose that $m>\max\{0,\ind \pi\}$. Then the set
of regular values of $\pi$ is a generic subset of $P$. In
particular, for a generic choice of $p\in P$, the level set
$\pi^{-1}(p)$ is either empty or a $C^m$-submanifold of $M$ with
dimension equal to $\ind \pi$.
\end{theorem}

The well-known Baire category Theorem states that a generic
subset of a complete metric space is necessarily dense. We can
thus combine Theorem \ref{sard-smale} with Proposition
\ref{pert:levelset} to find arbitrarily small
perturbations such that the level sets are regular.\\

\noindent\textbf{Parametrized level sets.} We now consider a
family of $C^m$-Fredholm maps $\{\gF^g:X\to Y\}$ depending
smoothly on an additional parameter $g\in R$, where $R$ denotes a
connected Banach manifold. In our applications---where for
example, $R$ is a completion of the space of all Riemannian
metrics on a fixed 3-manifold---we want to compare the level sets
$M(g):=(\gF^{g})^{-1}(y_0)$ for different parameters. For this
let $g_1$ and $g_{-1}$ be two distinct elements of $R$ and let
$g=g_t:[-1,1]\to R$ be a smooth path connecting them. Then the
corresponding level sets at the endpoints are related by a
cobordism given by the $y_0$-level set of
\begin{equation}\label{param:levelset:0}
\Psi:X\times [-1,1] \to Y,\quad \Psi(x,t):= \gF^{g_t}(x).
\end{equation}
Since
\begin{equation}\label{param:levelset:1}
D_{(x,t)}\Psi=D_x(\gF^{g_t}) + D_t( \gF^g(x)),
\end{equation}
the map $\Psi$ is a $C^m$-Fredholm map of index $1+\ind \gF$.
This is because the first summand of \eqref{param:levelset:1} is a
Fredholm operator $T_xX\oplus\R\to T_{\Psi(x,t)}Y$ of index
$1+\ind \gF^{g_t}$, whereas the second term is rank 1 so that it
does not affect the Fredholm index.

As in the preceding paragraph, the idea is now that we get
regularity of the $y_0$-level set of $\Psi$ if we consider small
perturbations. For this we consider a smooth family of
perturbation maps
\[
\Hat \gF^g:X\times P\to Y,\quad g\in R,
\]
such that each $\Hat \gF^g$ is a $C^m$-Fredholm map with regular
value $y_0$. Applying the Sard-Smale Theorem to the projections
$(\Hat\gF^{g_i})^{-1}(y_0)\to P$, we find parameters $p_1$ and
$p_{-1}$ in $P$ such that $y_0$ is a regular value of $\Hat
\gF_{p_i}^{g_i}:=\Hat\gF^{g_i}(\cdot,p_i)$ for both
$i\in\{-1,1\}$. To relate the corresponding level sets
\[
M_{p_i}(g_i):=(\Hat\gF_{p_i}^{g_i})^{-1}(y_0),\quad i\in\{-1,1\},
\]
by a regular cobordism we define a new perturbation space by
letting
\[
\widehat P:=\bigsetdef{p:[-1,1]\to P} {p \text{ is } C^m \text{
and } p(i)=p_i \text{ for } i\in\{-1,1\}}.
\]
Since $P$ is an affine space, modelled on some separable Banach
space $E$ the set $\widehat P$ is also an affine space, modelled
on the Banach space
\[
\bigsetdef{w:[-1,1]\to E}{w \text{ is } C^m \text{ and } w(i)=0
\text{ for } i\in\{-1,1\}}.
\]
Note that this is indeed a Banach space with respect to uniform
convergence of all derivatives of $w$ up to order $m$. In the
spirit of the last paragraph the map $\Psi$ given in
\eqref{param:levelset:0} is now perturbed by
\begin{equation*}
\Hat \Psi:X\times [-1,1]\times \widehat P \to Y,\quad
\Hat\Psi(x,t,p):= \Hat \gF^{g_t}(x,p_t).
\end{equation*}
To apply Proposition \ref{pert:levelset} we have to ascertain the
following:
\begin{lemma}
Let $\Hat \Psi$ be defined as above. Then the point $y_0$ is a
regular value of $\Hat \Psi$.
\end{lemma}
\begin{proof}
Let $(x,t,p)\in \Hat\Psi^{-1}(y_0)$ and consider
\[
D_{(x,t,p)}\Hat\Psi (v,s,w) = D_{(x,p_t)}(\Hat\gF^{g_t})(v,w_t) +
s\cdot \lfrac{d}{dt}\big|_t\Hat \gF^{g_t}\big(x,p_t\big).
\]
Since $y_0$ is a regular value of $\Hat\gF^{g_t}$ for any $t$,
the first operator on the right hand side is surjective.
Therefore, $(x,p,t)$ is a regular point of $\Hat \Psi$ provided
that $(v,w_t)$ can be chosen arbitrarily. This is, however, only
true if $t\notin\{-1,1\}$ since $w$ is subject to the condition
$w_i=0$ for $i\in\{-1,1\}$. But in this case,
\[
D_{(x,p_i)}(\Hat\gF^{g_i})(v,0)= D_x (\gF^{g_i}_{p_i})(v)
\]
which is already surjective due to the choice of $p_{-1}$ and
$p_1$.
\end{proof}

As a consequence of Proposition \ref{pert:levelset}, the
projection map
\[
\Pi:\Hat\Psi^{-1}(y_0)\longrightarrow \widehat P
\]
is a $C^m$-Fredholm map with index equal to $1+\ind\gF$. We thus
can apply Sard-Smale again, deducing that for a generic path
$p=p_t\in \widehat P$, the parametrized level set
\[
\widehat M_p(g):=\Pi^{-1}(p)\cong\bigcup_{t\in
[-1,1]}\big(M_{p_t}(g_t)\times \{t\}\big) \subset X\times [-1,1]
\]
is either empty or an $(1+\ind \gF)$-dimensional
$C^m$-submanifold of $X\times [-1,1]$. Summarizing the above
considerations, we have proved:

\begin{prop}\label{param:levelset}
Let $g=g_t:[-1,1]\to R$ be a smooth path. Suppose that $p_{-1}$
and $p_1$ are chosen in such a way that $y_0$ is a regular value
of $\Hat\gF_{p_i}^{g_i}$ for $i\in\{-1,1\}$. Then for a generic
path $p=p_t\in \widehat P$, the parametrized level set
\[
\widehat M_p(g):=\bigcup_{t\in [-1,1]}\big(M_{p_t}(g_t)\times
\{t\}\big) \subset X\times [-1,1]
\]
is either empty or a $C^m$-submanifold of dimension $(1+\ind
\gF)$.
\end{prop}

\begin{remark}\label{transv:section}\index{Fredholm maps!Fredholm
section} As was pointed out before, we apply the preceding
results to the $L^2$-gradient vector field of the
Chern-Simons-Dirac functional. Therefore, we have to generalize
the above considerations slightly as the gradient vector field is
not a Fredholm map between Banach manifolds but a {\em Fredholm
section} of a bundle of Banach spaces, i.e., a section which in
local trivializations can be represented by Fredholm maps. We
will now describe the changes to be made:

Given a smooth Fredholm section $\gF:X\to V$ of a bundle of
Banach spaces $V\to X$, where $X$ denotes a connected,
paracompact Banach manifold, consider the intrinsic differential
$D_x\gF:T_xX \to V_x$. If this is surjective at any zero $x$ of
$\gF$, then the section $\gF$ is called {\em transversal to the
zero section}. In this case, the zero set $\gF^{-1}(0)$ is a
smooth submanifold of $X$ of dimension equal to $\ind \gF$. As
before, we might have to perturb $\gF$ to obtain transversality.
Similarly, this is done by studying an extension map $\Hat \gF$,
which is in this case a section of the pullback bundle
$\pr_1^*V\to X\times P$, where $P$ is some perturbation space. If
we can achieve that $\Hat \gF$ is transversal to the zero section
of $\pr_1^*V\to X\times P$, then the union of all perturbed zero
sets $\widehat M:=\Hat\gF^{-1}(0)$ is a smooth submanifold of
$X\times P$, and Proposition \ref{pert:levelset} and
Prop.\ref{param:levelset} continue to hold also in this setting.
\end{remark}\index{Fredholm maps|)}

\section{The perturbed Seiberg-Witten equations}\label{pert:SW}

Using the preceding section as a guideline, we now consider
perturbations of the Chern-Simons-Dirac functional to obtain
transversality of the Seiberg-Witten map. To put it another way,
we want to produce a moduli space consisting solely of
non-degenerate critical points. Moreover, the perturbed moduli
space should preferably contain only irreducible points. As we
shall see, there are topological obstructions to the latter which
lie encoded in the first Betti number.

To begin with, we have to specify an appropriate set of
perturbations. Let $Z^2_k(M;i\R)$ denote the space of pure
imaginary valued, closed\footnote{Many authors choose co-closed
1-forms of an appropriate Sobolev class as perturbations which is
in more agreement with 4-dimensional Seiberg-Witten theory. We
follow Lim's approach in \cite{Lim:SW} since the space of closed
2-forms does not depend on the metric which is of some
convenience later.} 2-forms of some fixed Sobolev class $k\ge 2$.

\begin{dfn}\index{Chern-Simons-Dirac functional!perturbed}
Let $(M,g)$ be a closed, oriented Riemannian 3-manifold equipped
with a \spinc structure $\gs$. For $\eta\in Z^2_k(M;i\R)$ we
define the {\em $\eta$-perturbed} Chern-Simons-Dirac functional by
\begin{equation*}
\begin{split}
\csd_\eta&(\psi,A):=\csd(\psi,A)+\int_M (A-A_0)\wedge \eta \\
&=\frac{1}{2}\int_M\Big(\scalar{\psi}{\cD_A\psi}dv_g +
(A-A_0)\wedge (F_A+F_{A_0}+2\eta)\Big),
\end{split}
\end{equation*}
where $A_0$ is a fixed element of $\cA(\gs)$.
\end{dfn}
Using the computations at the end of Chapter \ref{SW:mon} as a
pattern, one finds that the $L^2$-gradient of $\csd_\eta$ is
given by\footnote{\label{convention}In this chapter we are going
to drop the reference to the Sobolev class of configurations and
gauge transformations, i.e., we write $\cC$ instead of $\cC$ and
so on.}
\begin{equation}\label{pert:SW:map}
\begin{split}
\SW_\eta:\cC(\gs)&\longrightarrow L^2(M,E) \\
(\psi,A)&\longmapsto \SW(\psi,A) - (0,*\eta).
\end{split}\end{equation}

\begin{dfn}
Let $(M,g)$ be a closed, oriented Riemannian 3-manifold equipped
with a \spinc structure $\gs$. For $\eta\in Z^2_k(M;i\R)$ we call
$(\psi,A)\in \cC(\gs)$ an {\em $\eta$-monopole} if it is a
critical point of $\csd_\eta$, i.e., if it solves the {\em
$\eta$-perturbed Seiberg-Witten equations}:
\begin{equation}\label{pert:SW:eqn}\index{Seiberg-Witten
equations!perturbed}
\fbox{$\begin{array}{rcl}
\cD_A\psi&=&0 \\
* (F_A+\eta) &= &\lfrac{1}{2}q(\psi)
\end{array}$}\;.\end{equation}
The moduli space of $\eta$-monopoles modulo gauge equivalence is
denoted by
$\cM_\eta(\gs)\subset\cB(\gs)$.\index{>@$\cM_\eta(\gs)$,
perturbed moduli space}\index{moduli space!perturbed}
\end{dfn}

Note that $\cM_\eta(\gs)$ is well-defined since $\SW_\eta$ is
$\cG$-equivariant. This is because the group of gauge
transformations acts only on the spinor part of $L^2(M,E)$ so
that equivariance of $\SW_\eta$ follows from equivariance of
$\SW$. Moreover, the section $\SW_\eta$ restricts to a section of
the bundle $\ker G^* \to \cC^*$ because $\eta$ is closed and
therefore,
\[
G_{(\psi,A)}^*\circ \SW_\eta(\psi,A)= G_{(\psi,A)}^*\circ
\SW(\psi,A) - 2d^*(*\eta)=0.
\]
As in Section \ref{moduli}.\ref{loc:struc}, we shall usually work
equivariantly on the bundle $\ker G^*\to \cC^*$ instead of
regarding $\SW_\eta$ as an $L^2$-gradient vector field on
$\cB^*$. In this context, the covariant derivative of $\SW_\eta:
\cC\to \ker G^*$ is again the index 0 Fredholm operator
\[
H_{(\psi,A)}:\ker (G_{(\psi,A)}^*|_{L_1^2})\to\ker G_{(\psi,A)}^*
\]
as defined in \eqref{H}. For this note that the dependence on
$\eta$ vanishes when we differentiate $\SW_\eta$.\\

\noindent\textbf{Regularity of the perturbed moduli space.}
Following the guideline of Section \ref{reg:val}, our next task
is to establish that the perturbation map\index{>@$\widehat{\SW}$,
perturbation map}
\[
\widehat{\SW}:\cC^*\times Z^2_k(M;i\R)\to \ker G^*
,\quad\widehat{\SW}\big((\psi,A),\eta\big):=\SW_\eta(\psi,A)
\]
is transversal to the zero section. It is straightforward to see
that at a point $\big((\psi,A),\eta\big)$, its derivative is
given by the index 1 Fredholm operator
\begin{equation}\label{pert:map:der}\begin{split}
\ker (G_{(\psi,A)}^*|_{L_1^2})\oplus Z^2_k(M;i\R)&\to\ker
G_{(\psi,A)}^* \\
\big((\gf,a),\nu\big)&\mapsto H_{(\psi,A)}(\gf,a) + (0,-*\nu).
\end{split}\end{equation}

\begin{prop}\label{pert:map:trans}
The perturbation map $\widehat\SW$ is transversal to the zero
section.
\end{prop}
\begin{proof}
Suppose that $\big((\psi,A),\eta\big)$ is a zero of
$\widehat{\SW}$. Since we are interested only in gauge
equivalence classes of $\eta$-monopoles, we may assume that
modulo a gauge transformation, the $\eta$-monopole $(\psi,A)$ is
at least of Sobolev class $k$ (see Remark \ref{moduli:eta:ord}
below). Under this assumption we now have to show that the
derivative of $\widehat\SW$ at $\big((\psi,A),\eta\big)$ is
surjective. Note that the derivative has closed image as it is
Fredholm. Thus we consider $(\gf,a)\in \ker G^*$, orthogonal to
the image, wanting to show that this implies $(\gf,a)=0$.

For each $(\gf,a)$ which is orthogonal to the image of
\eqref{pert:map:der}, we have
\begin{equation}\label{pert:map:trans:1}
(\gf,a)\in (\im H)^\perp=\ker H \quad\text{ and }\quad a \perp
* Z_k^2(M;i\R),
\end{equation}
where we are using that $H$ is self-adjoint and are dropping the
reference to $(\psi,A)$. From the definition of $H$ it is
immediate that
\[
\ker H= \ker F\cap \ker G^* = \ker T\cap \ker G^*,
\]
where $F$ is the differential of $\SW$ and $T$ is defined in
\eqref{T}. This shows that
\begin{equation}\label{pert:map:trans:2}
(\gf,a)\in \ker F\quad\text{ and }\quad (\gf,a)\in L_k^2(M,E)
\end{equation}
because $(\psi,A)$ is of Sobolev class $k$ so that the operator
$T$ is elliptic with $L^2_k$-coefficients and we may thus apply
elliptic regularity. Therefore, $\gf$ and $a$ are at least
continuous as $k\ge 2$.

From \eqref{pert:map:trans:1} we deduce that $a$ must be closed.
Together with the fact that $(\gf,a)\in\ker F$, this yields
\begin{equation}\label{pert:map:trans:3}
0=q(\psi,\gf)-* da = q(\psi,\gf).
\end{equation}
Let $U$ be the open subset of $M$ on which $\psi$ is nowhere
vanishing. Invoking Proposition \ref{q:prop}, we infer from
\eqref{pert:map:trans:3} that there exists $f\in L^2_k(U,i\R)$
such that
\[
\gf|_U = f\psi|_U.
\]
Note that $L^2_k$-regularity of $f$ follows form the explicit
formula
\[
f=i|\psi|^{-2}\scalar{\gf}{i\psi}.
\]
To proceed, we require more information about the structure of
$U$. This is provided by the so-called \textit{unique
continuation principle} which we state without proof.
\begin{theorem*}{\rm(cf. \cite{BW}, Thm.~8.2)}.\index{Dirac
operator!unique continuation principle} Let $M$ be a compact
Riemannian manifold and $S$ a $\cl(M)$-module with
$\cl(M)$-compatible connection. Then the unique continuation
principle is valid for the corresponding Dirac operator $\cD$.
That is, any solution $\gf$ of $\cD\gf=0$ which vanishes on an
open subset on $M$ also vanishes on the whole connected component
of $M$.
\end{theorem*}
As we have chosen $(\psi,A)$ irreducible, the spinor $\psi$ is
nonzero so that $U$ is nonempty. Furthermore, $\psi$ is harmonic
with respect to $\cD_A$. Hence, the unique continuation principle
ensures that $U$ is dense since the complement of $U$ cannot
contain any open subset of $M$.

Using the first part of the equation $F(\gf,a)=0$, we find that
\[
0=\cD_A(f\psi|_U) + \lfrac{1}{2}c(a)\psi|_U= f\cD_A\psi|_U +
c\big(df+\lfrac{1}{2}a\big)\psi|_U.
\]
Therefore, since $\cD_A\psi=0$ and as $\psi|_U$ is nowhere
vanishing,
\begin{equation}\label{pert:map:trans:4}
df+\lfrac{1}{2}a|_U=0\,.
\end{equation}
Now $(\gf,a)\in\ker G^*$ ensures that
\[
0=G^*(\gf|_U,a|_U)=G^*(f\psi|_U,a|_U)=
2d^*a|_U-i\Im\scalar{f\psi|_U}{\psi|_U}.
\]
Applying $d^*$ to \eqref{pert:map:trans:4} then yields
\[
0=4 d^*d f + 2d^*a|_U = 4d^*d f + i\Im \Scalar{f\psi|_U}{\psi|_U}
= 4d^*d f + \big|(\psi|_U)\big|^2f.
\]
Multiplying this equation with $f$ and integrating over $U$
results in
\[
0=\int_U\Big(4\scalar{df}{df}+|\psi|^2f^2\Big)dv_g.
\]
For $\psi|_U$ is nowhere vanishing, we deduce that $f=0$ and thus
also $\gf|_U=0$. In view of \eqref{pert:map:trans:4}, we hence
obtain that $\gf$ and $a$ vanish on an open and dense subset of
$M$. Because of continuity, we can finally draw the conclusion
that $(\gf,a)=0$.
\end{proof}

As a corollary to Proposition \ref{pert:levelset} and to Theorem
\ref{sard-smale} we now obtain from the above result:

\begin{theorem}\label{generic:eta}\index{moduli space!irreducible
part} Let $(M,g)$ be a closed, oriented Riemannian 3-manifold
with \spinc structure $\gs$. Then, for a generic choice of
$\eta\in Z_k^2(M,i\R)$, the irreducible part of the
$\eta$-perturbed moduli space consists of non-degenerate points.
\end{theorem}

Introducing a perturbation does not change the basic topological
features of the moduli space. As it is a topological subspace of
$\cB(\gs)$, the set $\cM_\eta(\gs)$ is Hausdorff. Modifying the
proof of the key estimate \eqref{key:estimate} slightly, one
readily gets the following:
\begin{prop}
Let $\eta\in Z^2_k(M;i\R)$, where $k\ge 4$. Suppose $(\psi,A)$ is
an $\eta$-monopole of some Sobolev class $\ge 4$. Then
\begin{equation}\label{pert:key:estimate}
\|\psi\|_\infty^2\le \max\big\{0,-2\min_{x\in
M}s_g(x)+2\|\eta\|_\infty\big\}.
\end{equation}
\end{prop}
This proposition implies compactness of $\cM_\eta(\gs)$ in the
same way as in Section \ref{moduli}.\ref{compact}.

\begin{remark}\label{moduli:eta:ord}
It should be pointed out that in contrast to Chapter \ref{moduli}
the $\eta$-perturbed moduli spaces need not consist of gauge
equivalence classes of smooth configurations. Reviewing the
corresponding proofs shows that we can only guarantee regularity
up to the Sobolev order of $\eta$.\\
\end{remark}

\noindent\textbf{Reducible $\boldsymbol{\eta}$-monopoles.}
\index{moduli space!reducible part} Compactness of
$\cM_\eta(\gs)$ in combination with Theorem \ref{generic:eta} does
not necessarily imply that the irreducible part of $\cM_\eta(\gs)$
is a finite set. This is because some {\em reducible}
$\eta$-monopole might be an accumulation point. We thus have to
understand the structure of the reducible locus.

\begin{prop}\label{red:exist}
Let $M$ be a closed, oriented 3-manifold with \spinc structure
$\gs$. Then, for every perturbation $\eta\in Z^2_k(M;i\R)$, there
exist reducible $\eta$-monopoles if and only if the cohomology
class of $\eta$ satisfies
\begin{equation}\label{red:exist:1}
2\pi ic(\gs)=[\eta],
\end{equation}
where $c(\gs)\in H^1_{dR}(M;\R)$ denotes the canonical class of
$\gs$. If nonempty, the set of gauge equivalence classes of
reducible $\eta$-monopoles is homeomorphic to the
$b_1$-dimensional torus
\[
H^1_{dR}(M;i\R)/H^1_{dR}(M;4\pi i\Z),
\]
where $b_1$ denotes the first Betti number of $M$.
\end{prop}

\begin{proof}
Suppose that there exists a reducible $\eta$-monopole $A$. Then
$*(F_A +\eta)$ is zero, and this implies that
\[
2\pi i c(\gs)=2\pi i [\lfrac{i}{2\pi}F_A]=2\pi i
[\lfrac{1}{2\pi i}\eta]=[\eta].
\]
On the other hand, $2\pi i c(\gs)= [\eta]$ can only hold if
$[F_{A_0}]=-[\eta]$ for each gauge field $A_0$. This means that
there exists an imaginary valued 1-form $a$ such that $-\eta =
F_{A_0} + da$. We conclude that $A_0+a$ is a reducible
$\eta$-monopole.

Let $Z^1_1(M;i\R)$ be the space of closed, imaginary valued
1-forms of Sobolev class 1. Supposing for the rest of the proof
that the set of such elements is nonempty, we fix a reducible
$\eta$-monopole $A_0$. Then the map
\[
\gF: Z^1_1(M;i\R)\to \cC,\quad \gF(a):=A_0+a,
\]
parametrizes the whole set of reducible $\eta$-monopoles. As we
have seen before, two elements of the form $A_0+a$ and $A_0+a+df$
lie in the same gauge orbit which shows that the map $\gF$
descends to a continuous map
\[
\Bar \gF: H^1_{dR}(M;i\R)\to \cM_\eta(\gs).
\]
The image of $\Bar\gF$ consists of the reducible part of
$\cM_\eta(\gs)$. It remains to insure that
\[
\Bar\gF([a])=\Hat\gF([a']) \Longleftrightarrow [a'-a]\in
\im\big(H^1_{dR}(M;4\pi i \Z)\to H^1_{dR}(M;\R)\big).
\]
This is, however, only a restatement of the considerations in
\eqref{derham}. Hence, $\Bar \gF$ factors to a bijective and
continuous map between the $b_1$-dimensional torus
$H^1_{dR}(M;i\R)/H^1_{dR}(M;4\pi i\Z)$ and the reducible part of
the moduli space. Since the torus is compact, this map is
necessarily a homeomorphism.
\end{proof}
An immediate conclusion of this result shall play a major role
later: If $b_1=0$, condition \eqref{red:exist:1} is always
fulfilled. Since a 0-dimensional torus is simply a point, we get:

\begin{cor}\label{red:exist:cor}
Let $M$ be a closed, oriented 3-manifold with vanishing first
Betti number, and let $\gs$ be a \spinc structure on $M$. Then,
for every perturbation $\eta\in Z^2_k(M;i\R)$, there exists
exactly one reducible point in
$\cM_\eta(\gs)$.\\
\end{cor}

\noindent\textbf{Suitable perturbations on manifolds with
$\boldsymbol{b_1>0}$.} Since we want to achieve that the perturbed
moduli space consists only of irreducible points, Proposition
\ref{red:exist} naturally leads to considering the restricted
perturbation space\index{>@$\cP_k(\gs)$, perturbation space}
\begin{equation}\label{rest:pert}
\cP_k(\gs):=\bigsetdef{\eta\in Z^2_k(M;i\R)}{2\pi
ic(\gs)\neq[\eta]}.
\end{equation}
As we have seen above the restricted perturbation space is empty
if $b_1=0$. Yet, if $b_1>0$, then $\cP_k(\gs)$ is an open, dense
subset of $Z^2_k(M;i\R)$.
\begin{dfn}\label{suit:pert:b>0}
Let $(M,g)$ be a closed and oriented Riemannian 3-manifold with
\spinc structure $\gs$ and first Betti number $b_1>0$.
\begin{enumerate}
\item An element $\eta\in \cP_k(\gs)$ is called a {\em suitable
perturbation} with respect to $g$ if the $\eta$-perturbed moduli
space $\cM_\eta(\gs;g)$ consists only of finitely many
non-degenerate, irreducible points.
\item If $\eta\in\cP_k(\gs)$ is a suitable perturbation with
respect to $g$, we define
\[
\sw_\eta(\gs;g):=\sum_{[\psi,A]\in\cM_\eta(\gs;g)} \eps(\psi,A).
\]
\end{enumerate}
\end{dfn}
Recall that $\eps(\psi,A)$ is defined as the orientation transport
along the family $T_{(t\psi,A)}$ associated to the linear path
connecting $(0,A)$ to $(\psi,A)$. Moreover, as an immediate
consequence of Theorem \ref{generic:eta} we obtain that the set
of suitable perturbations is a generic subset of $Z^2_k(M;i\R)$.
In particular, there exist suitable perturbation with respect to
arbitrary metrics on $M$.

\begin{remark*}
In Section \ref{=0} we will return to the question of how to
choose perturbations appropriately in the remaining case $b_1=0$.
\end{remark*}

\section[Invariance for manifolds with $b_1>1$]{Invariance
for manifolds with $\boldsymbol{b}_1\mathbf{>1}$}\label{>1}

We now want to prove that on 3-manifolds with $b_1>1$, the number
$\sw_\eta(\gs;g)$ is independent of the metric $g$ and the
perturbation $\eta$. \\

\noindent\textbf{The parametrized moduli space.} Before we take
up the proof of the theorem we have in mind, we make some general
observations concerning the perturbed moduli spaces associated to
two different metrics $g_{-1}$ and $g_1$.

Consider a Sobolev completion of the space of Riemannian metrics
and let $\{g_t\}_{t\in[-1,1]}$ be a continuous family of metrics.
For every continuous path of perturbations, $\eta_t:[-1,1]\to
Z_k^2(M;i\R)$, we define the {\em parametrized moduli space}
by\index{moduli space!parametrized}
\begin{equation}\label{param:moduli}
\index{>@$\widehat{\cM}_\eta(\gs;g)$, parametrized moduli space}
\widehat{\cM}_\eta(\gs;g) :=
\bigcup_{t\in[-1,1]}\cM_{\eta_t}(\gs;g_t)\times \{t\} \subset
\cB\times [-1,1]\,.
\end{equation}
\begin{remark*}
Note that this definition necessitates a procedure to identify
the spaces $\cB(\gs;g_t)$ for different parameters $t$. The
material we need for this is summarized in Section
\ref{app:spinc}.\ref{met:dep}. We choose, say, $g_0$ as a fixed
reference metric and redefine $\cM_{\eta_t}(\gs;g_t)$ as the zero
set (modulo gauge equivalence) of the map
\[
(\psi,A)\mapsto
\SW_{\eta_t}^{g_t}(\psi,A):=\big(\cD_A^t\psi,\lfrac{1}{2}q_t(\psi)
-*_t(F_A+\eta_t)\big),
\]
where $*_t$ denotes the Hodge-star-operator on $T^*M$ induced by
the metric $g_t$. Moreover, we employ the notation of Section
\ref{app:spinc}.\ref{met:dep} and write
\[
\cD_A^t\psi:=\Hat\gk^{-1}_t\cD_A^{g_t}\Hat\gk_t\psi\qquad
\text{ and }\qquad q_t(\psi):=q^{g_t}(\Hat\gk_t\psi),
\]
where $\Hat\gk_t: L^2(M,S;g_0)\to L^2(M,S;g_t)$ is induced by
identifying the spinor bundles associated to different metrics.
\end{remark*}

Irrespective of the value of $b_1$, the following holds:

\begin{prop}\label{pert:moduli:comp}\index{moduli
space!parametrized} Let $M$ be a closed, oriented 3-manifold with
\spinc structure $\gs$, and let $\{g_t\}_{t\in [-1,1]}$ be a
continuous family of Riemannian metrics on $M$. Then, for every
continuous path $\eta:[-1,1]\to Z_k^2(M;i\R)$, the parametrized
moduli space $\widehat{\cM}_\eta(\gs;g)\subset \cB(\gs)\times
[-1,1]$ is sequentially compact.
\end{prop}
\begin{proof}
Let $\big([\psi_n,A_n,t_n]\big)_{n\ge 1}$ be a sequence in
$\widehat{\cM}_\eta(\gs;g)$. Without loss of generality, we can
assume that $t_n$ converges to some $t_0\in [-1,1]$. We then have
to show that there exists a subsequence of $([\psi_n,A_n])_{n\ge
1}$ which converges to an element in
$\cM_{\eta_{t_0}}(\gs;g_{t_0})$. Essentially, this amounts to
transferring the corresponding arguments in the proof of the
compactness theorem of Chapter \ref{moduli}. We shall only give a
brief sketch of how this is done.

We choose $g_{t_0}$ as a reference metric. As before, we may
represent $[\psi_n,A_n]$ by configurations $(\psi_n, A_0+a_n)$,
where $a_n$ satisfies $d^{*_{t_0}}a_n=0$. Reviewing the proof of
Proposition \ref{coulomb:gauge}, we draw the conclusion that this
yields $(\psi_n,a_n)\in L^2_k(M,E)$, where $k\ge 2$ is the
Sobolev class of $\eta_t$. Moreover, we can achieve---by possibly
invoking another gauge transformation---that the sequence of the
harmonic parts of $a_n$ (with respect to $g_{t_0}$) is bounded.
The elements $(\psi_n,a_n)$ are solutions to
\begin{align*}
\cD_{A_0}^{t_n}\psi_n&=-\lfrac{1}{2}c(\Hat k_{t_n}^{-1}a_n)\psi_n \\
(d+d^{*_{t_n}})a_n&=
*_{t_n}\lfrac{1}{2}q_{t_n}(\psi_n)-F_{A_0}-\eta_{t_n}
\end{align*}
Since $(\Hat\gk_{t_n}\psi_n,A_0+a_n)$ is an $\eta_{t_n}$-monopole
(with respect to the metric $g_{t_n}$), we deduce as in Chapter
\ref{moduli} that the $L^2(g_{t_n})$-norm of the second equation's
left hand side is smaller than some number depending on the scalar
curvature $s_{t_n}$ and the $g_{t_n}$-norm of $\eta_{t_n}$. Here,
we have to use the estimate \eqref{pert:key:estimate} instead of
the key estimate in Chapter \ref{moduli}. Since $g_t$ and $\eta_t$
are continuous paths, this data depends continuously on the
parameter $t$. We thus also obtain a $L^2(g_{t_0})$-bound on
$\big(*_{t_n}\lfrac{1}{2}q_{t_n}(\psi_n)-F_{A_0}-\eta_{t_n}\big)$.
Arguing exactly as before, we deduce that $a_n$ is bounded with
respect to $L^2_1(g_{t_0})$. Proceeding in this manner, we can
finally achieve that $(\psi_n,a_n)$ is a bounded sequence with
respect to the $L^2_2(g_{t_0})$-norm, and this implies the
assertion.\qedhere\\
\end{proof}

\noindent\textbf{The Seiberg-Witten invariant.} We are now in the
position to prove the main result of this thesis. Our
presentation follows the proof of the corresponding result in the
four dimensional case as given in Nicolaescu's book
\cite{Nic:SW}, Sec.~2.3.2.

\begin{theorem}\label{b>1}\index{Seiberg-Witten invariant}
Let $M$ be a closed, oriented 3-manifold with \spinc structure
$\gs$ and first Betti number $b_1>1$. Suppose that $g_{-1}$ and
$g_1$ are two Riemannian metrics on $M$ and that $\,\eta_{-1}$ and
$\eta_1$ are respectively chosen suitable perturbations. Then
\[
\sw_{\eta_{-1}}(\gs;g_{-1})=\sw_{\eta_1}(\gs;g_1).
\]
We thus can define the ``Seiberg-Witten invariant" of the \spinc
manifold $(M,\gs)$ by letting
\[
\sw(\gs):=\sw_\eta(\gs;g),
\]
where $g$ is an arbitrary Riemannian metric, and $\eta$ is a
suitable perturbation with respect to $g$.
\end{theorem}

\begin{proof}
Let $\widehat{\cP}_k(\gs)$ denote the space of all continuously
differentiable paths $\eta:[-1,1]\to \cP_k(\gs)$ which connect
$\eta_{-1}$ and $\eta_1$. We endow this space with its natural
$C^1$-topology thus providing it with the structure of a Banach
manifold. Since we are using only the restricted perturbation
space, the parametrized moduli space associated to a path
$\eta\in\widehat{\cP}_k(\gs)$ is entirely contained in the
irreducible part $\cB^*\times [-1,1]$. For all $\eta$ in a
certain generic subset of $\widehat{\cP}_k(\gs)$, Proposition
\ref{param:levelset} implies that $\widehat{\cM}_\eta(\gs;g)$ is
either empty or carries the structure of a 1-dimensional
$C^1$-submanifold of $\cB^*\times [-1,1]$. In the first case it
is immediate that
\[
\sw_{\eta_{-1}}(\gs;g_{-1})=\sw_{\eta_1}(\gs;g_1)=0.
\]
We shall thus assume from now on that the parametrized moduli
space is nonempty, hence a 1-dimensional manifold. According to
Proposition \ref{pert:moduli:comp} this manifold is sequentially
compact. Moreover, its boundary is given by
\[
\big(\cM_{\eta_{-1}}(\gs;g_{-1})\times\{-1\}\big)\;\cup\;
\big(\cM_{\eta_1}(\gs;g_1)\times\{1\}\big).
\]
Therefore, $\widehat{\cM}_\eta(\gs;g)$ consists of a finite union
of continuously differentiable arcs, say, $c_i:[a_i,b_i]\to
\cB^*\times [-1,1]$, $i=1,\ldots,n$, whose endpoints lie on the
boundary (see Fig. \ref{param:b>1}). Note that we are neglecting
closed arcs since they neither contribute to
$\sw_{\eta_{-1}}(\gs;g_{-1})$ nor to $\sw_{\eta_1}(\gs;g_1)$.

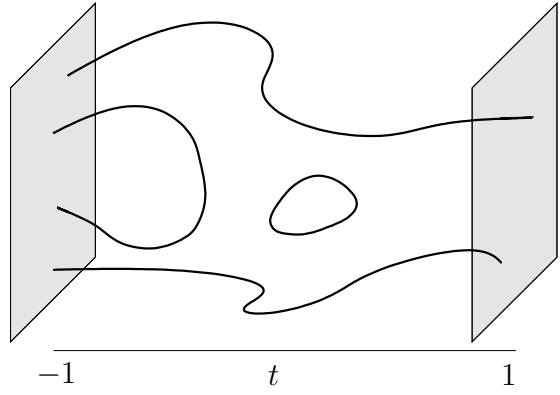
\begin{figure}
\begin{center}
\input{param_b.eepic}
\caption{The parametrized moduli space}\label{param:b>1}
\end{center}
\end{figure}

To simplify notation we want to embed $\widehat{\cM}_\eta(\gs;g)$
in $\cC^*\times[-1,1]$. This can be done in the following way: We
choose representatives for $c_i(a_i)$ and lift the paths
$c_i:[a_i,b_i]\to \cB^*\times [-1,1]$ horizontally to
$\cC^*\times [-1,1]$. This can be done since $\cC^*\to \cB^*$ is
a principal bundle with $\ker G^*\to \cC^*$ as a horizontal
structure. Then the tangent space to $\widehat{\cM}_\eta(\gs;g)$
at a point $(\psi,A,t)$ is given by the kernel of
\[
D_{(\psi,A,t)}\widehat{\SW}_\eta:\ker
(G_{(\psi,A)}^*|_{L_1^2})\oplus \R \to \ker G_{(\psi,A)}^*,
\]
i.e., it is given by all $(\gf,a,x)\in \ker
(G_{(\psi,A)}^*|_{L_1^2})\oplus \R$ such that
\begin{equation}\label{b>1:1}
H_{(\psi,A)}^{g_t}(\gf,a)+ x\cdot\lfrac{d}{dt} \big|_t
\SW_{\eta_t}^{g_t}(\psi,A) =0,
\end{equation}
where $H_{(\psi,A)}^{g_t}$ denotes the Hessian at $(\psi,A)$ with
respect to the metric $g_t$. To relate this with the operator
$T$, we infer from Lemma \ref{H=F} that for each zero of
$\widehat{\SW}_\eta$,
\[
T_{(\psi,A,t)}:=T_{(\psi,A)}^{g_t}=
\begin{pmatrix}
H_{(\psi,A)}^{g_t} &0 &0 \\
0 &0 &G_{(\psi,A)} \\
0 &G^{*_t}_{(\psi,A)} &0
\end{pmatrix},
\]
where we are using the decomposition $\big(\ker G^*\oplus \im
G\oplus L^2(M,i\R)\big)$ of the space $L^2(M,E\oplus i\R)$. We
now define $K_{(\psi,A,t)}:\R\to L^2(M,E\oplus i\R)$ by letting
\[
K_{(\psi,A,t)}(x):= x\cdot\lfrac{d}{dt} \big|_t
\SW_{\eta_t}^{g_t}(\psi,A).
\]
Since $\left(\begin{smallmatrix} 0 &G \\
G^*&0\end{smallmatrix}\right)$ yields an isomorphism from $\im
(G|_{L_2^2})\oplus L^2_1(M,i\R)$ to $\im G\oplus L^2(M,i\R)$, we
deduce from surjectivity in \eqref{b>1:1} that the operator
\[
T_{(\psi,A,t)} + K_{(\psi,A,t)} :L^2_1(M,E\oplus i\R)\oplus \R
\to L^2(M,E\oplus i\R)
\]
is onto. Moreover, its kernel is isomorphic to \eqref{b>1:1},
i.e., to the tangent space of $\widehat{\cM}_\eta(\gs;g)$ at the
point $(\psi,A,t)$. The above considerations show the following:
If $(\psi,A,t)$ is an element of the parametrized moduli space,
then the tangent space of $\widehat{\cM}_\eta(\gs;g)$ is
naturally isomorphic to the 1-dimensional space $\ker
\big(T_{(\psi,A,t)} + K_{(\psi,A,t)}\big)$. In the language of
Appendix \ref{app:det}, the map $K$ is a stabilizer of $T$ over
the parametrized moduli space, hence we have a natural
isomorphism of vector bundles
\[
T\widehat{\cM}_\eta(\gs;g) = \big(\det T\to
\widehat{\cM}_\eta(\gs;g)\big).
\]
We now use this observation to show that, given a path $c:[a,b]\to
\widehat{\cM}_\eta(\gs;g)$ which connects two boundary points,
the contribution to
\begin{equation}\label{b>1:2}
\sw_{\eta_1}(\gs;g_1)- \sw_{\eta_{-1}}(\gs;g_{-1})
\end{equation}
given by the endpoints cancel each other out. Summing over all
paths yields that \eqref{b>1:2} vanishes which is the assertion of
the theorem.

Writing $c(s) = \big(\psi(s),A(s),t(s)\big)$, it is immediate
from homotopy invariance of the orientation transport that
\begin{equation}\label{b>1:3}
\eps(T_s)=\eps\big(\psi(a),A(a)\big) \cdot
\eps(T_s^0)\cdot\eps\big(\psi(b),A(b)\big),
\end{equation}
where $T_s:=T_{(\psi(s),A(s))}^{g_{t(s)}}$ and
\[
T_s^0:=\begin{pmatrix} \cD_{A(s)}^{t(s)} &0 &0 \\
                            0 &-*_{t(s)} d &2d \\
                            0 & 2d^{*_{t(s)}}&0 \end{pmatrix}.
\]
The path $T_s^0$ is the direct sum of a complex family
and the family $\left(\begin{smallmatrix} -*_s d &2d \\
2d^{*_s}&0 \end{smallmatrix}\right)$. As in the proof of Lemma
\ref{eps:prop}, part (i), one sees that hence, $\eps(T_s^0)=1$.

Therefore, it remains to compute $\eps(T_s)$. Since $\eta_1$ and
$\eta_{-1}$ are suitable perturbations with respect to $g_1$ and
$g_{-1}$, the operators $T_a$ and $T_b$ are invertible. Hence, we
may apply Lemma \ref{OT:inv}, using the above observation that
$K_s:=K_{(\psi(s),A(s),t(s))}$ is a stabilizer for $T_s$. If we
choose a parametrization in such a way that $c'(s)\neq 0$, then
\begin{equation}\label{b>1:4}
\ker (T_s+K_s) = \Span_\R \big(\psi'(s),A'(s),0,t'(s)\big).
\end{equation}
Note that the---slightly confusing---ordering of variables stems
from the fact that $\ker (T_s+K_s)\subset L^2(M,E\oplus
i\R)\oplus \R$, whereas $c'(s)\in L^2(M,E)\oplus \R$. The
trivialization \eqref{b>1:4} of $\ker (T+K)$ induces an
isomorphism
\[
\Psi_a^b:\ker (T+K)_a\to\ker (T+K)_b,
\]
given by
\[
\Psi_a^b\big(\psi'(a),A'(a),0,t'(a)\big):=\big(\psi'(b),A'(b),0,t'(b)\big).
\]
The diagram \eqref{OT:CD} in the case at hand is
\[
\begin{CD}
\ker(T_a\oplus K_a)  @>{\Psi_a^b}>>   \ker(T_b\oplus K_b)\\
@VV{P_\R}V                    @VV{P_\R}V \\
\R                   @>{\gF_\R}>>        \R
\end{CD}
\]
where is uniquely determined by $\gF_\R(t'(a))=t'(b)$. Observe
that it follows from the fact that $T_a$ and $T_b$ are invertible
that both, $t'(a)$ and $t'(b)$, are nonzero. According to Lemma
\ref{OT:inv}, the orientation transport along $T_s$ is then given
by the parity of $\gF_\R$, i.e.,
\[
\eps(T_s)=\sgn \lfrac{t'(b)}{t'(a)}.
\]
\begin{figure}
\begin{center}
\input{orient.eepic}
\caption{Computing the orientation transport}\label{eps:orient}
\end{center}
\end{figure}
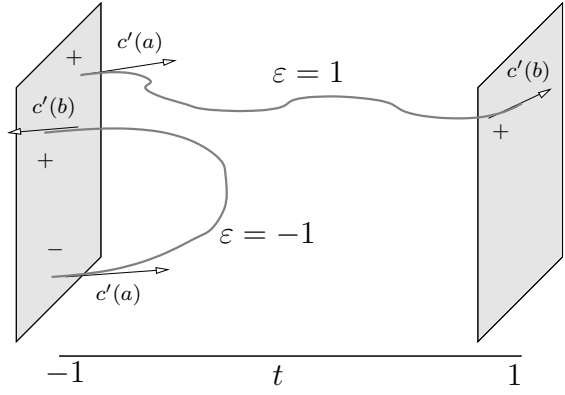
If $c$ connects elements in distinct parts of the boundary, i.e.,
$t(b) = - t(a)$, then necessarily $\sgn t'(a)= \sgn t'(b)$ (cf.
Fig. \ref{eps:orient}). The above formula then implies that
$\eps(T_s)=1$. If on the contrary, $c$ connects elements lying in
the same part of the boundary, then the signs of $t'(a)$ and
$t'(b)$ differ. Hence, in this case, $\eps(T_s)=-1$. Applying
these considerations to \eqref{b>1:3}, we get
\[
t(a)\cdot \eps\big(\psi(a),A(a)\big)+t(b)\cdot
\eps\big(\psi(b),A(b)\big)=0.
\]
From this one readily deduces that the contribution given by the
endpoints of $c$ to \eqref{b>1:2} cancel each other out.
\end{proof}

\section[The case $b_1=1$: Wall-crossing formula]{The case
$\boldsymbol{b}_1\mathbf{=1}$: Wall-crossing formula}\label{=1}

If $M$ is a 3-manifold with first Betti number $b_1=1$, then the
complement of the space of suitable perturbations,
\[
\cP_k(\gs)=\bigsetdef{\eta\in Z^2_k(M;i\R)}{2\pi
ic(\gs)\neq[\eta]},
\]
is a 1-codimensional affine subspace of $Z_k^2(M;i\R)$. The set
$Z_k^2(M;i\R)\setminus \cP_k$ is called {\em the wall}, and we
denote it by
\[
\cW_k(\gs):=\bigsetdef{\eta\in Z^2_k(M;i\R)}{2\pi
ic(\gs)=[\eta]}.
\]
Note that
\begin{equation}\label{wall}
\cW_k(\gs)=\eta_0 + d\big(L^2_{k+1}(M,iT^*M)\big).
\end{equation}
The parameter space now decomposes into two connected components
which are separated by the wall. In general, we thus cannot choose
a path in $\cP_k(\gs)$ connecting arbitrary suitable
perturbations. The goal of this section is to establish a
so-called {\em wall-crossing formula} which relates
$\sw_{\eta_{-1}}(\gs)$ to $\sw_{\eta_1}(\gs)$ if $\eta_{-1}$ and
$\eta_1$ lie in different components of $\cP_k(\gs)$.

The original version of the wall-crossing formula seems to be due
to Y. Lim \cite{Lim:SW}. Our presentation largely follows the
lecture notes by Nicolaescu \cite{Nic:SW3} and the corresponding
proof in the four dimensional case as it can be found in
\cite{Nic:SW}, Sec.~2.3.3.\\

\noindent\textbf{Positive and negative chambers.} First of all, we
have to provide a way of distinguishing the two components of
$\cP_k(\gs)$. This is done via an orientation of the second
cohomology of $M$. Since $b_1=1$ such an orientation is given by
a closed 2-form $\mu$ such that the cohomology class $[\mu]$ in
$H^2(M;\R)$ is nonzero. If $\eta_0\in\cW_k(\gs)$, we have $\eta_0
+ i\mu\notin \cW_k(\gs)$. This follows from \eqref{wall} and the
fact that the chosen 2-form $\mu$ is complementary to $\im d$.
This motivates the following definition.

\begin{dfn}
Let $M$ be a closed, oriented 3-manifold with $b_1=1$. Moreover,
let $M$ be equipped with a \spinc structure $\gs$ and an
orientation of the second cohomology. Let $\mu$ be a closed
2-form inducing the given cohomology orientation. Then, for
$\eta_0\in \cW_k(\gs)$, we define the {\em positive chamber}
$\cP^+_k(\gs)$ to be the component of $\cP_k(\gs)$ containing
$\eta_0+i\mu$. The complement of $\cP^+_k(\gs)$ in $\cP_k(\gs)$
is called the {\em negative chamber} $\cP_k^-(\gs)$. We thus
obtain a decomposition\index{>@$\cP_k^\pm(\gs)$, pos./neg.
chamber}
\[
Z_k^2(M;i\R)=\cP^-_k(\gs)\,
\dot\cup\;\cW_k(\gs)\;\dot\cup\;\cP^+_k(\gs).
\]
\end{dfn}
One readily checks that if $\mu, \mu'$ and $\eta_0, \eta_0'$ are
respectively chosen as above and if $[\mu']$ is a positive
multiple of $[\mu]$, then the linear path connecting
$\eta_0+i\mu$ with $\eta_0'+i\mu'$ never crosses the wall.
Therefore, the above definition depends only on the chosen
orientation of $H^2(M;\R)$. Moreover, note that if $\mu$ is
harmonic with respect to some Riemannian metric, then
\[
\cP_k^\pm(\gs)=\eta_0+ \bigsetdef{\eta\in
Z^2_k(M,i\R)}{\pm\Lscalar{\eta}{i\mu}>0}
\]
and
\[
\cW_k(\gs)=\eta_0+ \bigsetdef{\eta\in
Z^2_k(M,i\R)}{\Lscalar{\eta}{i\mu}=0}.
\]

Exactly as in Theorem \ref{b>1} one concludes that the number
$\sw_\eta(\gs;g)$ is independent of $\eta$ and $g$ as long as
$\eta$ is a suitable perturbation taken from only one of the two
chambers. We thus have the following result:
\begin{prop}
Let $(M,\gs)$ be a closed, oriented \spinc 3-manifold with
$b_1=1$. Moreover, let $M$ be equipped with an orientation of the
second cohomology. Then, for an arbitrary Riemannian metric $g$
and corresponding suitable perturbations $\eta^\pm \in
\cP_k^\pm(\gs)$ the numbers
\[
\sw^\pm(\gs) := \sw_{\eta^\pm}(\gs;g)
\]
are independent of $g$ and $\eta^\pm$.\\
\end{prop}

\noindent\textbf{The circle of reducibles.} Whenever
$\eta_0\in\cW_k(\gs)$, we know from Proposition \ref{red:exist}
that the reducible part of $\cM_{\eta_0}(\gs)$ is homeomorphic to
the circle $H^1_{dR}(M;i\R)/H^1_{dR}(M;4\pi i\Z)$. If we let
$\go_0$ be a generator of the lattice $H^1_{dR}(M;4\pi i\Z)$,
then it follows from the proof of Proposition \ref{red:exist}
that the circle of reducibles is parametrized by
$\{A_0+r\go_0\}_{r\in[0,1)}$, where $A_0$ is a fixed reducible
$\eta_0$-monopole.

If $M$ is equipped with a Riemannian metric $g$, we may assume
that $\go_0$ is harmonic with respect to $g$. This implies
$*\go_0\in i\cH^2(M;g)$ so that $*\go_0$ defines an orientation of
$i\cH^2(M;g)$ and thus an orientation of the second cohomology.

\begin{dfn}
Let $(M,g)$ be a closed, oriented Riemannian 3-manifold with
$b_1=1$, equipped with an orientation of $H^2(M;\R)$. Moreover,
let $\gs$ be a \spinc structure on $M$, and let $\eta_0\in
\cW_k(\gs)$. A harmonic 1-form $\go_0\in i\cH^1(M;g)$ is called a
\emph{generator of reducible $\eta_0$-monopoles} if $[\go_0]$
generates the lattice $H^1_{dR}(M;4\pi i\Z)$ and if $*\go_0$
induces the given orientation of the second cohomology.
\end{dfn}
Note that this characterizes $\go_0$ uniquely since there are
only two possible generators of $H^1_{dR}(M;4\pi i\Z)$.\\

\noindent\textbf{Finding a suitable path.} To find the relation
between $\sw^+(\gs)$ and $\sw^-(\gs)$ we have to consider a path
connecting two suitable perturbations $\eta^\pm\in\cP^{\pm}(\gs)$. We
will then derive the wall-crossing formula from a detailed
analysis of the parametrized moduli space near the circle of
reducibles. Finding an appropriate path requires some preliminary
considerations.

\begin{prop}\label{Dirac:pert:b=1}
Let $(M,g)$ be a closed, oriented Riemannian 3-manifold with
$b_1=1$. Moreover, let $M$ be equipped with an orientation of the
second cohomology and a \spinc structure $\gs$. For
$\eta_0\in\cW_k(\gs)$, let $\go_0$ be the generator of reducible
$\eta_0$-monopoles. Then for a generic element $A$ in $\cA(\gs)$,
the family $\{\cD_{A+r\go_0}\}_{r\in[0,1]}$ is transversal with
only simple crossings, cf. Definition \ref{transversal}.
\end{prop}

\begin{proof}
We use the perturbation methods of Section \ref{reg:val}.
Consider the open set $X:=\setdef{\psi\in L^2_1(M,S)}{\psi\neq 0}$
of non-vanishing spinors, and define a vector bundle $V\to X$ via
\[
V_\psi:=\ker \Re \Lscalar{i\psi}{.},\quad 0\neq \psi\in
L_1^2(M,S).
\]
Then $V\to X$ is a bundle of $\R$-Hilbert spaces. More precisely,
$V_\psi$ is a subspace of $L_1^2(M,S)$, which we consider as an
$\R$-Hilbert space. $V_\psi$ has codimension 1 because
\[
\big(\ker\Re\Lscalar{i\psi}{.}\big)^\perp=\Span_\R i\psi.
\]
Next consider the section $\gF:X\to V$ given by
\[
\gF(\psi):=\cD_{A_0}\psi,
\]
where $A_0$ is a fixed gauge field. Observe that $\gF$ is
well-defined since formal self-adjointness of the Dirac operator
implies that $\Re\Lscalar{i\psi}{\cD_{A_0}\psi}=0$. At a zero
$\psi$ of $\gF$ we have
\[
D_\psi\gF =\cD_{A_0}:L^2_1(M,S)\to \ker\Re\Lscalar{i\psi}{.}.
\]
Hence, $D_\psi\gF$ is a Fredholm operator of index 1. This is
because the Fredholm operator $\cD_{A_0}:L^2_1(M,S)\to L^2(M,S)$
of index 0 produces a Fredholm operator index 1 when the target
space is restricted to a 1-codimensional subspace.

Proceeding as in Section \ref{reg:val}, we now perturb $\gF$ to
make it transversal to the zero section. As the perturbation
space $P$ we take the space of imaginary valued 1-forms of
Sobolev class $1$, and the perturbation map is defined as the
section $\Hat \gF$ of $\pr_1^*V\to X \times P$ given by
\[
\Hat \gF(\psi,a):= \cD_{A_0+a}\psi.
\]
The next thing to establish is that for each zero $(\psi,a)$ of
$\Hat\gF$, the differential $D_{(\psi,a)}\Hat \gF$ is surjective.
We thus compute
\begin{equation}\label{Dirac:pert:b=1:1}
D_{(\psi,a)}\Hat \gF (\gf,b)= \cD_{A_0+a}\gf+\lfrac{1}{2}c(b)\psi.
\end{equation}
Assume that $\gf_0$ is orthogonal to the closed subspace $\im
D_{(\psi,a)}\Hat \gF$. Then the above equation shows that for
every $\gf\in L^2_1(M,S)$, the scalar product
$\Re\LScalar{\cD_{A_0+a}\gf}{\gf_0}$ vanishes. In combination
with formal self-adjointness of the Dirac operator, this yields
that
\[
\gf_0\in \ker \cD_{A_0+a}\;.
\]
On the other hand, according to equation \eqref{Dirac:pert:b=1:1},
the scalar product $\Re\LScalar{c(b)\psi}{\gf_0}$ vanishes for all
1-forms $b$. According to part (i) of Proposition \ref{q:prop}, ,
we thus have
\[
\LScalar{b}{q(\psi,\gf_0)}=0
\]
for each 1-form $b$. This implies that $q(\psi,\gf_0)=0$. As in
the proof of Proposition \ref{pert:map:trans} we may therefore
deduce that there exists $f\in L^2_1(M,i\R)$ such that
\[
\gf_0=f\cdot\psi.
\]
From the fact that $\cD_{A_0+a}\gf_0=0$ we infer that
\[
0= \cD_{A_0+a}(f\cdot\psi)= f\cdot\cD_{A_0+a}\psi + c(df)\psi =
c(df)\psi.
\]
This shows that $df=0$, since $\psi\neq 0$ and can thus only
vanish on the complement of a dense open set. Therefore, $f$ is
an imaginary constant and $\gf_0$ is a multiple of $\psi$ by $f$.
As $\gf_0\in\ker \Re\Lscalar{i\psi}{.}$, this demands
$\gf_0\equiv 0$. Hence, $D_{(\psi,a)}\Hat \gF$ is surjective.

We now apply similar considerations as in Proposition
\ref{param:levelset} and find that
\[
\Hat\Psi:X\times [0,1]\times P\to
\pr_1^*V,\quad\Hat\Psi(\psi,r,a)=\cD_{A+a+r\go_0}\psi,
\]
is transversal to the zero section. Therefore, given a generic
$a\in P$, the set
\begin{equation}\label{Dirac:pert:b=1:2}
\bigcup_{r\in[0,1]}\big(\ker\cD_{A_0+a+r\go_0}\setminus
\{0\}\big)\times \{r\}
\end{equation}
is either empty or a 2-dimensional real submanifold of $X\times
[0,1]$. Let us assume the latter holds (otherwise, there is
nothing left to prove). Then the projection
\[
\bigcup_{r\in[0,1]}\big(\ker\cD_{A_0+a+r\go_0}\setminus
\{0\}\big)\times \{r\}\to [0,1]
\]
is a smooth map from a 2-dimensional manifold to a 1-dimensional
one. By virtue of (the finite dimensional version of) Sard's
Theorem, we can ascertain that for each $r$ chosen from a dense
subset of $[0,1]$, the set
\[
\ker \big(\cD_{A_0+a+r\go_0}\big)\setminus \{0\}
\]
is either empty or a manifold of $\R$-dimension 1. On the other
hand, if nonempty, it necessarily has $\C$-dimension $\ge 1$.
This can only be possible if this manifold is empty and
therefore, $\ker \cD_{A_0+a+r\go_0}= 0$ for every $r$ in a dense
subset of $[0,1]$. For any other $r$, the kernel of
$\cD_{A_0+a+r\go_0}$ is of $\R$-dimension 2, since otherwise
\eqref{Dirac:pert:b=1:2} could not form a 2-dimensional manifold.

The last thing remaining to prove is that the family
$\{\cD_{A_0+a+r\go_0}\}_{r\in[0,1]}$ is transversal if $a$ is
chosen from the generic subset. Fix $r_0\in[0,1]$ such that the
corresponding operator $\cD_{A_0+a+r\go_0}$ is not invertible, and
let $\psi_0$ be an element of the kernel, of norm 1.
Transversality of the section $(\psi,r)\mapsto \Hat\Psi(\psi,r,a)$
guarantees that
\[
D_{(\psi_0,r_0)}\Hat\Psi(\gf,r,a)= \cD_{A_0+a+r_0\go_0}\gf +
r\cdot\lfrac{d}{ds}\big|_{r=r_0}\cD_{A_0+a+r\go_0}\psi_0
\]
yields a surjection $L^2_1(M,S)\oplus\R \to \ker
\Re\Lscalar{i\psi_0}{.}$. In particular, there exists $(\gf,r)\in
L^2_1(M,S)\oplus\R$ such that
\[
\cD_{A_0+a+r_0\go_0}\gf +
r\cdot\lfrac{d}{ds}\big|_{r=r_0}\cD_{A_0+a+r\go_0}\psi_0 = \psi_0.
\]
Since $\psi_0$ is a nonzero harmonic spinor with respect to
$\cD_{A_0+a+r_0\go_0}$, taking the $L^2$-product with $\psi_0$
enforces that
\[
\LScalar{\lfrac{d}{ds}\big|_{r=r_0}\cD_{A_0+a+r\go_0}\psi_0}{\psi_0}\neq
0.
\]
This shows that the zero eigenvalue at the point $r_0$ crosses
transversally.
\end{proof}

\begin{cor}\label{eta:path:b=1}
Let $(M,g)$ be a closed, oriented Riemannian 3-manifold with
$b_1=1$. Moreover, let $M$ be equipped with an orientation of the
second cohomology and a \spinc structure $\gs$. Suppose that
$\eta_{\pm 1}\in\cP^\pm_k(\gs)$ are suitable perturbations with
respect to $g$. Then there exists a connecting $C^3$-path
\[
\eta:[-1,1]\to Z_k^2(M;i\R)
\]
such that
\begin{enumerate}
\item The path $\eta$ meets the wall transversally and does so only in
0. That is, $\eta_t\in \cW_k(\gs)$ if and only if $t=0$ and
$\Lscalar{\eta_0'}{*\go_0}>0$, where $\go_0$ is the generator of
reducible $\eta_0$-monopoles.
\item The irreducible part of the parametrized moduli space is
either empty or a 1-dimensional $C^3$-submanifold of $\cB^*\times
[-1,1]$.
\item If $A_0$ is a reducible $\eta_0$-monopole, then the family
$\{\cD_{A_0+r\go_0}\}_{r\in[0,1]}$ is transversal with only simple
crossings.
\end{enumerate}
\end{cor}
\begin{proof}
Fix a gauge field $A_0$ lying in the generic set given by
Proposition \ref{Dirac:pert:b=1}. We define $\eta_0:= -F_{A_0}\in
Z^2_k(M;i\R)$, forcing $A_0$ to become a reducible
$\eta_0$-monopole. In particular, $\eta_0\in\cW_k(\gs)$. We now
provide an appropriate perturbation space to employ the results
of Section \ref{pert:levelset} again. Let us additionally fix a
constant $C>0$ and consider the set of all $C^3$-paths
$\eta:[-1,1]\to Z^2_k(M;i\R)$ satisfying
\begin{itemize}
\item $\eta(t)\in\cP^\pm_k(\gs)\;$ if $\;\pm t >0$,
\item $\eta(i)=\eta_i$ for $i\in\{-1,0,1\}$,
\item $\LScalar{\eta'(0)}{*\go_0}>C$,
\end{itemize}
where $\go_0$ is the generator of reducible $\eta_0$-monopoles.
Equipped with the $C^3$-topology, this set becomes a Banach
manifold since it is an open subset of the affine Banach space
\[
\bigsetdef{\eta:[-1,1]\to Z^2_k(M;i\R)}{\eta\text{ is }C^3\text{
and } \eta(i)=\eta_i\text{ for }i\in\{-1,0,1\}}.
\]
Note for this, that the requirement
$\LScalar{\eta'(0)}{*\go_0}>C$ is an open condition. It implies
that near 0 a path $\eta$ satisfies $\eta(t)\in\cP^\pm_k(\gs)$ if
$\pm t >0$. Away from 0, the first property is also an open
condition since $\cP_k^\pm(\gs)$ are open subsets of
$Z^2_k(M;i\R)$.

We now use the above Banach manifold as a perturbation space.
Each corresponding path automatically satisfies (i) and (iii).
Similarly as in the proof of Theorem \ref{b>1}, an application of
Proposition \ref{pert:levelset} and the Sard-Smale Theorem yields
that for a generic choice of such $\eta$, property (ii) can also
be achieved.\qedhere\\
\end{proof}

\noindent\textbf{The singular cobordism.} Let us now fix a path
$\eta$ as in the above Corollary. As we have seen in Section
\ref{>1}, the parametrized moduli space
\[
\widehat{\cM}_\eta(\gs)=\bigcup_{t\in[-1,1]}
\cM_{\eta_t}(\gs)\times\{t\}
\]
is compact. However, it does in general not necessarily form a
$C^3$-cobordism between the moduli spaces
$\cM_{\eta_{-1}}(\gs;g)$ and $\cM_{\eta_1}(\gs;g)$ for
singularities may occur at $\cM_{\eta_0}(\gs;g)\times\{0\}$, i.e.,
when the path $\eta$ crosses the wall. As the reducible part of
$\cM_{\eta_0}(\gs)$ is a circle, the singular cobordism will look
roughly as in Fig. \ref{param:b1}.

\begin{figure}
\begin{center}
\input{param_b1.eepic}
\caption{The parametrized moduli space in the case
$b_1=1$}\label{param:b1}
\end{center}
\end{figure}
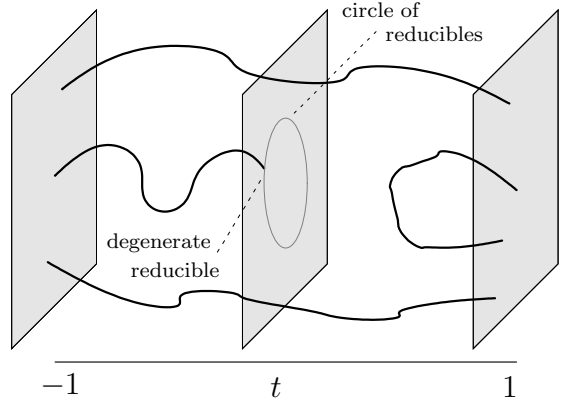

We now have to understand the nature of these singularities. Due
to property (ii) of the path $\eta$, they may only occur at the
circle of reducibles.

\begin{dfn}
Let $(M,g)$ be a closed, oriented Riemannian 3-manifold.
Moreover, let $\gs$ be a \spinc structure on $M$ and let
$\eta_0\in \cW_k(\gs)$. Then a reducible $\eta_0$-monopole $A$ is
called {\em non-degenerate} if $\cD_A$ is invertible and {\em
slightly degenerate} provided that $\dim_\C\ker\cD_A=1$.
\end{dfn}

In the case at hand, since the reducible part of
$\cM_{\eta_0}(\gs)$ can be parametrized by
$\{A_0+r\go_0\}_{r\in[0,1]}$, part (iii) of Corollary
\ref{eta:path:b=1} ensures that there may only occur reducibles
which are at most slightly degenerate. \\

\noindent\textbf{Local structure near a reducible.} The
difficulty in understanding the parametrized moduli space near the
circle of reducibles stems from the fact that reducible points lie
in a different stratum of the quotient $\cB$ than the irreducible
part of the moduli space. Recall that according to the slice
theorem, the local model is as follows: If $(0,A_0)\in\cC$ is a
reducible configuration, then there exists a $\U_1$-invariant open
neighbourhood $U$ of 0 in the slice, i.e.,
\begin{equation*}\label{U}
U\subset \ker (G_{(0,A)}^*|_{L_1^2})=L^2_1(M,S)\oplus \ker
(d^*|_{L_1^2}),
\end{equation*}
such that $U/\U_1$ models an open neighbourhood of $[0,A_0]$ in
$\cB$. Recall that in accordance with our earlier considerations,
$\U_1$ acts only on the spinor part of $L^2_1(M,S)\oplus \ker
(d^*|_{L_1^2})$.

If $\eta:[-1,1]\to Z_k^2(M,i\R)$ is a path of perturbations such
that $\eta_0$ admits a reducible monopole $(0,A_0)$, then the
above shows that the parametrized moduli space
$\widehat\cM_\eta(\gs)$ near $[0,A_0,0]$ is homeomorphic to
\begin{equation}\label{loc:struc:b=1:1}
\bigsetdef{(\gf,a,t)\in U\times (-\eps,\eps)}{\SW_{\eta_t}(\gf,
A_0+a)=0}/\U_1,
\end{equation}
where $U$ is a suitable $\U_1$-invariant open neighbourhood of 0
in the slice at $(0,A_0)$ given by the slice theorem. To
understand the above zero set, we let
\begin{equation}\label{s:b=1}
s:L^2_1(M,S)\oplus \ker (d^*|_{L_1^2})\times
(-\eps,\eps)\longrightarrow L^2(M,S)\oplus \ker d^*
\end{equation}
be defined by
\begin{align*}
s(\gf,a,t):=&\Proj_{L^2(M,S)\oplus\ker d^*}\circ
\,\SW_{\eta_t}(\gf,A_0+a)\\
=&\Proj_{L^2(M,S)\oplus\ker d^*}\big(\cD_{A_0+a}\gf
,\,\lfrac{1}{2}q(\gf)-*(F_{A_0+a} + \eta_t)\big).
\end{align*}
Therefore, $s(\gf,a,t)$ and $\SW_{\eta_t}(\gf,A_0+a)$ coincide
whenever the 1-form factor of $\SW_{\eta_t}(\gf,A_0+a)$ is
co-closed. Since $s(\gf,a,t)=0$ implies that $\gf$ is harmonic,
Proposition \ref{dastq} shows that this is true at every zero of
$s$, i.e.,
\[
s(\gf,a,t)=0\Longleftrightarrow \SW_{\eta_t}(\gf,A_0+a)=0.
\]
Hence, the local model \eqref{loc:struc:b=1:1} can be replaced by
a neighbourhood of 0 in $s^{-1}(0)/{\U_1}$. Summarizing this, we
have the following:
\begin{lemma}\label{s:lemma}
Let $(M,g)$ be a closed, oriented Riemannian 3-manifold, and let
$\eta:[-1,1]\to Z_k^2(M,i\R)$ be a a path of perturbations such
that $\eta_0$ admits a reducible monopole $(0,A_0)$. Then there
exists a $\U_1$-invariant neighbourhood $U$ of 0 in
$L^2_1(M,S)\oplus \ker (d^*|_{L_1^2})$ such that
\[
\bigsetdef{(\gf,a,t)\in U\times (-\eps,\eps)}{s(\gf,a,t)=0} / \U_1
\]
is homeomorphic to an open neighbourhood of $[0,A_0,0]$ in
$\widehat\cM_\eta(\gs)$.
\end{lemma}

For the remainder of this section, we consider the situation of
the last paragraph, i.e., we shall always assume that $b_1=1$ and
that $\eta$ is chosen as in Corollary \ref{eta:path:b=1}.
Moreover, we fix a reducible $\eta_0$-monopole $(0,A_0)$, and let
$\go_0$ denote the generator of reducible $\eta_0$-monopoles. To
study $\widehat \cM_\eta(\gs)$ near $[0,A_0,0]$, we compute that
the differential of $s$ at 0,
\[
D_0s:L^2_1(M,S)\oplus \ker (d^*|_{L_1^2})\oplus \R \to
L^2(M,S)\oplus \ker d^*,
\]
is given by
\begin{equation*}
\begin{split}
D_0s(\gf,a,t)&=\Proj_{L^2(M,S)\oplus\ker d^*}
\big(\cD_{A_0}\gf,-* da -* (D_0\eta)(t)\big) \\
&= \big(\cD_{A_0}\gf, -* da - t* \eta'_0\big),
\end{split}
\end{equation*}
where we are using that $* da$ and $*\eta_0'$ are co-closed.
Observe that $da+t\cdot\eta'_0=0$ if and only if $da=0$ and
$t=0$. This is because $\Lscalar{\eta'_0}{* \go_0}>0$ so that
necessarily $[\eta'_0]\neq 0$, whereas $[da]=0$. We conclude that
\begin{align*}
\ker D_0s &= \ker \cD_{A_0}\oplus \ker d|_{\ker d^*}\oplus
\{0\}\\
&=\ker \cD_{A_0}\oplus i\cH^1(M)\oplus \{0\}.
\end{align*}
The cokernel of $D_0s$ is given by the orthogonal complement of
$\im D_0s$ in $L^2(M,S)\oplus \ker d^*$. Since $\ker d^*$ is
spanned by $\eta_0'$ and the image of $* d$, we deduce that
\[
\coker D_0s= \coker \cD_{A_0} \oplus \{0\}=\ker \cD_{A_0}
\oplus\{0\}.
\]

\begin{lemma}\label{loc:red:b=1:nd}
If $A_0$ is non-degenerate, then a neighbourhood of $[0,A_0,0]$ in
the parametrized moduli space $\widehat{\cM}_\eta(\gs)$ is
homeomorphic to a neighbourhood of $A_0$ in the circle of
reducibles.
\end{lemma}
\begin{proof}
Recall that the circle of reducibles near $[0,A_0,0]$ is given by
\[
\bigsetdef{[0,A_0+r\go_0,0]\in\cB\times [-1,1]}{r\in (-\gd,\gd)},
\]
where $\gd<\frac12$, and $\go_0$ is the generator of reducible
$\eta_0$-monopoles. In terms of the local model $s^{-1}(0)/\U_1$,
this corresponds to the line
\begin{equation}\label{loc:red:b=1:nd:1}
\bigsetdef{(0,r\go_0,0)\in L_1^2(M,S)\oplus \ker
d^*|_{L_1^2}\times (-\eps,\eps)}{r\in (-\gd,\gd)},
\end{equation}
where we are using that $\U_1$ acts only on the spinor part,
hence has no effect on the above set. As $\ker \cD_{A_0}=0$ at a
non-degenerate monopole, the differential $D_0s$ is surjective
and has kernel equal to $i\cH^1(M)$. As this is 1-dimensional,
the implicit function theorem implies that
\eqref{loc:red:b=1:nd:1} is exactly the zero set of $s$ near 0.
\end{proof}
\begin{remark*}
The fact that $\ker D_0s= i\cH^1(M)$ motivates the remark
following \eqref{T:red} that $i\cH^1(M)$ is the tangent space to
the reducible part of the moduli space.
\end{remark*}

\begin{prop}\label{loc:red:b=1}
Let $A_0$ be slightly degenerate and fix $\psi_0\in L^2_1(M,S)$ of
norm 1 spanning $\ker \cD_{A_0}$. Then in a neighbourhood of 0 the
following holds:
\[
s(\gf,a,t)=0 \Longleftrightarrow (\gf,a,t)=
\begin{cases}
(0,r\go_0,0) &\text{ if } \gf=0 \\
\big(z\psi_0,g(z)\go_0,0\big) + f(z,g(z)) &\text{ if } \gf\neq 0
\end{cases}\;,
\]
with small $(z,r)\in\C\times \R$, and where $\go_0$ is the
generator of reducible $\eta_0$-monopoles. The map $f$ is a
$\U_1$-equivariant $C^3$-map
\[
f:\C\times \R \to (\ker \cD_{A_0}\oplus i\cH^1)^\perp,
\]
where the orthogonal complement is taken in $L_1^2(M,S)\oplus
\ker (d^*|_{L_1^2})\oplus \R$. The map $g:\C\to \R$ is $C^1$ and
$\U_1$-invariant. Both maps vanish of second order in 0.
\end{prop}

\begin{proof}
We will use the so-called local Kuranishi technique to study the
zero set of $s$ near 0. Let
\[
\gF:=\Proj_{\im D_0s}\circ s\quad \text{ and } \quad
\Psi:=\Proj_{\coker D_0s}\circ s.
\]
As $s=\gF\oplus \Psi$, we have to find the common zeros of $\gF$
and $\Psi$. Since we have forced $\gF$ to have a surjective
differential at 0, we can apply the implicit function theorem to
obtain a $C^3$-map
\[
f:\ker \cD_{A_0}\oplus i\cH^1\to (\ker \cD_{A_0}\oplus
i\cH^1)^\perp
\]
such that the graph of $f$ locally describes the zero set of
$\gF$. More precisely, using coordinates $(z,r)$ with respect to
$(\psi_0,\go_0)$ on $\ker \cD_{A_0}\oplus i\cH^1$, we have for
small $(\gf,a,t)$ that
\[
\gF(\gf,a,t)=0 \Longleftrightarrow (\gf,a,t)=(z\psi_0,r\go_0,0)+
f(z,r)\,.
\]
The first observation is that $f(0,r)=0$ for every $r$. This
follows from the fact that $\gF(0,r\go_0,0)=0$. Moreover, as the
zero set of $\gF$ is $\U_1$-invariant, we infer that $f$ is
necessarily $\U_1$-equivariant, i.e.,
\begin{equation*}
f(\gl z,r)=\gl f(z,r),\quad \gl \in\U_1.
\end{equation*}
Recall that on the right hand side, $\gl$ only operates on the
spinor part.

To extract information about the zeros of $s$ we now investigate
\[
\Hat \Psi:\C\times \R\to \ker \cD_{A_0},\quad \Hat\Psi(z,r):=
\Psi\big((z\psi_0,r\go_0,0)+ f(z,r)\big).
\]
This map is called the \textit{Kuranishi obstruction map}. It
reduces the original infinite dimensional problem of finding the
zeros of $s$ to finite dimensions. From the corresponding
property of $f$ it is immediate that $\Hat\Psi(0,r)=0$ for every
$r$ and that $\Hat \Psi$ is a $\U_1$-equivariant map. Note that
the latter implies that for fixed $r$, the map $\Hat \Psi(.,r)$ is
complex differentiable in 0. Therefore, we can factor out $z$
writing
\[
\Hat\Psi(z,r)=z\cdot \Hat\Psi_1(z,r),\quad \text{where }
\Hat\Psi_1(z,r):=\begin{cases} \frac{\Hat\Psi(z,r)}{z}, &z\neq
0,\\ \frac{\partial}{\partial z}\big|_{(0,r)}\Hat\Psi(z,r), & z=0.
\end{cases}
\]
Since $\Hat\Psi(z,r)$ is $C^3$, the map $\Hat\Psi_1$ is at least
$C^1$. This easily follows from the Taylor Formula. We now claim
that
\begin{equation}\label{kappa}
\lfrac{\partial}{\partial r}\big|_{(0,0)}\Hat\Psi_1(z,r)
=\gk\cdot\psi_0,\quad \text{ where }\quad\gk:=
\lfrac{1}{2}\Lscalar{c(\go_0)\psi_0}{\psi_0}\,.
\end{equation}
Observe that $\frac{\partial}{\partial
r}\big|_{(0,0)}\Hat\Psi_1(z,r)= \frac{\partial^2}{\partial
r\partial z}\big|_{(0,0)}\Hat\Psi(z,r)$. Making use of the
definition of $\Psi$ and letting $\Pi:=\Proj_{\ker\cD_{A_0}}$ we
have
\begin{align*}
\lfrac{\partial}{\partial z}\big|_{(0,r)}&\Hat\Psi(z,r) =
\lfrac{\partial}{\partial z}\big|_{(0,r)} \Pi\circ s\circ
\big((z\psi_0,r\go_0,0) + f(z,r)\big)\\
&=\Pi\circ D_{(0,r\go_0,0)}s\circ \big(\lfrac{\partial}{\partial
z}\big|_{(0,r)}(z\psi_0,r\go_0,0) + \lfrac{\partial}{\partial
z}\big|_{(0,r)} f(z,r) \big)\\
&=\Pi\circ \big(\cD_{A_0+r\go_0}\psi_0\big) + \Pi\circ
\lfrac{\partial}{\partial z}\big|_{(0,r)} \big(
\cD_{A_0+a(z,r)}\gf(z,r)\big),
\\
\end{align*}
where $\gf(z,r)$ and $a(z,r)$ denote the spinor and the 1-form
part of $f(z,r)$ respectively. Note that we have written down only
the spinor part of $D_{(0,r\go_0,0)}s$ which is sufficient since
the 1-form part vanishes when $\Pi=\Proj_{\ker \cD_{A_0}}$ is
applied. Since $\psi_0\in \ker \cD_{A_0}$, we find that
\[
\Pi\circ \big(\cD_{A_0+r\go_0}\psi_0\big) =
\LScalar{\cD_{A_0+r\go_0}\psi_0}{\psi_0} =
\lfrac{1}{2}\Lscalar{c(r\go_0)\psi_0}{\psi_0}= \gk\cdot r.
\]
The second term in the above expression of
$\lfrac{\partial}{\partial z}\big|_{(0,r)}\Hat\Psi(z,r)$ is equal
to zero because
\begin{align*}
\Pi\circ &\lfrac{\partial}{\partial z}\big|_{(0,r)} \big(
\cD_{A_0+a(z,r)}\gf(z,r)\big)= \lfrac{\partial}{\partial
z}\big|_{(0,r)} \LScalar{ \cD_{A_0+a(z,r)}\gf(z,r)}{\psi_0}\\ &=
\lfrac{\partial}{\partial
z}\big|_{(0,r)}\LScalar{\cD_{A_0}\gf(z,r)}{\psi_0}
+\lfrac{1}{2}\lfrac{\partial}{\partial
z}\big|_{(0,r)}\LSCalar{c(a(z,r))\gf(z,r)}{\psi_0} \\
&= \lfrac{\partial}{\partial
z}\big|_{(0,r)}\LScalar{\cD_{A_0}\gf(z,r)}{\psi_0} +
\lfrac{1}{2}\LSCalar{c\big(\lfrac{\partial}{\partial
z}\big|_{(0,r)} a(z,r)\big)\gf(0,r) \\ &\hspace{33ex} +
c\big(a(0,r)\big)\big(\lfrac{\partial}{\partial z}
\big|_{(0,r)}\gf(z,r)\big)}{\psi_0}.
\end{align*}
Here, the first summand vanishes as $\psi_0$ is harmonic and
$\cD_{A_0}$ is formally self-adjoint, while the second and the
third term equal zero since $a(0,r)=\gf(0,r)=0$. For this recall
that $f(0,r)=0$. Putting these computations together proves
\eqref{kappa}. Note that we have also proved that
$\Hat\Psi_1(0,0)=0$.

The next observation is that
\begin{equation}\label{sign:kappa}
\sgn \gk = \SF(\cD_{A_0+r\go_0};\;{\scriptstyle |r|\ll 1} )\,.
\end{equation}
This is because the family $\{\cD_{A_0+r\go_0}\}$ has only simple
crossings so that
\[
\SF(\cD_{A_0+r\go_0};\;{\scriptstyle |r|\ll 1})=
\sgn\LScalar{\lfrac{d}{dr}\big|_{r=0}\cD_{A_0+r\go_0}\psi_0}{\psi_0}
= \sgn\lfrac{1}{2}\LScalar{c(\go_0)\psi_0}{\psi_0}\,.
\]
In particular, $\gk\neq 0$ since the family is transversal.

Combining this information with \eqref{kappa} shows that
$\Hat\Psi_1(0,0)=0$ and $\lfrac{\partial}{\partial
r}\big|_{(0,0)}\Hat\Psi_1(z,r)\neq 0$. This allows us to apply
the implicit function theorem to the map $\Hat\Psi_1$ near 0 which
produces a $C^1$-map $g:\C\to\R$ such that in a neighbourhood of
0,
\[
\Hat \Psi_1(z,r) = 0\; \Longleftrightarrow\; r=g(z)\,.
\]
By definition of $\Hat\Psi$ and according to the equivariance
property of $f$, one readily ensures that the map $g$ must be
$\U_1$-invariant.

Putting all pieces of information together yields
\begin{align*}
s(\gf,a,t)=0 &\Longleftrightarrow \gF(\gf,a,t)=0\;\text{ and
}\;\Psi(\gf,a,t)=0 \\
& \Longleftrightarrow (\gf,a,t)=\big(z\psi_0,r\go_0,0\big) +
f(z,r)\;\text{ and }\;\Hat\Psi(z,r)=0 \\
&\Longleftrightarrow (\gf,a,t)=
\begin{cases}
(0,r\go_0,0) &\text{ if } \gf=0 \\
\big(z\psi_0,g(z)\go_0,0\big) + f(z,g(z)) &\text{ if } \gf\neq 0.
\end{cases}
\end{align*}

From the fact that $g$ is a $\U_1$-invariant map satisfying
$g(0)=0$, we immediately obtain that it vanishes of second order.
Moreover, if we write
\[
f(z,g(z))=\big(\gf(z),a(z),t(z)\big),
\]
then $\U_1$-equivariance of $f$ means
\[
\gf(\gl z)=\gl\gf(z),\quad a(\gl z)=a(z)\quad\text{and} \quad
t(\gl z)=t(z),\quad \gl\in\U_1.
\]
Hence, $a$ and $t$ are also $\U_1$-invariant and thus vanish of
second order in 0. To compute $\gf'(0)$ we note that
\[
\begin{split}
0 & =D_0s\big(\lfrac{\partial}{\partial z}
\big|_{z=0}(z\psi_0,g(z)\go_0,0)+\lfrac{\partial}{\partial z}
\big|_{z=0}f(z,g(z))\big)\\
&=\big(\cD_{A_0}\psi_0+\cD_{A_0} \gf'(0),0\big)= \big(\cD_{A_0}
\gf'(0),0\big),
\end{split}
\]
where we are using that $a'(0)=g'(0)=0$ and that $\psi_0$ is
harmonic with respect to $\cD_{A_0}$. We deduce that
$\cD_{A_0}\gf'(0)=0$. On the other hand, $\gf'(0)\perp
\ker\cD_{A_0}$ for the image of $f$ is contained in $(\ker
\cD_{A_0}\oplus i\cH^1)^\perp$. Hence, $\gf'(0)=0$ which shows
that $\gf$ also has a zero of order 2 in 0.
\end{proof}

\begin{remark*}
From the $\U_1$-equivariance and $\U_1$- invariance properties
and the fact that $f$ and $g$ vanish of second order in 0 one
deduces from the above proposition that a neighbourhood of
$[0,A_0,0]$ in $\widehat{\cM}_\eta(\gs)$ is homeomorphic to a
neighbourhood of 0 in
\[
\bigsetdef{(z,x)\in \R_+^0\times \R}{z=0 \text{ or } x=0},
\]
where $\R_+^0:= \R_+\cup \{0\}$. The branch $\R_+\times \{0\}$
corresponds to the irreducible part of the parametrized moduli
space near $[0,A_0,0]$.
\end{remark*}

The next result shows that the spectral flow of
$\{\cD_{A_0+r\go_0}\}$ at $r=0$ determines whether the
irreducible branch hits the circle of reducibles coming from the
left or from the right. More precisely,

\begin{prop}\label{orient:b=1}
If $A_0$ is a slightly degenerate $\eta_0$-monopole, then the
following holds:
\begin{itemize}
\item If $\,\SF(\cD_{A_0+r\go_0}; {\scriptstyle |r|\ll 1})=-1$, then the
irreducible part of $\widehat{\cM}_\eta(\gs)$ near $[0,A_0,0]$ is
entirely contained in $\cB^*\times [-1,0)$.
\item If $\,\SF(\cD_{A_0+r\go_0}; {\scriptstyle |r|\ll 1})=1$, then
the irreducible branch lies in $\cB^*\times (0,1]$.
\end{itemize}
Moreover, if $\,[\psi,A,t]$ is an element of the irreducible
branch close\footnote{The proof will provide a more precise
meaning for that.} to $[0,A_0,0]$, then
\[
\eps(\psi,A)=1.
\]
\end{prop}

\begin{proof}
As the path $\eta$ is chosen according to Corollary
\ref{eta:path:b=1}, we infer from part (i) of this result that
\[
\eta_t\in\cP_k^\pm(\gs)\;\Longleftrightarrow\;\pm\,t>0.
\]
Let $f$ and $g$ be as in Proposition \ref{loc:red:b=1}, and write
$f(z,g(z))=\big(\gf(z),a(z),t(z)\big)$. Then the first first part
of the proposition will be established if we prove that that
\[
\SF(\cD_{A_0+r\go_0}; {\scriptstyle |r|\ll 1}) =\pm 1
\;\Longleftrightarrow\; \eta_{t(z)}\in \cP^\pm_k(\gs)\quad \text{
for }|z|\ll 1.
\]
By definition of $\go_0$, this amounts to the same as proving that
\begin{equation}\label{eta:chamber}
\SF(\cD_{A_0+r\go_0}; {\scriptstyle |r|\ll 1})= \sgn
\LScalar{\eta_{t(z)}-\eta_0}{*\go_0} \quad\text{ for } |z|\ll 1.
\end{equation}
We recall that $\big(z \psi_0+\gf(z),A_0+g(z)\go_0+a(z)\big)$ is
an $\eta_{t(z)}$-monopole. In particular,
\[
\lfrac12 q\big(z\psi_0 + \gf(z)\big) -*\big(F_{A_0}+ g(z) d\go_0
+ d a(z) + \eta_{t(z)}\big)=0.
\]
As $F_{A_0}=-\eta_0$, taking the inner product with $* \go_0$
shows that
\begin{align*}
\LScalar{\eta_{t(z)}-\eta_0}{*\go_0} &= \LSCalar{*
\lfrac{1}{2}q\big(z\psi_0 + \gf(z)\big) - da(z)}{*\go_0}
\\ &=\LSCalar{\lfrac12 q\big(z\psi_0 + \gf(z)\big)}{\go_0}\,,
\end{align*}
In the last line, we have used that $da(z)$ is orthogonal to the
harmonic form $* \go$. Invoking Proposition \ref{q:prop}, we find
that
\begin{align*}
\LScalar{\eta_{t(z)}-&\eta_0}{*\go_0}
=\lfrac{1}{4}\LSCalar{c(\go_0)(z\psi_0+
\gf(z))}{z\psi_0+\gf(z))}\\
& = \lfrac{1}{2}|z|^2\gk + \lfrac{1}{2}\Re \LScalar{z
c(\go_0)\psi_0}{\gf(z )} + \lfrac{1}{4}\LScalar{c(\go_0)\gf(z
)}{\gf(z )}.
\end{align*}
Here, we employ the term $\gk$ as defined in \eqref{kappa}. Since
$\gf$ vanishes of second order in 0, the right hand side of the
above equation equals $\lfrac{1}{2}|z|^2\gk $ up to a term which
vanishes of order three at $z=0$. From the Taylor expansion of
the right hand side we thus infer that the signs of
$\LScalar{\eta_{t(z)}-\eta_0}{*\go_0}$ and $\gk$ coincide for
small $z$. On the other hand, we have already observed in
\eqref{sign:kappa} that $\sgn \gk =\SF(\cD_{A_0+r\go_0};
{\scriptstyle |r|\ll 1})$. Putting these observations together
proves \eqref{eta:chamber} and thus the first part of the
proposition.\\

Given $[\psi,A,t]$ in the irreducible branch close to $[0,A_0,0]$,
we now want to compute $\eps(\psi,A)$. Recall that this number is
defined as the orientation transport along the family
$\{T_{(s\psi,A)}\}$ associated to the affine path from $(0,A)$ to
$(\psi,A)$. Part (i) of Lemma \ref{eps:prop} shows that in order
to compute $\eps(\psi,A)$ we can equally use the affine path from
$(0,A_0)$ to $(\psi,A)$. This path is, however, homotopic to the
path of configurations associated to the part of the irreducible
branch connecting $[0,A_0,0]$ with $[\psi,A,t]$, i.e., to the path
\[
[0,z_0]\to \cC(\gs),\quad z \mapsto \big(\psi(z),A(z)\big),
\]
with
\begin{equation}\label{orient:b=1:1}
\big(\psi(z),A(z)\big):=\big(z\psi_0,A_0+g(z)\go_0\big) +
\big(\gf(z),a(z)\big).
\end{equation}
Here, $z_0$ is chosen in such a way that
$[\psi(z_0),A(z_0),t(z_0)]=[\psi,A,t]$. Recall that $\gf(z)$ and
$a(z)$ denote the spinor and the 1-form part of $f(z,g(z))$
respectively. We now want to compute the orientation transport
along $T_z := T_{(\psi(z),A(z))}$ for $z\in [0,z_0]$. To simplify
the corresponding calculation, we make the following
consideration: As in the proof of Theorem \ref{b>1}, we can
represent the tangent space to the irreducible part of the
paramterized moduli space at $[\psi,A,t]$ via
\[
T_{[\psi,A,t]}\widehat\cM_{\eta}^* \;\cong\; \ker
\big(T_{(\psi,A)} + K_{(\psi,A,t)}\big)\subset L_1^2(M,E\oplus
i\R)\oplus \R,
\]
where $K_{(\psi,A,t)}(x):= x\cdot \frac d{dt}|_t
\SW_{\eta_t}(\psi,A)$. This shows that $T_{(\psi,A)}$ is
injective whenever the projection
$T_{[\psi,A,t]}\widehat\cM_{\eta}^*\to \R$ is an isomorphism.
According to Sard's Theorem, this is true for $t$ in a dense open
subset of $[-1,1]$ because the projection $\widehat\cM_\eta^*\to
[-1,1]$ is a $C^3$ map between 1-dimensional manifolds.

This implies that by possibly choosing $[\psi,A,t]$ closer to
$[0,A_0,0]$ it can be guaranteed that for each $0<z \le z _0$,
the operator $T_z := T_{(\psi(z),A(z))}$ is invertible.
Therefore, the only contribution to the orientation transport
along $T_z$ is encoded in the spectral flow of $T_z$ at $z=0$.
The latter can be understood by means of the crossing operator
\[
C_T(0):=\Proj_{\ker T_0}\circ \big(\lfrac{d}{dz}\big|_{z=0}
T_z\big)\big|_{\ker T_0}\,,
\]
cf. Definition \ref{transversal}. The operator $T_z$ is
explicitly given by
\begin{multline*}
T_z\begin{pmatrix} \gf \\ a \\ f \end{pmatrix} =
\begin{pmatrix}
\cD_{A_0} &0 &0\\
0 &-* d & 2d \\
0 &2d^* &0
\end{pmatrix}\begin{pmatrix} \gf \\ a \\ f \end{pmatrix}\\ +
\begin{pmatrix}
\lfrac12 c\big(g(z)\go_0+a(z)\big)\gf + \lfrac{1}{2}c(a)\psi(z) -
f\psi(z) \\
q\big(\psi(z),\gf\big)\\
-i\Im \Scalar{\gf}{\psi(z)}
\end{pmatrix}.
\end{multline*}
Note that the first term does not depend on $z$. According
to\eqref{orient:b=1:1}, we have that $\psi'(0)=\psi_0$.
Furthermore, $a(z)$ and $g(z)$ vanish of second order in 0. We
thus conclude
\[
\big(\lfrac{d}{dz}\big|_{z=0} T_z\big)
\begin{pmatrix} \gf \\ a \\ f \end{pmatrix}
= \begin{pmatrix} \lfrac{1}{2}c(a)\psi_0 - f\psi_0 \\
q\big(\psi_0,\gf\big)\\
-i\Im \Scalar{\gf}{\psi_0}
\end{pmatrix}.
\]
Recalling that $\ker T_0 = \ker \cD_{A_0} \oplus \cH^1(M)\oplus
i\R$, we employ real coordinates
\[
(u,v,x,y)\mapsto \Big((u+iv)\psi_0,\,x
\lfrac{\go_0}{\|\go_0\|},\,iy\Big)\in \ker T_0.
\]
The operator $\big(\lfrac{d}{dz}\big|_{z=0} T_z\big)\big|_{\ker
T_0}$ is then represented by
\[
(u,v,x,y)\mapsto
\begin{pmatrix}
\lfrac{1}{2}c(x\lfrac{\go_0}{\|\go_0\|})\psi_0 - iy\psi_0 \\
q\big(\psi_0,u\psi_0+vi\psi_0\big)\\
-i\Im \Scalar{u\psi_0+vi\psi_0}{\psi_0}
\end{pmatrix}=\begin{pmatrix}
\lfrac{1}{2}c(x\lfrac{\go_0}{\|\go_0\|})\psi_0 - iy\psi_0 \\
u q(\psi_0)\\
-i v |\psi_0|^2
\end{pmatrix}.
\]
On the other hand, the orthogonal projection $\Proj_{\ker T_0}$
is given by
\[
(\gf,a,f)\mapsto \Big(\Re\LScalar{\gf}{\psi_0}\psi_0+
\Re\LScalar{\gf}{i\psi_0}i\psi_0,\,\LScalar{a}{\lfrac{\go_0}{\|\go_0\|}}
\lfrac{\go_0}{\|\go_0\|},\,i\|f\|_{L^2} \Big)
\]
since $(\psi_0,i\psi_0,\lfrac{\go_0}{\|\go_0\|},i)$ forms an
orthonormal basis of $\ker T_0$. With respect to this basis, the
operator $C_T(0)$ corresponds to
\[
(u,v,x,y)\mapsto
\begin{pmatrix}
\Re\LScalar{\lfrac{1}{2}c\big(x\lfrac{\go_0}{\|\go_0\|}
\big)\psi_0}{\psi_0}\\
\Re\LScalar{-iy\psi_0}{i\psi_0}\\
\LScalar{u q(\psi_0)}{\lfrac{\go_0}{\|\go_0\|}}\\
-v\|\psi_0\|_{L^2}\end{pmatrix}=
\begin{pmatrix}
x\lfrac{\gk}{\|\go_0\|}\\
-y\\
u\lfrac{\gk}{\|\go_0\|}\\
-v \end{pmatrix}.
\]
Note that we have applied Proposition \ref{q:prop} to express the
third row in terms of $\gk$.

We now conclude that the crossing operator has the matrix
description
\[
C_T(0)\begin{pmatrix} u\\ v\\ x\\ y\\ \end{pmatrix}=
\begin{pmatrix}
0 &0 &\lfrac{\gk}{\|\go_0\|} &0 \\
0 &0 &0 &-1 \\
\lfrac{\gk}{\|\go_0\|} &0 &0 &0\\
0 &-1 &0 &0
\end{pmatrix}
\begin{pmatrix} u \\ v \\ x \\ y \end{pmatrix}.
\]
Therefore, $\det C_T(0) = \lfrac{\gk^2}{\|\go_0\|^2}>0$ which
proves that the crossing operator at the starting point of the
path $T_z$ can only have an even number of negative eigenvalues.
Therefore, the spectral flow of $T_z$ is even. According to the
orientation transport formula of Theorem \ref{OT=SF}, we may
finally deduce that $\eps(T_z)=1$ which proves the claimed
formula.
\end{proof}

After having analysed the local structure of the parametrized
moduli space near the circle of reducibles, we are now able to
prove the main result of this section.

\begin{theorem}[Wall-Crossing Formula]\label{b=1}
\index{Seiberg-Witten invariant!wall-crossing formula} Let $M$ be
a closed, oriented 3-manifold with first Betti number $b_1=1$.
Suppose that $M$ is equipped with a \spinc structure $\gs$ and an
orientation of the second cohomology. For an arbitrary metric $g$
on $M$ let $\eta_{\pm 1}\in\cP_k^\pm(\gs)$ be suitable
perturbations with respect to $g$. Then for any path $\eta$
connecting $\eta_{-1}$ with $\eta_1$ and satisfying the
properties of Corollary \ref{eta:path:b=1},
\begin{equation}\label{b=1:form}
\sw_{\eta_1}(\gs)-\sw_{\eta_{-1}}(\gs) =
\SF(\cD_{A_0+r\go_0};\;{\scriptstyle r\in[0,1)}),
\end{equation}
where $A_0$ is an arbitrary reducible $\eta_0$-monopole, while
$\go_0$ is the generator of reducible $\eta_0$ monopoles.
Therefore,
\begin{equation*}
\sw^+(\gs) - \sw^-(\gs) = -\frac{1}{2} \int_{M} [\lfrac{1}{4\pi
i}\go_0]\wedge c(\gs).
\end{equation*}
\end{theorem}

\begin{proof}
According to part (ii) of Corollary \ref{eta:path:b=1}, the
irreducible part of $\widehat{\cM}_\eta(\gs)$ is a 1-dimensional
$C^3$-submanifold of $\cB^*\times [-1,1]$. Hence, it can be
written as the union of a finite number of $C^3$-arcs
\[
c_i:[a_i,b_i]\to \cB^*\times[-1,1],\quad
c_i(s)=[\psi_i(s),A_i(s),t_i(s)],\quad i=1,\ldots, n.
\]
As in the proof of Theorem \ref{b>1}, the contribution to
$\sw_{\eta_1}(\gs)-\sw_{\eta_{-1}}(\gs)$ of an arc connecting
points lying in the boundary $\big(\cM_{\eta_{-1}}\times
\{-1\}\big) \cup \big(\cM_{\eta_1}\times \{1\}\big)$ is always 0.
However, some endpoints may lie in the circle of reducibles (cf.
Fig. \ref{param:b1}). We have to show that the contribution of
these arcs is exactly the of the spectral flow of
$\cD_{A_0+r\go_0}$. Let us assume from now on that $c[a,b]\to
\cB\times [-1,1]$ is an arc for which, say, $c(b)$ lies in the
reducible part. As a consequence of Lemma \ref{loc:red:b=1:nd},
this implies that $A(b)$ is degenerate. Part (iii) of Corollary
\ref{eta:path:b=1} guarantees, though, that $A(b)$ is only
slightly degenerate. Therefore, Proposition \ref{orient:b=1}
describes the way in which way the image set of $c$ near $c(b)$
meets the circle of reducibles. To relate
$\eps\big(\psi(a),A(a)\big)$ with the spectral flow of
$\cD_{A(b)+r\go_0}$ we have to account for several different
situations.

Let us assume first that $c(a)\in\cM_{\eta_{-1}}(\gs)\times
\{-1\}.$ Then there are two possibilities of how $c$ might meet
the circle of reducibles (cf. Fig. \ref{b1:case1}).

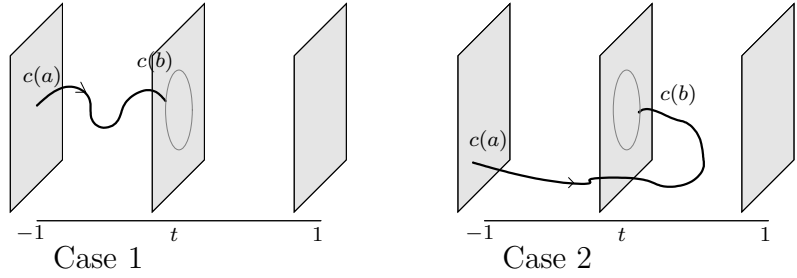
\begin{figure}
\begin{center}
\input{b1_case1.eepic}
\caption{The two possibilities how an arc might approach a
reducible}\label{b1:case1}
\end{center}
\end{figure}
\begin{Cases}

\item \textit{$c(s)$ is contained in $\cB^*\times [-1,0)$ for $s$
close to $b$:} Under this assumption, we can find $s_0$ in any
arbitrarily small neighbourhood of $b$ such that $t'(s_0)>0$.
Exactly as in the proof of Theorem \ref{b>1}, this results in
\[
\eps\big(T_{(\psi(s),A(s))};{\scriptstyle s\in[a,s_0]}\big)=1.
\]
In combination with homotopy invariance of the orientation
transport this proves
\[
\eps\big(\psi(a),A(a)\big)=\eps\big(\psi(s_0),A(s_0)\big).
\]
On the other hand, we infer from the last assertion of
Proposition \ref{orient:b=1} that
\[
\eps\big(\psi(s_0),A(s_0)\big)=1,
\]
where we possibly have to adjust $s_0$. Moreover, since $c(s)$
meets $[0,A(b),0]$ coming from $\cB^*\times [-1,0)$, we deduce
from the first part of the same result that
\[
\SF(\cD_{A(b)+r\go_0}; {\scriptstyle |r|\ll 1})=-1.
\]
Combining the above observations, we conclude that
\[
\eps\big(\psi(a),A(a)\big)= -\SF(\cD_{A(b)+r\go_0};{\scriptstyle
|r|\ll 1}).
\]
Therefore, the contribution of $c$ to
$\sw_{\eta_1}(\gs)-\sw_{\eta_{-1}}(\gs)$ is
$\SF(\cD_{A(b)+r\go_0};{\scriptstyle |r|\ll 1})$.

\item \textit{$c(s)$ is contained in $\cB^*\times (0,1]$ for $s$ close to
1:} Invoking Proposition \ref{orient:b=1} in the same way as
before, we find that in this case
\[
\eps\big(\psi(a),A(a)\big)=-\eps\big(\psi(s_0),A(s_0)\big)=-1,
\]
where $s_0$ is chosen appropriately. On the other hand,
\[
\SF(\cD_{A(b)+r\go_0}; {\scriptstyle |r|\ll 1})=1,
\]
and therefore again,
\[
\eps\big(\psi(a),A(a)\big)= - \SF(\cD_{A(b)+r\go_0};{\scriptstyle
|r|\ll 1}).
\]

Let us now assume that $c$ connects an element of
$\cM_{\eta_1}(\gs)\times \{1\}$ with a reducible. Similar
arguments show that independently of the direction from which $c$
meets the circle of reducibles,
\[
\eps(\psi(a),A(a))= \SF(\cD_{A(b)+r\go_0}; {\scriptstyle |r|\ll
1})
\]
so that again, the contribution to
$\sw_{\eta_1}(\gs)-\sw_{\eta_{-1}}(\gs)$ is the spectral flow of
$\cD_{A(b)+r\go_0}$ near $A(b)$.

However, we have not yet taken all possibilities into account for
$c$ might also connect two distinct reducibles $[0,A(a),0]$ and
$[0,A(b),0]$. In this situation, there are further cases to
distinguish (cf. Fig. \ref{b1:case3}).

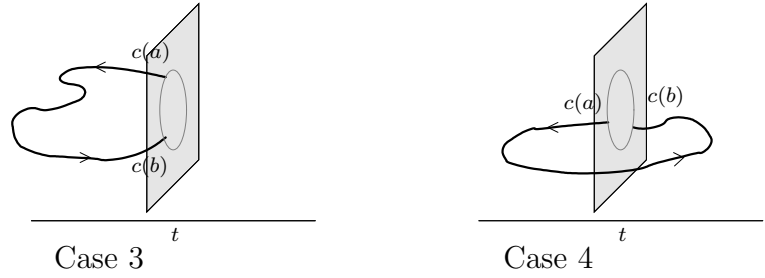
\begin{figure}
\begin{center}
\input{b1_case3.eepic}
\caption{Paths connecting distinct reducibles}\label{b1:case3}
\end{center}
\end{figure}

\item \textit{$c(s)\in \cB^*\times [-1,0)$ for both, $s\to a$ and
$s\to b$:} Choose $s_a$ and $s_b$ appropriately close to $a$ and
$b$ respectively. Then the last assertion of Proposition
\ref{orient:b=1} shows that
\[
\eps\big(\psi(s_a),A(s_a)\big)=1= \eps\big(\psi(s_b),A(s_b)\big).
\]
On the other hand, transferring the arguments of Theorem
\ref{b>1} to the situation at hand yields
\[
\eps\big(T_{(\psi(s),A(s))};{\scriptstyle s\in[s_a,s_b]}\big)=-1
\]
which contradicts the above. Hence, this case can actually not
occur. Clearly, the same occurs if we assume $c(s)\in \cB^*\times
(0,1]$ for $s$ close to $a$ and $b$.

\item \textit{$c(s)\in \cB^*\times [-1,0)$ for $s\to a$, and
$c(s)\in \cB^*\times (0,1]$ for $s\to b$:} In contrast to the
above case such an arc may indeed exist. According to the first
part of Proposition \ref{orient:b=1}, we then have
\[
\SF(\cD_{A(a)+r\go_0}; {\scriptstyle |r|\ll 1}) = -1 = -
\SF(\cD_{A(b)+r\go_0}; {\scriptstyle |r|\ll 1}).
\]
Therefore, such an arc neither contributes to the left hand side
nor to the right hand side of the equation we aim to prove.
\end{Cases}

These considerations show that every arc $c$ hitting the circle
of reducibles in some point $[0,A_i,0]$ yields a summand
$\SF(\cD_{A_i+r\go_0}; {\scriptstyle |x|\ll 1})$ in the left hand
side of \eqref{b=1:form}. Therefore,
\[
\sw_{\eta_1}(\gs)-\sw_{\eta_{-1}}(\gs) = \sum_{[0,A_i,0]}
\SF(\cD_{A_i+r\go_0}; {\scriptstyle |r|\ll 1}),
\]
where the sum is taken over all gauge equivalence classes
$[0,A_i,0]$ hit by an arc in the parametrized moduli space.

Letting $A_0$ be an arbitrary reducible $\eta_0$-monopole, part
(iii) of Corollary \ref{eta:path:b=1} ensures that
$\{\cD_{A_0+r\go_0}\}_{r\in[0,1)}$ is a transversal family with
only simple crossings. Hence, the spectral flow of
$\{\cD_{A_0+r\go_0}\}_{r\in[0,1)}$ is given by summing up the
contributions near all slightly degenerate reducibles. These are
exactly the points where an arc in the parametrized moduli space
hits the circle of reducibles. Therefore,
\[
\sum_{[0,A_i,0]} \SF(\cD_{A_i+r\go_0}; {\scriptstyle |r|\ll 1}) =
\SF(\cD_{A_0+r\go_0};{\scriptstyle r\in [0,1)})
\]
which proves the first version of the wall-crossing formula.\\

We will only briefly sketch how the second version follows from
the first one. Since $[\go_0]\in H_{dR}^1(M;4\pi i\Z)$, it
follows from \eqref{derham} that there exists a gauge
transformation $\gamma$ such that
\[
 [\go_0]=[2\gamma^{-1}d\gamma].
\]
This gauge transformation determines a line bundle $\Hat
L(\gs)\to S^1\times M$ by means of identifying the ends of the
pullback of $L(\gs)$ to $[0,1]\times M$ via $\gamma$. The family
$\{A+x\go_0\}_{x\in[0,1]}$ then gives rise to a connection $\Hat
A$ on $\Hat L(\gs)$. It turns out that $\Hat L(\gs)$ is the
canonical line bundle associated to the pullback\footnote{For a
similar discussion concerning manifolds of the type $[0,1]\times
M$ see \cite{Nic:SW}, Sec.~2.4.1.} of the \spinc structure $\gs$
on $M$ to $S^1\times M$. A result\footnote{The theorem we are
referring to seems to be rather a folklore result than a
well-established fact; it is, however, well motivated in loc. cit.
Moreover, Robbin \& Salamon \cite{RobSal:SF} give a rigorous
proof in a context which slightly differs from the situation at
hand.} of Atiyah et. al. \cite{AtiPatSin:SAR} ensures that the
spectral flow of $\{\cD_{A_0+x\go_0}\}_{x\in[0,1]}$ equals the
index of the \spinc Dirac operator associated to $\Hat A$. This
index can be computed using the Atiyah-Singer index Theorem and it
turns out that (see Lim \cite{Lim:SW}, Sec.~4.2)\index{Dirac
operator!Atiyah-Singer index Theorem}
\[
\SF(\cD_{A_0+x\go_0};{\scriptstyle x\in[0,1)})=\frac{1}{8}
\int_{[0,1]\times M}\lfrac{i}{2\pi}F_{\Hat A}\wedge
\lfrac{i}{2\pi}F_{\Hat A}\,.
\]
Since $\Hat A= A_0+x\go_0$, we compute
\begin{align*}
\lfrac{i}{2\pi}F_{\Hat A}\wedge \lfrac{i}{2\pi}F_{\Hat A}&=
(\lfrac{i}{2\pi}F_{A_0} + \lfrac{i}{2\pi}dx\wedge \go_0)^2 = 2
dx\wedge \lfrac{i}{2\pi}\go_0\wedge \lfrac{i}{2\pi}F_{A_0}.
\end{align*}
Therefore,
\[
\SF(\cD_{A_0+x\go_0};{\scriptstyle x\in[0,1)})=-\frac{1}{2}
\int_{M} \lfrac{1}{4\pi i}\go_0\wedge \lfrac{i}{2\pi}F_{A_0}=
-\frac{1}{2} \int_{M} [\lfrac{1}{4\pi i}\go_0]\wedge c(\gs).
\qedhere
\]
\end{proof}

\section[Manifolds with $b_1=0$]{Manifolds with
$\boldsymbol{b}_1\mathbf{=0}$}\label{=0}

Let $M$ be a closed, oriented 3-manifold with first Betti number
$b_1=0$. Hence, $H^1(M;\R)=H^2(M;\R)=0$. A manifold of this type
is usually called a {\em rational homology sphere} which refers to
the fact that its rational cohomology or, equivalently, its real
cohomology is the one of $S^3$. A manifold whose singular
homology equals $H_*(S^3;\Z)$ is then called an {\em integer
homology sphere}.

Let us first consider some specialities related to the vanishing
of $b_1$. First of all, the possible number of \spinc structures
is very limited. Since $H^2(M;\R)=0$, the image of the canonical
class of a \spinc structure $\gs$ in the real cohomology always
vanishes. Therefore, the canonical line bundle $L(\gs)$ is flat.
Hence, there exists up to gauge equivalence a unique flat
connection, which we denote\footnote{Since $A^\flat$ is a 1-form,
no notational confusion with the isomorphism $\flat:TM\to T^*M$
induced by the metric should be feasible.} by $A^\flat$.
Moreover, we have already observed in Corollary
\ref{red:exist:cor} that for each $\eta\in Z_k^2(M;i\R)$, the
perturbed moduli space $\cM_\eta(\gs)$ contains exactly one
reducible point. Up to a choice of $A^\flat$, a canonical
representative of this reducible is given by $A^\flat -
d^{-1}\eta$, where $d^{-1}\eta$ denotes the unique co-closed
1-form $\go$ such that $d\go=\eta$. Here, we are using that
$H^1(M;\R)=0$ which in association with Hodge decomposition
implies that $\gO^1(M)=\ker d\oplus \ker d^*$. Note, however,
that the map $d^{-1}$ depends on the metric.\\

\noindent\textbf{Count of monopoles in the case
$\boldsymbol{b_1=0}$.} Since the definition of $\sw_\eta(\gs)$ in
Definition \ref{suit:pert:b>0} is not well-suited if the
underlying manifold is a rational homology sphere, we have to
invoke some further considerations. As before, we call a reducible
$\eta$-monopole $A$ {\em non-degenerate} if $\cD_A$ is invertible
and {\em slightly degenerate} if $\cD_A$ has a one dimensional
kernel.

\begin{prop}\label{red:isol}
Let $M$ be a rational homology 3-sphere with Riemannian metric
$g$. Furthermore, let $\gs$ be a \spinc structure on $M$ and let
$A^\flat$ be a flat connection on $L(\gs)$. Then, for any $\eta\in
Z_k^2(M;i\R)$ such that $A^\flat-d^{-1}\eta$ is non-degenerate,
the reducible $\eta$-monopole lies isolated in $\cM_\eta(\gs;g)$.
\end{prop}
\begin{proof}
Let $A_0:=A^\flat-d^{-1}\eta$. We infer from the slice theorem
that a neighbourhood of $[0,A_0]$ in $\cB$ is homeomorphic to a
neighbourhood of 0 in $\big(L^2_1(M,S)\oplus \ker
(d^*|_{L_1^2})\big)/\U_1$. Modulo $\U_1$, the moduli space near
$[0,A_0]$ is then given by the zeros of
\[
\begin{split}
L^2_1(M,S)\oplus \ker (d^*|_{L_1^2})&\to
L^2(M,S)\oplus\ker d^*,\\
(\gf,a)&\mapsto \Proj_{L^2(M,S)\oplus\ker d^*}\SW(\gf,A_0+a).
\end{split}
\]
Note that the projection onto $L^2(M,S)\oplus\ker d^*$ does not
produce any new zeros. This is for the same reason as in the case
of the parametrized moduli space in Lemma \ref{s:lemma}. The
differential of the above map at the point 0 is given by
\[
\begin{split}
L^2_1(M,S)\oplus \ker (d^*|_{L_1^2})&\to L^2(M,S)\oplus\ker
d^*,\\ (\gf,a)&\mapsto \big(\cD_{A_0}\gf,-* da\big).
\end{split}
\]
Its kernel is clearly equal to $\ker\cD_{A_0}\oplus\ker d|_{\ker
d^*}$. Since $A_0$ is non-degenerate, $\ker\cD_{A_0}=0$.
Furthermore, $\ker d\cap\ker d^*=0$ for rational homology spheres.
Hence, the differential is invertible so that the inverse
function theorem shows that $0$ is an isolated point of the zero
set. Therefore, $[0,A_0]$ is an isolated point of
$\cM_\eta(\gs;g)$.
\end{proof}

\begin{dfn}\label{suit:pert:b=0}
Let $M$ be a rational homology 3-sphere equipped with a
Riemannian metric $g$ and a \spinc structure $\gs$.
\begin{enumerate}
\item An element $\eta\in Z^2_k(M;i\R)$ is called a {\em suitable
perturbation} with respect to $g$ if the $\eta$-perturbed moduli
space $\cM_\eta(\gs;g)$ consists only of non-degenerate points.
\item If $\eta\in Z^2_k(M;i\R)$ is a suitable perturbation with
respect to $g$, we define
\[
\sw_\eta(\gs;g):=\sum_{[\psi,A]\in\cM_\eta^*(\gs;g)} \eps(\psi,A),
\]
\end{enumerate}
\end{dfn}
Note that the sum is taken over the irreducible part of the
moduli space which is finite since, according to Proposition
\ref{red:isol}, the reducible point is isolated and thus cannot
be an accumulation point for irreducible monopoles. Recall from
Theorem \ref{generic:eta} that the condition that irreducible
$\eta$-monopoles are non-degenerate is a generic property. As we
shall see below, this is true also for the condition that the
reducible $\eta$-monopole is non-degenerate.\\

\noindent\textbf{Finding suitable paths.} Ultimately, we are
going to establish is a formula describing the dependence of
$\sw_\eta(\gs;g)$ on $g$ and $\eta$. Therefore, as in the
previous sections, we need to find an appropriate path connecting
two suitable perturbations.

\begin{prop}\label{Dirac:pert:b=0}
Let $M$ be a rational homology 3-sphere with \spinc structure
$\gs$, and let $A^\flat$ be a fixed flat connection on $L(\gs)$.
If $\{g_t\}_{t\in [-1,1]}$ is a smooth path of Riemannian metrics
on $M$, then a generic $C^m$-path $a_t:[-1,1]\to L^2_k(M,iT^*M)$
has the following properties:
\begin{enumerate}
\item The family $\{\cD_{A^\flat+a_t}^t\}_{t\in[-1,1]}$ is
transversal with only simple crossings,
\item If $\psi$ is a non-vanishing harmonic spinor with respect to
$\cD_{A^\flat+a_t}^t$, then
\[
\LScalar{d^{-1}*_t q_t(\psi)}{q_t(\psi)}^t\neq 0,
\]
\end{enumerate}
where the index $t$ refers to the metric $g_t$.
\end{prop}

\begin{proof}
We may consider both parts independently since the intersection
of two generic sets is again generic. The proof of the first part
is  very similar to the corresponding proof in the case $b_1=1$.
First of all, we fix a Riemannian metric $g$ on $M$ and consider
again the Hilbert bundle $V\to X$ over
$X:=L^2_1(M,S)\setminus\{0\}$, defined by
$V_\psi:=\ker\Re\Lscalar{i\psi}{.}$. Define
\[
\gF: X\to V,\quad \gF(\psi):=\cD_{A^\flat}\psi.
\]
This section is Fredholm of $\R$-index 1. Exactly as before, we
can make this section transversal by invoking the perturbation
\[
\Hat\gF: X\times P\to \pr_1^*V,\quad
\Hat\gF(\psi,a):=\cD_{A^\flat+a}\psi\;,
\]
where $P:=L^2_k(M,iT^*M)$ denotes the perturbation space.

Now let $g_t$ be a smooth path of Riemannian metrics on $M$.
Applying Proposition \ref{param:levelset}, we draw the conclusion
that for a generic $C^m$-path $a_t$, the set
\[
\bigcup_{t\in[-1,1]}\ker \cD_{A_0+a_t}^t\setminus \{0\}\times
\{t\}
\]
is either empty or carries the structure of a 2-dimensional real
$C^m$-sub\-mani\-fold of $X\times [-1,1]$. The same arguments as
in Proposition \ref{Dirac:pert:b=1} then ensure that $a_t$
satisfies part (i) of the assertion.\\

\noindent Turning our attention to part (ii), we fix the metric
again and consider the section
\[
\gF':X\to V\oplus \R,\quad \gF'(\psi):=
\Big(\cD_{A^\flat}\psi,\LScalar{d^{-1}*\Pi\circ
q(\psi)}{q(\psi)}\Big),
\]
where $\Pi:=\Proj_{\ker d^*}$. The relation
$(d^{-1})^*=*\,d^{-1}*$ together with a simple computation shows
that
\[
D_\psi\gF'(\gf)=\Big(\cD_{A^\flat}\gf, \LScalar{d^{-1}*\Pi\circ
q(\psi)}{4q(\psi,\gf)}\Big).
\]
We want to prove that $D_\psi\gF' : L_1^2(M,S)\to \ker
\Re\Lscalar{i\psi}{.}\oplus \R$ is a Fredholm operator of index
0. For this we consider the formal adjoint. Let $\gf_0\in
\big(\ker\Re\Lscalar{i\psi}{.}\cap L^2_1\big)$ and $r\in \R$.
Then for any $\gf\in L^2_1(M,S)$,
\begin{align*}
\LSCalar{D_\psi&\gF'(\gf)}{(\gf_0,r)}\\ &=
\Re\LScalar{\cD_{A^\flat}\gf}{\gf_0}+ r\cdot
\LScalar{d^{-1}*\Pi\circ q(\psi)}{4\cdot q(\psi,\gf)} \\
&=\Re\LSCalar{\gf}{\cD_{A^\flat}\gf_0 + 2r\cdot
c\big(d^{-1}*\Pi\circ q(\psi)\big)\psi}\;,
\end{align*}
where we have used the formula
$\LScalar{a}{q(\psi,\gf)}=\lfrac{1}{2}\Re\LScalar{c(a)\psi}{\gf}$
of Proposition \ref{q:prop}. Therefore, the formal adjoint of
$D_\psi\gF'$ is given by
\[
\begin{split}
\big(\ker\Re\Lscalar{i\psi}{.}\cap L^2_1\big)\oplus\R&\to L^2(M,S),\\
(\gf_0,r)&\mapsto \cD_{A^\flat}\gf_0 + 2r\cdot
c\big(d^{-1}*\Pi\circ q(\psi)\big)\psi\;.
\end{split}
\]
The first summand is a Fredholm operator of index 0. Since the
second term is compact, we deduce that the formal adjoint of
$D_\psi\gF'$ is Fredholm of index 0. Therefore, $D_\psi\gF'$ is
also Fredholm of index 0.

We now proceed in the spirit of the perturbation results in
Section \ref{reg:val} and make $\gF'$ transversal to the zero
section. Let
\[
\Hat\gF': X \times L^2_k(M,iT^*M) \to V\oplus\R
\]
be defined by
\[
\Hat\gF'(\psi,a):=\gF'(\psi)+\big(\lfrac{1}{2}c(a)\psi,0\big).
\]
Since $\cD_{A^\flat+a}\psi=0$ at a zero $(\psi,a)$ of $\Hat\gF'$,
we deduce from Proposition \ref{dastq} that $q(\psi)$ is
co-closed, i.e., $\Pi\circ q(\psi)=q(\psi)$. The differential of
$\Hat\gF'$ is then given by
\[
D_{(\psi,a)}\Hat\gF'(\gf,b)= \Big(\cD_{A^\flat+a}\gf + \lfrac12
c(b)\psi, \LScalar{d^{-1}* q(\psi)}{4q(\psi,\gf)}\Big).
\]
We claim that this map is surjective. Let $(\gf_0,r)$ be
$L^2$-orthogonal to the image of $D_{(\psi,a)}\Hat\gF'$. This
implies that
\[
\cD_{A^\flat+a}\gf_0 + 2r\cdot c\big(d^{-1}* q(\psi)\big)\psi=0
\quad\text{and}\quad \lfrac{1}{2}\Re\LScalar{c(b)\psi}{\gf_0)}=0
\]
for any $b\in L^2_k(M,iT^*M)$. As in the proof of Proposition
\ref{pert:map:trans}, the latter equation implies that there
exists $f\in L^2_1(M,i\R)$ with $\gf_0=f\psi$. Inserting this in
the first equation then shows that
\[
c\big(df + 2r\cdot d^{-1}* q(\psi)\big)\psi=0.
\]
By virtue of the unique continuation principle, $\psi$ is nowhere
vanishing on a dense open subset of $M$. Hence, necessarily $df +
2r\cdot d^{-1}* q(\psi)=0$. As the first summand of left hand
side is closed and the second summand is co-closed, we infer that
both terms vanish. Therefore, $f$ is constant and $r\cdot d^{-1}*
q(\psi)=0$. The first fact implies $\gf_0\equiv 0$ because
$\Re\Lscalar{i\psi}{f\psi}=\Re\Lscalar{i\psi}{\gf_0}=0$. For the
second term note that $q(\psi)$ is zero only at points where
$\psi$ vanishes. Therefore, $d^{-1}* q(\psi)$ cannot vanish
everywhere which implies that $r=0$.

To finish the proof of part (ii), we now consider an arbitrary
smooth path $g_t$ of Riemannian metrics on $M$. An $m$-times
continuously differentiable path $a_t:[-1,1]\to L^2_k(M,iT^*M)$
defines the parametrized zero set
\[
\bigcup_{t\in[-1,1]}\Hat\gF_{g_t}'(.,a_t)^{-1}(0)\times \{t\}.
\]
Proposition \ref{param:levelset} guarantees that for a generic
choice of such path, this set is either empty or a 1-dimensional
submanifold of $X\times [-1,1]$. Part (ii) may now be proved by
contradiction. Assume that there exists a tuple $(\psi,t)$ such
that $\cD_{A^\flat+a_t}^t\psi=0$ and $\LScalar{d^{-1}*_t
q_t(\psi)}{q_t(\psi)}=0$. Clearly, every multiple of $\psi$ by a
real constant also satisfies these two equations. On the other
hand, they also hold for $i\psi$ and its real multiples because
$q(i\psi)=q(\psi)$. Therefore, the fibre
$\Hat\gF_{g_t}'(.,a_t)^{-1}(0)$ is at least of real dimension 2,
which contradicts the fact that the parametrized zero set is
1-dimensional.
\end{proof}

\begin{cor}\label{eta:path:b=0}
Let $M$ be a rational homology 3-sphere, endowed with a \spinc
structure $\gs$ and a flat connection $A^\flat$ on $L(\gs)$.
Moreover, let $g_{-1}$ and $g_1$ be Riemannian metrics on $M$.
Suppose $\eta_{-1}$ and $\eta_1$ are respectively chosen suitable
perturbations. Then there exist $C^m$-paths $g_t$ and $\eta_t$ of
metrics and perturbations connecting $\,g_{-1}$ with $g_1$ and
$\eta_{-1}$ with $\eta_1$ respectively and having the following
properties:
\begin{enumerate}
\item The family $\{\cD_{A^\flat-d^{-1}\eta_t}^t\}_{t\in[-1,1]}$
is transversal with only simple crossings.
\item $\LScalar{d^{-1}*_t q_t(\psi)}{q_t(\psi)}^t\neq 0$
whenever $\psi$ is a non-vanishing harmonic spinor with respect to
$\cD_{A^\flat-d^{-1}\eta_t}^t$.
\item The irreducible part of the parametrized moduli space
$\widehat{\cM}_\eta(\gs;g)$ is a 1-dimensional $C^m$-submanifold
of $\cB^*\times [-1,1]$.
\end{enumerate}
\end{cor}
\begin{proof}
Let $g_t$ be a path of Riemannian metrics connecting $g_{-1}$ and
$g_1$. Proposition \ref{Dirac:pert:b=0} shows that a generic
$C^m$-path of imaginary valued 1-forms defines a family
$\{\cD_{A^\flat+a_t}^t\}$ satisfying the first two properties
respectively. Due to Hodge decomposition we can write
\[
a_t=df_t + d^*\mu_t,
\]
where $f_t:[-1,1]\to L^2_k(M,i\R)$ and $\mu_t:[-1,1]\to
L^2_k(M,i\gL^2 T^*M)$ are $C^m$-paths. Using the bounded inverse
$d^{-1}$ of $d:\im d^*\to Z^2(M;i\R)$, we find a path $\eta_t$ of
closed, imaginary valued 1-forms such that
\[
d^*\mu_t = - d^{-1}\eta_t.
\]
Therefore, the path $a_t$ is gauge equivalent to $-d^{-1}\eta_t$.
Making the simple but important observation that properties (i)
and (ii) are preserved if we apply a path of gauge
transformations to $a_t$, we deduce that the family
$\{\cD_{A^\flat-d^{-1}\eta_t}^t\}$ satisfies these properties as
well. Moreover, one achieves that the paths $\eta_t$ obtained in
this manner form a generic subset of the set paths connecting
$\eta_{-1}$ and $\eta_1$.

On the other hand, an application of Theorem \ref{generic:eta} as
in the proof of Theorem \ref{b>1} shows that a generic path
$\eta$ of closed imaginary valued 2-forms satisfies property
(iii). Since the intersection of two generic sets is again
generic, the assertion of the proposition follows.
\end{proof}

\begin{remark*}
A slightly modified consideration shows that the set of all
$\eta\in Z^2_k(M;i\R)$ such that $\cD_{A^\flat-d^{-1}\eta}$
is invertible forms a generic subset: \\
From the proof of Proposition \ref{Dirac:pert:b=0}, we know that a
generic choice of $a\in L^2_{k+1}(M,iT^*M)$ gives rise to an
invertible operator $\cD_{A^\flat+a}$. Possibly applying a gauge
transformation, we may assume that $a=-d^{-1}\eta$ for some
closed, imaginary valued 2-form $\eta$.\\
\end{remark*}

\noindent\textbf{Local structure of the parametrized moduli
space.} Fixing a $C^m$-path $\eta$ with $m\ge 2$ as in the above
corollary, we will now make a similar analysis as in Section
\ref{=1}.

The parametrized moduli space consists of a finite union of arcs,
one of which is the reducible branch parametrized by
$[0,A^\flat-d^{-1}\eta_t,t]$. The irreducible part forms a
1-dimensional $C^m$-submanifold of $\cB^*\times [-1,1]$, and
singularities occur whenever an irreducible arc meets the
reducible one (cf. Fig. \ref{param:b0}). We thus need to
understand the local structure of $\widehat{\cM}_{\eta}(\gs;g)$
near a reducible point $[0,A^\flat-d^{-1}\eta_{t_0},t_0]$.
Without loss of generality, we may assume in the following that
$t_0=0$. We then define $A_0:=A^\flat-d^{-1}\eta_0$.

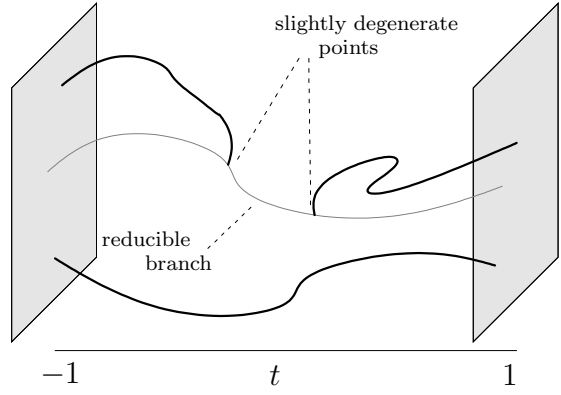
\begin{figure}
\begin{center}
\input{param_b0.eepic}
\caption{The parametrized moduli space in the case
$b_1=0$}\label{param:b0}
\end{center}
\end{figure}

Using the notation of Section \ref{app:spinc}.\ref{met:dep}, we
fix $g_0$ as a reference metric and employ the isometries $\Hat
k_t:L^2(M,T^*M;g_0)\to L^2(M,T^*M;g_t)$ and $\Hat
\gk_t:L^2(M,S;g_0)\to L^2(M,S;g_t)$ to identify configurations
associated to different metrics. According to the slice theorem,
a neighbourhood of $[0,A_0,0]$ in $\cB\times [-1,1]$ is
homeomorphic to $U/\U_1\times (-\eps,\eps)$, where $U$ is a
$\U_1$-invariant open subset of $L^2_1(M,S;g_0)\times \ker
(d^*|_{L_1^2})$.\footnote{For notational convenience, we are
dropping the reference to the metric $g_0$ in the adjoint of
$d$.} Modulo $\U_1$, the parametrized moduli space near
$[0,A_0,0]$ is then readily seen to be given by elements
$(\gf,a,t)\in U\times (-\eps,\eps)$ satisfying
\begin{equation}\label{loc:sol:b=0}
\big(\cD_{A^\flat-d^{-1}\eta_t+a}^t\gf,\;
q^t(\Hat\gk_t\gf)-*_tda\big)=0.
\end{equation}
As in the analogous situation in Section \ref{=1}, we now define
\[
s:L^2_1(M,S)\oplus\ker (d^*|_{L_1^2})\times
(-\eps,\eps)\longrightarrow L^2(M,S)\oplus \ker d^*
\]
by letting
\[
s(\gf,a,t):=\Proj_{L^2(M,S)\oplus\ker d^*} \Big(
\cD_{A^\flat-d^{-1}\eta_t+a}^t\gf,\; \Hat k_t^{-1}\big(
q^t(\Hat\gk_t\gf)-*_tda\big)\Big).
\]
Observe that we have to employ the isometry $\Hat k_t^{-1}$ since
$q^t(\Hat\gk_t\gf)-*_tda\in \ker d^{*_t}$ for any solution of
\eqref{loc:sol:b=0} but not necessarily
$q^t(\Hat\gk_t\gf)-*_tda\in \ker d^*$. Formula \eqref{dast:met}
shows that $\Hat k_t$ induces an isomorphism $\ker d^*\to \ker
d^{*_t}$, and we deduce exactly in the same manner as before that
the solutions of \eqref{loc:sol:b=0} coincide with the zeros of
$s$.

We shall need the differential
\[
D_0s:L^2_1(M,S)\oplus \ker (d^*|_{L_1^2})\oplus \R \to
L^2(M,S)\oplus \ker d^*
\]
of $s$ at the point 0 in order to invoke the implicit function
theorem. A short computation yields
\begin{equation*}
D_0s(\gf,a,t)=\Proj_{L^2(M,S)\oplus\ker d^*}\big(\cD_{A_0}\gf, -*
da\big)= \big(\cD_{A_0}\gf, -* da\big)
\end{equation*}
Since $\ker d^*\cap \ker d= 0$ on a rational homology sphere, we
readily infer that
\begin{align*}
\ker D_0s &= \ker \cD_{A_0}\oplus \{0\}\oplus \R \\
\coker D_0s &= \ker \cD_{A_0} \oplus\{0\}\;.
\end{align*}

Very similar as before we then deduce the following.
\begin{lemma}
If $A_0$ is non-degenerate, then the parametrized moduli space
near $[0,A_0,0]$ is locally homeomorphic to the reducible branch
\[
\bigcup_{t\in [-1,1]} [0,A^\flat-d^{-1}\eta_t,t].
\]
\end{lemma}

\begin{prop}\label{loc:red:b=0}
Suppose that $A_0$ is slightly degenerate, and let $\psi_0\in
L^2_1(M,S)$ be a spinor of norm 1 spanning $\ker \cD_{A_0}$. Then
in a neighbourhood of 0 in $L^2_1(M,S)\oplus \ker (d^*|_{L_1^2})
\times (-\eps,\eps)$, the following holds:
\[
s(\gf,a,t)=0 \Longleftrightarrow (\gf,a,t)=
\begin{cases}
(0,0,t) &\text{ if } \gf=0 \\
\big(z\psi_0,0,h(z)\big) + f(z,h(z)) &\text{ if } \gf\neq 0
\end{cases}\;,
\]
where
\[
f:\C\times\R\to (\ker \cD_{A_0}\oplus \R)^\perp
\]
is a $\U_1$-equivariant $C^m$-map, and $h:\C\to \R$ is a
$\U_1$-invariant $C^{m-2}$-map. Both maps vanish of second order
in 0.
\end{prop}
\begin{proof}
We apply the Kuranishi technique again. Let
\[
\gF:=\Proj_{\im D_0s}\circ s\quad \text{ and } \quad
\Psi:=\Proj_{\coker D_0s}\circ s.
\]
Transferring the arguments from the case $b_1=1$, we obtain a
$\U_1$-equivariant $C^m$-map
\[
f:\ker \cD_{A_0}\oplus \R\to (\ker \cD_{A_0}\oplus \R)^\perp
\]
such that
\[
\gF(\gf,a,t)=0 \Longleftrightarrow (\gf,a,t)=(z\psi_0,0,t)+
f(z,t),
\]
where we are using coordinates $z\in\C$ with respect to $\psi_0$.
The Kuranishi obstruction map is then given by
\[
\Hat \Psi:\C\times\R\to \ker \cD_{A_0},\quad \Hat\Psi(z,t):=
\Psi\big((z\psi_0,0,t)+ f(z,t)\big). \;
\]
Since $t\mapsto A^\flat-d^{-1}\eta_t$ parametrizes the reducible
part of $\widehat{\cM}_\eta(\gs;g)$, we find that $s(0,0,t)=0$ and
hence also $f(0,t)=\Hat\Psi(0,t)=0$. As before, we thus study
\[
\Hat\Psi(z,t)=z\cdot \Hat\Psi_1(z,r),\quad \text{where }
\Hat\Psi_1(z,t):=\begin{cases} \frac{\Hat\Psi(z,t)}{z}, &z\neq
0,\\ \frac{\partial}{\partial z}\big|_{(0,t)}\Hat\Psi(z,t), & z=0.
\end{cases}
\]
As $\Hat\Psi$ is a $C^m$-map, the function $\Hat\Psi_1$ is at
least $C^{m-2}$.  Before applying the implicit function theorem
again, we have to ascertain that the $t$ derivative of this
function does not vanish in $(0,0)$. Letting
$\Pi:=\Proj_{\ker\cD_{A_0}}$, we compute
\begin{multline*}
\lfrac{\partial}{\partial z}\big|_{(0,t)}\Hat\Psi(z,t)
=\Pi\Big(\lfrac{\partial}{\partial z}\big|_{(0,t)}
s(z\psi_0,0,t)\Big) + \Pi \Big(\lfrac{\partial}{\partial
z}\big|_{(0,t)} (s\circ f)\Big)
\\
=\LScalar{\cD_{A^\flat-d^{-1}\eta_t}^t\psi_0}{\psi_0}\cdot\psi_0 +
\LScalar{\lfrac{\partial}{\partial z}\big|_{(0,t)}
\cD_{A_0+a(z,t)}\gf(z,t)}{\psi_0}\cdot \psi_0\,,
\end{multline*}
where $\gf(z,t)$ and $a(z,t)$ denote the spinor and the 1-form
part of $f(z,t)$ respectively. With exactly the same arguments as
in the proof of Proposition \ref{loc:red:b=1}, one deduces that
the second term in the above equation's last line vanishes. Thus,
\begin{equation}\label{kappa:b=0}
\begin{split}
\sgn\LScalar{&\lfrac{\partial}{\partial t}\big|_{(0,0)}
\Hat\Psi_1(z,t)}{\psi_0} =
\sgn\LScalar{\lfrac{\partial^2}{\partial t\partial
z}\big|_{(0,0)}\Hat\Psi(z,t)}{\psi_0} \\
&=\sgn \LScalar{\lfrac{d}{dt}\big|_{t=0}
\cD_{A^\flat-d^{-1}\eta_t}^t\psi_0}{\psi_0}\\
&=\SF\big(\cD_{A^\flat-d^{-1}\eta_t}^t;\;{\scriptstyle |t|\ll
1}\big)=\pm 1\;,
\end{split}
\end{equation}
where we are using that $\{\cD_{A^\flat-d^{-1}\eta_t}^t\}$ has
transversal spectral flow with only simple crossings. In
particular, $\lfrac{\partial}{\partial t}\big|_{(0,0)}
\Hat\Psi_1(z,t)\neq 0$. Hence, the implicit function theorem
produces a $\U_1$-invariant $C^{m-2}$-map $h:\C\to\R$ such that in
a neighbourhood of (0,0),
\[
\Hat \Psi_1(z,t) = 0 \Longleftrightarrow t=h(z).
\]

Very similar to the situation in the preceding chapter, we then
infer that in a neighbourhood of 0, the claimed condition holds.
The arguments concerning the property that $f$ and $h$ vanish of
second order in 0 are also the same as before.
\end{proof}

\begin{remark*}
As a result, we find that a neighbourhood of $[0,A_0,0]$ in
$\widehat{\cM}_\eta(\gs)$ is homeomorphic to a neighbourhood of 0
in
\[
\bigsetdef{(z,x)\in \R_+^0\times \R}{z=0 \text{ or } x=0},
\]
where $\R_+\times \{0\}$ corresponds to the irreducible part. The
branch $\{0\}\times\R$ corresponds to the reducible arc
$[0,A^\flat-d^{-1}\eta_t,t]$.
\end{remark*}

According to Proposition \ref{loc:red:b=0}, the $t$-component of
the irreducible branch near $[0,A_0,0]$ is given by the value of
$h$, which we shall henceforth regard as a function $\R_+^0\to
\R$. Therefore, $h$ encodes information about on which side of
$\cB\times \{0\}$ the irreducible branch near $[0,A_0,0]$ is
located. Since $h$ is a $C^{m-2}$-function vanishing of second
order in 0, it is promising to assume that $m\ge 4$ and study the
second derivative of $h$: If $h''(0)<0$, then $h$ has a maximum
in 0 so that the irreducible branch near $[0,A_0,0]$ is contained
in $\cB\times [-1,0)$. If $h''(0)>0$, then this branch lies in
$\cB\times (0,1]$.

The next result shows how to relate this with the spectral flow
$\SF(\cD_{A^\flat-d^{-1}\eta_t};\;{\scriptstyle |x|\ll 1})$ and
the number $\LScalar{d^{-1}* q(\psi_0)}{q(\psi_0)}$. Recall that
both numbers do not vanish according to the choice of $\eta$.

\begin{prop}\label{orient:b=0}
Assume that $m\ge 4$. Then the second derivative of $h$ is given
by the formula
\begin{equation*}
h''(0)\cdot\LScalar{\lfrac{d}{dt}\big|_{t=0}
\cD_{A^\flat-\eta_t}^t\psi_0}{\psi_0} = -2\LScalar{d^{-1}*
q(\psi_0)}{q(\psi_0)}
\end{equation*}
and thus never equals 0. Moreover, if $m\ge 6$ and if $[\psi,A,t]$
is an element of the irreducible branch close to $[0,A_0,0]$, then
\begin{equation*}
\eps(\psi,A)= \sgn h''(0)\cdot
\SF(\cD_{A^\flat-d^{-1}\eta_t};\;{\scriptstyle |t|\ll 1}).
\end{equation*}
\end{prop}

\begin{proof}
Let us write $f(z,h(z))=\big(\gf(z),a(z),0\big)$. The irreducible
branch is then locally parametrized by
\[
(0,z_0)\to \cB\times[-1,1],\quad z\mapsto [\psi(z),A(z),h(z)],
\]
where $z_0\in\R_+$ is appropriately small, and where
\[
\psi(z):=z\psi_0 +\gf(z)\quad\text{and}\quad
A(z):=A^\flat-d^{-1}\eta_{h(z)}+a(z)\,.
\]
According to the above proposition, the path
$\big(\psi(z),A(z),h(z)\big)$ can be extended to $z=0$ in a twice
continuously differentiable way since we have chosen $m\ge 4$.

We differentiate the Dirac equation for $(\psi(z),A(z))$ two
times and obtain
\[
0=\lfrac{d^2}{dz^2} \Big(\cD_{A(z)}^{h(z)}\Big)\psi(z) +
2\lfrac{d}{dz}\Big(\cD_{A(z)}^{h(z)}\Big)\psi'(z) +
\cD_{A(z)}^{h(z)}\psi''(z).
\]
If we now take the inner product with $\psi_0$, then a careful
analysis shows that we can differentiate the resulting equation
at $z=0$ once again. For this we observe by making use of
self-adjointness of $\cD_{A(z)}^{h(z)}$ and the chain rule
\begin{align*}
\lfrac1z\LSCalar{\cD_{A(z)}^{h(z)}\psi''(z)}{&\psi_0}=
\LSCalar{\psi''(z)}{\lfrac1z\cD_{A(z)}^{h(z)}\psi_0}\\
&\xrightarrow{z\to 0}
\LSCalar{\psi''(0)}{h'(0)\lfrac{d}{dt}\big|_{t=0}
\cD_{A^\flat-\eta_t}^t\psi_0+c(a'(0))\psi_0}=0
\end{align*}
because $a$ and $h$ vanish of second order. Moreover,
\begin{align*}
\lfrac1z\LSCalar{\lfrac{d^2}{dz^2}
\Big(\cD_{A(z)}^{h(z)}\Big)\psi(z)}{\psi_0} &\xrightarrow{z\to 0}
\LSCalar{\lfrac{d^2}{dz^2}\big|_{z=0}
\Big(\cD_{A(z)}^{h(z)}\Big)\psi'(0)}{\psi_0}.
\end{align*}
Thus, the chain rule implies that
\begin{align*}
0&=\LSCalar{\lfrac{d^2}{dz^2}\big|_{z=0}
\big(\cD_{A(z)}^{h(z)}\big)\psi'(0)}{\psi_0}\\
&=\LSCalar{\lfrac{d}{dz}\big|_{z=0}\Big(h'(z)\lfrac{d}{dt}
\big|_{t=h(z)}\cD_{A^\flat-\eta_t}^t+
\lfrac12c(a'(z))\Big)\psi'(0)}{\psi_0} \\
&=\LSCalar{h''(0)\lfrac{d}{dt}\big|_{t=0}
\cD_{A^\flat-\eta_t}^t\psi_0+ \lfrac12c(a''(0))\psi_0 }{\psi_0},
\end{align*}
where we invoke that $\psi'(0)=\psi_0$ and that $h$ and $a$ vanish
of second order in 0. Therefore,
\[
h''(0)\cdot\LScalar{\lfrac{d}{dt}\big|_{t=0}
\cD_{A^\flat-\eta_t}^t\psi_0}{\psi_0} = -
\lfrac{1}{2}\LScalar{c(a''(0))\psi_0}{\psi_0}.
\]
The term on the right hand side can be computed by making use of
the second part of the Seiberg-Witten equations. Since
\[
q_{h(z)}\big(\Hat\gk_{h(z)}\psi(z)\big)= *_{h(z)}da(z),
\]
differentiating twice and using that $\psi(0)=0$ and
$a(0)=a'(0)=0$ yields:
\[
2q(\psi_0)=* da''(0),\quad\text{i.e.,}\quad a''(0)=2d^{-1}*
q(\psi_0).
\]
Combining this with the above result and invoking Proposition
\ref{q:prop}, we infer that
\begin{equation*}
h''(0)\cdot\LScalar{\lfrac{d}{dt}\big|_{t=0}
\cD_{A^\flat-\eta_t}^t\psi_0}{\psi_0} = -2\LScalar{d^{-1}*
q(\psi_0)}{q(\psi_0)}.
\end{equation*}
This proves the first assertion.\\

We shall now determine the number $\eps(\psi,A)$ for an
irreducible $\eta_t$-monopole such that $[\psi,A,t]$ lies on the
irreducible branch close to $[0,A_0,0]$. As in the corresponding
situation before, this amount to compute the spectral flow of the
family
\[
T_z:=T_{(\psi(z),A(z))}^{h(z)}\quad z\in[0,z_0],
\]
where we suppose that
$\big(\psi(z_0),A(z_0),h(z_0)\big)=(\psi,A,t)$. In addition, we
may assume that $T_z$ is invertible for any $z\neq 0$. We write
\[
T_z=D_z+ K_z,
\]
where $D_z$ and $K_z$ are given by
\[
D_z \begin{pmatrix}\gf \\a \\ f\end{pmatrix} =
\begin{pmatrix}
\cD_{A^\flat-d^{-1}\eta_{h(z)}}^{h(z)}\gf\\
-*_{\scriptscriptstyle h(z)}da + 2df \\
2d^{*_{h(z)}}a\end{pmatrix}, \quad
\begin{pmatrix}\gf
\\a \\ f\end{pmatrix}\in L^2_1(M,E\oplus i\R)
\]
and
\[
K_z
\begin{pmatrix}\gf \\a \\ f\end{pmatrix} =
\begin{pmatrix}\lfrac{1}{2}c(a(z))\gf+\lfrac{1}{2}c(a)\psi(z)
-f\psi(z) \\
q_{h(z)}\big(\psi(z),\gf\big) \\
-i\Im\Scalar{\gf}{\psi(z)}_{h(z)}\end{pmatrix}, \quad
\begin{pmatrix}\gf
\\a \\ f\end{pmatrix}\in L^2_1(M,E\oplus i\R)
\]
respectively. As $h(z)$ vanishes of second order in 0, we infer
from the chain rule that
\[
\lfrac{d}{dz}\big|_{z=0} D_z = 0.
\]
Since $\psi'(0)=\psi_0$ and $a'(0)=0$, this results in:
\begin{equation}\label{orient:b=0:1}
\big(\lfrac{d}{dz}\big|_{z=0} T_z\big)
\begin{pmatrix} \gf \\ a \\ f \end{pmatrix} =
\begin{pmatrix}
\lfrac{1}{2}c(a)\psi_0 - f\psi_0 \\
q\big(\psi_0,\gf\big)\\
-i\Im \Scalar{\gf}{\psi_0}
\end{pmatrix},\quad
\begin{pmatrix} \gf\\ a\\ f\end{pmatrix}\in L^2_1(M,E\oplus i\R).
\end{equation}
Next observe that
\[
\ker T_0=\ker\cD_{A_0}\oplus\{0\}\oplus i\R= \Span_\R
\{\psi_0,i\psi_0,i\}.
\]
If we employ real coordinates $x_1$, $x_2$ and $y$ with respect
to the above orthonormal basis, then a very similar computation
as in the proof of Proposition \ref{orient:b=1} shows that the
crossing operator $C_T(0)=\Proj_{\ker T_0}\circ
\big(\lfrac{d}{dz}\big|_{z=0} T_z\big)|_{\ker T_0}$ has the
matrix description
\begin{equation}\label{orient:b=0:2}
C_T(0)\begin{pmatrix}x_1 \\ x_2 \\ y \end{pmatrix} =
\begin{pmatrix}
0 &0 &0 \\ 0 &0 &-1 \\ 0 &-1 &0
\end{pmatrix}\begin{pmatrix}
x_1 \\ x_2 \\ y \end{pmatrix}.
\end{equation}
We conclude that $\spec(C_T(0))=\{-1,0,1\}$.

In association with Kato's Selection Theorem \ref{Kato:sel}, this
corresponds to the following situation: We can choose three
$C^1$-paths $\gt_i:[0,z_0]\to \R$, $i\in\{-1,0,1\}$ parametrizing
the eigenvalues of $T_z$ which equal zero for $z=0$. The
structure of the crossing operator then shows that $\gt'_i(0)=i$.
In particular, the eigenvalue $\gt_0$ vanishes of second order in
0 so that we cannot compute the spectral flow by means of the
crossing operator. We thus need a more refined analysis of
$\gt_0$'s behaviour near 0.
\begin{remark*}
In the literature, the path of operators $T_z$ is usually assumed
to be analytically. One then finds an analytic parametrization of
eigenvalues and corresponding eigenvectors (cf. \cite{Che:CI},
\cite{Nic:SW3}). However, the author of this thesis does not see
how to achieve analyticity of $T$. Fortunately, the
considerations in Appendix \ref{app:SF&OT} provide a way out of
this trouble if we choose $m\ge 6$. Nevertheless, the involved
computations remain essentially the same.
\end{remark*}
According to Example \ref{gl:sec}, we may assume that $\gt_0$ is
twice continuously differentiable, if $m\ge 6$.  Moreover, we
find a $C^2$-family $v_z\in L^2_1(M,E\oplus i\R)$ locally
corresponding to the eigenvectors associated to $\gt_0$. Formula
\eqref{gl:sec:der} then shows that
\begin{equation}\label{gt''}
\begin{split}
\gt_0''(0)&=\lfrac{d}{dz}\big|_{z=0}
\Re\LScalar{\big(\lfrac{d}{dz}\big|_z T_z\big)v_z}{v_z} \\
&=\Re\LScalar{T_0''v_0 + T_0'v_0'}{v_0} +
\Re\LScalar{T_0'v_0}{v_0'} \\ &=\Re\LScalar{T_0''v_0 +
2T_0'v_0'}{v_0},
\end{split}
\end{equation}
where we are using obvious abbreviations and invoke
self-adjointness of $T_0'$.

First of all, note that $v_0$ lies in the kernel of $T_0$. Letting
$P:=\Proj_{\ker T_0}$, we thus have $v_0=Pv_0$. Moreover,
differentiating the equation $T_z v_z =\gt_0(z)v_z$ implies that
\[
T_0'v_0 + T_0 v_0' = \gt_0'(0)v_0 + \gt_0(0)v_0'= 0
\]
for $\gt_0=\gt_0'=0$. From this and the fact that $PT_0=0$ we
deduce
\[
0=PT_0'v_0 + PT_0 v_0' = PT_0'Pv_0 = C_T(0)v_0.
\]
Therefore, $v_0\in \ker C_T(0)$. As a consequence of
\eqref{orient:b=0:2}, the kernel of the crossing operator is given
by $\Span_\R \psi_0$. Henceforth, we may thus assume that
$v_0=\psi_0$. To determine $v_0'$ we deduce from the above
computations and the explicit formula of $T_0'$ in
\eqref{orient:b=0:1} that
\[
-T_0v_0'=T_0'v_0=T_0'\psi_0 = \begin{pmatrix} 0\\ q(\psi_0)\\
0 \end{pmatrix}.
\]
Recalling that $T_0$ is explicitly given by
\[
T_0 = \begin{pmatrix} \cD_{A_0} &0 &0 \\
                0 &-* d &2d \\
                0 &2d^* &0 \end{pmatrix},
\]
we infer that
\[
v_0'\in d^{-1}* q(\psi_0) +\ker T_0,\quad\text{i.e., }
v_0'=\begin{pmatrix} z\psi_0 \\ d^{-1}* q(\psi_0)\\ iy
\end{pmatrix}
\]
for suitable $z\in\C$ and $y\in\R$. According to
\eqref{orient:b=0:1}, we thus find the following:
\begin{align*}
\Re\LScalar{T_0'v_0'}{v_0} &= \Re\LScalar{\lfrac12
c\big(d^{-1}* q(\psi_0)\big)\psi_0-iy\psi_0}{\psi_0} \\
&=\LScalar{ d^{-1}* q(\psi_0)}{q(\psi_0)}.
\end{align*}
Considering the second term in \eqref{gt''}, we immediately deduce
from the definition of $T_z$ that
\[
\Re\LScalar{T_0''v_0}{v_0}=\Re
\LSCalar{\lfrac{d^2}{dz^2}\big|_{z=0}
\Big(\cD_{A(z)}^{h(z)}\Big)\psi_0}{\psi_0}.
\]
As a result of the computations performed in the first part of
this proof, this term equals zero. Hence, we finally draw the
following conclusion:
\begin{align*}
\gt_0''(0)&=2\Re\LScalar{T_0'v_0'}{v_0}= 2\LScalar{ d^{-1}*
q(\psi_0)}{q(\psi_0)}\\
&=-h''(0)\cdot\LScalar{\lfrac{d}{dt}\big|_{t=0}
\cD_{A^\flat-\eta_t}^t\psi_0}{\psi_0},
\end{align*}
where we have employed the first part of this result.

\begin{figure}
\begin{center}
\input{teta.eepic}
\caption{Spectral flow of $T_z$}\label{teta}
\end{center}
\end{figure}
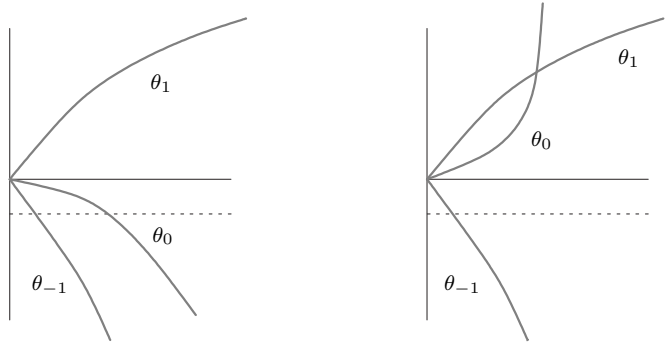

The spectral flow of $T_z$ in 0 is now obtained in the following
way (cf. Fig. \ref{teta}): If $\gt_0''(0)>0$, then the function
$\gt$ has a minimum in 0. Therefore, the only eigenvalue leaving 0
in the negative direction is $\gt_{-1}$. Due to our convention of
counting eigenvalues at the endpoints, the spectral flow of $T_z$
then equals 1. In accordance with the orientation transport
formula in Theorem \ref{OT=SF}, we conclude that
$\eps(\psi,A)=-1$. If, on the other hand, $\gt_0''(0)<0$, then the
eigenvalue $\gt_0(z)$ is also negative for small $z>0$. It thus
contributes to the spectral flow of $T_z$, and we find that
$\eps(\psi,A)=1$. Therefore,
\[
\eps(\psi,A)= -\sgn \gt_0''(0)= \sgn
h''(0)\cdot\sgn\LScalar{\lfrac{d}{dt}\big|_{t=0}
\cD_{A^\flat-\eta_t}^t\psi_0}{\psi_0}.
\]
This proves the assertion since the last term in the above
equation equals the spectral flow of
$\{\cD_{A^\flat-d^{-1}\eta_t}\}$ in $t=0$.
\end{proof}

\begin{theorem}\label{b=0}
Let $M$ be a rational homology 3-sphere endowed with a \spinc
structure $\gs$ and a flat connection $A^\flat$ on $L(\gs)$.
Suppose that $g_{-1}$ and $g_1$ be Riemannian metrics on $M$
together with respective suitable perturbations $\eta_{-1}$ and
$\eta_1$. Then
\begin{equation}\label{wallcross:b=0}
\sw_{\eta_{-1}}(\gs;g_{-1})-\sw_{\eta_1}(\gs;g_1)=
-\SF\big(\cD_{A^\flat-d^{-1}\eta_t}^t;{\scriptstyle t\in
[-1,1]}\big),
\end{equation}
where $g_t$ and $\eta_t$ are arbitrary $C^1$-path of metrics and
perturbations  connecting $g_{-1}$ with $g_1$ and $\eta_{-1}$
with $\eta_1$ respectively.
\end{theorem}
\begin{proof}
We will proceed exactly as in the proof of the wall-crossing
formula in the preceding section. Therefore, we immediately
restrict our attention to the arcs of the parametrized moduli
space which meet the reducible branch. Let
\[
c(s)=[\psi(s),A(s),t(s)]:[a,b]\to \cB\times [-1,1]
\]
be an arc meeting a slightly degenerate reducible in
$c(b)=[0,A(b),t(b)]$. To begin with, we make the assumption that
$c(a)$ does not lie on the reducible branch (cf. Fig.
\ref{b0:case1}).

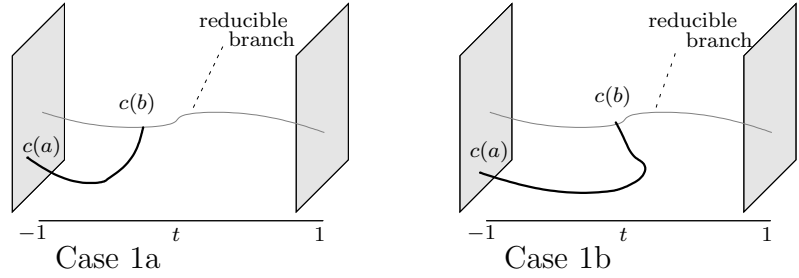
\begin{figure}
\begin{center}
\input{b0_case1.eepic}
\caption{An irreducible branch meeting the reducible one
}\label{b0:case1}
\end{center}
\end{figure}

\begin{Cases}
\item $c(a)\in\cM_{\eta_{-1}}(\gs;g_{-1})$: Let us primarily
assume that $t(s)<t(b)$ for $s$ close to $b$. In the notation of
Proposition \ref{loc:red:b=0}, we have $t=h(z)$, and this yields
that $h''(0)<0$. The familiar considerations of Theorem \ref{b>1}
imply that for an appropriate choice of $s_0$ close to $b$, we
have
\[
\eps\big(\psi(a),A(a)\big)=\eps\big(\psi(s_0),A(s_0)\big)
\]
On the other hand, the fact that $h''(0)<0$ together with
Proposition \ref{orient:b=0} shows that
\[
\eps\big(\psi(s_0),A(s_0)\big)=
-\SF(\cD_{A^\flat-d^{-1}\eta_t};\;{\scriptstyle |t-t(b)|\ll 1}),
\]
where we might possibly have to adjust $s_0$ a little. Hence,
\[
\eps\big(\psi(a),A(a)\big)=
-\SF(\cD_{A^\flat-d^{-1}\eta_t};\;{\scriptstyle |t-t(b)|\ll 1}).
\]

If, on the other hand, $t(s)>t(b)$ for $s$ close to $b$, then
\[
\eps\big(\psi(a),A(a)\big)= -\eps\big(\psi(s_0),A(s_0)\big)
\]
for appropriately chosen $s_0$. Another application of
Proposition \ref{orient:b=0} implies that in this situation
\[
\eps\big(\psi(s_0),A(s_0)\big)=
\SF(\cD_{A^\flat-d^{-1}\eta_t};\;{\scriptstyle |t-t(b)|\ll 1}).
\]
We thus obtain the same formula for $\eps\big(\psi(a),A(a)\big)$
again.

\item $c(a)\in\cM_{\eta_1}(\gs;g_1)$:
With exactly the same arguments, one straightforwardly deduces
that
\[
\eps\big(\psi(a),A(a)\big)=
\SF(\cD_{A^\flat-d^{-1}\eta_t};\;{\scriptstyle |t-t(b)|\ll 1})
\]
irrespective of how $c$ meets the reducible branch.

We are therefore left to consider paths connecting two points on
the reducible branch.

\item\textit{$t(s)<t(a)$ as $s\to a$ and $t(s)<t(b)$ as $s\to
b$:} Under this assumption, we find that for $s_a$ close to $a$
and $s_b$ close to $b$,
\[
\eps\big(\psi(s_a),A(s_a)\big)=-\eps\big(\psi(s_b),A(s_b)\big).
\]
On the other hand, the analysis of Proposition \ref{orient:b=0}
shows that for $i=a,b$,
\[
\eps\big(\psi(s_i),A(s_i)\big)=
\SF(\cD_{A^\flat-d^{-1}\eta_t};\;{\scriptstyle |t-t(i)|\ll 1}).
\]
Therefore, when studying the spectral flow of the entire path
$\cD_{A^\flat-d^{-1}\eta_t}$, the contributions at $t_a$ and at
$t_b$ cancel each other out.

\item\textit{$t(s)<(t(a))$ if $s\to a$ and $t(s)>t(b)$ if $s\to
b$:} In this case,
\[
\eps\big(\psi(s_a),A(s_a)\big)=\eps\big(\psi(s_b),A(s_b)\big).
\]
However, we also have
\[
\eps\big(\psi(s_a),A(s_a)\big)=
\SF(\cD_{A^\flat-d^{-1}\eta_t};\;{\scriptstyle |t-t(a)|\ll 1})
\]
and
\[
\eps\big(\psi(s_b),A(s_b)\big)= -
\SF(\cD_{A^\flat-d^{-1}\eta_t};\;{\scriptstyle |t-t(b)|\ll 1}).
\]
Therefore, the contributions at $t_a$ and at $t_b$ to the
spectral flow cancel each other in this case as well.

The remaining cases are treated accordingly.
\end{Cases}

Proceeding with those paths which do not meet the reducible branch
exactly as in the proof of the wall-crossing formula, the above
considerations establish the claimed formula.
\end{proof}

\noindent\textbf{Producing an invariant for rational homology
spheres.} The above result shows that $\sw_\eta(\gs;g)$ depends
profoundly on the metric $g$ and the perturbation parameter
$\eta$. Thus we do not obtain an invariant for rational homology
spheres. Due to the resemblance with the gauge theoretical
construction of the Casson invariant as performed by Taubes
\cite{Tau:CI}, it was soon conjectured by Kronheimer that adding
an appropriate counter term to $\sw_\eta(\gs;g)$, one should
obtain an invariant which on a homology sphere again equals the
Casson invariant. This is indeed true as Lim \cite{Lim:CI} and
Chen \cite{Che:CI} independently proved. In the remaining part of
this thesis, we will only give a formula for the counter term,
briefly motivating why the resulting number is an invariant of
the underlying manifold.

Let $P$ be a first-order self-adjoint elliptic operator on a
closed, oriented manifold $M$ of odd dimension. Atiyah et al.
\cite{AtiPatSin:SAR} proved that the function
\[
\eta_P(s):=\sum_{z\in \spec(P)\setminus \{0\}}\sgn(z)|z|^{-s}
\]
converges for $\Re s\gg 0$ and extends to a meromorphic function
on the whole $s$-plane, with a finite value at $s=0$. The {\em
$\eta$-invariant} of $P$ is then defined as\index{>@$\eta_P$,
$\eta$-invariant}\index{Dirac operator!$\eta$-invariant}
\[
\eta_P:=\eta_P(0).
\]

Let $\gs$ be a \spinc structure on a rational homology 3-sphere
$M$, and let $g$ be a Riemannian metric with corresponding
suitable perturbation $\nu\in Z^2_k(M;i\R)$. Here, we are using
the notation $\nu$ instead of $\eta$ to avoid a confusion with the
$\eta$-invariant. We then define $\eta_{\dir}(g,\nu)$ as the
$\eta$-invariant of the \spinc Dirac operator
$\cD_{A^\flat-d^{-1}\nu}^g$, where $A^\flat$ is a fixed flat
connection on $L(\gs)$.\index{>@$\eta_{\dir}(g,\nu)$}

If $g_t$ and $\nu_t$ are paths as in Theorem \ref{b=0}, then it
follows from the work of Atiyah, Patodi and Singer that
\begin{align*}
\frac{1}{2}\big(\eta_{\dir}(g_1,\nu_1)&
-\eta_{\dir}(g_{-1},\nu_{-1})\big) \\ &=
\SF(\cD_{A^\flat-d^{-1}\nu_t};{\scriptstyle t\in [-1,1]}) +
\frac{1}{8}\int_{[-1,1]\times
M}\big(-\lfrac{1}{3}p_1(\Hat\nabla)+c_1(\Hat A)^2\big),
\end{align*}
where $p_1$ is the first Pontryagin class, and $\Hat
A=A^\flat-d^{-1}\nu$ and $\Hat \nabla$ respectively denote the
pullbacks of the connection on $L(\gs)$ and the Levi-Civita
covariant derivative to the corresponding bundles over
$[-1,1]\times M$. A survey of the (nontrivial) results leading to
this formula can be found in Nicolaescu's book \cite{Nic:SW},
Sec.~4.1.3. Note that we have already invoked that the operator
$\cD_{A^\flat-d^{-1}\nu_t}$ is invertible for $t=-1,1$ so that
the dimensions of the kernels do not occur on the above formula's
left hand side.

On the other hand, we have the odd signature operator on $(M,g)$,
i.e.,\index{>@$\eta_{\sign}(g)$}
\[
\cD_{\sign}^g:= \begin{pmatrix} * d &-d\\ -d^*
&0\end{pmatrix}:\gO^1\oplus\gO^0\to \gO^1\oplus\gO^0.
\]
This is also a first-order self-adjoint elliptic differential
operator, and we denote the associated $\eta$-invariant by
$\eta_{\sign}(g)$. If $g_t$ is a path of metrics, then the
corresponding formula is (see Atiyah et al. \cite{AtiPatSin:SAR}):
\[
\eta_{\sign}(g_1)-\eta_{\sign}(g_{-1})=\frac{1}{3}
\int_{[-1,1]\times M}p_1(\Hat\nabla).
\]
Note that there is no spectral flow term since the kernels of
$\cD_{\sign}^{g_t}$ have constant dimensions.

Considering paths $g_t$ and $\nu_t$ as in Theorem \ref{b=0}, we
deduce that the term $4\eta_{\dir}(g,\nu)+\eta_{\sign}(g)$
behaves in the following way:
\begin{equation}\label{kreck}
\begin{split}
4\eta_{\dir}(g_1,\nu_1)&+\eta_{\sign}(g_1)-
4\eta_{\dir}(g_{-1},\nu_{-1})-\eta_{\sign}(g_{-1})\\
&=8\SF(\cD_{A^\flat-d^{-1}\nu_t};{\scriptstyle t\in [-1,1]}) +
\int_{[-1,1]\times M}c_1(\Hat A)^2 \\
\end{split}
\end{equation}

The second summand in the last line can be split up into a
difference of two terms which only depend on $\nu_{-1}$ and
$\nu_1$ respectively. For notational convenience, we let $a_t:=
-d^{-1}\nu_t$. Then
\[
F_{\Hat A}=F_{A^\flat} +da_t + dt\wedge\lfrac{\partial}{\partial
t}a_t = da_t + dt\wedge\lfrac{\partial}{\partial t}a_t.
\]
Hence, we compute
\begin{align*}
F_{\Hat A}\wedge F_{\Hat A}&= dt\wedge da_t\wedge
\lfrac{\partial}{\partial t}a_t +
dt\wedge\lfrac{\partial}{\partial t}a_t \wedge da_t =
dt\wedge\big(da_t\wedge \lfrac{\partial}{\partial
t}a_t + \lfrac{\partial}{\partial t}a_t \wedge da_t  \big)\\
&= dt\wedge\big(d (a_t\wedge \lfrac{\partial}{\partial t}a_t) +
a_t\wedge\lfrac{\partial}{\partial t}da_t +
\lfrac{\partial}{\partial
t}a_t \wedge da_t  \big)\\
&= dt\wedge\big(d (a_t\wedge \lfrac{\partial}{\partial t}a_t) +
\lfrac{\partial}{\partial t}(a_t\wedge da_t)\big).
\end{align*}
Using Stoke's Theorem and performing the integration over
$[-1,1]$, one finds
\begin{align*}
\int_{[-1,1]\times M}c_1(\Hat A)^2 &=
-\frac{1}{4\pi^2}\int_{[-1,1]\times M} dt\wedge
\lfrac{\partial}{\partial t}(a_t\wedge da_t) \\
&=-\frac{1}{4\pi^2}\int_M\big(a_1\wedge da_1 - a_{-1}\wedge
da_{-1}\big)\\
&= -\frac{1}{4\pi^2}\int_M\big(d^{-1}\nu_1\wedge \nu_1 -
d^{-1}\nu_{-1}\wedge \nu_{-1}\big).
\end{align*}
A combination of this computation with formula
\eqref{wallcross:b=0} of Theorem \ref{b=0} and formula
\eqref{kreck} easily establishes the following result:

\begin{theorem}\index{Seiberg-Witten invariant!modified}
Let $M$ be a rational homology 3-sphere with \spinc structure
$\gs$. If $\nu$ is a suitable perturbation with respect to a
Riemannian metric $g$, then the ``modified Seiberg-Witten
invariant"
\[
\gl_\nu(\gs;g):= \sw_\nu(\gs;g)-\lfrac12 \eta_{\dir}(g,\nu) -
\lfrac18 \eta_{\sign}(g) - \lfrac{1}{32\pi^2}\int_M
d^{-1}\nu\wedge \nu
\]
is independent of $\nu$ and $g$. Therefore, it gives rise to a
smooth invariant of $M$.
\end{theorem}

\begin{remark*}\quad
\begin{enumerate}
\item Note that $\gl_\nu(\gs;g)$ is generally not integer valued.
In fact, it is $\Z$-valued provided that $H_1(M;\Z)=0$.
Otherwise, $8h\cdot \gl_\nu(\gs;g)\in \Z$ where $h:=|H_1(M;\Z)|$
denotes the order of the first homology group (cf. Lim
\cite{Lim:SW}, Prop.~17).
\item If $M$ happens to be an integer homology sphere, then the
cohomology group $H^2(M;\Z)$ is trivial. Hence, there exists only
one \spinc structure on $M$. In \cite{Lim:CI}, Y. Lim establishes
that the unique number obtained in this way equals the {\em
Casson invariant} of an integer homology sphere.
\item For rational homology spheres, Marcolli \& Wang
\cite{MarWan:SW} investigate an averaged version of the modified
Seiberg-Witten invariants---obtained by summing over all \spinc
structures. They prove that it equals the so-called {\em
Casson-Walker invariant}.
\item On the other hand, Nicolaescu \cite{Nic:Rat} shows that a
combination of the Casson-Walker invariant and certain refined
torsion invariants (due to V.G. Turaev) determine all modified
Seiberg-Witten invariants of a rational homology sphere.
\end{enumerate}
\end{remark*}

\cleardoublepage

%% file: param_b.eepic
\setlength{\unitlength}{0.0005in}
\begingroup\makeatletter\ifx\SetFigFont\undefined%
\gdef\SetFigFont#1#2#3#4#5{%
  \reset@font\fontsize{#1}{#2pt}%
  \fontfamily{#3}\fontseries{#4}\fontshape{#5}%
  \selectfont}%
\fi\endgroup%
{\renewcommand{\dashlinestretch}{30}
\begin{picture}(5937,4001)(0,-10)
\texture{0 0 0 888888 88000000 0 0 80808
    8000000 0 0 888888 88000000 0 0 80808
    8000000 0 0 888888 88000000 0 0 80808
    8000000 0 0 888888 88000000 0 0 80808 }
\color{gray}
\shade\path(12,450)(12,3093)(893,3974)
    (893,1331)(12,450)
\shade\path(4812,450)(4812,3093)(5693,3974)
    (5693,1331)(4812,450)
\path(12,450)(12,3093)(893,3974)
    (893,1331)(12,450)
\path(4812,450)(4812,3093)(5693,3974)
    (5693,1331)(4812,450)
\path(462,360)(5262,360)
\color{black}
\thicklines
\path(462,2625)(463,2626)(466,2627)
    (472,2630)(480,2635)(493,2641)
    (509,2649)(529,2659)(553,2671)
    (580,2685)(612,2700)(646,2716)
    (683,2733)(722,2751)(764,2769)
    (807,2787)(851,2804)(895,2821)
    (941,2837)(987,2851)(1034,2865)
    (1082,2876)(1130,2886)(1178,2894)
    (1228,2900)(1278,2903)(1329,2903)
    (1381,2900)(1433,2894)(1485,2883)
    (1537,2869)(1587,2850)(1635,2827)
    (1679,2801)(1720,2773)(1756,2744)
    (1789,2715)(1817,2687)(1842,2659)
    (1863,2633)(1881,2608)(1896,2585)
    (1909,2562)(1920,2541)(1929,2521)
    (1937,2501)(1943,2482)(1950,2462)
    (1955,2443)(1960,2423)(1966,2402)
    (1971,2380)(1977,2357)(1983,2331)
    (1989,2304)(1996,2274)(2003,2241)
    (2011,2206)(2018,2168)(2025,2127)
    (2031,2085)(2035,2040)(2038,1995)
    (2037,1950)(2033,1904)(2025,1860)
    (2015,1819)(2003,1782)(1991,1749)
    (1979,1720)(1967,1695)(1955,1673)
    (1944,1654)(1934,1637)(1924,1623)
    (1915,1611)(1905,1600)(1896,1590)
    (1887,1581)(1877,1573)(1867,1564)
    (1855,1556)(1842,1547)(1827,1537)
    (1810,1527)(1791,1516)(1768,1504)
    (1742,1492)(1713,1479)(1680,1466)
    (1643,1453)(1602,1442)(1558,1432)
    (1512,1425)(1464,1422)(1416,1422)
    (1369,1426)(1325,1432)(1283,1440)
    (1244,1450)(1209,1461)(1177,1472)
    (1148,1485)(1122,1498)(1098,1512)
    (1076,1526)(1056,1540)(1037,1554)
    (1018,1569)(1000,1583)(981,1598)
    (962,1613)(941,1629)(919,1644)
    (896,1660)(870,1676)(842,1692)
    (813,1708)(781,1725)(747,1741)
    (713,1757)(678,1772)(644,1787)
    (612,1800)(591,1809)(572,1817)
    (555,1824)(541,1830)(529,1835)
    (519,1839)(511,1842)(505,1845)
    (500,1847)(497,1848)(496,1848)
    (497,1848)(499,1847)(502,1846)
    (506,1844)(510,1842)(515,1840)
    (521,1838)(528,1835)(534,1832)
    (541,1829)(548,1827)(555,1824)
    (562,1821)(569,1818)(576,1815)
    (582,1812)(588,1810)(593,1808)
    (597,1806)(601,1804)(605,1803)
    (607,1802)(609,1801)(610,1801)
    (611,1800)(612,1800)
\path(612,3225)(613,3226)(615,3227)
    (619,3229)(625,3233)(634,3238)
    (645,3244)(660,3253)(679,3264)
    (701,3276)(726,3290)(755,3306)
    (787,3324)(822,3343)(860,3364)
    (901,3386)(944,3408)(989,3432)
    (1036,3456)(1085,3480)(1135,3504)
    (1186,3528)(1237,3552)(1290,3575)
    (1343,3597)(1397,3619)(1451,3640)
    (1506,3660)(1561,3679)(1618,3696)
    (1674,3712)(1732,3727)(1790,3740)
    (1849,3751)(1908,3760)(1969,3767)
    (2030,3771)(2092,3773)(2154,3772)
    (2216,3768)(2277,3761)(2337,3750)
    (2408,3732)(2473,3710)(2533,3684)
    (2584,3656)(2628,3626)(2664,3595)
    (2693,3564)(2714,3532)(2728,3501)
    (2735,3469)(2738,3439)(2736,3409)
    (2729,3379)(2720,3349)(2708,3320)
    (2695,3291)(2681,3262)(2667,3233)
    (2654,3204)(2643,3175)(2634,3144)
    (2629,3113)(2628,3081)(2632,3049)
    (2642,3015)(2659,2980)(2683,2945)
    (2715,2909)(2756,2872)(2806,2835)
    (2864,2799)(2931,2764)(3006,2730)
    (3087,2700)(3153,2679)(3221,2660)
    (3289,2644)(3356,2630)(3423,2619)
    (3488,2610)(3550,2603)(3611,2597)
    (3668,2594)(3723,2592)(3775,2592)
    (3824,2593)(3870,2595)(3915,2599)
    (3956,2603)(3996,2608)(4034,2614)
    (4070,2620)(4105,2627)(4139,2635)
    (4173,2642)(4206,2650)(4239,2658)
    (4272,2666)(4305,2674)(4339,2682)
    (4374,2690)(4411,2698)(4448,2706)
    (4488,2713)(4529,2720)(4572,2727)
    (4617,2734)(4664,2740)(4713,2745)
    (4763,2750)(4816,2755)(4870,2759)
    (4924,2763)(4979,2766)(5034,2769)
    (5088,2771)(5139,2773)(5187,2775)
    (5229,2777)(5267,2778)(5300,2779)
    (5329,2780)(5354,2781)(5375,2782)
    (5393,2783)(5407,2783)(5417,2784)
    (5425,2784)(5430,2784)(5432,2785)
    (5429,2785)(5425,2785)(5418,2784)
    (5410,2784)(5401,2784)(5390,2784)
    (5377,2783)(5364,2783)(5350,2782)
    (5335,2782)(5319,2782)(5303,2781)
    (5286,2781)(5270,2780)(5254,2780)
    (5238,2779)(5222,2779)(5207,2778)
    (5193,2778)(5180,2777)(5167,2777)
    (5156,2776)(5147,2776)(5138,2776)
    (5131,2776)(5125,2775)(5121,2775)
    (5117,2775)(5115,2775)(5113,2775)(5112,2775)
\path(462,1200)(463,1200)(466,1200)
    (470,1200)(477,1200)(488,1201)
    (502,1201)(519,1201)(541,1202)
    (567,1203)(597,1203)(630,1204)
    (668,1205)(709,1205)(754,1206)
    (801,1207)(852,1208)(904,1208)
    (958,1209)(1014,1209)(1071,1209)
    (1130,1209)(1189,1209)(1248,1209)
    (1308,1208)(1369,1208)(1429,1206)
    (1490,1205)(1551,1203)(1612,1201)
    (1673,1198)(1734,1194)(1796,1191)
    (1858,1186)(1919,1181)(1981,1175)
    (2043,1169)(2104,1162)(2165,1154)
    (2224,1145)(2282,1135)(2337,1125)
    (2432,1103)(2511,1080)(2571,1057)
    (2612,1033)(2636,1010)(2645,987)
    (2640,965)(2625,943)(2602,922)
    (2574,901)(2543,881)(2512,862)
    (2484,842)(2460,824)(2444,806)
    (2437,789)(2443,773)(2464,760)
    (2501,750)(2555,744)(2626,744)
    (2712,750)(2764,757)(2818,766)
    (2872,776)(2926,788)(2979,801)
    (3030,815)(3079,830)(3126,845)
    (3170,860)(3211,875)(3249,890)
    (3284,906)(3318,921)(3349,936)
    (3378,951)(3405,966)(3432,981)
    (3457,995)(3481,1010)(3506,1025)
    (3530,1040)(3555,1055)(3580,1070)
    (3607,1085)(3635,1101)(3666,1116)
    (3698,1132)(3732,1149)(3770,1165)
    (3810,1182)(3854,1200)(3900,1217)
    (3950,1235)(4003,1253)(4059,1270)
    (4117,1288)(4178,1305)(4239,1321)
    (4301,1336)(4362,1350)(4440,1366)
    (4512,1379)(4579,1389)(4639,1396)
    (4694,1401)(4742,1403)(4786,1404)
    (4825,1402)(4861,1399)(4892,1395)
    (4920,1389)(4946,1382)(4969,1374)
    (4990,1366)(5009,1357)(5026,1348)
    (5041,1338)(5055,1329)(5067,1320)
    (5077,1311)(5086,1303)(5093,1296)
    (5099,1290)(5104,1285)(5107,1281)
    (5109,1278)(5111,1276)(5112,1275)
\path(2712,1725)(2711,1703)(2715,1683)
    (2721,1665)(2729,1650)(2738,1637)
    (2748,1626)(2758,1617)(2768,1610)
    (2778,1603)(2788,1598)(2798,1594)
    (2808,1591)(2818,1587)(2829,1585)
    (2840,1582)(2853,1580)(2866,1577)
    (2881,1575)(2898,1573)(2918,1571)
    (2939,1569)(2964,1568)(2991,1567)
    (3021,1568)(3053,1570)(3087,1575)
    (3125,1583)(3162,1592)(3198,1603)
    (3231,1614)(3262,1625)(3290,1636)
    (3316,1646)(3340,1655)(3361,1663)
    (3381,1672)(3400,1680)(3418,1688)
    (3436,1696)(3454,1704)(3471,1714)
    (3489,1724)(3508,1737)(3526,1750)
    (3545,1766)(3563,1784)(3580,1804)
    (3595,1826)(3606,1850)(3612,1875)
    (3612,1898)(3608,1920)(3600,1942)
    (3590,1962)(3578,1980)(3565,1997)
    (3551,2013)(3537,2028)(3524,2041)
    (3510,2054)(3496,2066)(3482,2077)
    (3468,2088)(3454,2098)(3439,2108)
    (3422,2118)(3405,2128)(3386,2138)
    (3366,2147)(3343,2156)(3318,2164)
    (3291,2170)(3262,2176)(3230,2179)
    (3196,2179)(3162,2175)(3128,2167)
    (3094,2156)(3063,2143)(3034,2128)
    (3007,2112)(2983,2096)(2962,2079)
    (2942,2064)(2925,2048)(2909,2033)
    (2895,2017)(2881,2002)(2868,1987)
    (2856,1972)(2843,1957)(2831,1941)
    (2818,1924)(2805,1906)(2791,1887)
    (2778,1866)(2764,1845)(2750,1822)
    (2737,1798)(2726,1773)(2717,1749)(2712,1725)
\put(5112,0){\makebox(0,0)[lb]{$1$}}
\put(280,0){\makebox(0,0)[lb]{$-1$}}
\put(2690,0){\makebox(0,0)[lb]{$t$}}
\end{picture}
}

%% file: orient.eepic
\setlength{\unitlength}{0.0005in}
\begingroup\makeatletter\ifx\SetFigFont\undefined%
\gdef\SetFigFont#1#2#3#4#5{%
  \reset@font\fontsize{#1}{#2pt}%
  \fontfamily{#3}\fontseries{#4}\fontshape{#5}%
  \selectfont}%
\fi\endgroup%
{\renewcommand{\dashlinestretch}{30}
\begin{picture}(6012,4046)(0,-10)
\texture{0 0 0 888888 88000000 0 0 80808
    8000000 0 0 888888 88000000 0 0 80808
    8000000 0 0 888888 88000000 0 0 80808
    8000000 0 0 888888 88000000 0 0 80808 }
\color{gray}
\shade\path(87,495)(87,3138)(968,4019)
    (968,1376)(87,495)
\path(87,495)(87,3138)(968,4019)
    (968,1376)(87,495)
\shade\path(4887,495)(4887,3138)(5768,4019)
    (5768,1376)(4887,495)
\path(4887,495)(4887,3138)(5768,4019)
    (5768,1376)(4887,495)
\path(537,345)(5337,345)
\path(537,345)(5337,345)
\color{black}
\path(762,3270)(1737,3420)
\whiten\path(1622.957,3372.102)(1737.000,3420.000)(1613.834,3431.404)(1622.957,3372.102)
\path(600,1170)(1662,1245)
\whiten\path(1544.654,1205.885)(1662.000,1245.000)(1540.053,1265.708)(1544.654,1205.885)
\path(5000,2820)(5637,3120)
\whiten\path(5547.695,3034.416)(5637.000,3120.000)(5517.927,3086.511)(5547.695,3034.416)
\path(730,2730)(12,2670)
\whiten\path(127.953,2713.068)(12.000,2670.000)(134.579,2653.435)(127.953,2713.068)
\color{gray2}
\thicklines
\path(762,3270)(764,3270)(767,3271)
    (774,3272)(784,3273)(798,3275)
    (816,3277)(837,3280)(861,3282)
    (888,3285)(917,3288)(947,3291)
    (979,3294)(1011,3296)(1044,3298)
    (1078,3299)(1111,3299)(1146,3299)
    (1181,3297)(1217,3295)(1253,3291)
    (1290,3285)(1327,3279)(1362,3270)
    (1410,3255)(1447,3238)(1471,3223)
    (1484,3209)(1486,3197)(1480,3187)
    (1469,3178)(1456,3170)(1443,3161)
    (1432,3152)(1428,3141)(1433,3127)
    (1450,3111)(1481,3091)(1527,3069)
    (1587,3045)(1626,3032)(1665,3019)
    (1704,3008)(1740,2997)(1773,2988)
    (1802,2979)(1828,2972)(1850,2965)
    (1869,2960)(1885,2955)(1898,2950)
    (1910,2946)(1921,2942)(1931,2939)
    (1941,2935)(1953,2932)(1966,2929)
    (1982,2925)(2002,2922)(2025,2918)
    (2054,2915)(2087,2911)(2127,2907)
    (2173,2904)(2225,2900)(2283,2898)
    (2346,2896)(2412,2895)(2471,2896)
    (2528,2897)(2583,2900)(2634,2903)
    (2680,2907)(2722,2912)(2758,2918)
    (2789,2923)(2815,2929)(2838,2936)
    (2856,2942)(2871,2949)(2884,2956)
    (2894,2962)(2903,2969)(2912,2976)
    (2921,2983)(2930,2990)(2941,2997)
    (2954,3003)(2971,3010)(2990,3016)
    (3014,3022)(3042,3027)(3076,3032)
    (3115,3037)(3160,3041)(3211,3044)
    (3268,3046)(3329,3047)(3395,3047)
    (3462,3045)(3530,3042)(3596,3038)
    (3658,3032)(3717,3026)(3770,3020)
    (3817,3012)(3859,3005)(3895,2997)
    (3926,2989)(3953,2981)(3976,2973)
    (3995,2965)(4011,2957)(4025,2949)
    (4038,2941)(4050,2932)(4061,2924)
    (4074,2916)(4087,2908)(4103,2900)
    (4121,2892)(4142,2884)(4167,2876)
    (4197,2868)(4230,2860)(4269,2853)
    (4313,2845)(4361,2839)(4414,2833)
    (4470,2827)(4528,2823)(4587,2820)
    (4658,2818)(4724,2818)(4786,2820)
    (4843,2824)(4895,2829)(4942,2835)
    (4986,2842)(5025,2850)(5062,2858)
    (5096,2867)(5128,2877)(5157,2887)
    (5185,2897)(5210,2908)(5234,2918)
    (5255,2927)(5274,2937)(5290,2945)
    (5304,2952)(5315,2958)(5324,2963)
    (5330,2966)(5334,2968)(5336,2969)(5337,2970)
\path(387,2670)(388,2670)(391,2670)
    (395,2671)(403,2672)(414,2673)
    (428,2674)(446,2676)(468,2678)
    (494,2681)(524,2683)(557,2686)
    (594,2689)(634,2692)(677,2696)
    (722,2699)(769,2702)(818,2705)
    (867,2708)(918,2710)(970,2712)
    (1023,2713)(1075,2714)(1128,2714)
    (1182,2713)(1236,2712)(1289,2710)
    (1344,2706)(1398,2702)(1453,2696)
    (1508,2690)(1563,2681)(1619,2671)
    (1674,2660)(1729,2646)(1783,2631)
    (1836,2614)(1887,2595)(1943,2571)
    (1994,2545)(2040,2519)(2081,2494)
    (2115,2469)(2145,2446)(2170,2425)
    (2191,2405)(2208,2387)(2221,2371)
    (2232,2356)(2240,2342)(2246,2329)
    (2250,2318)(2253,2306)(2256,2295)
    (2257,2284)(2259,2272)(2260,2259)
    (2261,2246)(2262,2231)(2264,2214)
    (2266,2195)(2268,2174)(2270,2151)
    (2272,2124)(2274,2096)(2275,2064)
    (2275,2030)(2274,1995)(2269,1958)
    (2262,1920)(2250,1881)(2235,1843)
    (2217,1808)(2199,1776)(2181,1748)
    (2163,1723)(2147,1702)(2131,1684)
    (2117,1669)(2104,1657)(2092,1647)
    (2081,1638)(2070,1631)(2060,1625)
    (2049,1620)(2039,1615)(2027,1609)
    (2014,1603)(1999,1596)(1981,1587)
    (1961,1577)(1937,1565)(1909,1551)
    (1877,1534)(1841,1515)(1799,1494)
    (1753,1471)(1701,1446)(1646,1420)
    (1587,1395)(1530,1372)(1473,1351)
    (1416,1331)(1360,1313)(1305,1297)
    (1252,1283)(1200,1270)(1150,1258)
    (1101,1247)(1053,1238)(1006,1229)
    (960,1222)(915,1215)(871,1209)
    (827,1203)(785,1198)(745,1194)
    (706,1190)(668,1186)(634,1183)
    (601,1180)(572,1178)(546,1176)
    (524,1174)(505,1173)(490,1172)
    (479,1171)(471,1171)(466,1170)
    (463,1170)(462,1170)
\color{black}
\put(600,3370){\makebox(0,0)[lb]{$\scriptstyle +$}}
\put(5037,2600){\makebox(0,0)[lb]{$\scriptstyle +$}}
\put(400,1370){\makebox(0,0)[lb]{$\scriptstyle -$}}
\put(300,2300){\makebox(0,0)[lb]{$\scriptstyle +$}}
\put(2200,1500){\makebox(0,0)[lb]{$\eps=-1$}}
\put(2750,3200){\makebox(0,0)[lb]{$\eps=1$}}

\put(5187,45){\makebox(0,0)[lb]{$1$}}
\put(387,45){\makebox(0,0)[lb]{$-1$}}
\put(2750,45){\makebox(0,0)[lb]{$t$}}
\put(1137,3495){\makebox(0,0)[lb]{$\scriptstyle c'(a)$}}
\put(5187,3195){\makebox(0,0)[lb]{$\scriptstyle c'(b)$}}
\put(250,2820){\makebox(0,0)[lb]{$\scriptstyle c'(b)$}}
\put(912,870){\makebox(0,0)[lb]{$\scriptstyle c'(a)$}}
\end{picture}
}

%% file: param_b1.eepic
\setlength{\unitlength}{0.0005in}
\begingroup\makeatletter\ifx\SetFigFont\undefined%
\gdef\SetFigFont#1#2#3#4#5{%
  \reset@font\fontsize{#1}{#2pt}%
  \fontfamily{#3}\fontseries{#4}\fontshape{#5}%
  \selectfont}%
\fi\endgroup%
{\renewcommand{\dashlinestretch}{30}
\begin{picture}(5937,4050)(0,-10)
\thinlines
\texture{0 0 0 888888 88000000 0 0 80808
    8000000 0 0 888888 88000000 0 0 80808
    8000000 0 0 888888 88000000 0 0 80808
    8000000 0 0 888888 88000000 0 0 80808 }
\color{gray}
\path(462,360)(5262,360)
\shade\path(2412,499)(2412,3142)(3293,4023)
    (3293,1380)(2412,499)
\path(2412,499)(2412,3142)(3293,4023)
    (3293,1380)(2412,499)
\shade\path(4812,499)(4812,3142)(5693,4023)
    (5693,1380)(4812,499)
\path(4812,499)(4812,3142)(5693,4023)
    (5693,1380)(4812,499)
\shade\path(12,499)(12,3142)(893,4023)
    (893,1380)(12,499)
\path(12,499)(12,3142)(893,4023)
    (893,1380)(12,499)
\color{black}
\thicklines
\path(387,1399)(388,1398)(391,1397)
    (396,1394)(403,1390)(414,1383)
    (428,1375)(445,1365)(466,1353)
    (491,1340)(518,1324)(549,1307)
    (582,1289)(617,1270)(653,1250)
    (692,1230)(731,1210)(772,1189)
    (813,1169)(855,1148)(898,1128)
    (942,1109)(987,1089)(1033,1070)
    (1080,1052)(1128,1034)(1177,1017)
    (1228,1001)(1279,986)(1332,972)
    (1385,959)(1437,949)(1504,939)
    (1564,932)(1617,929)(1661,930)
    (1695,933)(1722,939)(1740,946)
    (1753,955)(1760,965)(1763,976)
    (1763,987)(1762,999)(1761,1011)
    (1761,1023)(1763,1034)(1768,1046)
    (1779,1056)(1795,1066)(1818,1076)
    (1848,1083)(1886,1090)(1931,1095)
    (1982,1098)(2037,1099)(2089,1098)
    (2139,1096)(2185,1093)(2226,1089)
    (2262,1085)(2292,1080)(2316,1076)
    (2335,1071)(2349,1067)(2360,1062)
    (2369,1058)(2375,1053)(2381,1049)
    (2387,1044)(2394,1040)(2403,1035)
    (2415,1029)(2432,1023)(2453,1016)
    (2481,1009)(2515,1001)(2556,992)
    (2605,982)(2661,972)(2722,960)
    (2787,949)(2845,939)(2901,929)
    (2955,920)(3005,911)(3051,903)
    (3092,895)(3129,888)(3160,882)
    (3188,876)(3211,871)(3232,866)
    (3250,861)(3265,857)(3280,853)
    (3293,849)(3307,845)(3321,841)
    (3336,838)(3353,834)(3373,830)
    (3395,826)(3421,822)(3451,818)
    (3485,814)(3524,810)(3566,807)
    (3612,804)(3661,801)(3711,800)
    (3762,799)(3828,800)(3886,803)
    (3936,808)(3975,813)(4005,819)
    (4026,825)(4039,831)(4045,837)
    (4047,843)(4046,849)(4043,855)
    (4040,862)(4039,868)(4040,875)
    (4045,882)(4057,890)(4075,899)
    (4101,908)(4136,918)(4180,928)
    (4231,939)(4287,949)(4336,957)
    (4386,965)(4434,972)(4480,979)
    (4525,984)(4567,989)(4608,994)
    (4648,998)(4686,1001)(4723,1004)
    (4760,1007)(4795,1010)(4828,1012)
    (4861,1014)(4891,1016)(4920,1018)
    (4946,1019)(4969,1021)(4989,1022)
    (5005,1022)(5017,1023)(5027,1024)
    (5032,1024)(5036,1024)(5037,1024)
\path(5262,2149)(5261,2150)(5258,2152)
    (5253,2156)(5246,2162)(5236,2171)
    (5222,2182)(5206,2195)(5187,2210)
    (5166,2227)(5142,2245)(5117,2265)
    (5091,2285)(5063,2305)(5035,2326)
    (5006,2347)(4976,2367)(4945,2387)
    (4913,2407)(4881,2426)(4847,2445)
    (4812,2463)(4775,2480)(4738,2496)
    (4700,2511)(4662,2524)(4615,2537)
    (4573,2545)(4537,2550)(4509,2552)
    (4488,2552)(4472,2550)(4462,2546)
    (4456,2541)(4452,2536)(4449,2530)
    (4447,2524)(4444,2517)(4439,2511)
    (4431,2503)(4419,2496)(4401,2488)
    (4379,2479)(4352,2470)(4321,2460)
    (4287,2449)(4250,2436)(4217,2424)
    (4189,2414)(4166,2406)(4147,2400)
    (4133,2397)(4122,2395)(4113,2393)
    (4106,2393)(4099,2392)(4091,2390)
    (4082,2387)(4071,2381)(4058,2372)
    (4042,2360)(4024,2343)(4005,2323)
    (3987,2299)(3972,2274)(3961,2249)
    (3953,2228)(3949,2210)(3946,2196)
    (3945,2186)(3946,2178)(3947,2173)
    (3950,2168)(3952,2163)(3955,2156)
    (3958,2147)(3961,2135)(3965,2118)
    (3969,2096)(3974,2068)(3980,2036)
    (3987,1999)(3995,1965)(4003,1933)
    (4011,1904)(4017,1879)(4022,1857)
    (4026,1840)(4028,1825)(4029,1813)
    (4030,1802)(4031,1793)(4032,1783)
    (4033,1773)(4036,1762)(4041,1749)
    (4048,1734)(4059,1715)(4073,1695)
    (4090,1672)(4112,1648)(4137,1624)
    (4164,1603)(4190,1586)(4214,1572)
    (4234,1562)(4250,1555)(4262,1551)
    (4271,1548)(4278,1548)(4283,1548)
    (4287,1549)(4292,1550)(4297,1552)
    (4306,1553)(4318,1554)(4335,1554)
    (4358,1553)(4387,1552)(4423,1551)
    (4465,1550)(4512,1549)(4553,1549)
    (4594,1551)(4635,1553)(4674,1556)
    (4711,1559)(4747,1563)(4782,1567)
    (4816,1571)(4848,1576)(4880,1581)
    (4911,1586)(4941,1591)(4969,1596)
    (4996,1601)(5021,1606)(5044,1610)
    (5063,1614)(5079,1617)(5092,1620)
    (5101,1622)(5107,1623)(5111,1624)(5112,1624)
\path(462,2299)(463,2300)(466,2303)
    (472,2308)(480,2316)(491,2327)
    (506,2341)(524,2357)(545,2376)
    (569,2397)(595,2419)(622,2442)
    (652,2465)(683,2488)(714,2510)
    (747,2531)(780,2551)(815,2569)
    (851,2584)(888,2598)(926,2608)
    (966,2615)(1007,2618)(1050,2617)
    (1093,2611)(1137,2599)(1176,2583)
    (1213,2564)(1246,2541)(1275,2516)
    (1300,2489)(1320,2462)(1336,2434)
    (1348,2405)(1357,2377)(1362,2348)
    (1366,2320)(1368,2292)(1368,2264)
    (1368,2236)(1368,2209)(1369,2181)
    (1370,2154)(1374,2127)(1379,2100)
    (1387,2073)(1399,2047)(1414,2022)
    (1433,1998)(1457,1976)(1484,1957)
    (1515,1941)(1550,1930)(1587,1924)
    (1625,1924)(1661,1930)(1696,1941)
    (1728,1955)(1757,1973)(1782,1992)
    (1804,2013)(1822,2036)(1837,2059)
    (1850,2083)(1861,2107)(1871,2131)
    (1879,2155)(1887,2180)(1895,2205)
    (1903,2230)(1913,2256)(1924,2281)
    (1937,2307)(1952,2334)(1970,2360)
    (1992,2386)(2017,2412)(2046,2438)
    (2078,2462)(2113,2485)(2149,2506)
    (2187,2524)(2230,2540)(2271,2551)
    (2308,2558)(2343,2560)(2375,2558)
    (2404,2554)(2431,2546)(2455,2537)
    (2478,2525)(2500,2512)(2520,2497)
    (2539,2482)(2556,2466)(2572,2450)
    (2587,2435)(2599,2421)(2610,2408)
    (2619,2397)(2626,2388)(2631,2382)
    (2634,2377)(2636,2375)(2637,2374)
\path(537,3199)(538,3200)(540,3201)
    (543,3204)(548,3208)(556,3214)
    (566,3222)(579,3232)(594,3244)
    (612,3257)(633,3273)(657,3290)
    (682,3308)(710,3327)(740,3347)
    (771,3368)(804,3389)(839,3410)
    (875,3431)(912,3452)(951,3472)
    (992,3492)(1034,3511)(1078,3530)
    (1124,3548)(1173,3565)(1224,3580)
    (1278,3595)(1335,3609)(1394,3621)
    (1457,3631)(1523,3639)(1592,3645)
    (1662,3649)(1729,3650)(1794,3648)
    (1857,3645)(1916,3640)(1970,3633)
    (2020,3625)(2065,3616)(2105,3606)
    (2140,3595)(2170,3584)(2196,3573)
    (2218,3561)(2237,3549)(2253,3537)
    (2266,3524)(2278,3512)(2289,3499)
    (2300,3486)(2310,3474)(2321,3461)
    (2333,3448)(2346,3435)(2362,3422)
    (2381,3409)(2403,3396)(2429,3383)
    (2459,3370)(2494,3357)(2534,3345)
    (2579,3333)(2629,3321)(2683,3310)
    (2742,3299)(2805,3289)(2870,3281)
    (2937,3274)(3011,3268)(3083,3264)
    (3150,3262)(3211,3262)(3266,3264)
    (3314,3267)(3356,3272)(3391,3277)
    (3420,3284)(3443,3292)(3462,3300)
    (3477,3310)(3489,3319)(3498,3329)
    (3505,3339)(3512,3349)(3519,3359)
    (3526,3369)(3535,3379)(3546,3388)
    (3561,3398)(3579,3406)(3602,3414)
    (3630,3421)(3663,3426)(3703,3431)
    (3750,3434)(3802,3436)(3861,3436)
    (3924,3434)(3992,3430)(4062,3424)
    (4128,3417)(4193,3408)(4256,3397)
    (4316,3386)(4374,3374)(4429,3361)
    (4482,3347)(4532,3333)(4580,3319)
    (4626,3304)(4671,3288)(4713,3273)
    (4754,3257)(4794,3241)(4833,3225)
    (4870,3208)(4906,3192)(4940,3176)
    (4973,3160)(5004,3145)(5033,3130)
    (5060,3117)(5085,3104)(5107,3092)
    (5126,3082)(5142,3073)(5156,3066)
    (5167,3060)(5175,3056)(5181,3053)
    (5184,3051)(5186,3049)(5187,3049)
\thinlines
\color{gray2}
\put(2862,2224){\ellipse{450}{1350}}
\color{black}
\put(312,0){\makebox(0,0)[lb]{$-1$}}
\put(5112,0){\makebox(0,0)[lb]{$1$}}
\put(2700,0){\makebox(0,0)[lb]{$t$}}
\put(3450,3950){\makebox(0,0)[lb]{\scriptsize circle of}}
\put(3900,3700){\makebox(0,0)[lb]{\scriptsize reducibles}}
\dashline{40.000}(2950,2950)(3800,3800)
\put(1000,1500){\makebox(0,0)[lb]{\scriptsize degenerate}}
\put(1250,1250){\makebox(0,0)[lb]{\scriptsize reducible}}
\dashline{50.000}(2600,2250)(2150,1500)
\end{picture}
}

%% file: b1_case1.eepic
\setlength{\unitlength}{0.00035in}
\begingroup\makeatletter\ifx\SetFigFont\undefined%
\gdef\SetFigFont#1#2#3#4#5{%
  \reset@font\fontsize{#1}{#2pt}%
  \fontfamily{#3}\fontseries{#4}\fontshape{#5}%
  \selectfont}%
\fi\endgroup%
{\renewcommand{\dashlinestretch}{30}
\begin{picture}(11915,3974)(0,-10)
\texture{0 0 0 888888 88000000 0 0 80808
    8000000 0 0 888888 88000000 0 0 80808
    8000000 0 0 888888 88000000 0 0 80808
    8000000 0 0 888888 88000000 0 0 80808 }
\color{gray}
\shade\path(12,838)(12,3170)(789,3947)
    (789,1615)(12,838)
\path(12,838)(12,3170)(789,3947)
    (789,1615)(12,838)
\shade\path(4247,838)(4247,3170)(5025,3947)
    (5025,1615)(4247,838)
\path(4247,838)(4247,3170)(5025,3947)
    (5025,1615)(4247,838)
\shade\path(2130,838)(2130,3170)(2907,3947)
    (2907,1615)(2130,838)
\path(2130,838)(2130,3170)(2907,3947)
    (2907,1615)(2130,838)
\path(409,706)(4644,706)
\path(409,706)(4644,706)
\shade\path(6687,838)(6687,3170)(7464,3947)
    (7464,1615)(6687,838)
\path(6687,838)(6687,3170)(7464,3947)
    (7464,1615)(6687,838)
\shade\path(10922,838)(10922,3170)(11700,3947)
    (11700,1615)(10922,838)
\path(10922,838)(10922,3170)(11700,3947)
    (11700,1615)(10922,838)
\path(7084,706)(11319,706)
\path(7084,706)(11319,706)
\shade\path(8805,838)(8805,3170)(9582,3947)
    (9582,1615)(8805,838)
\path(8805,838)(8805,3170)(9582,3947)
    (9582,1615)(8805,838)
\color{black}
\path(1062,2775)(1137,2625)(987,2625)
\path(8337,1350)(8412,1275)(8337,1200)
\thicklines
\path(409,2426)(410,2427)(414,2431)
    (420,2436)(429,2445)(441,2457)
    (457,2471)(476,2489)(498,2508)
    (522,2529)(549,2551)(577,2573)
    (606,2595)(637,2616)(668,2636)
    (700,2654)(734,2670)(768,2684)
    (804,2695)(842,2703)(881,2707)
    (922,2707)(963,2702)(1005,2691)
    (1042,2676)(1077,2657)(1108,2635)
    (1134,2610)(1156,2585)(1173,2558)
    (1187,2531)(1196,2504)(1203,2477)
    (1207,2451)(1208,2424)(1209,2397)
    (1209,2371)(1209,2345)(1210,2318)
    (1211,2292)(1215,2267)(1221,2241)
    (1230,2216)(1243,2191)(1259,2168)
    (1280,2147)(1305,2128)(1334,2112)
    (1367,2101)(1402,2095)(1438,2096)
    (1472,2102)(1505,2113)(1535,2127)
    (1561,2145)(1583,2164)(1602,2185)
    (1618,2206)(1631,2229)(1642,2251)
    (1651,2275)(1659,2298)(1667,2321)
    (1674,2345)(1682,2369)(1691,2393)
    (1702,2418)(1715,2443)(1731,2468)
    (1750,2493)(1772,2517)(1798,2542)
    (1828,2566)(1861,2588)(1895,2608)
    (1931,2625)(1969,2639)(2005,2649)
    (2038,2655)(2069,2656)(2097,2655)
    (2122,2651)(2146,2644)(2168,2636)
    (2188,2626)(2207,2614)(2225,2601)
    (2241,2587)(2257,2573)(2271,2559)
    (2283,2546)(2295,2533)(2304,2522)
    (2312,2512)(2318,2505)(2323,2499)
    (2326,2495)(2327,2493)(2328,2492)
\path(6912,1575)(6913,1575)(6915,1574)
    (6920,1573)(6927,1571)(6937,1568)
    (6950,1564)(6967,1559)(6988,1553)
    (7012,1546)(7039,1538)(7070,1530)
    (7103,1520)(7140,1510)(7178,1499)
    (7219,1488)(7262,1477)(7306,1465)
    (7351,1453)(7397,1441)(7444,1429)
    (7493,1417)(7541,1405)(7591,1393)
    (7641,1382)(7693,1370)(7745,1359)
    (7798,1348)(7853,1337)(7909,1327)
    (7966,1317)(8024,1307)(8083,1298)
    (8143,1290)(8203,1282)(8262,1275)
    (8340,1267)(8412,1262)(8475,1259)
    (8529,1257)(8572,1258)(8605,1259)
    (8630,1262)(8647,1266)(8657,1271)
    (8661,1276)(8662,1282)(8659,1288)
    (8656,1294)(8652,1300)(8649,1306)
    (8649,1313)(8653,1319)(8662,1324)
    (8677,1330)(8699,1335)(8730,1339)
    (8770,1343)(8819,1346)(8876,1348)
    (8942,1350)(9012,1350)(9071,1350)
    (9130,1349)(9187,1347)(9241,1344)
    (9292,1341)(9338,1336)(9381,1331)
    (9420,1325)(9455,1319)(9486,1311)
    (9514,1304)(9540,1296)(9563,1288)
    (9584,1279)(9605,1271)(9625,1262)
    (9644,1254)(9664,1246)(9685,1239)
    (9706,1232)(9730,1226)(9756,1221)
    (9784,1217)(9815,1215)(9848,1214)
    (9885,1215)(9924,1218)(9965,1224)
    (10008,1232)(10052,1243)(10096,1258)
    (10137,1275)(10175,1295)(10209,1317)
    (10238,1340)(10264,1363)(10285,1386)
    (10303,1408)(10318,1429)(10329,1449)
    (10339,1468)(10345,1486)(10350,1503)
    (10354,1518)(10355,1533)(10356,1547)
    (10356,1561)(10356,1575)(10355,1589)
    (10353,1603)(10351,1618)(10349,1634)
    (10347,1651)(10345,1669)(10343,1690)
    (10340,1711)(10337,1735)(10334,1761)
    (10329,1789)(10324,1819)(10318,1851)
    (10310,1883)(10300,1917)(10287,1950)
    (10268,1989)(10246,2026)(10222,2058)
    (10198,2086)(10175,2111)(10152,2131)
    (10131,2147)(10111,2160)(10092,2170)
    (10074,2178)(10057,2185)(10041,2190)
    (10024,2194)(10008,2198)(9991,2202)
    (9973,2206)(9953,2212)(9932,2219)
    (9909,2227)(9883,2237)(9855,2249)
    (9825,2263)(9792,2278)(9758,2294)
    (9722,2310)(9687,2325)(9644,2342)
    (9605,2355)(9571,2364)(9542,2370)
    (9517,2373)(9495,2373)(9477,2372)
    (9460,2369)(9446,2365)(9434,2360)
    (9423,2354)(9414,2349)(9406,2343)
    (9400,2337)(9395,2333)(9391,2329)
    (9389,2327)(9388,2326)(9387,2325)
\thinlines
\color{gray2}
\put(2527,2360){\ellipse{398}{1192}}
\put(9202,2360){\ellipse{398}{1192}}
\color{black}
\put(650,0){\makebox(0,0)[lb]{Case 1}}
\put(100,399){\makebox(0,0)[lb]{$\scriptstyle -1$}}
\put(2394,399){\makebox(0,0)[lb]{$\scriptstyle t$}}
\put(4512,399){\makebox(0,0)[lb]{$\scriptstyle 1$}}
\put(200,2700){\makebox(0,0)[lb]{$\scriptstyle c(a)$}}
\put(1900,2925){\makebox(0,0)[lb]{$\scriptstyle c(b)$}}
\put(7350,0){\makebox(0,0)[lb]{Case 2}}
\put(6800,399){\makebox(0,0)[lb]{$\scriptstyle -1$}}
\put(9069,399){\makebox(0,0)[lb]{$\scriptstyle t$}}
\put(11187,399){\makebox(0,0)[lb]{$\scriptstyle 1$}}
\put(6850,1750){\makebox(0,0)[lb]{$\scriptstyle c(a)$}}
\put(9700,2425){\makebox(0,0)[lb]{$\scriptstyle c(b)$}}
\end{picture}
}

%% file: b1_case3.eepic
\setlength{\unitlength}{0.00035in}
\begingroup\makeatletter\ifx\SetFigFont\undefined%
\gdef\SetFigFont#1#2#3#4#5{%
  \reset@font\fontsize{#1}{#2pt}%
  \fontfamily{#3}\fontseries{#4}\fontshape{#5}%
  \selectfont}%
\fi\endgroup%
{\renewcommand{\dashlinestretch}{30}
\begin{picture}(11229,3974)(0,-10)
\texture{0 0 0 888888 88000000 0 0 80808
    8000000 0 0 888888 88000000 0 0 80808
    8000000 0 0 888888 88000000 0 0 80808
    8000000 0 0 888888 88000000 0 0 80808 }
\color{gray}
\path(307,706)(4542,706)
\path(307,706)(4542,706)
\path(6982,706)(11217,706)
\path(6982,706)(11217,706)
\shade\path(8703,838)(8703,3170)(9480,3947)
    (9480,1615)(8703,838)
\path(8703,838)(8703,3170)(9480,3947)
    (9480,1615)(8703,838)
\shade\path(2028,838)(2028,3170)(2805,3947)
    (2805,1615)(2028,838)
\path(2028,838)(2028,3170)(2805,3947)
    (2805,1615)(2028,838)
\color{black}
\path(1185,1650)(1035,1725)
\path(1185,1650)(1035,1575)
\path(1260,3000)(1410,3075)
\path(1260,3000)(1410,2925)
\path(8160,2175)(8010,2100)(8160,2025)
\path(9885,1725)(10035,1650)(9900,1525)
\thicklines
\path(2310,2850)(2309,2850)(2306,2851)
    (2301,2852)(2293,2854)(2282,2856)
    (2268,2859)(2250,2863)(2228,2868)
    (2203,2873)(2175,2880)(2144,2886)
    (2111,2893)(2075,2900)(2037,2908)
    (1998,2916)(1958,2923)(1917,2931)
    (1875,2938)(1832,2946)(1789,2953)
    (1745,2960)(1700,2966)(1655,2972)
    (1608,2978)(1561,2983)(1512,2988)
    (1462,2992)(1412,2995)(1361,2998)
    (1310,2999)(1260,3000)(1197,2999)
    (1139,2997)(1087,2993)(1043,2989)
    (1006,2985)(975,2980)(950,2976)
    (929,2972)(913,2968)(900,2964)
    (888,2960)(879,2956)(870,2952)
    (861,2948)(852,2943)(841,2937)
    (830,2931)(817,2923)(803,2914)
    (788,2904)(772,2893)(757,2880)
    (744,2865)(735,2850)(732,2831)
    (736,2813)(748,2797)(764,2783)
    (784,2772)(806,2763)(831,2757)
    (856,2751)(883,2747)(910,2744)
    (937,2740)(964,2736)(990,2730)
    (1016,2723)(1040,2713)(1062,2701)
    (1081,2686)(1096,2668)(1106,2647)
    (1110,2625)(1107,2605)(1100,2586)
    (1090,2568)(1079,2553)(1067,2541)
    (1056,2530)(1045,2522)(1034,2515)
    (1024,2509)(1014,2504)(1004,2500)
    (993,2496)(981,2491)(967,2486)
    (952,2480)(933,2473)(911,2464)
    (884,2454)(853,2441)(818,2428)
    (778,2414)(735,2400)(694,2388)
    (654,2378)(615,2370)(580,2364)
    (548,2360)(521,2358)(496,2358)
    (475,2360)(456,2363)(440,2366)
    (424,2371)(410,2375)(396,2379)
    (381,2383)(365,2387)(348,2389)
    (329,2389)(307,2388)(282,2385)
    (255,2378)(226,2369)(196,2358)
    (165,2343)(135,2325)(107,2304)
    (83,2282)(64,2260)(49,2240)
    (38,2222)(31,2206)(26,2192)
    (23,2180)(22,2169)(22,2160)
    (23,2150)(24,2140)(26,2130)
    (27,2119)(30,2106)(32,2091)
    (34,2073)(37,2053)(40,2030)
    (45,2005)(51,1978)(60,1950)
    (73,1921)(87,1895)(101,1873)
    (113,1855)(124,1842)(133,1831)
    (140,1824)(145,1819)(149,1815)
    (154,1812)(159,1810)(165,1807)
    (173,1803)(184,1798)(200,1791)
    (220,1781)(246,1769)(279,1756)
    (317,1740)(360,1725)(398,1713)
    (436,1703)(471,1694)(502,1686)
    (529,1680)(552,1676)(571,1673)
    (586,1671)(598,1669)(607,1669)
    (615,1669)(623,1669)(630,1669)
    (639,1669)(650,1669)(663,1669)
    (681,1668)(704,1667)(732,1665)
    (766,1663)(806,1660)(853,1657)
    (905,1653)(960,1650)(1012,1647)
    (1062,1645)(1109,1643)(1152,1641)
    (1189,1638)(1221,1636)(1247,1634)
    (1269,1631)(1288,1629)(1302,1626)
    (1315,1624)(1325,1621)(1335,1619)
    (1345,1616)(1355,1614)(1368,1613)
    (1382,1612)(1401,1611)(1423,1611)
    (1449,1612)(1481,1615)(1518,1618)
    (1561,1624)(1608,1631)(1658,1639)
    (1710,1650)(1765,1663)(1818,1678)
    (1867,1694)(1913,1710)(1954,1727)
    (1993,1744)(2028,1761)(2061,1778)
    (2092,1795)(2121,1813)(2148,1830)
    (2174,1847)(2198,1864)(2220,1879)
    (2240,1894)(2258,1908)(2273,1920)
    (2285,1930)(2295,1937)(2302,1943)
    (2306,1947)(2309,1949)(2310,1950)
\path(8910,2175)(8909,2175)(8906,2175)
    (8900,2174)(8891,2174)(8879,2173)
    (8864,2172)(8844,2170)(8822,2168)
    (8796,2166)(8767,2164)(8736,2162)
    (8703,2159)(8669,2156)(8633,2153)
    (8596,2150)(8558,2147)(8520,2144)
    (8481,2140)(8441,2137)(8400,2133)
    (8358,2129)(8315,2125)(8270,2120)
    (8225,2116)(8178,2111)(8131,2105)
    (8085,2100)(8024,2093)(7970,2086)
    (7924,2081)(7889,2077)(7862,2075)
    (7843,2075)(7830,2075)(7823,2077)
    (7819,2079)(7816,2081)(7814,2083)
    (7811,2085)(7806,2086)(7797,2085)
    (7783,2082)(7764,2076)(7739,2068)
    (7708,2057)(7673,2042)(7635,2025)
    (7602,2007)(7571,1989)(7543,1972)
    (7518,1956)(7497,1942)(7479,1929)
    (7463,1919)(7449,1910)(7437,1902)
    (7426,1894)(7416,1887)(7407,1881)
    (7398,1873)(7388,1865)(7379,1855)
    (7370,1843)(7360,1829)(7352,1812)
    (7344,1793)(7337,1772)(7334,1749)
    (7335,1725)(7340,1705)(7347,1686)
    (7355,1670)(7363,1654)(7371,1642)
    (7378,1631)(7384,1622)(7388,1614)
    (7391,1608)(7393,1604)(7395,1600)
    (7396,1597)(7398,1594)(7400,1591)
    (7403,1588)(7408,1585)(7415,1581)
    (7426,1576)(7441,1570)(7461,1564)
    (7486,1555)(7517,1546)(7555,1535)
    (7600,1524)(7652,1512)(7710,1500)
    (7757,1491)(7806,1483)(7854,1476)
    (7901,1469)(7945,1463)(7987,1458)
    (8025,1453)(8059,1449)(8091,1445)
    (8119,1441)(8144,1438)(8166,1436)
    (8186,1433)(8205,1431)(8222,1429)
    (8238,1427)(8254,1425)(8270,1423)
    (8286,1422)(8304,1420)(8324,1419)
    (8346,1418)(8371,1416)(8399,1415)
    (8431,1415)(8467,1414)(8508,1413)
    (8553,1413)(8603,1414)(8658,1415)
    (8716,1416)(8779,1418)(8844,1421)
    (8910,1425)(8980,1430)(9047,1436)
    (9110,1443)(9169,1450)(9222,1457)
    (9269,1463)(9310,1470)(9346,1477)
    (9377,1483)(9403,1488)(9425,1494)
    (9444,1499)(9460,1504)(9474,1509)
    (9486,1514)(9498,1519)(9509,1524)
    (9520,1529)(9533,1534)(9547,1539)
    (9564,1545)(9583,1552)(9605,1559)
    (9631,1566)(9660,1575)(9694,1584)
    (9731,1593)(9773,1604)(9817,1615)
    (9864,1626)(9912,1638)(9960,1650)
    (10031,1668)(10094,1684)(10146,1697)
    (10187,1707)(10218,1714)(10241,1719)
    (10257,1721)(10267,1721)(10274,1720)
    (10279,1719)(10283,1717)(10287,1717)
    (10292,1717)(10301,1720)(10312,1725)
    (10327,1733)(10346,1745)(10367,1760)
    (10389,1779)(10410,1800)(10425,1821)
    (10436,1841)(10443,1861)(10448,1878)
    (10451,1894)(10452,1907)(10451,1918)
    (10450,1928)(10447,1936)(10445,1943)
    (10441,1950)(10437,1957)(10433,1964)
    (10428,1972)(10422,1982)(10415,1993)
    (10407,2006)(10397,2022)(10385,2039)
    (10371,2059)(10354,2079)(10335,2100)
    (10311,2122)(10287,2141)(10265,2157)
    (10244,2171)(10227,2182)(10212,2191)
    (10200,2198)(10190,2204)(10181,2209)
    (10172,2213)(10164,2216)(10154,2220)
    (10143,2223)(10129,2227)(10111,2232)
    (10089,2236)(10063,2241)(10032,2245)
    (9997,2249)(9960,2250)(9923,2249)
    (9889,2245)(9859,2240)(9833,2233)
    (9813,2226)(9798,2219)(9787,2211)
    (9778,2203)(9772,2195)(9766,2187)
    (9761,2180)(9755,2171)(9746,2163)
    (9735,2154)(9721,2145)(9702,2136)
    (9678,2127)(9650,2117)(9619,2108)
    (9585,2100)(9544,2092)(9506,2086)
    (9473,2083)(9444,2081)(9417,2081)
    (9394,2082)(9372,2083)(9353,2086)
    (9335,2089)(9320,2092)(9307,2094)
    (9297,2097)(9291,2099)(9287,2100)(9285,2100)
\thinlines
\color{gray2}
\put(2425,2360){\ellipse{398}{1192}}
\put(9100,2360){\ellipse{398}{1192}}
\color{black}
\put(650,0){\makebox(0,0)[lb]{Case 3}}
\put(2394,399){\makebox(0,0)[lb]{$\scriptstyle t$}}
\put(1800,3050){\makebox(0,0)[lb]{$\scriptstyle c(a)$}}
\put(1800,1350){\makebox(0,0)[lb]{$\scriptstyle c(b)$}}
\put(7350,0){\makebox(0,0)[lb]{Case 4}}
\put(9069,399){\makebox(0,0)[lb]{$\scriptstyle t$}}
\put(9500,2400){\makebox(0,0)[lb]{$\scriptstyle c(b)$}}
\put(8250,2300){\makebox(0,0)[lb]{$\scriptstyle c(a)$}}
\end{picture}
}

%% file: param_b0.eepic
\setlength{\unitlength}{0.0005in}
\begingroup\makeatletter\ifx\SetFigFont\undefined%
\gdef\SetFigFont#1#2#3#4#5{%
  \reset@font\fontsize{#1}{#2pt}%
  \fontfamily{#3}\fontseries{#4}\fontshape{#5}%
  \selectfont}%
\fi\endgroup%
{\renewcommand{\dashlinestretch}{30}
\begin{picture}(5937,4001)(0,-10)
\texture{0 0 0 888888 88000000 0 0 80808
    8000000 0 0 888888 88000000 0 0 80808
    8000000 0 0 888888 88000000 0 0 80808
    8000000 0 0 888888 88000000 0 0 80808 }
\color{gray}
\shade\path(12,450)(12,3093)(893,3974)
    (893,1331)(12,450)
\path(12,450)(12,3093)(893,3974)
    (893,1331)(12,450)
\shade\path(4812,450)(4812,3093)(5693,3974)
    (5693,1331)(4812,450)
\path(4812,450)(4812,3093)(5693,3974)
    (5693,1331)(4812,450)
\path(462,360)(5262,360)
\color{black}
\thicklines
\thicklines
\path(462,1320)(463,1320)(464,1318)
    (468,1316)(473,1313)(480,1308)
    (490,1302)(503,1294)(518,1284)
    (537,1273)(558,1260)(583,1245)
    (610,1228)(640,1210)(673,1191)
    (708,1170)(746,1149)(786,1127)
    (827,1104)(870,1081)(914,1057)
    (960,1034)(1007,1011)(1056,988)
    (1105,965)(1156,942)(1207,921)
    (1260,899)(1315,879)(1371,859)
    (1428,840)(1487,822)(1548,805)
    (1611,789)(1676,775)(1744,761)
    (1813,750)(1885,740)(1958,731)
    (2034,725)(2110,721)(2187,720)
    (2267,721)(2345,726)(2420,732)
    (2490,740)(2555,751)(2615,762)
    (2669,775)(2718,788)(2760,802)
    (2798,817)(2830,832)(2858,848)
    (2881,864)(2901,879)(2917,895)
    (2931,911)(2942,928)(2952,944)
    (2960,960)(2968,976)(2976,993)
    (2984,1009)(2994,1026)(3005,1043)
    (3018,1060)(3033,1077)(3051,1094)
    (3073,1112)(3099,1130)(3129,1148)
    (3164,1166)(3204,1185)(3249,1203)
    (3299,1221)(3354,1240)(3414,1257)
    (3478,1274)(3546,1291)(3616,1306)
    (3687,1320)(3766,1334)(3843,1345)
    (3917,1354)(3989,1362)(4057,1367)
    (4123,1371)(4184,1373)(4243,1374)
    (4300,1373)(4353,1371)(4405,1368)
    (4454,1365)(4501,1360)(4547,1354)
    (4591,1348)(4633,1342)(4674,1335)
    (4714,1327)(4751,1319)(4787,1311)
    (4822,1304)(4854,1296)(4884,1288)
    (4911,1281)(4936,1274)(4958,1268)
    (4977,1263)(4994,1258)(5007,1254)
    (5017,1251)(5025,1249)(5031,1247)
    (5034,1246)(5036,1245)(5037,1245)
\color{gray2}
\thinlines
\path(387,2220)(388,2221)(389,2222)
    (392,2225)(396,2229)(403,2234)
    (411,2242)(422,2252)(435,2263)
    (450,2276)(468,2291)(488,2308)
    (511,2326)(535,2345)(562,2364)
    (590,2385)(620,2405)(652,2426)
    (685,2447)(720,2467)(756,2487)
    (794,2506)(833,2523)(875,2540)
    (918,2556)(964,2570)(1012,2583)
    (1062,2594)(1116,2603)(1173,2610)
    (1234,2615)(1298,2617)(1366,2617)
    (1437,2613)(1511,2606)(1587,2595)
    (1653,2583)(1719,2569)(1783,2552)
    (1844,2535)(1902,2516)(1956,2496)
    (2005,2476)(2050,2455)(2091,2434)
    (2127,2413)(2158,2392)(2186,2371)
    (2210,2350)(2230,2329)(2247,2309)
    (2262,2288)(2275,2268)(2285,2248)
    (2295,2228)(2304,2208)(2312,2189)
    (2320,2169)(2330,2149)(2340,2129)
    (2351,2109)(2365,2089)(2382,2068)
    (2401,2048)(2424,2027)(2451,2006)
    (2482,1984)(2518,1963)(2559,1941)
    (2605,1919)(2657,1898)(2715,1877)
    (2778,1856)(2847,1836)(2920,1817)
    (2998,1800)(3079,1784)(3162,1770)
    (3242,1759)(3323,1751)(3402,1744)
    (3480,1740)(3556,1738)(3630,1738)
    (3703,1739)(3773,1742)(3841,1746)
    (3907,1752)(3972,1759)(4034,1767)
    (4095,1776)(4155,1786)(4213,1796)
    (4271,1808)(4326,1820)(4381,1833)
    (4435,1847)(4488,1861)(4539,1875)
    (4589,1890)(4638,1904)(4686,1919)
    (4731,1934)(4775,1948)(4817,1962)
    (4856,1976)(4893,1989)(4928,2001)
    (4959,2012)(4988,2023)(5013,2032)
    (5035,2040)(5055,2048)(5070,2054)
    (5083,2059)(5094,2063)(5101,2066)
    (5106,2068)(5110,2069)(5111,2070)(5112,2070)
\color{black}
\thicklines
\path(2262,2295)(2263,2297)(2264,2301)
    (2267,2309)(2270,2320)(2275,2335)
    (2280,2353)(2285,2373)(2290,2396)
    (2295,2420)(2298,2445)(2300,2472)
    (2301,2500)(2299,2530)(2295,2562)
    (2288,2596)(2277,2632)(2262,2670)
    (2246,2703)(2229,2733)(2214,2758)
    (2201,2778)(2190,2792)(2182,2803)
    (2177,2809)(2173,2812)(2170,2813)
    (2168,2814)(2166,2814)(2161,2816)
    (2155,2819)(2145,2826)(2130,2837)
    (2110,2853)(2083,2874)(2049,2901)
    (2008,2934)(1962,2970)(1927,2997)
    (1891,3024)(1857,3050)(1825,3075)
    (1796,3098)(1771,3120)(1748,3140)
    (1730,3158)(1714,3175)(1700,3191)
    (1689,3206)(1679,3219)(1670,3232)
    (1662,3245)(1653,3257)(1644,3270)
    (1633,3282)(1621,3295)(1605,3308)
    (1587,3322)(1565,3335)(1538,3350)
    (1507,3364)(1471,3378)(1430,3391)
    (1386,3403)(1337,3413)(1287,3420)
    (1236,3423)(1186,3423)(1138,3420)
    (1092,3414)(1048,3406)(1006,3396)
    (967,3384)(929,3370)(893,3355)
    (858,3338)(825,3321)(793,3303)
    (761,3284)(732,3265)(703,3246)
    (676,3227)(651,3209)(628,3192)
    (607,3176)(589,3162)(574,3150)
    (561,3140)(552,3132)(545,3127)
    (540,3123)(538,3121)(537,3120)
\path(3162,1770)(3162,1771)(3161,1773)
    (3160,1777)(3159,1783)(3158,1791)
    (3156,1801)(3154,1813)(3153,1827)
    (3151,1843)(3150,1860)(3150,1878)
    (3151,1897)(3153,1916)(3157,1936)
    (3162,1956)(3169,1978)(3179,1999)
    (3192,2022)(3208,2045)(3228,2069)
    (3251,2094)(3280,2120)(3312,2145)
    (3345,2168)(3379,2189)(3414,2208)
    (3448,2226)(3480,2242)(3510,2256)
    (3538,2268)(3564,2279)(3589,2289)
    (3612,2298)(3633,2306)(3654,2313)
    (3675,2320)(3695,2327)(3715,2333)
    (3736,2339)(3758,2345)(3781,2351)
    (3805,2357)(3830,2362)(3857,2367)
    (3884,2371)(3912,2374)(3939,2375)
    (3965,2374)(3987,2370)(4007,2361)
    (4019,2348)(4024,2331)(4020,2313)
    (4008,2292)(3991,2271)(3968,2249)
    (3942,2226)(3913,2203)(3881,2180)
    (3849,2157)(3818,2135)(3788,2113)
    (3762,2091)(3739,2070)(3722,2050)
    (3713,2032)(3713,2016)(3725,2003)
    (3748,1995)(3786,1992)(3837,1995)
    (3873,2000)(3912,2008)(3954,2017)
    (3998,2028)(4043,2040)(4091,2054)
    (4139,2069)(4188,2085)(4238,2102)
    (4288,2120)(4339,2139)(4391,2158)
    (4443,2178)(4495,2198)(4548,2219)
    (4601,2240)(4653,2261)(4706,2283)
    (4758,2304)(4809,2326)(4860,2347)
    (4909,2368)(4956,2388)(5000,2407)
    (5042,2425)(5081,2441)(5116,2457)
    (5148,2470)(5175,2482)(5199,2492)
    (5218,2501)(5233,2507)(5245,2512)
    (5253,2516)(5258,2518)(5261,2519)(5262,2520)
\thinlines
\put(5112,0){\makebox(0,0)[lb]{$1$}}
\put(312,0){\makebox(0,0)[lb]{$-1$}}
\put(2690,0){\makebox(0,0)[lb]{$t$}}
\put(950,1450){\makebox(0,0)[lb]{\scriptsize reducible}}
\put(1400,1200){\makebox(0,0)[lb]{\scriptsize branch}}
\dashline{50.000}(2500,1850)(2050,1420)
\put(2750,3650){\makebox(0,0)[lb]{\scriptsize slightly degenerate}}
\put(3200,3400){\makebox(0,0)[lb]{\scriptsize points}}
\dashline{50.000}(2362,2350)(3000,3400)
\dashline{50.000}(3120,1875)(3100,3400)
\end{picture}
}

%% file: teta.eepic
\setlength{\unitlength}{0.00035in}
\begingroup\makeatletter\ifx\SetFigFont\undefined%
\gdef\SetFigFont#1#2#3#4#5{%
  \reset@font\fontsize{#1}{#2pt}%
  \fontfamily{#3}\fontseries{#4}\fontshape{#5}%
  \selectfont}%
\fi\endgroup%
{\renewcommand{\dashlinestretch}{30}
\begin{picture}(9794,5084)(0,-10)
\path(6247,322)(6247,4672)
\path(6247,2422)(9547,2422)
\path(22,322)(22,4672)
\path(22,2422)(3322,2422)
\dashline{60.000}(22,1900)(3322,1900)
\dashline{60.000}(6247,1900)(9547,1900)
\color{gray2}
\thicklines
\path(6247,2422)(6248,2423)(6250,2426)
    (6254,2431)(6261,2439)(6271,2451)
    (6284,2467)(6300,2487)(6320,2511)
    (6343,2539)(6370,2571)(6400,2607)
    (6433,2647)(6468,2689)(6506,2734)
    (6545,2781)(6586,2829)(6627,2878)
    (6669,2927)(6711,2975)(6752,3024)
    (6794,3071)(6834,3117)(6874,3162)
    (6913,3206)(6951,3248)(6987,3288)
    (7023,3326)(7058,3363)(7092,3399)
    (7125,3433)(7157,3465)(7189,3496)
    (7220,3526)(7250,3555)(7280,3583)
    (7310,3610)(7340,3636)(7370,3661)
    (7399,3686)(7429,3711)(7460,3734)
    (7491,3759)(7524,3784)(7557,3808)
    (7590,3832)(7624,3856)(7658,3879)
    (7693,3903)(7728,3926)(7764,3950)
    (7801,3973)(7838,3996)(7876,4019)
    (7914,4042)(7952,4064)(7991,4087)
    (8030,4109)(8070,4131)(8109,4152)
    (8149,4173)(8188,4194)(8228,4215)
    (8267,4235)(8306,4254)(8344,4274)
    (8382,4292)(8420,4310)(8457,4328)
    (8493,4345)(8529,4361)(8564,4377)
    (8598,4393)(8631,4408)(8664,4422)
    (8696,4436)(8727,4449)(8757,4462)
    (8787,4474)(8816,4487)(8844,4498)
    (8872,4509)(8908,4524)(8944,4538)
    (8978,4552)(9013,4565)(9047,4578)
    (9081,4591)(9115,4604)(9150,4616)
    (9186,4629)(9222,4642)(9260,4655)
    (9298,4668)(9338,4682)(9378,4695)
    (9419,4709)(9461,4723)(9502,4736)
    (9543,4749)(9582,4762)(9619,4774)
    (9653,4785)(9683,4794)(9709,4802)
    (9731,4809)(9747,4814)(9759,4818)
    (9766,4820)(9770,4821)(9772,4822)
\path(6247,2422)(6248,2421)(6251,2417)
    (6255,2411)(6263,2401)(6273,2387)
    (6287,2368)(6305,2344)(6326,2315)
    (6351,2282)(6379,2244)(6410,2201)
    (6444,2156)(6480,2107)(6518,2055)
    (6557,2002)(6597,1948)(6637,1894)
    (6676,1839)(6716,1786)(6754,1733)
    (6792,1681)(6828,1631)(6863,1583)
    (6896,1537)(6928,1492)(6958,1449)
    (6987,1409)(7015,1369)(7041,1332)
    (7066,1296)(7089,1262)(7112,1229)
    (7134,1197)(7154,1166)(7174,1136)
    (7193,1106)(7212,1078)(7230,1050)
    (7247,1022)(7268,988)(7289,954)
    (7309,921)(7328,888)(7347,854)
    (7366,820)(7385,786)(7404,751)
    (7423,716)(7442,679)(7461,641)
    (7480,601)(7500,561)(7521,519)
    (7541,476)(7562,431)(7583,387)
    (7604,342)(7624,298)(7643,255)
    (7662,214)(7679,175)(7695,141)
    (7708,110)(7720,84)(7729,63)
    (7736,47)(7741,35)(7744,28)
    (7746,24)(7747,22)
\path(6247,2422)(6249,2423)(6253,2424)
    (6262,2428)(6275,2432)(6293,2439)
    (6316,2448)(6345,2459)(6379,2472)
    (6417,2487)(6459,2503)(6503,2521)
    (6550,2539)(6597,2558)(6645,2577)
    (6692,2596)(6738,2615)(6782,2634)
    (6825,2652)(6865,2670)(6904,2687)
    (6940,2704)(6975,2720)(7008,2736)
    (7039,2751)(7068,2767)(7096,2782)
    (7123,2797)(7149,2812)(7174,2828)
    (7198,2844)(7222,2859)(7247,2877)
    (7272,2895)(7296,2913)(7320,2932)
    (7344,2952)(7368,2973)(7392,2994)
    (7416,3017)(7439,3040)(7463,3064)
    (7486,3089)(7509,3115)(7531,3142)
    (7553,3170)(7574,3198)(7595,3227)
    (7616,3256)(7635,3287)(7654,3317)
    (7672,3348)(7689,3379)(7706,3411)
    (7721,3442)(7736,3474)(7750,3506)
    (7763,3539)(7776,3571)(7788,3605)
    (7799,3638)(7810,3672)(7818,3701)
    (7826,3730)(7834,3761)(7841,3792)
    (7848,3825)(7855,3858)(7861,3893)
    (7868,3930)(7874,3969)(7880,4009)
    (7886,4052)(7891,4097)(7897,4145)
    (7903,4194)(7908,4246)(7914,4301)
    (7919,4357)(7924,4415)(7929,4474)
    (7934,4534)(7939,4594)(7944,4653)
    (7948,4711)(7953,4766)(7956,4818)
    (7960,4865)(7963,4908)(7965,4944)
    (7967,4975)(7969,5000)(7970,5019)
    (7971,5032)(7972,5040)(7972,5045)(7972,5047)
\path(22,2422)(23,2423)(25,2426)
    (29,2431)(36,2439)(46,2451)
    (59,2467)(75,2487)(95,2511)
    (118,2539)(145,2571)(175,2607)
    (208,2647)(243,2689)(281,2734)
    (320,2781)(361,2829)(402,2878)
    (444,2927)(486,2975)(527,3024)
    (569,3071)(609,3117)(649,3162)
    (688,3206)(726,3248)(762,3288)
    (798,3326)(833,3363)(867,3399)
    (900,3433)(932,3465)(964,3496)
    (995,3526)(1025,3555)(1055,3583)
    (1085,3610)(1115,3636)(1145,3661)
    (1174,3686)(1204,3711)(1235,3734)
    (1266,3759)(1299,3784)(1332,3808)
    (1365,3832)(1399,3856)(1433,3879)
    (1468,3903)(1503,3926)(1539,3950)
    (1576,3973)(1613,3996)(1651,4019)
    (1689,4042)(1727,4064)(1766,4087)
    (1805,4109)(1845,4131)(1884,4152)
    (1924,4173)(1963,4194)(2003,4215)
    (2042,4235)(2081,4254)(2119,4274)
    (2157,4292)(2195,4310)(2232,4328)
    (2268,4345)(2304,4361)(2339,4377)
    (2373,4393)(2406,4408)(2439,4422)
    (2471,4436)(2502,4449)(2532,4462)
    (2562,4474)(2591,4487)(2619,4498)
    (2647,4509)(2683,4524)(2719,4538)
    (2753,4552)(2788,4565)(2822,4578)
    (2856,4591)(2890,4604)(2925,4616)
    (2961,4629)(2997,4642)(3035,4655)
    (3073,4668)(3113,4682)(3153,4695)
    (3194,4709)(3236,4723)(3277,4736)
    (3318,4749)(3357,4762)(3394,4774)
    (3428,4785)(3458,4794)(3484,4802)
    (3506,4809)(3522,4814)(3534,4818)
    (3541,4820)(3545,4821)(3547,4822)
\path(22,2422)(23,2421)(26,2417)
    (30,2411)(38,2401)(48,2387)
    (62,2368)(80,2344)(101,2315)
    (126,2282)(154,2244)(185,2201)
    (219,2156)(255,2107)(293,2055)
    (332,2002)(372,1948)(412,1894)
    (451,1839)(491,1786)(529,1733)
    (567,1681)(603,1631)(638,1583)
    (671,1537)(703,1492)(733,1449)
    (762,1409)(790,1369)(816,1332)
    (841,1296)(864,1262)(887,1229)
    (909,1197)(929,1166)(949,1136)
    (968,1106)(987,1078)(1005,1050)
    (1022,1022)(1043,988)(1064,954)
    (1084,921)(1103,888)(1122,854)
    (1141,820)(1160,786)(1179,751)
    (1198,716)(1217,679)(1236,641)
    (1255,601)(1275,561)(1296,519)
    (1316,476)(1337,431)(1358,387)
    (1379,342)(1399,298)(1418,255)
    (1437,214)(1454,175)(1470,141)
    (1483,110)(1495,84)(1504,63)
    (1511,47)(1516,35)(1519,28)
    (1521,24)(1522,22)
\path(22,2422)(24,2422)(29,2421)
    (37,2419)(50,2416)(69,2412)
    (94,2407)(124,2401)(160,2394)
    (202,2385)(247,2375)(297,2364)
    (349,2353)(403,2341)(458,2328)
    (512,2315)(567,2302)(620,2289)
    (672,2276)(722,2263)(770,2250)
    (815,2237)(859,2225)(900,2212)
    (940,2199)(978,2186)(1014,2173)
    (1048,2160)(1081,2146)(1113,2132)
    (1145,2118)(1175,2103)(1205,2088)
    (1235,2072)(1264,2056)(1293,2038)
    (1322,2021)(1351,2002)(1380,1983)
    (1409,1963)(1438,1942)(1467,1920)
    (1496,1898)(1526,1875)(1555,1851)
    (1585,1827)(1614,1801)(1644,1776)
    (1673,1749)(1702,1723)(1731,1696)
    (1760,1668)(1788,1641)(1816,1613)
    (1843,1586)(1870,1558)(1897,1531)
    (1922,1503)(1947,1476)(1972,1449)
    (1996,1423)(2019,1397)(2041,1371)
    (2063,1346)(2085,1321)(2106,1296)
    (2127,1271)(2147,1247)(2168,1221)
    (2190,1196)(2211,1170)(2232,1144)
    (2253,1117)(2274,1090)(2296,1063)
    (2319,1034)(2342,1005)(2365,974)
    (2390,942)(2416,908)(2442,873)
    (2470,837)(2498,799)(2527,761)
    (2556,722)(2586,682)(2615,643)
    (2643,605)(2670,569)(2695,535)
    (2718,503)(2739,476)(2756,453)
    (2770,434)(2781,419)(2788,409)
    (2793,402)(2796,399)(2797,397)
\color{black}
\put(6500,697){\makebox(0,0)[lb]{$\scriptstyle\gt_{-1}$}}
\put(9097,4097){\makebox(0,0)[lb]{$\scriptstyle\gt_1$}}
\put(7797,2872){\makebox(0,0)[lb]{$\scriptstyle\gt_0$}}
\put(2150,1447){\makebox(0,0)[lb]{$\scriptstyle\gt_0$}}
\put(2122,3722){\makebox(0,0)[lb]{$\scriptstyle\gt_1$}}
\put(350,697){\makebox(0,0)[lb]{$\scriptstyle\gt_{-1}$}}
\end{picture}
}

%% file: b0_case1.eepic
\setlength{\unitlength}{0.00035in}
\begingroup\makeatletter\ifx\SetFigFont\undefined%
\gdef\SetFigFont#1#2#3#4#5{%
  \reset@font\fontsize{#1}{#2pt}%
  \fontfamily{#3}\fontseries{#4}\fontshape{#5}%
  \selectfont}%
\fi\endgroup%
{\renewcommand{\dashlinestretch}{30}
\begin{picture}(11712,3974)(0,-10)
\texture{0 0 0 888888 88000000 0 0 80808
    8000000 0 0 888888 88000000 0 0 80808
    8000000 0 0 888888 88000000 0 0 80808
    8000000 0 0 888888 88000000 0 0 80808 }
\color{gray}
\shade\path(12,838)(12,3170)(789,3947)
    (789,1615)(12,838)
\path(12,838)(12,3170)(789,3947)
    (789,1615)(12,838)
\shade\path(4247,838)(4247,3170)(5025,3947)
    (5025,1615)(4247,838)
\path(4247,838)(4247,3170)(5025,3947)
    (5025,1615)(4247,838)
\path(409,706)(4644,706)
\path(409,706)(4644,706)
\shade\path(6687,838)(6687,3170)(7464,3947)
    (7464,1615)(6687,838)
\path(6687,838)(6687,3170)(7464,3947)
    (7464,1615)(6687,838)
\shade\path(10922,838)(10922,3170)(11700,3947)
    (11700,1615)(10922,838)
\path(10922,838)(10922,3170)(11700,3947)
    (11700,1615)(10922,838)
\path(7084,706)(11319,706)
\path(7084,706)(11319,706)
\color{gray2}
\path(462,2325)(463,2325)(465,2324)
    (469,2323)(474,2321)(483,2318)
    (494,2315)(508,2310)(525,2305)
    (545,2299)(568,2292)(594,2284)
    (623,2275)(654,2266)(688,2257)
    (723,2247)(760,2237)(800,2227)
    (840,2216)(882,2206)(926,2196)
    (970,2186)(1017,2177)(1064,2167)
    (1114,2158)(1165,2150)(1219,2142)
    (1274,2134)(1332,2127)(1393,2120)
    (1456,2115)(1523,2110)(1592,2106)
    (1664,2103)(1737,2101)(1812,2100)
    (1887,2101)(1959,2103)(2029,2106)
    (2094,2110)(2153,2115)(2206,2121)
    (2254,2127)(2296,2134)(2331,2142)
    (2362,2150)(2387,2158)(2408,2166)
    (2425,2175)(2439,2183)(2450,2192)
    (2459,2201)(2467,2210)(2475,2219)
    (2482,2228)(2490,2237)(2499,2245)
    (2511,2254)(2525,2262)(2542,2271)
    (2564,2279)(2590,2287)(2621,2294)
    (2658,2301)(2701,2307)(2750,2313)
    (2806,2318)(2867,2322)(2935,2324)
    (3007,2326)(3083,2326)(3162,2325)
    (3237,2322)(3311,2318)(3384,2313)
    (3455,2307)(3524,2300)(3590,2292)
    (3653,2284)(3715,2275)(3774,2265)
    (3831,2255)(3886,2244)(3939,2233)
    (3991,2222)(4041,2210)(4090,2198)
    (4138,2186)(4184,2174)(4229,2162)
    (4273,2149)(4315,2137)(4356,2125)
    (4395,2113)(4432,2101)(4466,2090)
    (4499,2080)(4528,2070)(4555,2062)
    (4578,2054)(4599,2047)(4616,2041)
    (4630,2036)(4641,2032)(4650,2029)
    (4655,2027)(4659,2026)(4661,2025)(4662,2025)
\path(7137,2325)(7138,2325)(7140,2324)
    (7144,2323)(7149,2321)(7158,2318)
    (7169,2315)(7183,2310)(7200,2305)
    (7220,2299)(7243,2292)(7269,2284)
    (7298,2275)(7329,2266)(7363,2257)
    (7398,2247)(7435,2237)(7475,2227)
    (7515,2216)(7557,2206)(7601,2196)
    (7645,2186)(7692,2177)(7739,2167)
    (7789,2158)(7840,2150)(7894,2142)
    (7949,2134)(8007,2127)(8068,2120)
    (8131,2115)(8198,2110)(8267,2106)
    (8339,2103)(8412,2101)(8487,2100)
    (8562,2101)(8634,2103)(8704,2106)
    (8769,2110)(8828,2115)(8881,2121)
    (8929,2127)(8971,2134)(9006,2142)
    (9037,2150)(9062,2158)(9083,2166)
    (9100,2175)(9114,2183)(9125,2192)
    (9134,2201)(9142,2210)(9150,2219)
    (9157,2228)(9165,2237)(9174,2245)
    (9186,2254)(9200,2262)(9217,2271)
    (9239,2279)(9265,2287)(9296,2294)
    (9333,2301)(9376,2307)(9425,2313)
    (9481,2318)(9542,2322)(9610,2324)
    (9682,2326)(9758,2326)(9837,2325)
    (9912,2322)(9986,2318)(10059,2313)
    (10130,2307)(10199,2300)(10265,2292)
    (10328,2284)(10390,2275)(10449,2265)
    (10506,2255)(10561,2244)(10614,2233)
    (10666,2222)(10716,2210)(10765,2198)
    (10813,2186)(10859,2174)(10904,2162)
    (10948,2149)(10990,2137)(11031,2125)
    (11070,2113)(11107,2101)(11141,2090)
    (11174,2080)(11203,2070)(11230,2062)
    (11253,2054)(11274,2047)(11291,2041)
    (11305,2036)(11316,2032)(11325,2029)
    (11330,2027)(11334,2026)(11336,2025)(11337,2025)
\color{black}
\thicklines
\path(237,1650)(238,1649)(240,1648)
    (245,1645)(251,1640)(261,1633)
    (273,1625)(288,1614)(306,1602)
    (327,1588)(351,1572)(377,1555)
    (404,1537)(434,1519)(464,1500)
    (496,1481)(529,1462)(563,1443)
    (598,1424)(634,1406)(671,1389)
    (709,1372)(748,1355)(788,1340)
    (831,1326)(874,1312)(920,1300)
    (967,1290)(1014,1281)(1062,1275)
    (1116,1271)(1166,1270)(1212,1270)
    (1252,1272)(1286,1275)(1315,1279)
    (1339,1283)(1357,1287)(1372,1291)
    (1384,1294)(1392,1298)(1400,1302)
    (1406,1306)(1412,1311)(1418,1316)
    (1426,1322)(1435,1330)(1447,1338)
    (1462,1349)(1481,1362)(1503,1377)
    (1529,1396)(1559,1417)(1592,1442)
    (1627,1469)(1662,1500)(1698,1535)
    (1732,1572)(1761,1609)(1788,1645)
    (1811,1681)(1832,1716)(1850,1751)
    (1866,1784)(1881,1818)(1893,1851)
    (1905,1883)(1915,1915)(1924,1945)
    (1932,1974)(1939,2001)(1945,2025)
    (1950,2046)(1954,2064)(1957,2078)
    (1960,2088)(1961,2095)(1962,2098)(1962,2100)
\path(6987,1425)(6988,1425)(6990,1424)
    (6994,1423)(7000,1421)(7009,1418)
    (7020,1414)(7035,1409)(7053,1403)
    (7074,1396)(7098,1388)(7126,1380)
    (7156,1370)(7189,1360)(7225,1349)
    (7262,1337)(7302,1325)(7343,1313)
    (7385,1301)(7429,1288)(7474,1276)
    (7520,1264)(7567,1252)(7615,1240)
    (7664,1228)(7714,1217)(7764,1206)
    (7816,1195)(7870,1185)(7924,1175)
    (7981,1166)(8038,1158)(8098,1150)
    (8159,1143)(8221,1137)(8285,1132)
    (8348,1128)(8412,1125)(8491,1124)
    (8565,1124)(8635,1126)(8698,1129)
    (8755,1132)(8805,1136)(8848,1140)
    (8886,1144)(8918,1147)(8945,1151)
    (8969,1155)(8989,1158)(9007,1162)
    (9022,1165)(9037,1169)(9051,1173)
    (9065,1177)(9080,1182)(9096,1188)
    (9113,1194)(9133,1202)(9154,1211)
    (9179,1221)(9205,1234)(9234,1248)
    (9265,1264)(9297,1283)(9329,1303)
    (9360,1326)(9387,1350)(9414,1382)
    (9433,1414)(9445,1445)(9451,1473)
    (9451,1498)(9446,1520)(9438,1540)
    (9428,1556)(9415,1571)(9400,1584)
    (9385,1595)(9368,1606)(9351,1617)
    (9334,1629)(9316,1642)(9298,1657)
    (9280,1674)(9263,1694)(9245,1717)
    (9228,1744)(9211,1773)(9194,1806)
    (9178,1840)(9162,1875)(9142,1919)
    (9123,1958)(9107,1993)(9093,2022)
    (9080,2048)(9069,2070)(9059,2089)
    (9051,2106)(9042,2121)(9035,2134)
    (9029,2146)(9023,2155)(9019,2163)
    (9016,2168)(9014,2172)(9013,2174)(9012,2175)
\thinlines
\put(650,0){\makebox(0,0)[lb]{Case  1a}}
\put(100,399){\makebox(0,0)[lb]{$\scriptstyle -1$}}
\put(2394,399){\makebox(0,0)[lb]{$\scriptstyle t$}}
\put(4512,399){\makebox(0,0)[lb]{$\scriptstyle 1$}}
\put(162,1650){\makebox(0,0)[lb]{$\scriptstyle c(a)$}}
\put(2800,3600){\makebox(0,0)[lb]{\scriptsize reducible}}
\put(3237,3315){\makebox(0,0)[lb]{\scriptsize branch}}
\put(1600,2300){\makebox(0,0)[lb]{$\scriptstyle c(b)$}}
\dashline{60.000}(2712,2400)(3150,3300)

\put(7350,0){\makebox(0,0)[lb]{Case  1b}}
\put(6800,399){\makebox(0,0)[lb]{$\scriptstyle -1$}}
\put(9069,399){\makebox(0,0)[lb]{$\scriptstyle t$}}
\put(11187,399){\makebox(0,0)[lb]{$\scriptstyle 1$}}
\put(9500,3600){\makebox(0,0)[lb]{\scriptsize reducible}}
\put(10062,3300){\makebox(0,0)[lb]{\scriptsize branch}}
\put(8700,2400){\makebox(0,0)[lb]{$\scriptstyle c(b)$}}
\put(6837,1580){\makebox(0,0)[lb]{$\scriptstyle c(a)$}}
\dashline{60.000}(9900,3200)(9612,2400)

\end{picture}
}

%% file: app_A.tex
\chapter{Elliptic Equations on Compact Manifolds}\label{app:ell}

This appendix summarizes the basic notions and results we need
from the theory of elliptic differential equations on compact
manifolds. In Section \ref{sob}, we recall the definition of
Sobolev spaces on manifolds. As the Seiberg-Witten equations are
nonlinear, it is not sufficient to work only in the well-known
setting of the Hilbert spaces $L^2_k$; we also have to consider
the Banach spaces $L_k^p$ for arbitrary $1\le p<\infty$. We state
the versions of Sobolev embedding and Rellich's compactness
result in this more general context, and provide a Sobolev
multiplication theorem as a corollary. The corresponding proofs
can be found in standard references like the books of Gilbarg \&
Trudinger \cite{GT}, Adams \cite{Ad}, Taylor \cite{Tay:PDE}, and
Aubin \cite{Au}. The latter book includes the formulation on
manifolds which we need.

Section \ref{diffop} is dedicated to differential operators acting
on sections of vector bundles. Basically, we will only list the
fundamental properties of elliptic partial differential operators
and refer for most proofs to the wide range of literature (e.g.
Gilbarg \& Trudinger \cite{GT}). Brief expositions, yet including
proofs, can also be found in many books on differential geometry
(e.g. Warner \cite{War:LG} or Nicolaescu \cite{Nic:GM}).

\section{Sobolev spaces}\label{sob}\index{Sobolev spaces|(}
Let $(M,g)$ be a closed\footnote{As in the main part of the
thesis we use the convention that a closed manifold is compact,
connected, and has no boundary.} and oriented Riemannian
manifold. The volume form $dv_g$ induces a Lebesgue measure on
$M$. For each $1\le p<\infty$, we can thus define the space
$L^p(M,\K)$ of (equivalence classes of) $\K$-valued measurable
functions $f$ on $M$ for which
\[
\|f\|_p:=\Big(\int_M |f|^pdv_g\Big)^{1/p}<\infty.
\]
Suppose $\pi:E\to M$ is a Hermitian or Euclidean vector bundle
endowed with a connection $\nabla$ which is compatible with the
metric. We let $L^p(M,E)$ denote the space of $L^p$-sections of
$E$, i.e., the space of (equivalence classes of) measurable maps
$u:M\to E$ which satisfy $\pi\circ u = \id_M$ almost everywhere
and $|u|\in L^p(M,\R)$. For each $k\in\N$ the Sobolev space
$L_k^p(M,E)$ consists of all sections $u\in L^p(M,E)$ for
which\index{<@$L^p_k(E)$, Sobolev space} there exists $v\in
L^p\big(M,T^*M^{\otimes m}\otimes E\big)$ such that for all $w\in
C^\infty\big(M,T^*M^{\otimes m}\otimes E\big)$ and any $m\le k$,
\[
\int_M\scalar{v}{w}dv_g=\int_M\scalar{u}{(\nabla^m)^t w}dv_g
\]
Here,
\[
(\nabla^m)^t:C^\infty\big(M,T^*M^{\otimes m}\otimes E\big)\to
C^\infty(M,E)
\]
denotes the formal adjoint of $\nabla^m$ (cf. also Section
\ref{diffop} below). Then $v$ is called the \emph{weak} $m$-th
covariant derivative of $u$ and is denoted by $\nabla^m u$. Note
that it is always defined as a distribution. Therefore, the above
can be reformulated by saying
\[
u\in L^p_k(M,E)\quad\Longleftrightarrow\quad\nabla^mu\in
L^p\big(T^*M^{\otimes m}\otimes E\big).
\]
Each $L_k^p(M,E)$ is a Banach space with respect to the
norm\index{<@$\parallel.\parallel_{k,p}=
\parallel.\parallel_{L_k^p}$, Sobolev norm}
\[
\|u\|_{L_k^p}:=\sum_{m\le k}\|\nabla^m u\|_{L^p}= \sum_{m\le
k}\Big(\int_M|\nabla^m u|^pdv_g\Big)^{1/p}.
\]
Moreover, $L^p_k(M,E)$ lies dense in $L^p(M,E)$ and contains the
smooth functions as a dense subspace with respect to the norm
$\|.\|_{L_k^p}$ (cf. \cite{Au}, Thm.2.4). Furthermore, it turns
out that compactness of $M$ guarantees that the definition of
$L_k^p(M,E)$ is independent of all choices made. Any choice of
different metrics and connections yields equivalent norms (cf.
Aubin \cite{Au}, Thm.2.20).

\begin{remark*}
If $p\neq 2$, the spaces $L_k^p$ cannot conveniently be defined
for by making use of the Fourier transformation as in the case of
$L^2_k$. Introducing $L_s^p$ for every $s\in\R$ requires
interpolation theory (cf. Taylor \cite{Tay:PDE}, Sec.~13.6).\\
\end{remark*}

\noindent\textbf{Embedding theorems.} There are two very
important results relating Sobolev spaces both with each other
and with spaces of $C^r$ functions. These theorems are
generalizations of the well-known Sobolev embedding Theorem and
of Rellich's lemma in the case $p=2$. Let $n$ denote the dimension
of $M$. We define the \emph{scaling weight} of $L_k^p(M,E)$ by
letting
\[
w(k,p):=k-\lfrac{n}{p}.
\]
\begin{theorem}[Sobolev embedding]\label{sobolev}{\rm(cf. \cite{Au},
Thm.~2.20)}.\index{Sobolev spaces!Sobolev embedding}\\ Consider a
closed, oriented manifold $M$ and a vector bundle $E\to M$.
\begin{enumerate}
\item Suppose $k_1\ge k_2$ and $w(k_1,p_1)\ge w(k_2,p_2)$. Then there
is a bounded inclusion $L_{k_1}^{p_1}(M,E)\subset
L_{k_2}^{p_2}(M,E)$.
\item Suppose $w(k,p)> r\in\N$. Then every $L_k^p$-section
of $E$ can be represented by a $C^r$-section. Moreover, the
inclusion $L_k^p(M,E)\subset C^r(E)$ is bounded with respect to
the norm on $C^r(E)$ given by uniform convergence of the involved
derivatives.
\end{enumerate}
\end{theorem}
\begin{theorem}[Rellich-Kondrachov]\label{rellich}{\rm(cf.
\cite{Au}, Thm.~2.34)}.\index{Sobolev spaces!Rellich-Kondrachov}\\
Consider a closed, oriented manifold $M$ and a vector bundle
$E\to M$. If $k_1>k_2$ and $w(k_1,p_1)>w(k_2,p_2)$, then the
inclusion map $L_{k_1}^{p_1}(M,E)\subset L_{k_2}^{p_2}(M,E)$ is a
compact operator. Moreover, the embedding in part {\rm (ii)} is
always compact.
\end{theorem}
The Sobolev embedding Theorem combined with the well-known
H\"{o}lder inequality leads to multiplication theorems for the
Sobolev spaces $L_k^p$.
\begin{prop}[Sobolev
multiplication]\label{sob:mult}\index{Sobolev spaces!Sobolev
multiplication} Consider a closed, oriented manifold $M$ and
vector bundles $E_1,E_2,F\to M$, endowed with a bilinear bundle
map $b:E_1\oplus E_2\to F$. Let $k_1,k_2,l\in\N$ with $k_1,k_2\ge
l$, and $p_1,p_2,q\in(1,\infty)$ such that $p_1,p_2\ge
q$.\footnote{Observe that this implies that $w(k_i,p_i)\ge
w(l,q)$.} If $w(k_1,p_1)+w(k_2,p_2)>w(l,q)$, then $b$ extends to
a bounded bilinear map
\[
b:L_{k_1}^{p_1}(M,E_1)\times L_{k_2}^{p_2}(M,E_2)\to L_l^q(M,F).
\]
\end{prop}
\begin{remark*}
Although this theorem is frequently used in the analysis of
nonlinear partial differential equations, it is difficult to find
a reference in the literature. In a slightly different form, the
theorem can be found in Palais \cite{Pal:GA}, Ch.~9. To be
self-contained, we shall present a proof.
\end{remark*}
\begin{proof}
\begin{steps}
\item\textit{The H\"older inequality
revisited:}\index{inequalities!Holder}\\ We recall that for any
$p_1',p_2',q'\in(1,\infty)$ such that $p_1',p_2'> q'$, the
H\"older inequality implies that there is a continuous
multiplication $L^{p_1'}\times L^{p_2'}\to L^{q'}$ provided that
$\lfrac{1}{p_1'} +\lfrac{1}{p_1'}=\lfrac{1}{q'}$. Moreover, on
compact manifolds there is a continuous embedding of $L^{q'}$ in
$L^q$ whenever $q'\ge q$. Hence, under the assumption that
$\lfrac{1}{p_1'}+\lfrac{1}{p_1'}\le \lfrac{1}{q}$ or,
equivalently, if $w(0,p_1')+w(0,p_2')\ge w(0,q)$, we obtain a
continuous multiplication $L^{p_1'}\times L^{p_2'}\to L^q$.

\item \textit{The case $l=0$:}\\
We have to show that there exists a continuous multiplication
\begin{equation}\tag{$*$}
L_{k_1}^{p_1}\times L_{k_2}^{p_2}\to L^q.
\end{equation}
For this we have to study different cases:
\begin{Cases}
\item\textit{There exists $i\in\{1,2\}$ such that
$w(k_i,p_i)>0$.} Without loss of generality, we may suppose that
$i=1$. Then $L_{k_1}^{p_1}$ embeds continuously in $C^0$.
Furthermore, we have an inclusion of $L_{k_2}^{p_2}$ in
$L^{p_2}$. Therefore, there is a bounded map
\[
L_{k_1}^{p_1}\times L_{k_2}^{p_2} \to L^{p_2}.
\]
Moreover, $L^{p_2}$ embeds continuously in $L^q$ for $p_2\ge q$.
This proves ($*$).
\item\textit{$w(k_i,p_i)\le 0$ for $i=1,2$.} Since
$w(k_1,p_1)+w(k_2,p_2)> w(0,q)$, this implies that $w(k_i,p_i)>
w(0,q)$ for $i=1,2$. Hence, we can find $p_i'\in (q,\infty)$ such
that
\[
w(k_i,p_i)\ge w(0,p_i') > w(0,q)\quad\text{and}\quad
w(0,p_1')+w(0,p_2')\ge w(0,q).
\]
Applying the considerations of \textit{Step 1}, we deduce
that there exists a continuous multiplication
\[
L^{p_1'}\times L^{p_2'}\to L^q.
\]
On the other hand, there are bounded inclusions
$L_{k_i}^{p_i}\subset L^{p_i'}$ which yields ($*$) in the case at
hand.
\end{Cases}

\item\textit{The general case:} \\
Let us now assume that $l$ is chosen arbitrarily. Fixing $m\le
l$, we deduce that for every section $u\in L_{k_1}^{p_1}(M,E_1)$
and every $v\in L_{k_2}^{p_2}(M,E_2)$,
\begin{multline*}\quad
\big|\nabla^m(b(u,v))\big|\le
\const\cdot\Big|\sum_{\overset{\scriptstyle k+i+j}{=
m}}(\nabla^kb)(\nabla^iu,\nabla^jv)\Big| \\
\le\const\cdot\sum_{\overset{\scriptstyle k+i+j}{=
m}}\big|(\nabla^kb)(\nabla^iu,\nabla^jv)\big|
\le\const\cdot\sum_{i+j\le
m}\big|\nabla^iu\big|\cdot\big|\nabla^jv\big|\;.
\end{multline*}
Observe that $\nabla^iu\in L^{p_1}_{k_1-i}$ and $\nabla^jv\in
L^{p_2}_{k_2-j}$. Whenever $i+j\le m$,
\[
w(k_1-i,p_1)+w(k_2-j,p_2)>w(l-m,q)\ge w(0,q).
\]
Hence, $\nabla^iu$ and $\nabla^jv$ satisfy the conditions of
\textit{Step 2}. Therefore,
\begin{align*}
\big\|\nabla^m(b(u,v))\big\|_{L^q}&\le \const\cdot\sum_{i+j\le m}
\big\| |\nabla^iu|\cdot|\nabla^jv| \big\|_{L^q} \\
&\le\const\cdot\sum_{i+j\le m} \big\|\nabla^iu
\big\|_{L_{k_1-i}^{p_1}}
\cdot\big\|\nabla^jv\big\|_{L_{k_2-j}^{p_2}}\\
&\le \const\cdot\sum_{i+j\le m} \|u\|_{L_{k_1}^{p_1}}\cdot
\|v\|_{L_{k_2}^{p_2}} \le \const\cdot\|u\|_{L_{k_1}^{p_1}}\cdot
\|v\|_{L_{k_2}^{p_2}}
\end{align*}
so that
$\|b(u,v)\|_{L_l^q}\le\const\cdot\|u\|_{L_{k_1}^{p_1}}\cdot
\|v\|_{L_{k_2}^{p_2}}$. This proves the assertion.\qedhere
\end{steps}
\end{proof}

\begin{example}\label{sob:mult:n=3}
Let $M$ be a compact and oriented Riemannian 3-manifold, and let
$b:E_1\otimes E_2\to F$ a bilinear bundle morphism between
arbitrary vector bundles over $M$. The above proposition shows
that $b$ induces a bounded bilinear map
\begin{equation}\label{sob:mult:n=3:1}
b:L^2_k(M,E_1)\times L^2_k(M,E_2)\longrightarrow L^2_k(M,F)
\end{equation}
whenever $k$ satisfies $2k-\lfrac{6}{2}> k-\lfrac{3}{2}$, that is,
whenever $k\ge 2$.

In particular, if $E$ is a bundle of algebras, then the Hilbert
spaces $L^2_k(M,E)$ are Banach algebras provided that $k\ge 2$.
This observation yields natural choices for the involved Sobolev
orders when we want to equip a group of gauge transformations with
a Sobolev structure. Furthermore, if the bundle of algebras acts
on some vector bundle $F$, then the proposition shows that for an
appropriate choice of $l$, the Banach algebra $L^2_k(M,E)$ acts
continuously on $L^2_l(M,F)$. For example, we have an associated
continuous multiplication
\[
b:L^2_2(M,E)\times L^2_1(M,F)\to L^2_1(M,F),
\]
i.e., $L^2_1(M,F)$ is an $L^2_2(M,E)$-module. This fact plays an
important role when the action of the group of gauge
transformations on the space of sections of a vector bundle is
modelled in the context of Sobolev spaces.\index{Sobolev spaces|)}
\end{example}

\section{Analytic properties of elliptic
operators}\label{diffop}\index{differential operators|(}

Let $E$ and $F$ denote Euclidean or Hermitian vector bundles over
a compact and oriented Riemannian manifold $M$. Furthermore,
suppose that $P:C^\infty(M,E)\to C^\infty(M,F)$ is a differential
operator of order $m$, i.e., $P$ is expressed in local coordinates
as
\[
P=\sum_{|\ga|\le m}a_\ga(x)\frac{\partial^{|\ga|}}{\partial
x_\ga},
\]
where $a_\ga$ are smooth matrix valued functions, while $\ga$ is a
multi index. Then the \emph{principal symbol} of $P$ is locally
defined by
\[
\gs_m(P)_{(x,\xi)}=\sum_{|\ga|=m}a_\ga(x)\xi^\ga.
\]
It gives rise to a bundle map $\gs_m(P):\pi^*E\to\pi^*F$, where
$\pi:T^*M\to M$ denotes the bundle projection of the cotangent
bundle.

Associated to every differential operator $P$ there is the
so-called \textit{formal adjoint}\footnote{In contrast to the
main part of this text, we now denote the formal adjoint with the
superscript $t$ instead of $*$. This is because we want to
distinguish it clearly from the functional analytic adjoint (see
below).} $P^t:C^\infty(M,F)\to C^\infty(M,E)$. It is defined by
the property that for every $u\in C^\infty(M,E)$ and every $v\in
C^\infty(M,F)$,\index{differential operators!formal adjoint}
\[
\int_M \scalar{Pu}{v}_F\,dv_g = \int_M \scalar{u}{P^tv}_E\,dv_g.
\]
Deriving an explicit formula via integration by parts yields that
the formal adjoint always exists and is again a differential
operator of order $m$. Moreover, it is uniquely characterized by
the above property. The principal symbols of $P$ and $P^t$ are
related by
\[
\gs_m(P^t)=(-1)^m\gs_m(P)^*,
\]
where $\gs_m(P)^*$ denotes the adjoint of the bundle map
$\gs_m(P)$ with respect to the induced metrics.\\

\noindent\textbf{Unbounded operators.} Working in the context of
Sobolev spaces is readily appreciated by the elementary fact that
for each $k\in\N$ and $p\in [1,\infty)$, a differential operator
$P$ induces bounded linear maps
\[
P_{k,p}:L^p_{k+m}(M,E)\to L^p_k(M,F)
\]
and
\[
P^t_{k,p}=(P^t)_{k,p}:L^p_{k+m}(M,F)\to L^p_k(M,E).
\]
Investigating the functional analytic properties of these maps
leads naturally to the theory of unbounded operators in Banach
spaces (cf. Kato \cite{K}, Sec.~III.5 and Sec.~V.3). We briefly
want to fix some notation in this context. For notational
convenience we will drop the reference to $p$ and simply write
$P_k$, the value of $p$ being understood from the context.

As a consequence of the above observation, $P$ induces an
unbounded operator $P_0:L^p(M,E)\supset L^p_m(M,E)\to L^p(M,F)$
with a dense domain $\dom(P_0):=L^p_m(M,E)$. Mutatis mutandis,
the same holds for $P^t$. If $p=2$, there is no reason to expect,
though, that the operator $P^t_0$ coincides with the
\emph{functional analytic} adjoint $(P_0)^*$. Recall that
\[
v\in\dom(P_0)^*:\Longleftrightarrow\;\exists_{w\in L^2(M,E)}
\forall_{u\in
\dom(P_0)}:\;\int_M\scalar{Pu}{v}_F\,dv_g=\int_M\scalar{u}{w}_E\,dv_g,
\]
and $(P_0)^* v:=w$. The operator $(P_0)^*$ is linear and densely
defined because $C^\infty(M,F)\subset \dom(P_0)^*$. If $E=F$ and
$P_0=(P_0)^*$ as unbounded operators in Hilbert space, then $P_0$
is called \emph{self-adjoint}. If we only have $P=P^t$, then we
call $P$ \emph{formally} self-adjoint.

\begin{remark*}
Note that the notion of the \emph{functional analytic} adjoint is
only well-defined if $p=2$ because only then $L^p$ is a Hilbert
space. In contrast to that, the \emph{formal} adjoint of a
differential operator makes perfect sense in the $L^p$ context
for all $p$ since it is a differential operator in its own
right.\\
\end{remark*}

\noindent\textbf{Elliptic operators}. The unbounded operators
induced by $P$ have surprising functional analytic properties if
$P$ is \textit{elliptic}, i.e., if its principal symbol
$\gs_m(P):\pi^*E\to \pi^*F$ is an isomorphism off the zero
section. We note that this implies that $E$ and $F$ have the same
rank. Due to the relation of their principal symbols, ellipticity
of $P$ implies that $P^t$ is elliptic as well. The following
theorem lies at the heart of the theory of elliptic operators.

\begin{theorem}\label{ell:thm}\index{differential operators!elliptic
estimate}\index{differential operators!elliptic regularity} Let
$P:C^{\infty}(M,E)\to C^{\infty}(M,F)$ be an elliptic operator.
Then the following holds.
\begin{enumerate}
\item ``Elliptic estimate'': For $u\in L^p_{k+m}(M,E)$,
\begin{equation}\label{ell:est}
\|u\|_{L_{k+m}^p}\le \const\cdot
\big(\|Pu\|_{L_k^p}+\|u\|_{L_k^p}\big),
\end{equation}
where $k\in\N$ and $p\in [1,\infty)$.
\item ``Elliptic regularity'': If $u\in L^p(M,E)$ satisfies
$Pu\in L^p_k(M,F)$ \emph{weakly}, i.e., if there exists $v\in
L^p_k(M,F)$ such that
\[
\forall_{w\in C^\infty(M,F)}:\ \int_M \scalar{u}{P^tw}_E\,dv_g =
\int_M \scalar{v}{w}_E\,dv_g,
\]
then $u\in L^p_{k+m}(M,E)$.

In particular, if $u\in L^p_{k+m}(M,E)$ satisfies $Pu=0$, then
$u$ is smooth. Hence, the kernels of the operators $P_k$ coincide
and consist solely of smooth sections.
\end{enumerate}
\end{theorem}
The proof of the elliptic estimate is quite technical and can be
found in many textbooks (cf. Aubin \cite{Au}, Sec.~3.6 and
references therein, or Gilbarg \& Trudinger \cite{GT}). Elliptic
regularity follows from the elliptic estimate via smoothing
arguments and can also be found in the textbooks cited above.

This theorem together with well-known results from functional
analysis are the key to an understanding of the analytic
properties of elliptic operators. We now start including some
proofs since our discussion in the main part of this thesis
relies on a detailed understanding of the next results.

\begin{prop}
For any $k\in\N$, $P_k$ is a closed unbounded operator
$L^p_k(M,E)\to L^p_k(M,F)$ with domain $L^p_{k+m}(M,E)$.
\end{prop}
\begin{proof}
We have to show that the graph of $P_k$, which is given by
\[
\bigsetdef{(u,P_k u)}{u\in\dom(P_k)}\subset L^p_k(M,E)\times
L^p_k(M,F),
\]
is a closed subspace. Suppose that $(u_n)$ is a sequence in
$\dom(P_k)$, i.e., in $L^p_{k+m}(M,E)$ such that $u_n\to u$ in
$L^p_k(M,E)$ and $Pu_n\to v$ in $L^p_k(M,F)$. We claim that $u\in
L^p_{k+m}(M,E)$ and $Pu=v$. The elliptic estimate
\[
\|u_n-u_{n'}\|_{L_{k+m}^p}\le \const\cdot
\big(\|Pu_n-Pu_{n'}\|_{L_k^p}+\|u_n-u_{n'}\|_{L^p_k}\big)
\]
shows that $(u_n)$ is a Cauchy sequence in $L^p_{k+m}(M,E)$ thus
converging in $L^p_{k+m}(M,E)$ to (the same limit point) $u$. By
continuity of
\[
P:L^p_{k+m}(M,E)\to L^p_k(M,F),
\]
we infer that $Pu=v$.
\end{proof}

\begin{prop}\label{funct:ana:adj}
If $p=2$, $P$ and its formal adjoint $P^t$ satisfy
$(P_0)^*=P^t_0$ and $(P^t_0)^*=P_0$.
\end{prop}
\begin{proof}
It suffices to show that $\dom(P_0)^*=\dom(P^t_0)\,(=L^2_m(M,F))$.
Suppose that $v\in\dom(P_0)^*$. Then there exists $w\in L^2(M,E)$
such that for any $u\in\dom(P_0)$,
\[
\int_M\scalar{Pu}{v}_F\,dv_g=\int_M\scalar{u}{w}_E\,dv_g.
\]
Since $C^\infty(M,E)\subset\dom(P_0)$, this implies that $w=P^t v$
weakly in $L^2(M,E)$. By elliptic regularity of $P^t$, we conclude
that $v\in L^2_m(M,F)$. Hence, $(P_0)^*\subset P^t_0$. The other
inclusion is obvious. As $P^t$ is also elliptic, the second
equality is proved in the same way.
\end{proof}

In addition to the elliptic estimate, we shall need the following
result.
\begin{prop}[Poincar\'e inequality]\index{inequalities!Poincare}
Let $p\ge 2$, and
\[
u\in (\ker P)^\perp\cap \dom(P_k),
\]
where we take the orthogonal complement in $L^2(M,E)$. Then
\begin{equation}\label{poinc}
\|u\|_{L_{k+m}^p}\le \const\cdot \|Pu\|_{L_k^p},
\end{equation}
\end{prop}
\begin{proof}
Invoking the elliptic estimate, we have to show that
\[
\|u\|_{L_k^p}\le\const\cdot\|Pu\|_{L_k^p}.
\]
Arguing by contradiction, we assume that there exists a sequence
$(u_n)$ in $L^p_{k+m}(M,E)$ such that all $u_n$ are
$L^2$-orthogonal to $\ker P$ and
\[
\|u_n\|_{L_k^p} > n\cdot \|Pu_n\|_{L_k^p}.
\]
Without loss of generality we may also demand that
$\|u_n\|_{L_k^p}=1$. Then the last inequality shows that
$\|Pu_n\|_{L_k^p}\to 0$, and the elliptic estimate imposes an
$L^p_{k+m}$-bound on $(u_n)$. We deduce from the
Rellich-Kondrachov Theorem that a subsequence of $(u_n)$
converges in $L^p_k(M,E)$ to, say, $u$. In particular,
$\|u\|_{L_k^p}=1$. As $Pu_n\to 0$ and $P_k$ is closed, we infer
that $u\in L^p_{k+m}(M,E)$ and $Pu=0$. On the other hand, all
$u_n$ are $L^2$-orthogonal to $\ker P$ which yields that the same
holds for $u$. Hence, $u\in\ker P\cap (\ker P)^\perp$ so that
necessarily $u=0$. This contradicts $\|u\|_{L_k^p}=1$.
\end{proof}

As an application of this result, we now deduce the Fredholm
property of elliptic operators.

\begin{theorem}\label{Fred:prop}\index{differential
operators!Hodge decomposition} Let $M$ be a closed, oriented
manifold and let $P:C^\infty(M,E)\to C^\infty(M,F)$ be an elliptic
differential operator. Then the following holds.
\begin{enumerate}
\item For every $k\in\N$ and $p\ge 2$, the operators $P_k$ and
$P^t_k$ are semi-Fredholm, i.e., they have closed ranges and
finite dimensional kernels.
\item ``Hodge decomposition'': There is an $L^2$-orthogonal
decomposition
\begin{equation}\label{Hodge:dec}
L^p_k(M,F)=\im P_k \oplus \ker (P^t).
\end{equation}
\end{enumerate}
In particular, $P_k$ is Fredholm, i.e., its kernel and cokernel
are finite dimensional. Moreover, the Fredholm index
\[
\ind P_k:= \dim (\ker P_k) - \dim (\coker P_k)
\]
neither depends on $k$ nor on $p$.
\end{theorem}

\begin{proof}
Let $(v_n)$ be a sequence in the image of $P_k$ which converges
in $L^p_k(M,F)$ and let $v$ denote the limit point. We choose a
sequence $(u_n)$ in $L^p_{k+m}(M,E)$ such that $Pu_n=v_n$. Since
$\ker P\subset C^\infty(M,E)$ and $L^p_k\subset L^2$ (because
$p\ge 2$), we may assume that all $u_n$ are $L^2$-orthogonal to
$\ker P$. We then infer from the Poincar\'e inequality
\eqref{poinc} that $(u_n)$ is a Cauchy sequence in $L^p_{k+m}$
hence converging to, say, $u\in L^p_{k+m}(M,E)$. Since
$P_k:L_{k+m}^p(M,E)\to L_k^p(M,F)$ is continuous, $Pu=v$.
Therefore, $\im P_k$ is closed in $L^p_k(M,F)$.

To prove the finite dimensionality of $\ker P$, we want to use the
well-known fact that a Banach space is finite dimensional if and
only if the unit sphere is sequentially compact. Hence, suppose
$(u_n)$ is a sequence in $\ker P$ with $\|u_n\|_{L_{k+m}^p}=1$.
As a consequence of the Rellich-Kondrachov Theorem, $(u_n)$
contains a subsequence which converges in $L^p_k(M,E)$ to some
limit point $u$. Since $Pu_n=0$, the elliptic estimate yields
that this subsequence is also a Cauchy sequence in
$L^2_{k+m}(M,E)$ thus also converging to $u$ with respect to the
$L^2_{k+m}$-topology. Continuity of $P_k$ implies that $Pu=0$ so
that $u\in \ker P$. Clearly, $\|u\|_{L_{k+m}^p}=1$ and this shows
that the unit sphere in $\ker P$ is sequentially compact. In the
same way, replacing $P$ with $P^t$, we get that $P^t$ is a
semi-Fredholm operator. Thus we have proved (i).

Concerning (ii), let us first content ourselves to the case $k=0$
and $p=2$. Since $P_0$ has a closed range in $L^2(M,F)$, we have
an $L^2$-orthogonal decomposition
\[
L^2(M,F)=\im P_0 \oplus \ker (P_0)^*=\im P_0 \oplus \ker P^t_0.
\]
Here, we have used that $(P_0)^*$ and $P^t_0$ coincide. Since
$P^t$ is elliptic, $\ker P^t_0=\ker P^t$. Hence, the assertion is
proved in the case at hand.

To obtain the general case, we intersect the above decomposition
with $L^p_k(M,F)$. Since $L^p_k\subset L^2$ and $\ker P^t\in
C^\infty(M,F)$, this yields an $L^2$-orthogonal decomposition
\[
L^p_k(M,F)= \big(\im P_0\cap L^p_k(M,F)\big) \oplus \ker P^t.
\]
Elliptic regularity of $P$ implies that $\im P_0\cap
L^p_k(M,F)=\im P_k$ which establishes (ii).

From the Hodge decomposition we deduce that $\coker P_k\cong\ker
P^t$ where the right hand side is finite dimensional and
independent of $k$ and $p$. Moreover, we have already observed
that $\ker P$ is finite dimensional and neither depends on $k$
nor $p$. This implies the assertion about the independence of the
Fredholm index.\qedhere \\
\end{proof}

\noindent\textbf{Discrete spectrum.} Let $P:C^\infty(M,E)\to
C^\infty(M,E)$ be a formally self-adjoint, elliptic differential
operator. Then, according to Proposition \ref{funct:ana:adj}, the
associated operator $P_0$ in $L^2(M,E)$ is self-adjoint. This
implies that its spectrum,
\[
\spec(P_0):=\bigsetdef{\gl\in\C}{P_0-\gl:L^2_m(M,E)\to L^2(M,E)
\text{ is not invertible}},
\]
is a (closed) subset of $\R$. As a consequence of the Rellich
Lemma, any element $\gl\in\C\setminus\spec(P_0)$ defines a compact
operator $(P_0-\gl)^{-1}:L^2(M,E)\to L^2_m(M,E)\subset L^2(M,E)$.
Therefore, $P_0$ is said to have \textit{compact resolvent}. From
the spectral theory of compact operators, it is easy to deduce
that the spectrum of a self-adjoint operator having compact
resolvent consists solely of discrete eigenvalues of finite
multiplicity, which form an unbounded subset of $\R$ (cf. Kato
\cite{K}, Thm.~III.6.29). One sometimes refers to this property by
calling $P_0$ an operator with \textit{discrete spectrum}. In
addition to that, elliptic regularity of $P_0$ implies that the
corresponding eigenvectors are smooth.\\

\noindent\textbf{Injectively elliptic
operators.}\index{differential operators!injectively elliptic}
The requirement that $\gs_m(P)$ is an isomorphism naturally splits
into two parts, namely injectivity and surjectivity of the
symbol. A differential operator $P:C^\infty(M,E)\to C^\infty(M,F)$
of order $m$ is called \textit{injectively elliptic} if its
principal symbol is injective off the zero section. The elliptic
estimate \eqref{ell:est} also holds for differential operators of
this kind:

\begin{theorem}\label{inj:ell:thm}
Let $P:C^{\infty}(M,E)\to C^{\infty}(M,F)$ be an injectively
elliptic operator of order $m$. Then the following holds.
\begin{enumerate}
\item ``Elliptic estimate'': For all $k\in\N$ and $u\in
L^2_{k+m}(M,E)$,
\begin{equation*}
\|u\|_{L_{k+m}^2}\le \const\cdot
\big(\|Pu\|_{L_k^2}+\|u\|_{L_k^2}\big).
\end{equation*}
\item ``Elliptic regularity'': If $u\in L^2(M,E)$ satisfies
$Pu\in L^2_k(M,F)$ \emph{weakly}, then $u\in L^2_{k+m}(M,E)$.
\end{enumerate}
\end{theorem}

\begin{proof}[Sketch of proof] We shall only give a proof in the
case $k\ge m$, since then the assertions are immediate
consequences of Theorem \ref{ell:thm} applied to the elliptic
operator
\[
P^tP: C^\infty(M,E)\to C^\infty(M,E)
\]
of order $2m$. Since $P^t_{k-m}:L^2_k(M,F)\to L^2_{k-m}(M,E)$ is
bounded, the elliptic estimate for $P^tP$,
\[
\|u\|_{L_{k+m}^2}\le
\const\cdot\big(\|P^tPu\|_{L_{k-m}^2}+\|u\|_{L_{k-m}^2}\big),
\]
implies that also
\[
\|u\|_{L_{k+m}^2}\le \const\cdot\big(\|Pu\|_{L_k^2} +
\|u\|_{L_{k-m}^2}\big).
\]
Since $\|.\|_{L_{k-m}^2}\le \|.\|_{L_k^2}$, we obtain the elliptic
estimate for $P$ and thus (i). Similarly, we get (ii) in the
following way: Let $u\in L^2(M,E)$ and $Pu=v$ weakly with $v\in
L_k^2(M,F)$, $k\ge m$. Then it is immediate that
\[
P^tP u = P^t v\in L_{k-m}^2(M, E)
\]
weakly. Therefore, $u\in L_{k-m+2m}^2=L_{k+m}^2$. The proof in
the general case has to be carried out by introducing Sobolev
spaces of negative order as the dual spaces---endowed with the
operator norm---to the corresponding spaces of positive order.
The elliptic estimate continues to hold for these spaces, and the
above arguments carry over to this setting.
\end{proof}

\begin{remark*}
Although we have only given a rigid proof of the above theorem in
the case $k\ge m$, we shall use it for all $k\ge 0$. Yet, we have
formulated the theorem only for $p=2$ since otherwise, the above
proof does not easily carry over to the case $k<m$.
\end{remark*}

One now establishes functional analytic properties much as
before, taking care, however, that all arguments involving
ellipticity of $P$'s formal adjoint are no longer valid in the
context of injectively elliptic operators. The following
Proposition summarizes the results we need.

\begin{prop}\label{inj:ell}
Let $P:C^\infty(M,E)\to C^\infty(M,F)$ be an injectively elliptic
operator of order $m$ and let $k\in\N$. Then
\begin{enumerate}
\item The unbounded operator $P_k$ is a closed unbounded
semi-Fredholm operator
\[
P_k:L_k^2(M,E)\supset L_{k+m}^2(M,E)\to L_k^2(M,F)
\]
with finite dimensional kernel.
\item The formal adjoint of $P$ satisfies $(P^t_0)^* = P_0$.
\item There is an $L^2$-orthogonal decomposition
\begin{equation}\label{Hodge:dec:inj}
L^2_m(M,F)=\im P_m \oplus \ker P^t_0.
\end{equation}
\end{enumerate}
\end{prop}

\begin{proof}
Part (i) and (ii) are proved exactly as before. For this note that
the Poincar\'e inequality \eqref{poinc} also holds for injectively
elliptic operators. Regarding (iii), we first have an
$L^2$-orthogonal decomposition
\begin{equation}\label{Hodge:dec:inj:2}
L^2(M,F)=\im P_0 \oplus \ker (P_0)^*,
\end{equation}
for which we invoke that $P_0$ has closed range. Suppose $v\in
L^2_m(M,F)$. As an element of $L^2(M,F)$, we may decompose $v$
according to the above as $v=Pu + w$, with $u\in L^2_m(M,E)$ and
$w\in\ker P^*_0$. Applying $P^t$ to this equation, we deduce that
$P^tv=P^tPu$ weakly in $L^2(M,E)$. Elliptic regularity of $P^tP$
then guarantees that $u\in L^2_{2m}(M,E)$ so that in particular,
$w=v-Pu\in L^2_m(M,F)$. As $\ker P^*_0\cap L^2_m(M,F)=\ker
P^t_0$, the assertion follows.
\end{proof}

In contrast to the elliptic case, $\ker (P_0)^*$ is in general
neither finite dimensional nor does it coincide with the kernel
of the formal adjoint. However, the above proposition implies

\begin{cor}\label{L^2:compl}
Let $P:C^\infty(M,E)\to C^\infty(M,F)$ be an injectively elliptic
operator of order $m$. Then
\[
\im P_0 = \overline{\im P_m}^{L^2}\quad\text{ and }\quad \ker
(P_0)^*= \overline{\ker P^t_0}^{L^2}.
\]
\end{cor}
\begin{proof}
Since $\im P_m\subset \im P_0$ and $\ker P^t_0\subset \ker
(P_0)^*$,
\begin{equation}\label{L^2:compl:1}
\overline{\im P_m}^{L^2}\subset \im P_0 \quad\text{ and }\quad
\overline{\ker P^t_0}^{L^2}\subset \ker (P_0)^*
\end{equation}
because the subspaces $\im P_0$ and $\ker (P_0)^*$ are closed in
$L^2(M,F)$. As the subspace $L_m^2(M,F)$ is dense in $L^2(M,F)$,
we deduce from \eqref{Hodge:dec:inj:2} and \eqref{Hodge:dec:inj}
that
\[
\im P_0 \oplus \ker (P_0)^*= \overline{\im P_m}^{L^2} \oplus
\overline{\ker P^t_0}^{L^2}
\]
as $L^2$-orthogonal decompositions. Together with
\eqref{L^2:compl:1}, this implies the assertion.\qedhere\\
\end{proof}

\noindent\textbf{Non-smooth coefficients.} In gauge theory one
usually works with nonlinear partial differential equations. In
view of the implicit function theorem, it is promising to model
the nonlinear partial differential operator on a suitable Sobolev
space and study its differential in order to gather information
about the set of solutions. This linearization fits into the
context of linear differential operators we have described above,
though, with a slight modification: The point in the Sobolev
space, where we are linearizing the operator enters the
differential and we usually obtain a differential operator with
Sobolev coefficients. However, the actual generalization we have
to make is only a minor one. The underlying observation is that
in the situation we shall encounter, the nonlinear part of the
partial differential equation is of order 0. This implies that
linearizing the equation at different points yields operators
differing only by a compact operator. It is then a
well-established fact that the functional analytic properties of
the two linear operators are essentially the same.

\begin{theorem}\label{relative:comp}\index{differential
operators!relatively compact perturbation} Let $P:C^\infty(M,E)\to
C^\infty(M,F)$ be an injective elliptic operator of order $m$.
Moreover, let $T:L^2_m(M,E)\to L^2(M,E)$ be a compact\footnote{In
this situation, $T$ is called \textit{relatively compact} with
respect to $P$.} operator and consider $P+T$ with
$\dom(P+T)=L^2_m(M,E)$. Then the following holds
\begin{enumerate}
\item The unbounded operator $P+T$ is closed and semi-Fredholm
operator with finite dimensional kernel.
\item $L^2(M,F)=\im (P+T) \oplus \ker (P+T)^*$.
\item Suppose that $F=E$ and that $P$ is formally self-adjoint, and
$T$ is symmetric with respect to the $L^2$ scalar product. Then
$P+T$ is a self-adjoint Fredholm operator in $L^2(M,E)$ with
compact resolvent.
\end{enumerate}
\end{theorem}
Part (i) and (ii) follow from Proposition \ref{inj:ell} and
Theorem V.5.26 of \cite{K} about relatively compact perturbations
of semi-Fredholm operators. Part (iii) is an immediate consequence
of Theorem 9.9 in \cite{Wei:LO}, see also Sec.~V.4 in \cite{K}.
\index{differential operators|)}

\cleardoublepage

%% file: app_B.tex
\chapter{The determinant line bundle}\label{app:det}
\setcounter{section}{1}

In this appendix we present a version of how to construct a
canonical line bundle over the space of Fredholm operators. Since
we shall not need this notion in the context of Banach spaces, we
restrict ourselves immediately to Hilbert spaces although the
situation is more or less the same. For a discussion of different
possibilities to construct the determinant bundle, we refer to
\cite{BB}, Ch.~3.

Let $\sF(H_1,H_2)$ denote the set of bounded Fredholm operators
between two separable $\K$-Hilbert spaces $H_1$ and $H_2$.
\begin{dfn}\index{<@$\sF(H_1,H_2)$, Fredholm op.}
The {\em determinant line} of $T\in \sF(H_1,H_2)$ is the vector
space
\[
\det T:= \det(\ker T)\otimes \big(\det(\coker T)\big)^*\;,
\]
where $(\ldots)^*$ denotes the dual space. Recall that for each
$n$-dimensional $\K$-vector space $V$, the space $\det(V)$ is
defined as the top exterior power $\gL^nV$. In
particular, $\det\{0\}$ is the underlying scalar field $\K$.\\
\end{dfn}

\noindent\textbf{The space of Fredholm operators.} We equip
$\sF(H_1,H_2)$ with the topology induced by the operator norm on
$\sL(H_1,H_2)$, the latter denoting the Banach space of bounded
liner maps. Let $\sU\subset \sF(H_1,H_2)$ be a connected open
subset. It is well-known that the map $T\mapsto \ind(T)$ is
constant on $\sU$. Assume for a moment that the assignment
$T\mapsto \dim(\ker T)$ is also constant on $\sU$. Under this
assumption, defining $(Ker)_T:= \ker T$ and $(Coker)_T:=\coker T$
for every $T\in\sU$ yields vector bundles $Ker \to \sU$ and $Coker
\to \sU$. We can thus form the line bundle
\[
Det:=\det(Ker)\otimes \big(\det(Coker)\big)^* \longrightarrow
\sU\;.
\]
However, there is no immediate way of endowing the collection
$\bigcup_T\det T$ with the structure of a line bundle if
$\dim(\ker T)$ varies with $T$. To achieve this, we introduce the
following concept.

\begin{dfn}\index{families of operators!stabilizer}
Let $\sU\subset \sF(H_1,H_2)$ and let $K:\sU\to \sL(V,H_2)$ be a
continuous map, where $V$ is a finite dimensional Hilbert space.
If the operator
\[
T_K: H_1\oplus V\to H_2,\quad (e,v)\mapsto Te + K(T)v,
\]
is surjective for every $T\in \sU$, we call $K$ a
\emph{stabilizer over $\sU$}.
\end{dfn}

Let $K:\sU\to\sL(V,H_2)$ be a stabilizer over some open subset
$\sU \subset \sF_n$, the latter denoting the component of Fredholm
operators of index $n$. Then for every $T\in \sU$, the operator
$T_K$ is surjective and Fredholm with index equal to $n + \dim V$.
Therefore, the collection $\bigcup_{T\in\sU}\ker T_K$ forms a
well-defined vector bundle which we denote by $Ker_K\to\sU$. We
can then form the line bundle
\[
\det(Ker_K)\longrightarrow \sU.
\]
We shall see in Proposition \ref{can:isom} below, that there
exists a natural isomorphism
\begin{equation}\label{can:isom:prev}
\det T \cong \det(\ker T_K)\otimes (\det V)^*
\end{equation}
for any $T\in \sU$. Thus, the idea is to define the structure of a
line bundle on $\bigcup_{T\in \sU}\det T$ via the above
isomorphisms. Before doing so, let us first recall the well-known
fact that stabilizers exist in abundance.

\begin{lemma}\label{Stab:exists}
For each $T_0\in \sF(H_1,H_2)$, there exist a finite dimensional
subspace $V\subset H_2$ and an open neighbourhood $\sU$ of $T_0$
such that the constant map $K:=(V\hookrightarrow H_2)$ is a
stabilizer over $\sU$.
\end{lemma}

\begin{proof}
Let $V:=\im T_0^\perp$. Since $T_0$ is Fredholm, $V$ is a finite
dimensional subspace of $H_2$. We define $W:=(\ker
T_0)^\perp\subset H_1$ and $K:=(V\hookrightarrow H_2)$. Then the
map
\[
S:\sF(H_1,H_2)\to \sL(W\oplus V,H_2),\quad S(T):=T_K|_{W\oplus V}
\]
is continuous. Moreover, the construction is made in such a way
that $S(T_0)$ is invertible. It is well-known that the set of
invertible elements is open in $\sL(W\oplus V,H_2)$. Hence, there
exists an open neighbourhood $\sU$ of $T_0$ such that for every
$T\in\sU$, the operator $S(T):W\oplus V\to H_2$ is invertible. In
particular, the operator $T_K:H_1\oplus V\to H_2$ is
surjective.\qedhere\\
\end{proof}

\noindent\textbf{The determinant line of a Fredholm operator.}
Our next aim is to describe the isomorphism
\eqref{can:isom:prev}. This requires some linear algebra so that
we shall restrict to the case of a single Fredholm operator
$T:H_1\to H_2$ in order to simplify notation.

\begin{lemma}\label{det}
Let $K:V\to H_2$ be a stabilizer of $T\in\sF(H_1,H_2)$, and let
$P_V:=\Proj_V:H_1\oplus V\to V$, and $F:=\Proj_{\coker T}\circ K$,
where $\Proj$ denotes the orthogonal projection. Then the sequence
\begin{equation}\label{det:seq}
\begin{CD}
0 @>>> \ker T @>>> \ker T_K  @>{P_V}>> V @>{F}>> \coker T @>>> 0
\end{CD}
\end{equation}
is exact. Hence, there exists a natural isomorphism
\begin{equation}\label{det:isom}
\det(\ker T) \otimes \det V\; \cong\; \det(\ker
T_K)\otimes\det(\coker T).
\end{equation}
\end{lemma}
\begin{proof}
The existence of the isomorphism \eqref{det:isom} follows from
exactness of \eqref{det:seq} and Lemma \ref{det:gen} below.
Hence, we are left to show exactness of the sequence. Since this
is fairly trivial, we only mention that $P_V(\ker T_K) = \ker F$
because
\[
v\in P_V(\ker T_K)\;\Longleftrightarrow\; Kv\in \im T
\;\Longleftrightarrow\; F v = \Proj_{\coker T}\circ Kv =
0.\qedhere
\]
\end{proof}

\begin{lemma}\label{det:gen}
Let
\begin{equation}\label{det:gen:seq}
\begin{CD}
0 @>>> V_n @>{f_n}>> V_{n-1}  @>{f_{n-1}}>> \ldots @>{f_1}>> V_0
@>>> 0
\end{CD}
\end{equation}
be an exact sequence of finite dimensional vector spaces. Then
there exists a natural isomorphism
\begin{equation}\label{det:gen:isom}
\bigotimes_{2k \le n} \det V_{n-2k} \cong \bigotimes_{2k \le n-1}
\det V_{n-1-2k}
\end{equation}
\end{lemma}
\begin{proof}[Sketch of proof]
Let $\eta_n\otimes \eta_{n-2}\otimes \ldots$ be an element of the
left-hand side of \eqref{det:gen:isom}. For each $i$ let $c_i:=
\dim(V_i /\ker f_i)$. It follows from exactness of
\eqref{det:gen:seq} that there exist $\go_i\in \gL^{c_i}V_i$ such
that
\begin{equation*}
\eta_n\otimes \eta_{n-2}\otimes \ldots\; =\; \go_n \otimes
(f(\go_{n-1})\wedge \go_{n-2})\otimes \ldots
\end{equation*}
Then
\[
\go_n \otimes (f(\go_{n-1})\wedge \go_{n-2})\otimes \ldots
\longmapsto (f(\go_n)\wedge \go_{n-1})\otimes (f(\go_{n-2})\wedge
\go_{n-3})\otimes \ldots
\]
gives the desired isomorphism. It is routine to check that this
isomorphism does not depend on the particular choices of
$\go_n,\ldots,\go_0$.
\end{proof}

\begin{dfn}
We call a basis of the form
\[
\go_n \otimes (f(\go_{n-1})\wedge \go_{n-2})\otimes \ldots \in
\bigotimes_{2k \le n} \det V_{n-2k}
\]
an \emph{adapted basis} associated to the sequence
\eqref{det:gen:seq}.
\end{dfn}

\begin{example*}
Let us illustrate the independence of the adapted basis in the
case $n=2$, i.e., for a short exact sequence
\[
\begin{CD}
0 @>>> V_2  @>{f_2}>> V_1  @>{f_1}>> V_0 @>>> 0.
\end{CD}
\]
Let $\eta_2\otimes \eta_0\in \det V_2\otimes \det V_0$ and choose
$\go_i, \go_i' \in \gL^\bullet V_i$ such that
\[
\eta_2\otimes \eta_0 = \go_2\otimes f_1(\go_1) = \go_2'\otimes
f_1(\go_1').
\]
If this expression is nonzero---what we will assume
henceforth---there exists $\gl\in\K^*$ such that $\go_2'=\gl\cdot
\go_2$ and $f_1(\go_1')=\gl^{-1}\cdot f_1(\go_1')$. Note that
this does \emph{not} imply that $\go_1'=\gl^{-1}\cdot\go_1$ since
$\go_1$ and $\go_1'$ might span determinant lines of different
complements to $\ker f_1$. Yet, writing
$\go_2=u_1\wedge\ldots\wedge u_{c_2}$,
$\go_1=v_1\wedge\ldots\wedge v_{c_1}$, and
$\go_1'=v_1'\wedge\ldots\wedge v_{c_1}'$, we obtain two ordered
bases of $V_1$,
\[
\big(f_2(u_1),\ldots, f_2(u_{c_2}), v_1,\ldots, v_{c_1}\big)\quad
\text{and}\quad \big(f_2(u_1),\ldots, f_2(u_{c_2}), v_1',\ldots,
v_{c_1}'\big).
\]
The corresponding transition matrix has the form
\[
\begin{pmatrix} 1 & 0\\ * & A \end{pmatrix},
\]
and thus, $f_2(\go_2)\wedge \go_1'=\det(A)\cdot f_2(\go_2)\wedge
\go_1$. Moreover, since $f_1(\go_1')=\gl^{-1}\cdot f_1(\go_1)$ and
$f_2(u_i)\in \ker f_1$, it follows that $\det A=\gl^{-1}$. Then
\begin{align*}
f_2(\go_2')\wedge \go_1' &= f_2(\gl\cdot \go_2)\wedge \go_1' =
\gl\cdot f_2(\go_2)\wedge \go_1'\\ &= \gl\cdot (\det A)\cdot
f_2(\go_2)\wedge \go_1 = f_2(\go_2)\wedge \go_1.\\
\end{align*}
\end{example*}

\noindent\textbf{Sign conventions.} Recall that for any
1-dimensional $\K$-vector spaces $L$ and $L'$, there are canonical
isomorphisms
\begin{equation}\label{can:isoms}
\begin{split}
 L\otimes L'\; \cong\; L'\otimes L, \quad u\otimes u'\longmapsto
 u'\otimes u  \\
 L^*\otimes L\;\cong\; \K,\quad u^*\otimes v\longmapsto u^*[v].
\end{split}
\end{equation}
With these rules it is easy to get \eqref{can:isom:prev} from
\eqref{det:isom}. However, there are some subtleties concerning
signs. In Appendix \ref{app:SF&OT} we want to extract a sign, the
so-called \emph{orientation transport}, from the determinant line
bundle of a path of self-adjoint Fredholm operators. For this
reason we have to be very careful about how to deal with signs. We
shall use the Knudsen-Mumford sign conventions \cite{KnuMum:Det}
which we recall now (see also Nicolaescu \cite{Nic:Tor},
Sec.~1.2).

Let $V$ be a finite dimensional $\K$-vector space. Then the
determinant line $\det V$ carries a natural \emph{weight}, namely
the natural number $\dim V$. The general concept lying behind this
is the following:

\begin{dfn}
Let $L$ be a 1-dimensional vector space, and let $w\in\Z$. Then
the tuple $(L,w)$ is called a \emph{weighted line}. We define
\[
(L,w)^* := (L^*,-w),
\]
and, if $(L',w')$ is another weighted line,
\[
(L,w)\otimes (L',w'):= (L\otimes L', w+w').
\]
From now on we consider a determinant line $\det V$ as a weighted
line with weight $\dim V$.
\end{dfn}

In the context of weighted lines $(L,w)$ and $(L',w')$, the
canonical isomorphisms \eqref{can:isoms} are altered in the
following way:
\begin{equation}\label{sign:conv}
\begin{split}
 L\otimes L'\; \cong\; L'\otimes L, \quad u\otimes u'\longmapsto
(-1)^{ww'} u'\otimes u \\
 L^*\otimes L\;\cong\; \K,\quad u^*\otimes v\longmapsto (-1)^{\frac
{w(w-1)}2} u^*[v].
\end{split}
\end{equation}
These are the so-called \emph{Knudsen-Mumford sign conventions}.
For the remaining part of this appendix it is understood that we
are using \eqref{sign:conv}.

\begin{example*}
The application to Lemma \ref{det} in mind, we shall have a closer
look at a four term exact sequence
\[
\begin{CD}
0 @>>> V_3 @>{f_3}>> V_2  @>{f_2}>> V_1  @>{f_1}>> V_0 @>>> 0.
\end{CD}
\]
Then the isomorphism $\det V_3\otimes \det V_1\cong\det
V_2\otimes \det V_0$ together with \eqref{sign:conv} shows that
\[
\begin{split}
\det V_3\otimes (\det V_0)^*& \cong \det V_3\otimes \det
V_1\otimes (\det V_1)^*\otimes (\det V_0)^*\\ &\cong \det V_2
\otimes \det V_0 \otimes (\det V_1)^*\otimes (\det V_0)^*\\
&\cong \det V_2\otimes (\det V_1)^* \otimes \det V_0\otimes (\det
V_0)^* \\ &\cong \det V_2\otimes (\det V_1)^*
\end{split}
\]
Taking the sign conventions into account, one checks that, in
terms of an adapted basis
\[
\go_3\otimes (f_2(\go_2)\wedge \go_1) \in \det V_3\otimes \det
V_1,
\]
the induced isomorphism $\det V_3\otimes (\det V_0)^*\cong \det
V_2\otimes (\det V_1)^*$ is given by
\[
\go_3\otimes (f_1(\go_1))^* \longmapsto
(-1)^{\frac{(n_0+n_1)(n_0+n_1+1)}2}(f_3(\go_3)\wedge \go_2)
\otimes (f_2(\go_2)\wedge \go_1)^*,
\]
where $n_i:=\dim V_i$, and $(\ldots)^*$ denotes the operation of
taking the dual.
\end{example*}

Translating this example to the situation of Lemma \ref{det}, we
obtain the result we were aiming at in \eqref{can:isom:prev}:

\begin{prop}\label{can:isom}
Let $T\in\sF(H_1,H_2)$ and let $K:V\to H_2$ be a stabilizer of
$T$. Define $F:=\Proj_{\coker T}\circ K$. Then there is a natural
isomorphism
\[
\gF_K:\det(\ker T)\otimes \big(\det(\coker T)\big)^*\to \det(\ker
T_K)\otimes (\det V)^*.
\]
This isomorphism is uniquely defined in the following way. If
\[
\xi \otimes (P_V(\eta)\wedge \go) \in \det(\ker T)\otimes (\det
V),
\]
is an adapted basis associated to the sequence \eqref{det:seq},
then
\begin{equation}\label{can:isom:expl}
\gF_K\big(\xi\otimes (F(\go))^*\big)=
(-1)^{\frac{(n_0+n_1)(n_0+n_1+1)}2} (\xi\wedge \eta)\otimes
(P_V(\eta)\wedge \go)^*,
\end{equation}
where $n_0:=\dim(\coker T)$ and $n_1:=\dim V$.
\end{prop}

The advantage of regarding determinant lines as weighted lines
together with the above sign conventions is that the isomorphism
\eqref{det:gen:isom} behaves functorial with respect to morphisms
of exact sequences. Instead of going into further detail in the
abstract setting, we restrict to the application to determinant
lines of Fredholm operators.

Let $K_1:V_1\to H_2$ and $K_2:V_2\to H_2$ be stabilizers of
$T\in\sF(H_1,H_2)$. Then $K_1+K_2:V_1\oplus V_2\to H_2$ is also a
stabilizer of $T$. Hence, Proposition \ref{can:isom} yields
isomorphisms
\begin{align*}
\gF_{K_1+K_2}:\det(\ker T)\otimes (\det(\coker T))^* &\to
\det(\ker
T_{K_1+K_2})\otimes (\det(V_1\oplus V_2))^*\\
\intertext{and, for $i=1,2$,} \gF_{K_i}:\det(\ker T)\otimes
(\det(\coker T))^* &\to \det(\ker T_{K_i}) \otimes (\det V_i)^*.
\end{align*}
The next result shows that these isomorphisms are naturally
related.

\begin{prop}\label{can:isom:rel}
Let $K_1:V_1\to H_2$ and $K_2:V_2\to H_2$ be stabilizers of
$T\in\sF(H_1,H_2)$. Then, for $i\in\{1,2\}$, there exists a
natural isomorphisms
\[
\gF_i: \det(\ker T_{K_i}) \otimes (\det V_i)^*\to \det(\ker
T_{K_1+K_2})\otimes (\det(V_1\oplus V_2))^*
\]
such that the following diagram commutes.
\begin{equation}\label{can:isom:rel:diag}
\begindc[1]
\obj(74,1){$\det(\ker T)\otimes \det(\coker T)^*$}[T]
\obj(74,150){$\det (\ker T_{K_1+K_2})\otimes\det(V_1\oplus
V_2)^*$}[K1K2] \obj(1,75){$\det(\ker T_{K_1}) \otimes (\det
V_1)^*$}[K1] \obj(148,75){$\det(\ker T_{K_2}) \otimes (\det
V_2)^*$}[K2] \mor{T}{K1}{$\gF_{K_1}$}
\mor{T}{K2}{$\gF_{K_2}$}[\atright,\solidarrow]
\mor{K1}{K1K2}{$\gF_1$}
\mor{K2}{K1K2}{$\gF_2$}[\atright,\solidarrow] \mor{T}{K1K2}{}
\enddc
\end{equation}
\end{prop}

\begin{proof}
We show the assertion for the left triangle of
\eqref{can:isom:rel:diag}. It is easy to check that, with
$F_i:=\Proj_{\coker T}\circ K_i$,
\[
\begin{CD}
    @.  @.   0   @.    0     @.   @.\\
  @.  @.   @AAA     @AAA     @.      @.\\
    @.     0    @>>> V_2     @>{\id}>>     V_2    @>>>    0    @.\\
  @.     @AAA  @A{P_{V_2}}AA      @A{P_{V_2}}AA     @AAA   @.\\
  0 @>>> \ker T @>>> \ker T_{K_1+K_2}  @>{P_{V_1\oplus V_2}}>>
  V_1\oplus V_2 @>{F_1+F_2}>> \coker T @>>> 0 \\
  @.     @A{\id}AA   @AAA @AAA @A{\id}AA  @.\\
  0 @>>> \ker T @>>> \ker T_{K_1}   @>{P_{V_1}}>>  V_1  @>{F_1}>>
  \coker T @>>> 0 \\
  @.     @AAA        @AAA         @AAA     @AAA   @.\\
    @.     0    @.   0     @.    0    @.   0   @.\\
\end{CD}
\]
is a commutative diagram with exact rows and columns. Note that
the first vertical short exact sequence stems from the fact that
$T_{K_1+K_2}$ is a stabilizer of $T_{K_1}$. It yields that
\begin{equation}\label{can:isom:rel:1}
\det(\ker T_{K_1})\otimes \det V_2 \cong \det(\ker T_{K_1+K_2}).
\end{equation}
Moreover, we get from the second vertical short exact sequence
that
\begin{equation}\label{can:isom:rel:2}
\det V_1\otimes \det V_2 \cong \det(V_1\oplus V_2).
\end{equation}
Now, we deduce from these two isomorphisms that the isomorphism
$\gF_1$ we are looking for is naturally given by the composition
\[
\begin{split}
\det(\ker T_{K_1}) &\otimes (\det V_1)^* \\
&\cong \det(\ker T_{K_1}) \otimes (\det V_1)^* \otimes \det
V_2\otimes (\det V_2)^*\\
&\cong \det(\ker T_{K_1})\otimes \det V_2 \otimes (\det
V_1)^*\otimes (\det V_2)^*\\
&\cong \det(\ker T_{K_1+K_2}) \otimes \big(\det(V_1\oplus
V_2)\big)^*.
\end{split}
\]
Note that in the last line we have employed that $(L_1\otimes
L_2)^*\cong L_1^*\otimes L_2^*$ for any two weighted lines $L_1$
and $L_2$.

To give an explicit description of $\gF_1$, consider adapted
bases,
\[
\eta_1\otimes P_{V_2}(\eta_2)\in \det (\ker T_{K_1})\otimes \det
V_2
\]
associated to \eqref{can:isom:rel:1}, and
\[
\go_1\otimes P_{V_2}(\go_2)\in \det V_1 \otimes \det V_2
\]
associated to \eqref{can:isom:rel:2}. Note that we may take
$\go_2:=P_{V_1\oplus V_2}(\eta_2)$. This follows from
commutativity of the big diagram which particularly shows that
\begin{equation}\label{can:isom:rel:3}
P_{V_2}(\go_2)= P_{V_2}\circ P_{V_1\oplus
V_2}(\eta_2)=P_{V_2}(\eta_2).
\end{equation}
Using the sign conventions \eqref{sign:conv}, and letting
$n_1:=\dim V_1$ and $n_2:=\dim V_2$, one readily checks that
$\gF_1$ is now given by
\begin{equation}\label{phi1}
\begin{split}
\gF_1(\eta_1 \otimes \go_1^*) &= (-1)^{n_1n_2+\frac{n_2(n_2+1)}2}
\cdot (\eta_1 \wedge\eta_2 )\otimes \big(\go_1 \wedge
P_{V_2}(\eta_2)\big)^*\\
&= (-1)^{n_1n_2+\frac{n_2(n_2+1)}2} \cdot (\eta_1 \wedge\eta_2
)\otimes \big(\go_1 \wedge P_{V_1\oplus V_2}(\eta_2)\big)^*,
\end{split}
\end{equation}

It remains to check that $\gF_1\circ \gF_{K_1} = \gF_{K_1+K_2}$.
To use the explicit description \eqref{can:isom:expl} of
$\gF_{K_1}$, we consider an adapted basis
\[
\xi \otimes \big(P_{V_1}(\eta_1')\wedge \go_1'\big) \in \det(\ker
T)\otimes (\det V_1)
\]
so that
\[
\gF_{K_1}\big(\xi\otimes F_1(\go_1')^*\big) =
(-1)^{\frac{(n_0+n_1)(n_0+n_1+1)}2}(\xi\wedge \eta_1')\otimes
\big(P_{V_1}(\eta_1')\wedge \go_1'\big)^*,
\]
where $n_0:=\dim(\coker T)$. Letting $\eta_1:= \xi\wedge \eta_1'$
and $\go_1:= P_{V_1}(\eta_1')\wedge \go_1'$ in \eqref{phi1}, we
apply $\gF_1$ to this and get
\begin{multline*}
\gF_1\circ \gF_{K_1}\big(\xi\otimes F(\go_1')^*\big)=
(-1)^{\frac{(n_0+n_1)(n_0+n_1+1)}2+n_1n_2+\frac{n_2(n_2+1)}2} \\
\cdot \big((\xi\wedge \eta_1')\wedge\eta_2 \big)\otimes
\big((P_{V_1}(\eta_1')\wedge \go_1') \wedge
P_{V_2}(\eta_2)\big)^*.
\end{multline*}
To compute $\gF_{K_1+K_2}\big(\xi\otimes F(\go_1')^*\big)$, first
note that
\[
\xi\otimes \big(P_{V_1\oplus V_2}(\eta_1'\wedge\eta_2)\wedge
\go_1'\big)  = \xi \otimes \big(P_{V_1}(\eta_1')\wedge
P_{V_2}(\eta_2)\wedge \go_1'\big)
\]
is an adapted basis associated to the exact sequence
\eqref{det:seq} in the case $K=K_1+K_2$. Since $(F_1+F_2)(\go_1')
= F_1(\go_1')$,
\begin{multline*}
\gF_{K_1+K_2}\big(\xi\otimes F_1(\go_1')^*\big)
=(-1)^{\frac{(n_0+n_1+n_2)(n_0+n_1+n_2+1)}2}\\
\cdot\big(\xi\wedge (\eta_1'\wedge \eta_2)\big)\otimes
\big(P_{V_1}(\eta_1')\wedge P_{V_2}(\eta_2)\wedge \go_1'\big)^*.
\end{multline*}
Since $\eta_2\in \gL^{n_2}(\ker T_{K_1+K_2})$ and $\go_1'\in
\gL^{n_0}V_1$,
\begin{equation}\label{sign:must}
P_{V_1}(\eta_1')\wedge \go_1' \wedge P_{V_2}(\eta_2) = (-1)^{n_0
n_2}\cdot P_{V_1}(\eta_1')\wedge P_{V_2}(\eta_2)\wedge \go_1'.
\end{equation}
Therefore, to prove that $\gF_1\circ \gF_{K_1}\big(\xi\otimes
F(\go_1')^*\big)$ and $\gF_{K_1+K_2}\big(\xi\otimes
F_1(\go_1')^*\big)$ are equal it remains to observe that
\[
(-1)^{\frac{(n_0+n_1)(n_0+n_1+1)}2+n_1n_2+\frac{n_2(n_2+1)}2} =
(-1)^{\frac{(n_0+n_1+n_2)(n_0+n_1+n_2+1)}2 + n_0 n_2}.\qedhere
\]
\end{proof}
\begin{remark*}
We have been so explicit in the last proof to show that the sign
conventions are essential. If we had rather used the isomorphisms
of \eqref{can:isoms}, we would have had a problem with signs.
This is because the factor $(-1)^{n_0n_2}$ in \eqref{sign:must}
would not have cancelled out.\\
\end{remark*}

\noindent\textbf{The determinant line bundle.} The considerations
in the last paragraph enable us to define the determinant line
bundle over the space of Fredholm operators.

\begin{theorem}\label{det:line}
There exists a canonical line bundle $Det \to \sF$ with fibres
$Det_T = \det T$ and the following property: For every finite
dimensional Hilbert space $V$ and every open subset
$\sU\subset\sF$ such that there exists a stabilizer
$K:\sU\to\sL(V,H_2)$, the fibrewise isomorphisms
\begin{equation}\label{det:line:1}
\gF_{K,T}:\det T \longrightarrow \det(\ker T_K)\otimes (\det
V)^*,\quad T\in\sU,
\end{equation}
given by Proposition \ref{can:isom} yield a well-defined bundle
isomorphism
\begin{equation*}
\gF_K: Det|_{\sU} \longrightarrow \det(Ker_K)\otimes(\det
V)^*|_{\sU}.
\end{equation*}
\end{theorem}
\begin{proof}
Let $T_0\in \sF$, and consider a stabilizer $K$ over some open
neighbourhood $\sU$ of $T_0$. Then we use the isomorphisms
\eqref{det:line:1} to pull back the line bundle structure of
$\det(Ker_K)\otimes(\det V)^*|_{\sU}$ to $\bigcup_{T\in\sU} \det
T$. Note that according to Lemma \ref{Stab:exists}, we can cover
$\bigcup_{T\in\sF} \det T$ in this way. Yet, we still have to
prove that the thus obtained line bundles patch together and do
not depend on the stabilizers we have chosen. For this it
suffices to check that given two stabilizers
$K_1:\sU\to\sL(V_1,H_2)$ and $K_2:\sU\to\sL(V_2,H_2)$ over a
common subset $\sU\subset \sF$, there exists a bundle isomorphism
\[
\gF_{12}: \det(Ker_{K_1})\otimes(\det V_1)^*|_{\sU} \to
\det(Ker_{K_2})\otimes(\det V_2)^*|_{\sU}
\]
such that for every $T\in\sU$ the following diagram commutes:
\[
\begindc[1]
\obj(89,1){$\det T$}[T] \obj(1,90){$\det(\ker T_{K_1}) \otimes
(\det V_1)^*$}[K1] \obj(178,90){$\det(\ker T_{K_2}) \otimes (\det
V_2)^*$}[K2] \mor{T}{K1}{$\gF_{K_1,T}$}
\mor{T}{K2}{$\gF_{K_2,T}$}[\atright,\solidarrow]
\mor{K1}{K2}{$\gF_{12,T}$}
\enddc.
\]
It is immediately clear from \eqref{can:isom:rel:diag} that in
the notation of Proposition \ref{can:isom:rel}, letting $\gF_{12}
:= \gF_2^{-1} \circ \gF_1$ makes the above diagram commutative.
Moreover, the explicit formula $\eqref{phi1}$ shows that the
collection of fibrewise maps yields, in fact, a {\em bundle}
isomorphism $\gF_{12}$. This completes the proof.
\end{proof}

\begin{dfn}
Let $T:X\to\sF(H_1,H_2)$ be a continuous family of Fredholm
operators, where $X$ is an arbitrary topological space. Then the
pullback bundle
\[
\det T:= T^*(Det)\longrightarrow X
\]
is called the {\em determinant line bundle of the family $T$}. If
$U\subset X$ is an open set such that there exists a stabilizer
$K:T(U)\to \sL(H_1\oplus V,H_2)$ we shall call $K\circ T$ a {\em
stabilizer of $T$ over $U$} (and usually denote it also by $K$).
\end{dfn}

\begin{remark*}
In particular, the line bundle $\det T\to X$ is defined in the
following setting: Let $\{T_x\}_{x\in X}$ be a family of closed,
densely defined Fredholm operators in a separable Hilbert space
$H$ such that $\dom T_x=W$ is independent of $x$. Moreover, we
assume that the induced graph norms on $W$ are all equivalent.
Then $T$ can be regarded as a map $T:X\to \sF(W,H)$, and thus the
above definition applies provided that the map is continuous.
\end{remark*}

Now suppose that $T:X\to \sF(H_1,H_2)$ is continuous where $X$ is
{\em compact}. According to Lemma \ref{Stab:exists}, we can cover
$X$ with finitely many $U_i$ such that there exist finite
dimensional subspaces $V_i\subset H_2$ such that
$K_i:=(V_i\hookrightarrow H_2)$ is a stabilizer over $U_i$. It is
then clear that letting $V:=\bigoplus_i V_i$ we may take
$V\hookrightarrow H_2$ as a stabilizer over $X$. Hence,

\begin{lemma}\label{stab:compact}
If $X$ is compact and $T:X\to \sF(H_1,H_2)$ is continuous, then
there exists a finite dimensional subspace $V\subset H_2$ such
that the constant $K:=(V\hookrightarrow H_2)$ defines a
stabilizer of $T$ over $X$. In particular, the determinant line
bundle $\det T\to X$ is globally isomorphic to
\[
\det(\ker T_K)\otimes (\det V)^*\longrightarrow X.
\]
\end{lemma}

%% file: app_C.tex
\chapter{Spectral Flow and Orientation Transport}\label{app:SF&OT}

In this appendix, we summarize the definition and elementary
properties of the spectral flow assigned to a path of
self-adjoint elliptic operators. We shall not attempt to give the
most general definition and restrict ourselves to the special
cases which will actually occur in the applications we have in
mind. However, this requires some nontrivial considerations about
how the spectrum changes along a path of self-adjoint operators.
We will only state the results we need referring to the literature
for proofs.

In the second part we shall then see how to extract a
number---the so-called orientation transport---from the
determinant line bundle of a path of self-adjoint operators. We
will then prove a formula relating this notion with the spectral
flow. Due to the conventions we have used in Appendix
\ref{app:det} there are some subtleties about signs. Therefore, we
include all proofs concerning the orientation transport.

\section{Spectral flow}\label{app:SF}

The spectral flow is assigned to a continuous path of bounded
self-adjoint Fredholm operators by counting with multiplicity the
number of eigenvalues passing through zero in the positive
direction (see Fig. \ref{sf:1}). The original definition due to
M. Atiyah and G. Lusztig is via the intersection number of the
spectrum's graph with the zero line (cf. Atiyah et al.
\cite{AtiPatSin:SAR}). To get a rigorous definition one possibly
has to perturb the path slightly in order to make certain that
the intersection number is well-defined.

\begin{remark*}
Nowadays there are more satisfactory, yet logically equivalent
definitions of the spectral flow. J. Phillips' approach
\cite{Phi:SF} is perhaps the most appealing. An extensive
treatment for paths of unbounded self-adjoint operators can be
found in Boo\ss-Bavnbek et al. \cite{BoLePhi:SF}. Nevertheless,
we will follow the original ideas of Atiyah since they are better
suited for the actual computations which occur in this thesis.\\
\end{remark*}

\begin{figure}
\begin{center}
\input{sf_1.eepic}
\caption{Spectral flow}\label{sf:1}
\end{center}
\end{figure}

\noindent\textbf{Paths of Hermitian matrices.} Before we can
define the spectral flow, we have to recall some results from
finite dimensional perturbation theory.

Let $\{A_t\}_{t\in[a,b]}$ be a path of self-adjoint operators in
a finite dimensional $\K$-Hilbert space $H$. Modulo the choice of
a basis of $H$, the path $\{A_t\}_{t\in[a,b]}$ consists of
Hermitian\footnote{In what follows, ``Hermitian'' has to be
replaced with ``symmetric'' if $\K=\R$.} matrices. We want to
find suitable paths of eigenvalues parametrizing the graph of the
path's spectrum, i.e., the subset
\[
\bigcup_{t\in[a,b]} \{t\}\times\spec(A_t)\subset [a,b]\times \R
\]
If $A_t$ is a $C^k$-path such that there are only {\em simple}
eigenvalues, then the implicit function theorem ensures that it is
possible to find $C^k$-functions parametrizing the graph. In
addition, it is then easy to see that there exist $C^k$-families
of corresponding eigenvectors.

Whenever an eigenvalue has higher multiplicity, the situation
becomes much more delicate. Nevertheless, there is one
comparatively simple observation: We can find continuous paths of
eigenvalues whenever $A_t$ is continuous in $t$. Repeating each
eigenvalue according to its multiplicity we only have to number
the eigenvalues in ascending order and the paths obtained in this
way turn out to be continuous (cf. Kato \cite{K}, Sec.~II.5.2).
If $A_t$ is assumed to be $C^k$, smooth or analytic, it is
natural to ask whether the parametrization of the spectrum's
graph can be chosen to have a corresponding regularity. However,
according to classical results due to Rellich, this is only
partially true. We shall now recall some of those results.

\begin{theorem}\textnormal{(cf. \cite{K}, Thm.~II.6.1)}. Let
$A_t$ be a real analytic path of Hermitian matrices. Then there
are two sets of real analytic families representing the repeated
eigenvalues and the corresponding eigenvectors respectively.
\end{theorem}

\begin{theorem}\label{Kato:st1}\textnormal{(cf. \cite{K},
Thm.~II.6.8)}. Assume that $A_t$ is a $C^1$-path of Hermitian
matrices. Then there exist continuously differentiable functions
$\gl_i$ representing the repeated eigenvalues.
\end{theorem}

It should be noted that the proof of the second theorem is rather
complicated and does not carry over to higher orders of
differentiability. Moreover, we cannot hope that there exist
corresponding $C^1$-families of eigenvectors. This is illustrated
by a famous example due to Rellich (cf. Kato \cite{K},
Ex.~II.5.3). It is not difficult, though, to show that the total
projection on the eigenspaces near a $k$-fold eigenvalue is
continuously differentiable:

\begin{theorem}\label{Kato:st2}\textnormal{(cf. \cite{K},
Thm.~II.5.4)}. Let $A_t$ be a $C^1$-path of Hermitian matrices,
and suppose that $\gl$ is a $k$-fold eigenvalue of $A_{t_0}$. If
$\gl_1,\ldots,\gl_k$ denote $C^1$-functions representing
eigenvalues of $A_t$ such that $\gl_i(t_0)=\gl$, then the total
projection
\[
\Proj_{\ker(A-\gl_1)}+\ldots + \Proj_{\ker(A-\gl_k)}
\]
is continuously differentiable near $t_0$.
\end{theorem}

In addition to these well-known results, we want to point out a
recent result due to Alekseevsky, Kriegl, Michor and Losik:

\begin{theorem}\textnormal{(cf. \cite{AleKriMic:Smo}, Thm.~7.6)}.
Let $A_t$ be a smooth family of Hermitian matrices such that no
two of the continuous eigenvalues meet of infinite order at any
$t$ if they are not equal for all $t$. Then all the eigenvalues
and all eigenvectors can be chosen smoothly in $t$ on the whole
parameter domain.
\end{theorem}

Using the proof of the last results as a guideline, we now want to
consider an example which we will need in the main part of this
thesis.
\begin{example}\label{gl:sec}
Let $\{A_t\}_{t\in[-1,1]}$ be a $C^4$-path of self-adjoint
operators acting on a finite dimensional Hilbert space $H$ and
suppose that $A_0=0$ and that the eigenvalues of $A'_0$ are
simple. Then the following holds:
\begin{enumerate}
\item The eigenvalues of $A_t$ near 0 can be parametrized by
$C^2$-functions $\gl_i$, and there exist corresponding $C^2$-maps
of normalized eigenvectors $v_i$.
\item If in addition one of the eigenvalues of $A_0'$ vanishes
then the corresponding $C^2$-path of eigenvalues, say $\gl$,
satisfies
\begin{equation}\label{gl:sec:der}
\gl''(0)=\lfrac{d}{dt}\big|_{t=0} \scalar{A'_tv_t}{v_t}\;,
\end{equation}
where $v_t$ is the path of eigenvectors associated to $\gl$ near
0.
\end{enumerate}
\begin{proof}
Since $A_0=0$ and $A_t$ is $C^4$, we can use the Taylor
approximation to write
\[
A_t= A_0' t + \frac{A_0^{(2)}}2 t^2 + \frac{A_0^{(3)}}6 t^3 +
R(t),
\]
with the Taylor remainder $R(t)$, which is $C^4$ and satisfies
$\|R(t)\|\le C |t|^4$ for all $t$. This implies that
$\frac{R(t)}t$ is a $C^2$-map with a zero of order 3 in 0. We can
thus write $A_t=tB_t$, where $B_t$ is a $C^2$-path of
self-adjoint operators on $H$ satisfying $B_0=A_0'$. From Theorem
\ref{Kato:st1} we thus obtain $C^1$-functions $\mu_i$
parametrizing the eigenvalues of $B_t$. Due to our assumptions,
each $\mu_i(0)$ is a simple eigenvalue of $B_0=A_0'$ so that we
can find $\eps<1$ such that all $\mu_i(t)$ are simple eigenvalues
of $B_t$ for each $|t|<\eps$. Therefore, by virtue of the implicit
function theorem, each $\mu_i(t)$ is necessarily a $C^2$-function
on $(-\eps,\eps)$. This shows in particular that
$\ker(B_t-\mu_i(t))$ forms a $C^2$-line bundle over
$(-\eps,\eps)$. Choosing normalized sections $v_i(t)$ produces
$C^2$-paths of eigenvectors of $B_t$. Defining $C^2$-functions
$\gl_i(t):= t\mu_i(t)$ we now observe that for $|t|<\eps$,
\[
A_t v_i(t) = t B_t v_i(t)= t\mu_i(t) v_i(t) = \gl_i(t) v_i(t).
\]
Therefore, each $\gl_i$ is a $C^2$-path of eigenvalues of $A_t$
near 0 with corresponding $C^2$-path of eigenvectors $v_i$. This
proves part (i).

Now assume that $\gl=\gl_{i_0}$ satisfies $\gl'(0)=0$. For any
$t\in(-\eps,\eps)$, we then have $A_t v_t = \gl(t) v_t$ with a
$C^2$-path of eigenvectors $v_t$ of unit length. Differentiating
this yields
\begin{equation}\label{gl:sec:1}
A_t'v_t + A_t v_t' = \gl'(t) v_t + \gl(t) v_t'.
\end{equation}
Taking the inner product with $v_t$ results in
\begin{equation*}
\gl'(t)=\scalar{A_t'v_t}{v_t} + 2\Re\scalar{A_tv_t}{v'_t}\;,
\end{equation*}
where we have used self-adjointness of $A_t$. Differentiating
again and evaluating at $t=0$, we infer that
\[
\gl''(0)=\lfrac{d}{dt}\big|_{t=0} \scalar{A'_tv_t}{v_t} +
2\Re\big(\scalar{A_0'v_0}{v_0'} + \scalar{A_0v_0'}{v_0'} +
\scalar{A_0v_0}{v_0''}\big).
\]
Since $A_0=0$, $\gl(0)=0$ and $\gl'(0)=0$ we deduce from
\eqref{gl:sec:1} that $A_0' v_0= 0$. Therefore, the three
summands on the right hand side of the above equation vanish, and
this proves part (ii).\qedhere\\
\end{proof}
\end{example}

\noindent\textbf{Paths of self-adjoint operators.} After this
short digression into finite dimensional Hilbert space theory, we
return to the setting we actually have in mind. Let
$\{T_t\}_{t\in [a,b]}$ be a family of unbounded, self-adjoint
operators\footnote{In our applications, $T_t$ will be a family of
formally self-adjoint, elliptic operators of constant order $m$,
i.e., operators in $H=L^2$ with domain $W=L^2_m$ (cf. the
discussion in App.~\ref{app:ell}, Sec.~\ref{diffop}).} in a
separable $\K$-Hilbert space $H$. Assume that there exists a
Hilbert space $W$ such that all $T_t$ have $W$ as a common domain
with graph norm equivalent to the given norm on $W$, i.e., we
suppose that $T_t$ is a function with values in the set
\index{<@$\sL_{sa}(W,H)$, self-adjoint operators}
\begin{equation}\label{Lsa}
\sL_{sa}(W,H):= \bigsetdef{T\in\sL(W,H)}{T \text{ self-adjoint
operator in } H}.
\end{equation}
In contrast to that, let \index{<@$\sL_{sym}(W,H)$, symmetric
operators}
\begin{equation}\label{Lsym}
\sL_{sym}(W,H):=\bigsetdef{T\in\sL(W,H)}{T\text{ symmetric
operator in } H}.
\end{equation}
Recall that $\sL_{sym}(W,H)$ is a Banach space if endowed with
the operator norm topology. Due to the Kato-Rellich Theorem (cf.
\cite{K}, Thm.~V.4.3), $\sL_{sym}(W,H)$ contains $\sL_{sa}(W,H)$
as an open subset. Therefore, in the case at hand, we may speak of
continuity and differentiability of the family $T_t$ with respect
to the operator norm topology on $\sL_{sa}(W,H)$.

Moreover, we assume that $W$ embeds compactly in $H$ which ensures
that each $T_t$ has discrete spectrum consisting of real
eigenvalues with finite multiplicity.

The following result makes a precise definition of the spectral
flow possible. It is a consequence of Theorem \ref{Kato:st1} and
Theorem \ref{Kato:st2}. A proof can be found in Robbin \& Salamon
\cite{RobSal:SF}, Cor.~4.29.

\begin{theorem}[Kato's Selection
Theorem]\label{Kato:sel}\index{families of operators!Kato's
Selection Theorem} Suppose that $T_t$ is continuously
differentiable. Let $t_0\in [a,b]$ and $c>0$ such that $\pm
c\notin\spec(T_{t_0})$, and let $n$ be the dimension of the
subspace spanned by eigenvectors corresponding to eigenvalues in
$(-c,c)$. Then there exists a constant $\eps>0$ and
$C^1$-functions
\[
\gl_i:(t_0-\eps,t_0+\eps)\to (-c,c),\quad i\in\{1,\ldots,n\}
\]
with the following properties:
\begin{itemize}
\item $\gl_i(t)\in\spec(T_t)$
\item $\gl_i'(t)\in \spec\big(P_i(t)\circ T_t'\circ P_i(t)\big)$,
where $P_i(t)=\Proj_{\ker(T_t-\gl_i(t))}$.
\item If $\gl\in \spec(T_t)\cap (-c,c)$ with corresponding
spectral projection $P$ and $\mu\in \spec(P\circ T_t'\circ P)$ is
an eigenvalue of multiplicity $m$, then there are precisely $m$
indices $i_1,\ldots,i_m$ such that $\gl_{i_j}(t)=\gl$ and
$\gl_{i_j}'(t)=\mu$ for $j=1,\ldots,m$.
\end{itemize}
\end{theorem}

\begin{dfn}\label{transversal}\index{<@$C_T$, crossing
operator}\index{families of operators!transversal}\index{families
of operators!spectral flow|(} Assume that $T:[a,b]\to
\sL_{sa}(W,H)$ is continuously differentiable. Then we define the
{\em crossing operator} of $T$ at $t\in [a,b]$ by letting
\[
C_T(t):=\Proj_{\ker T_t}\circ\;T_t'|_{\ker T_t}:
\ker T_t \to \ker T_t\;.
\]
The path $T$ is called \textit{transversal} if $C_T(t)$ is
invertible for each $t\in [a,b]$. Furthermore, $T$ is said to
have {\em simple crossings} if $\dim \ker T_t \le 1$ for every
$t\in [a,b]$.
\end{dfn}

\begin{remark*}
Whenever $T$ is a transversal path of self-adjoint operators,
Kato's Selection Theorem ensures that the graph $\bigcup_{t\in
[a,b]}\{t\}\times \spec(T_t)$ intersects the line $[a,b]\times
\{0\}$ transversally. This explains the chosen terminology.
\end{remark*}

An application of Sard's Theorem yields

\begin{theorem}{\rm(cf. \cite{RobSal:SF}, Thm.~4.22)}. Suppose
$T:[a,b]\to \sL_{sa}(W,H)$ is continuously differentiable. Then
the path $T+\gd$ is transversal for almost every $\gd\in\R$.
\end{theorem}

\begin{dfn}\label{SF:def}\index{$\SF(T)$, spectral
flow}\index{<@$\SF(T)$, spectral flow} Let $T:[a,b]\to
\sL_{sa}(W,H)$ be a continuously differentiable path. Since $T_a$
and $T_b$ have discrete spectrum, there exists $\gd>0$ such that
the operators $T_a+\eps$ and $T_b+\eps$ are invertible for all
$0<\eps\le\gd$. According to the above result there exists $\eps$
among these such that the path $T+\eps$ is transversal. We define
the {\em spectral flow} of $T$ by letting
\[
\SF(T):=\sum_{t\in(a,b)} \sign C_{T+\eps}(t)\;,
\]
where ``sign" denotes the signature of a Hermitian endomorphism,
i.e., the number of positive eigenvalues minus the number of
negative ones.
\end{dfn}

By virtue of Kato's Selection Theorem, the above definition is
independent of the choices made provided that $\gd$ is chosen
sufficiently small.

If $\eps$ can be chosen in such a way that in addition, $T+\eps$
has only simple crossings, then the spectral flow of $T$ is
clearly given by
\[
\SF(T)=\sum_{t\in(a,b)}\sgn\Scalar{T_t'v_t}{v_t}\;,
\]
where each $v_t$ is a vector spanning $\ker(T_t+\eps)$. Here,
``sgn" denotes the sign whereas ``sign" is reserved for the
signature of a Hermitian endomorphism.

\begin{remark*}
Observe that our definition has a built-in convention of how to
treat zero eigenvalues of $T_a$ and $T_b$: Instead of counting
crossings with the zero line, we take the intersection number of
$\bigcup_{t\in [a,b]}\{t\}\times \spec(T_t)$ with the line
$[a,b]\times \{-\eps\}$ (cf. Fig. \ref{sf:eps}). Moreover, note
that this corresponds to subtracting negative eigenvalues of
$C_T(a)$ and adding positive eigenvalues of $C_T(b)$, i.e.,
\begin{multline*}
\sum_{t\in(a,b)} \sign C_{T+\eps}(t) =  \sum_{t\in(a,b)} \sign
C_T(t) - \#\setdef{\gl\in\spec C_T(a)}{\gl<0}\\ +
\#\setdef{\gl\in\spec C_T(b)}{\gl>0}.
\end{multline*}

\end{remark*}

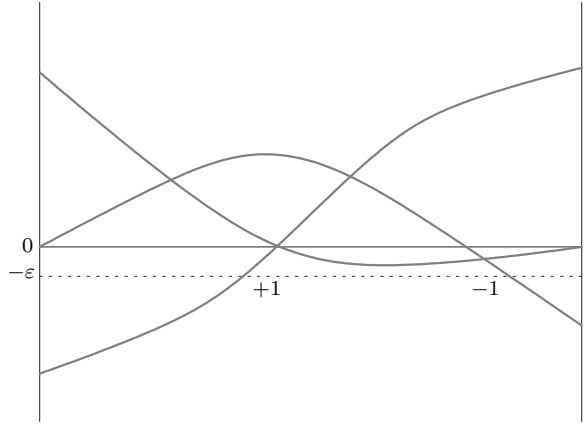
\begin{figure}
\begin{center}
\input{sf_eps.eepic}
\caption{Convention for counting zero eigenvalues at the endpoints
}\label{sf:eps}
\end{center}
\end{figure}

\noindent\textbf{Properties of the spectral flow.} For any pair
$T_1:[a,b]\to\sL_{sa}(W_1,H_1)$ and
$T_2:[a,b]\to\sL_{sa}(W_2,H_2)$ of paths of self-adjoint
operators, we can form the direct sum
\[
T_1\oplus T_2:[a,b]\to \sL_{sa}(W_1\oplus W_2, H_1\oplus H_2)\;.
\]
Given $T_1:[a,b]\to\sL_{sa}(W,H)$ and $T_2:[b,c]\to\sL_{sa}(W,H)$
satisfying $T_1(b)=T_2(b)$, we can also build the concatenation
$T_1\#T_2:[a,c]\to\sL_{sa}(W,H)$, defined by
\[
(T_1\#T_2)(t):=\begin{cases}T_1(t)\quad\text{\rm if } t\in [a,b] \\
                              T_2(t)\quad\text{\rm if } t\in
                              [b,c]\;.
\end{cases}
\]

\begin{prop}\label{SF}{\rm(cf. \cite{RobSal:SF}, Thm.~4.23)}.
Let $T_1$ and $T_2$ be continuously differentiable paths of
self-adjoint operators in $H$ with domain $W$.
\begin{enumerate}
\item If $T_1$ is a constant family, then $\SF(T_1)=0$.
\item If $T_1$ and $T_2$ are homotopic relative endpoints, then
$\SF(T_1)=\SF(T_2)$.
\item $\SF(T_1\oplus T_2)=\SF(T_1)+\SF(T_2)$, where we allow $T_1$
and $T_2$ to be defined in different Hilbert spaces.
\item $\SF(T_1\# T_2)=\SF(T_1)+\SF(T_2)$, whenever the left-hand
side is well-defined.
\end{enumerate}
\end{prop}

\begin{remark*}
It is not difficult to see that a homotopy invariant on the space
of $C^1$-paths gives also rise to a homotopy invariant for
continuous paths. Thus we could go on and define the spectral
flow in this more general setting like, for example, in Sec.~4 of
\cite{RobSal:SF}. However, the situation we shall actually
encounter does not require any further considerations so that we
refer to the literature for a more extensive treatment.
\end{remark*}
\index{families of operators!spectral flow|)}

\section{Orientation transport}\label{app:OT}

Let $H$ be a separable $\R$-Hilbert space and let $W\subset H$ be
a dense subspace. We denote by
\[
\sF_{sa}(W,H):= \sL_{sa}(W,H)\cap \sF(W,H)
\]
the space of closed, unbounded, self-adjoint Fredholm operators
with domain $W$. Note that $\sF_{sa}(W,H)$ is an open subset of
$\sL_{sym}(W,H)$ if we use the operator norm. We now consider the
restriction of the determinant line bundle $Det\to \sF(W,H)$ of
Appendix \ref{app:det} to $\sF_{sa}(W,H)$. For every
$T\in\sF_{sa}(W,H)$ we have $(\im T)^\perp=\ker T$. Using the
Knudsen-Mumford sign conventions \eqref{sign:conv}, we thus have:

\begin{lemma}
For every $T\in\sF_{sa}(W,H)$ let $n_0(T):=\dim(\ker T)$. Then
there is a canonical isomorphism
\begin{equation}\label{sa:fibres}
\Psi_T:Det_T=\det(\ker T)\otimes \det(\coker T)^*\to \R,
\end{equation}
given by
\[
\Psi_T(\xi\otimes \go^*):=(-1)^{\frac{n_0(T)(n_0(T)+1)}2}
\go^*[\xi].
\]
\end{lemma}
\begin{remark*}
Note that this lemma does not necessarily imply that the line
bundle $Det|_{\sF_{sa}}\to \sF_{sa}$ is trivial since the
collection $\{\Psi_T\}_{T\in\sF_{sa}}$ need not give rise to a
{\em continuous} map $Det|_{\sF_{sa}} \to \sF_{sa}\times \R$.\\
\end{remark*}

\noindent\textbf{Orientation transport.}\index{families of
operators!orientation transport|(}\index{<@$\eps(T)$, orientation
transport} Let $T:[a,b]\to \sF_{sa}(W,H)$ be a continuous path.
Since $[a,b]$ is contractible, there exists a trivialization of
the determinant line bundle $\det T\to [a,b]$. This induces an
isomorphism
\begin{equation}\label{det:triv:isom}
\Psi_{T_aT_b}:\det T_a \to \det T_b.
\end{equation}
By concatenation with the isomorphisms of \eqref{sa:fibres} at
the endpoints $T_a$ and $T_b$ we get a chain of isomorphisms
\[
\begin{CD}
\R @>{\Psi_{T_a}^{-1}}>> \det T_a @>{\Psi_{T_aT_b}}>> \det T_b
@>{\Psi_{T_b}}>> \R.
\end{CD}
\]
It is immediate that the parity of the isomorphism $\R\to \R$
given in this way is independent of the particular choice of
trivialization of $\det T\to [a,b]$. Hence we may define:

\begin{dfn}\label{OT:def}
Let $T:[a,b]\to\sF_{sa}(W,H)$ be a continuous path, and let
$\Psi_{T_aT_b}:\det T_a \to \det T_b$ be an isomorphism induced by
a trivialization of $\det T\to [a,b]$. Then
\[
\eps(T):=\sgn\big(\Psi_{T_b}\circ \Psi_{T_aT_b}\circ
\Psi_{T_a}^{-1}(1)\big)
\]
is called the {\em orientation transport} along $T$.
\end{dfn}

We are now going to derive an alternative formula for the
orientation transport in the case of a continuous path
$T:[a,b]\to\sF_{sa}(W,H)$ which has invertible endpoints. Let
$K:[a,b]\to\sL(W\oplus V,H_2)$ be a stabilizer of $T$. Modulo the
canonical isomorphism $\det T\cong \det(\ker T_K)\otimes (\det
V)^*$, the isomorphisms \eqref{sa:fibres} at the endpoints $T_a$
and $T_b$,
\[
\Psi_{K_t}:\det(\ker (T_K)_t)\otimes (\det V)^*\longrightarrow
\R,\quad t\in\{a,b\},
\]
are given by
\begin{equation}\label{can:isom:expl:inv}
\Psi_{K_t}(\eta\otimes P_V(\eta)^*)= (-1)^{\frac{n(n+1)}2},
\end{equation}
where $n:=\dim V$. The orientation transport along $T$ is given by
\[
\eps(T)=\sgn\big( \Psi_{K_b}\circ \Psi_{T_aT_b}\circ
\Psi_{K_a}^{-1}(1)\big),
\]
with an isomorphism
\[
\Psi_{T_aT_b}:\det\big(\ker (T_K)_a\big)\otimes (\det V)^* \to
\det\big(\ker (T_K)_b\big)\otimes (\det V)^*
\]
induced by a trivialization of $\det(\ker T_K)\otimes(\det
V)^*\to[a,b]$. Clearly, choosing a trivialization of $\ker
T_K\to[a,b]$ yields an isomorphism
\[
\Psi_a^b: \ker (T_K)_a\to \ker (T_K)_b,
\]
and we may then take
\[
\Psi_{T_aT_b}(\eta\otimes\go^*):=\Psi_a^b(\eta)\otimes \go^*.
\]
Formula \eqref{can:isom:expl:inv} implies that for any basis
$\eta\in\det \ker (T_K)_a$,
\begin{equation}\label{OT:inv:2}
\Psi_{K_b}\circ \Psi_{T_aT_b}\circ \Psi_{K_a}^{-1}(1) =
(-1)^{\frac {n(n+1)}2}\cdot \Psi_{K_b}\big( \Psi_a^b(\eta)\otimes
P_V (\eta)^* \big).
\end{equation}
Now, to use \eqref{can:isom:expl:inv} for $t:=b$, we observe that
\begin{equation}\label{OT:inv:3}
\Psi_a^b(\eta)\otimes P_V(\eta)^* = \det(\gF_V) \cdot
\Psi_a^b(\eta)\otimes \big[P_V\circ\Psi_a^b(\eta)\big]^*,
\end{equation}
where $\gF_V:V\to V$ is defined by the following commutative
diagram:
\begin{equation}\label{OT:CD}
\begin{CD}
\ker(T\oplus K)_a  @>{\Psi_a^b}>>   \ker(T\oplus K)_b\\
@VV{P_V}V                    @VV{P_V}V \\
V                   @>{\gF_V}>>        V
\end{CD}
\end{equation}
Here, the projection $P_V:\ker(T_K)_t \to V$ is an isomorphism
because $T_t$ is invertible for $t\in\{a,b\}$. Inserting
\eqref{OT:inv:3} in \eqref{OT:inv:2}, we conclude
\[
\Psi_{K_b}\circ \Psi_{T_aT_b}\circ \Psi_{K_a}^{-1}(1) =
(-1)^{\frac {n(n+1)}2}\cdot (-1)^{\frac {n(n+1)}2}\cdot \det
\gF_V =\det \gF_V.
\]
Hence, we have proved:

\begin{lemma}\label{OT:inv}
Let $T:[a,b]\to\sF_{sa}(W,H)$ be a continuous path with
invertible endpoints, and let $K:[a,b]\to \sL(W\oplus V,H)$ be a
stabilizer of $T$, and define $\gF_V:V\to V$ via the commutative
diagram \eqref{OT:CD}. Then
\[
\eps(T)= \sgn(\det \gF_V).
\]
\end{lemma}

\noindent\textbf{Properties of the orientation transport.} The
orientation transport has some properties which are reminiscent
of what we observed for the spectral flow in Proposition
\ref{SF}. Using the definition of direct sum and concatenation as
given there, we have:

\begin{prop}\label{OT}
Let $T_0\in\sF_{sa}(W,H)$ be a self-adjoint Fredholm
operator.
\begin{enumerate}
\item The orientation transport along the constant family $T_t:= T_0$
is 1.
\item If there exists $\gd>0$ such that $T_t:=T_0+t$ is invertible
for $0<t\le \gd$, then
\[
\eps(T_t;\;{\scriptstyle 0\le t\le \gd})=(-1)^{\dim(\ker T)}.
\]
\end{enumerate}
Moreover, if $T_0$, $T_1:[a,b]\to\sF_{sa}(W,H)$ are continuous
paths.
\begin{enumerate}
\setcounter{enumi}{2}
\item If $T_0$ and $T_1$ are homotopic relative
endpoints, then $\eps(T_0)=\eps(T_1)$.
\item $\eps(T_0\oplus T_1)=\eps(T_0)\cdot \eps(T_1)$, where we
allow $T_0$ and $T_1$ to be defined in different Hilbert spaces.
\item $\eps(T_0\# T_1)=\eps(T_0)\cdot \eps(T_1)$, whenever the
left-hand side is well-defined.
\end{enumerate}
\end{prop}

\begin{proof}
We only prove (ii) and (iii) since the other properties are
straightforward. We start with the proof of (ii) using the same
approach as in Lemma \ref{OT:inv}. Letting $V:=\ker T_0$, we
consider
\[
(T_V)_t: W\oplus V\to H,\quad (x,v)\mapsto T_t x + v, \quad
t\in[0,\gd].
\]
The determinant line bundle $\det T\to[0,\gd]$ obtains its line
bundle structure via the natural isomorphism $\det T \cong
\det(\ker T_V)\otimes (\det V)^*$. The bundle $\ker(T_V)\to
[0,\gd]$ is canonically trivial via
\[
[0,\gd]\times V \to \ker T_V ,\quad (t,v)\mapsto (v,-tv)\;.
\]
Hence, we get an isomorphism
\[
\Psi_0^\gd:\ker (T_V)_0\to\ker (T_v)_\gd,\quad (v,0)\mapsto
(v,-\gd v)
\]
which in turn induces an isomorphism on the level of determinants,
\[
\begin{split}
\Psi_{T_0T_\gd}:\det(\ker (T_V)_0)\otimes (\det V)^*&\to \det(\ker
(T_V)_\gd)\otimes (\det V)^*,\\
\xi\otimes \go^* &\mapsto \Psi_0^\gd(\xi)\otimes \go^*.
\end{split}
\]
Modulo the identification $\det T \cong \det(\ker T_V)\otimes
(\det V)^*$, the isomorphisms \eqref{sa:fibres} at the endpoints
$T_0$ and $T_\gd$ take the form
\[
\begin{split}
\Psi_{T_0}:\det (\ker (T_V)_0) \otimes (\det V)^* &\to \R,\quad
\Psi_{T_0}(\xi\otimes \go^*)= (-1)^{n+\frac{n(n+1)}2 }\cdot
\go^*[\xi],\\
\Psi_{T_\gd}:\det (\ker (T_V)_\gd) \otimes (\det V)^* &\to
\R,\quad \Psi_{T_\gd}\big(\eta\otimes P_V(\eta)^*\big) =
(-1)^{\frac{ n(n+1)}2},
\end{split}
\]
where $n:=\dim V$. Now, the orientation transport along $T$ is
given by
\[
\eps(T)=\sgn\big(\Psi_{T_\gd}\circ \Psi_{T_0T_\gd}\circ
\Psi_{T_0}^{-1}(1)\big).
\]
From the explicit formula for $\Psi_0^\gd$ one readily gets that
$P_V(\Psi_0^\gd(\xi))= (-\gd)^n\cdot \xi$. This implies
\[
\xi^*= (-\gd)^n\cdot P_V(\Psi_0^\gd(\xi))^*.
\]
Hence, the expressions for $\Psi_{T_0}$ and $\Psi_{T_\gd}$ imply
that
\[
\begin{split}
\Psi_{T_\gd}\circ \Psi_{T_0T_\gd}\circ \Psi_{T_0}^{-1}(1) &=
(-1)^{n+\frac{n(n+1)}2}\cdot\Psi_{T_\gd}\big(\Psi_0^\gd(\xi)\otimes
\xi^* \big) \\
&= (-1)^{n+\frac{n(n+1)}2}\cdot \Psi_{T_\gd}\big((-\gd)^n \cdot
\Psi_0^\gd(\xi)\otimes P_V(\Psi_0^\gd(\xi))^*\big)\\
&= (-1)^{n+\frac{n(n+1)}2}\cdot (-\gd)^n \cdot
(-1)^{\frac{n(n+1)}2} =\ \gd^n\ .
\end{split}
\]
Since $\gd>0$, the orientation transport along $T$ is 1.

\noindent To prove (iii), suppose that
\[
S:[a,b]\times [0,1]\to \sF_{sa}(W,H)
\]
is a homotopy connecting $T_0$ and $T_1$ and leaving the
endpoints fixed. We then consider the determinant bundle $\det
S\to [a,b]\times [0,1]$. Note that then, we have the following
isomorphism of line bundles:
\[
\det T_0= \det S|_{[a,b]\times \{0\}}\quad\text{ and }\quad \det
T_1=\det S|_{[a,b]\times \{1\}}.
\]
Since $S$ leaves the endpoints fixed, we may define for arbitrary
$s\in[0,1]$
\[
L_a:= \det S(a,s) =\det (T_0)_a = \det (T_1)_a
\]
and similarly $L_b$. As $[a,b]\times [0,1]$ is contractible, there
exists a nowhere vanishing section $\Psi: [a,b]\times [0,1] \to
\det S$. Restricting $\Psi$ to $[a,b]\times \{0\}$ and
$[a,b]\times \{1\}$, we get isomorphisms
\[
\Psi^0:L_a\to L_b\quad\text{ and }\quad \Psi^1:L_a\to L_b.
\]
As both are induced by $\Psi$, they are homotopic in the space of
isomorphisms $L_a\to L_b$. This shows that using $\Psi^0$ and
$\Psi^1$ as in the definition to compute the orientation transport
along $T_0$ and $T_1$ respectively, we get
$\eps(T_0)=\eps(T_1)$.\qedhere\\
\end{proof}

\noindent\textbf{Orientation transport and spectral flow.} The
above proposition suggests an important connection between the
spectral flow and the orientation transport of a path of
self-adjoint Fredholm operators. Since we defined the spectral
flow only in a more restrictive context, we now consider
$C^1$-paths and require that $W$ embeds
compactly\footnote{Consequently, elements of $\sL_{sa}(W,H)$ have
compact resolvent, so that they are automatically Fredholm.} in
$H$.

\begin{theorem}\label{OT=SF} \index{families of operators!spectral
flow}\index{families of operators!orientation transport formula}
Let $T:[a,b]\to\sL_{sa}(W,H)$ be a continuously differentiable
path. Then
\[
\eps(T)=(-1)^{\SF(T)}\;.
\]
\end{theorem}
\begin{proof}
Recall that $\SF(T)$ is defined as $\sum_{t\in(a,b)} \sign
C_{T+\gd}(t)$, where $C_{T+\gd}(t)$ is the crossing operator and
$\gd>0$ is such that $T+\gd$ is transversal and $T+t$ has
invertible endpoints for $0<t\le \gd$. Employing (ii), (iii) and
(v) of Proposition \ref{OT}, one straightforwardly shows that
$\eps(T)=\eps(T+\gd)$ so that without loss of generality, we may
assume that $T$ is already a transversal family with invertible
endpoints and drop $\gd$ from the notation. We are then to
show that
\begin{equation}\label{OT=dim}
\begin{split}
\eps(T)=&(-1)^{\sum_{t\in(a,b)} \sign C_T(t)}\\ &\qquad=
\prod_{t\in(a,b)}(-1)^{\sign C_T(t)} =\prod_{t\in(a,b)}
(-1)^{\dim(\ker T_t)}\;.
\end{split}
\end{equation}
Invoking (v) of Proposition \ref{OT}, we can clearly restrict to
the case of a continuously differentiable path $T:[-1,1]\to
\sL_{sa}(W,H)$ with the property that $T_t$ is invertible for
each $t\neq 0$. Under this assumption, $V:=\ker T_0$ is a
stabilizer of $T$ over $[-1,1]$. Since $\ker T_V\to [-1,1]$ is a
$C^1$-vector bundle over a contractible space, we may choose
trivializing $C^1$-sections
\[
e_i(t):[-1,1]\to \ker T_V.
\]
Moreover, we can clearly achieve that
$\big(e_1(0),\ldots,e_n(0)\big)$ is an orthonormal basis of $V$.
Write
\[
e_i(t)=\big(e_i(0)+w_i(t),v_i(t)\big)
\]
with appropriate $C^1$-maps $w_i:[-1,1]\to W$ and $v_i:[-1,1]\to
V$. Note that $w_i(t)\to 0$ and $v_i(t)\to 0$ as $t\to 0$.

We shall now compute $\eps(T)$ via $\eps\big(T|_{[-t_0,t_0]}\big)$
by taking the limit $t_0\to 0$. For this note that the
orientation transport along $T|_{[-t_0,t_0]}$ is independent of
$t_0$ because $T|_{[-1,-t_0]}$ and $T|_{[t_0,1]}$ are paths of
invertible operators, thus giving no contribution. We use Lemma
\ref{OT:inv} to compute $\eps\big(T|_{[-t_0,t_0]}\big)$. One
readily checks that the isomorphism $\gF_V:V\to V$ given by the
diagram \eqref{OT:CD} in the situation at hand is uniquely
determined by
\[
\gF_V\big(v_i(-t_0)\big)= v_i(t_0).
\]
Hence, the orientation transport along $T|_{[-t_0,t_0]}$ is given
by
\[
\eps\big(T|_{[-t_0,t_0]}\big)=\sgn
\det\Big(\scalar{v_i(-t_0)}{v_j(t_0)}\Big)_{ij}\;.
\]
As we want to determine this sign by letting $t_0\to 0$, we
write---using that $T$ is continuously differentiable---
\[
T_t=T_0+tT_0'+o(t).
\]
Since $(T_V)_t e_i(t)=0$, $T_0 e_i(0)=0$, and $t w_i(t)=o(t)$, we
deduce that
\[
0=(T_V)_t\big(e_i(0)+w_i(t),v_i(t)\big) = T_0w_i(t) + tT_0'e_i(0)
+ v_i(t) + o(t)\;.
\]
Applying $P:=\Proj_V=\Proj_{\ker T_0}$ to the above equation
deletes the first term, and thus,
\[
v_i(t)=- tPT_0'e_i(0)+o(t)
\]
From this we conclude that for small $t_0>0$,
\begin{align*}
\det\Big(\scalar{v_i(-t_0)}{v_j(t_0)}\Big)_{ij} &=
\det\Big(\Scalar{\lfrac{v_i(-t_0)}{t_0}}{\lfrac{v_j(t_0)}{t_0}}\Big)_{ij}
\\
&=\det\Big(\Scalar{PT_0'e_i(0) + \lfrac{o(t_0)}{t_0}}{-PT_0'e_j(0)
+ \lfrac{o(t_0)}{t_0}}\Big)_{ij}\;.
\end{align*}
This expression allows to perform the limit $t_0\to 0$, which
produces
\[
\begin{split}
\eps(T)&=\sgn\det\Big(\Scalar{-PT_0'e_i(0)}{PT_0'e_j(0)}\Big)_{ij}\\&=
\sgn\Big((-1)^n\det(PT_0'P)^2\Big)=(-1)^n.
\end{split}
\]
As $n=\dim(\ker T_0)$, equation \eqref{OT=dim} is established.
\end{proof}\index{families of operators!orientation transport|)}

\cleardoublepage

%% file: sf_1.eepic
\setlength{\unitlength}{0.0003in}
\begingroup\makeatletter\ifx\SetFigFont\undefined%
\gdef\SetFigFont#1#2#3#4#5{%
  \reset@font\fontsize{#1}{#2pt}%
  \fontfamily{#3}\fontseries{#4}\fontshape{#5}%
  \selectfont}%
\fi\endgroup%
{\renewcommand{\dashlinestretch}{30}
\begin{picture}(9344,7239)(0,-10)
\texture{0 0 0 888888 88000000 0 0 80808
    8000000 0 0 888888 88000000 0 0 80808
    8000000 0 0 888888 88000000 0 0 80808
    8000000 0 0 888888 88000000 0 0 80808 }
\color{white}
\shade\path(22,7212)(9322,7212)(9322,6912)
    (8872,6762)(7747,6687)(6847,6762)
    (5647,6612)(5122,6537)(3997,6537)
    (2797,6687)(1372,6612)(247,6687)
    (172,6687)(22,7212)
\path(22,7212)(9322,7212)(9322,6912)
    (8872,6762)(7747,6687)(6847,6762)
    (5647,6612)(5122,6537)(3997,6537)
    (2797,6687)(1372,6612)(247,6687)
    (172,6687)(22,7212)
\shade\path(22,12)(9322,12)(9247,387)
    (8497,312)(7372,462)(5122,387)
    (3097,462)(1222,312)(397,462)
    (22,312)(22,12)
\path(22,12)(9322,12)(9247,387)
    (8497,312)(7372,462)(5122,387)
    (3097,462)(1222,312)(397,462)
    (22,312)(22,12)
\thicklines
\color{gray2}
\path(22,5112)(23,5112)(24,5111)
    (26,5109)(30,5106)(35,5102)
    (42,5096)(52,5089)(63,5080)
    (77,5069)(93,5057)(112,5042)
    (133,5026)(157,5007)(183,4987)
    (211,4965)(241,4942)(273,4916)
    (307,4890)(343,4861)(380,4832)
    (419,4801)(459,4770)(500,4737)
    (542,4703)(585,4669)(630,4633)
    (675,4597)(720,4560)(767,4522)
    (815,4483)(863,4443)(913,4402)
    (964,4359)(1016,4316)(1069,4271)
    (1123,4225)(1179,4178)(1237,4129)
    (1296,4078)(1357,4025)(1419,3971)
    (1484,3914)(1549,3856)(1616,3797)
    (1684,3736)(1753,3674)(1822,3612)
    (1893,3547)(1963,3483)(2030,3420)
    (2093,3360)(2153,3302)(2208,3248)
    (2259,3197)(2305,3149)(2346,3105)
    (2382,3064)(2414,3026)(2441,2992)
    (2464,2961)(2483,2932)(2499,2907)
    (2512,2883)(2522,2862)(2530,2842)
    (2536,2824)(2540,2807)(2543,2792)
    (2545,2777)(2547,2762)(2549,2748)
    (2551,2733)(2554,2718)(2559,2703)
    (2565,2686)(2573,2669)(2584,2650)
    (2598,2630)(2615,2608)(2635,2584)
    (2660,2558)(2689,2530)(2722,2500)
    (2760,2468)(2804,2433)(2853,2397)
    (2906,2359)(2965,2319)(3029,2277)
    (3097,2236)(3169,2194)(3245,2152)
    (3322,2112)(3401,2074)(3479,2039)
    (3556,2006)(3630,1975)(3702,1948)
    (3769,1922)(3833,1900)(3891,1879)
    (3945,1861)(3995,1845)(4040,1831)
    (4080,1819)(4116,1809)(4149,1800)
    (4177,1792)(4203,1785)(4225,1780)
    (4245,1775)(4263,1771)(4279,1768)
    (4294,1766)(4308,1764)(4322,1762)
    (4336,1761)(4350,1760)(4366,1759)
    (4383,1758)(4401,1758)(4422,1757)
    (4446,1757)(4473,1757)(4503,1757)
    (4538,1757)(4577,1757)(4620,1758)
    (4669,1758)(4723,1759)(4782,1761)
    (4846,1763)(4916,1766)(4991,1770)
    (5070,1776)(5154,1782)(5241,1790)
    (5331,1800)(5422,1812)(5513,1826)
    (5603,1842)(5690,1859)(5773,1877)
    (5853,1895)(5927,1914)(5996,1932)
    (6060,1949)(6118,1967)(6171,1983)
    (6218,1999)(6261,2014)(6298,2028)
    (6331,2041)(6360,2054)(6386,2066)
    (6408,2077)(6427,2088)(6444,2098)
    (6459,2108)(6472,2118)(6485,2127)
    (6497,2137)(6509,2147)(6521,2157)
    (6535,2168)(6549,2179)(6566,2192)
    (6585,2205)(6606,2220)(6631,2235)
    (6659,2253)(6691,2272)(6727,2293)
    (6768,2316)(6813,2341)(6864,2368)
    (6920,2397)(6980,2429)(7046,2463)
    (7117,2500)(7192,2539)(7271,2580)
    (7353,2623)(7437,2667)(7522,2712)
    (7606,2758)(7689,2804)(7770,2849)
    (7849,2894)(7925,2939)(7998,2982)
    (8068,3024)(8136,3065)(8201,3105)
    (8263,3144)(8323,3183)(8380,3220)
    (8435,3256)(8489,3292)(8540,3327)
    (8590,3361)(8639,3395)(8685,3428)
    (8731,3460)(8775,3492)(8818,3523)
    (8860,3554)(8900,3584)(8939,3613)
    (8977,3641)(9013,3669)(9047,3695)
    (9080,3721)(9111,3745)(9140,3767)
    (9167,3789)(9192,3808)(9215,3826)
    (9235,3842)(9253,3857)(9269,3869)
    (9282,3880)(9293,3889)(9302,3896)
    (9309,3902)(9314,3906)(9318,3909)
    (9320,3911)(9321,3912)(9322,3912)
\path(22,2412)(22,2411)(22,2410)
    (23,2407)(24,2403)(25,2397)
    (27,2389)(30,2379)(33,2366)
    (37,2350)(42,2332)(48,2311)
    (55,2287)(62,2260)(71,2231)
    (81,2199)(92,2165)(104,2129)
    (117,2090)(132,2050)(147,2008)
    (164,1964)(182,1920)(202,1875)
    (223,1829)(245,1783)(269,1736)
    (294,1690)(321,1644)(349,1599)
    (379,1555)(411,1511)(444,1469)
    (480,1428)(518,1388)(558,1351)
    (600,1314)(645,1280)(693,1248)
    (744,1219)(798,1191)(855,1166)
    (916,1144)(981,1125)(1049,1110)
    (1123,1097)(1200,1088)(1283,1083)
    (1370,1082)(1463,1086)(1560,1094)
    (1663,1107)(1771,1125)(1883,1148)
    (2001,1177)(2122,1212)(2214,1241)
    (2306,1273)(2400,1307)(2494,1344)
    (2587,1382)(2679,1422)(2771,1463)
    (2861,1505)(2949,1548)(3035,1592)
    (3118,1636)(3199,1680)(3278,1724)
    (3354,1769)(3427,1813)(3498,1857)
    (3565,1900)(3630,1943)(3692,1986)
    (3752,2028)(3809,2070)(3863,2111)
    (3915,2152)(3965,2192)(4012,2232)
    (4058,2271)(4101,2309)(4143,2347)
    (4183,2385)(4222,2422)(4259,2459)
    (4295,2496)(4330,2532)(4365,2568)
    (4398,2604)(4431,2640)(4464,2676)
    (4497,2712)(4530,2748)(4563,2784)
    (4596,2821)(4630,2857)(4664,2894)
    (4699,2932)(4736,2970)(4773,3008)
    (4812,3047)(4853,3087)(4895,3127)
    (4939,3168)(4985,3210)(5033,3253)
    (5083,3297)(5136,3342)(5191,3387)
    (5249,3434)(5309,3482)(5372,3530)
    (5438,3580)(5507,3631)(5579,3683)
    (5653,3735)(5731,3789)(5811,3843)
    (5894,3899)(5980,3955)(6068,4011)
    (6158,4068)(6251,4125)(6345,4182)
    (6440,4239)(6536,4295)(6633,4351)
    (6730,4406)(6827,4460)(6922,4512)
    (7064,4588)(7201,4659)(7333,4725)
    (7458,4786)(7578,4842)(7691,4893)
    (7797,4940)(7898,4981)(7993,5019)
    (8082,5052)(8166,5082)(8245,5108)
    (8319,5131)(8389,5150)(8455,5167)
    (8518,5181)(8576,5192)(8632,5201)
    (8684,5208)(8734,5213)(8781,5217)
    (8826,5218)(8869,5219)(8909,5217)
    (8947,5215)(8983,5212)(9017,5207)
    (9049,5202)(9080,5197)(9108,5191)
    (9135,5184)(9159,5177)(9182,5170)
    (9203,5164)(9222,5157)(9239,5150)
    (9254,5144)(9268,5139)(9279,5133)
    (9290,5129)(9298,5125)(9305,5121)
    (9310,5118)(9315,5116)(9318,5115)
    (9320,5113)(9321,5113)(9322,5112)
\path(22,6087)(23,6086)(25,6085)
    (28,6084)(32,6081)(37,6078)
    (44,6074)(53,6069)(64,6063)
    (77,6056)(92,6048)(110,6039)
    (129,6028)(152,6017)(176,6004)
    (203,5990)(233,5975)(265,5959)
    (299,5942)(335,5925)(374,5906)
    (415,5887)(458,5867)(503,5847)
    (550,5826)(599,5805)(650,5784)
    (703,5762)(758,5740)(815,5718)
    (873,5696)(933,5675)(996,5653)
    (1060,5631)(1126,5610)(1194,5589)
    (1265,5568)(1338,5548)(1413,5528)
    (1491,5508)(1572,5489)(1656,5470)
    (1743,5451)(1833,5434)(1927,5416)
    (2025,5400)(2126,5384)(2232,5368)
    (2342,5354)(2457,5340)(2576,5327)
    (2701,5314)(2830,5303)(2964,5293)
    (3102,5284)(3245,5277)(3392,5270)
    (3543,5265)(3697,5262)(3823,5261)
    (3950,5260)(4078,5261)(4205,5262)
    (4332,5265)(4459,5268)(4584,5272)
    (4708,5277)(4831,5283)(4952,5289)
    (5072,5297)(5190,5304)(5307,5313)
    (5422,5321)(5536,5331)(5648,5341)
    (5758,5351)(5867,5362)(5975,5373)
    (6082,5385)(6187,5397)(6291,5410)
    (6394,5422)(6496,5436)(6597,5449)
    (6697,5463)(6796,5477)(6894,5491)
    (6991,5506)(7087,5520)(7182,5535)
    (7277,5550)(7370,5566)(7462,5581)
    (7554,5597)(7644,5612)(7733,5628)
    (7821,5643)(7907,5659)(7992,5674)
    (8075,5689)(8157,5705)(8236,5720)
    (8314,5734)(8390,5749)(8463,5763)
    (8534,5777)(8602,5790)(8667,5803)
    (8730,5815)(8790,5827)(8846,5839)
    (8900,5849)(8950,5860)(8997,5869)
    (9040,5878)(9080,5886)(9116,5894)
    (9150,5901)(9179,5907)(9206,5912)
    (9229,5917)(9249,5922)(9266,5925)
    (9281,5928)(9293,5931)(9302,5933)
    (9309,5934)(9314,5935)(9318,5936)
    (9320,5937)(9321,5937)(9322,5937)
\color{black}
\path(22,3012)(9322,3012)
\path(22,7212)(22,12)
\path(9322,7212)(9322,12)
\put(-350,2920){\makebox(0,0)[lb]{$\scriptstyle 0$}}
\put(2200,3312){\makebox(0,0)[lb]{$\scriptstyle -1$}}
\put(4600,2512){\makebox(0,0)[lb]{$\scriptstyle +1$}}
\put(8000,2512){\makebox(0,0)[lb]{$\scriptstyle +1$}}
\put(3200,6748){\makebox(0,0)[lb]{\scriptsize essential spectrum}}
\put(3200,4400){\makebox(0,0)[lb]{\scriptsize discrete spectrum}}
\end{picture}
}

%% file: sf_eps.eepic
\setlength{\unitlength}{0.0003in}
\begingroup\makeatletter\ifx\SetFigFont\undefined%
\gdef\SetFigFont#1#2#3#4#5{%
  \reset@font\fontsize{#1}{#2pt}%
  \fontfamily{#3}\fontseries{#4}\fontshape{#5}%
  \selectfont}%
\fi\endgroup%
{\renewcommand{\dashlinestretch}{30}
\begin{picture}(9622,7239)(0,-10)
\path(300,3012)(9600,3012)
\path(300,7212)(300,12)
\path(9600,7212)(9600,12)
\dashline{60.000}(300,2512)(9600,2512)
\thicklines
\color{gray2}
\path(300,3012)(301,3013)(304,3014)
    (309,3017)(317,3021)(329,3027)
    (344,3036)(365,3047)(390,3060)
    (420,3076)(456,3095)(497,3117)
    (543,3141)(594,3169)(650,3198)
    (711,3230)(776,3265)(845,3301)
    (918,3339)(993,3378)(1070,3419)
    (1149,3460)(1230,3502)(1311,3545)
    (1393,3587)(1475,3629)(1556,3671)
    (1636,3712)(1716,3753)(1794,3792)
    (1870,3831)(1945,3869)(2018,3905)
    (2089,3940)(2158,3974)(2226,4007)
    (2291,4039)(2354,4069)(2415,4098)
    (2474,4126)(2532,4153)(2588,4178)
    (2642,4203)(2694,4226)(2745,4248)
    (2795,4269)(2843,4290)(2890,4309)
    (2936,4328)(2982,4345)(3026,4362)
    (3069,4378)(3112,4394)(3154,4409)
    (3196,4423)(3238,4437)(3283,4452)
    (3329,4466)(3374,4480)(3419,4492)
    (3464,4505)(3509,4516)(3554,4527)
    (3600,4538)(3645,4547)(3691,4556)
    (3737,4564)(3783,4572)(3829,4579)
    (3876,4584)(3923,4590)(3971,4594)
    (4018,4597)(4067,4600)(4115,4601)
    (4164,4602)(4213,4601)(4263,4600)
    (4313,4598)(4363,4595)(4413,4590)
    (4464,4585)(4515,4579)(4566,4572)
    (4617,4563)(4669,4554)(4720,4544)
    (4772,4533)(4823,4521)(4875,4508)
    (4927,4494)(4979,4479)(5031,4463)
    (5083,4446)(5135,4428)(5187,4409)
    (5240,4390)(5292,4370)(5345,4348)
    (5398,4326)(5451,4303)(5505,4279)
    (5559,4255)(5614,4229)(5669,4202)
    (5725,4175)(5725,4174)(5768,4153)
    (5811,4131)(5856,4108)(5900,4084)
    (5946,4059)(5992,4034)(6039,4008)
    (6088,3981)(6137,3953)(6187,3924)
    (6239,3893)(6292,3862)(6346,3830)
    (6402,3796)(6459,3761)(6518,3725)
    (6578,3687)(6641,3648)(6705,3608)
    (6771,3566)(6839,3523)(6909,3478)
    (6981,3431)(7055,3383)(7131,3333)
    (7209,3282)(7289,3230)(7371,3176)
    (7454,3120)(7540,3063)(7627,3005)
    (7716,2946)(7806,2885)(7897,2824)
    (7989,2762)(8081,2700)(8174,2637)
    (8267,2574)(8359,2512)(8451,2449)
    (8541,2388)(8630,2327)(8717,2268)
    (8801,2211)(8883,2155)(8961,2101)
    (9036,2050)(9106,2002)(9172,1957)
    (9234,1914)(9290,1875)(9342,1840)
    (9388,1808)(9430,1779)(9466,1755)
    (9497,1733)(9524,1715)(9545,1700)
    (9563,1688)(9576,1679)(9586,1672)
    (9593,1667)(9597,1664)(9599,1663)(9600,1662)
\path(300,6012)(301,6011)(303,6010)
    (306,6007)(311,6002)(319,5995)
    (330,5986)(344,5974)(361,5959)
    (383,5940)(408,5918)(438,5893)
    (472,5864)(510,5831)(553,5794)
    (601,5753)(653,5709)(709,5661)
    (770,5610)(835,5555)(903,5497)
    (975,5437)(1050,5374)(1129,5308)
    (1210,5241)(1293,5172)(1378,5102)
    (1465,5030)(1553,4958)(1643,4885)
    (1733,4811)(1823,4738)(1914,4665)
    (2005,4593)(2095,4521)(2185,4450)
    (2274,4380)(2362,4312)(2449,4244)
    (2535,4178)(2620,4114)(2704,4051)
    (2786,3990)(2867,3931)(2947,3873)
    (3025,3817)(3101,3763)(3176,3710)
    (3250,3660)(3323,3611)(3394,3564)
    (3464,3519)(3532,3475)(3600,3433)
    (3666,3392)(3732,3354)(3796,3316)
    (3859,3281)(3922,3246)(3984,3213)
    (4046,3182)(4106,3152)(4167,3123)
    (4227,3095)(4286,3068)(4346,3043)
    (4405,3019)(4464,2995)(4523,2973)
    (4582,2952)(4641,2932)(4700,2912)
    (4761,2893)(4822,2874)(4884,2857)
    (4945,2841)(5007,2825)(5070,2810)
    (5133,2797)(5197,2784)(5261,2772)
    (5326,2761)(5392,2751)(5459,2742)
    (5527,2734)(5597,2726)(5667,2720)
    (5739,2714)(5812,2709)(5887,2705)
    (5964,2702)(6042,2700)(6122,2698)
    (6203,2698)(6287,2698)(6372,2699)
    (6459,2701)(6548,2703)(6639,2706)
    (6733,2710)(6827,2715)(6924,2720)
    (7023,2727)(7123,2733)(7224,2741)
    (7327,2749)(7432,2758)(7537,2767)
    (7643,2777)(7749,2787)(7856,2797)
    (7962,2808)(8068,2819)(8174,2831)
    (8278,2842)(8380,2854)(8480,2866)
    (8578,2877)(8673,2889)(8764,2900)
    (8852,2911)(8936,2922)(9015,2932)
    (9090,2942)(9160,2951)(9224,2959)
    (9283,2967)(9337,2975)(9385,2981)
    (9427,2987)(9464,2992)(9496,2997)
    (9523,3001)(9545,3004)(9562,3006)
    (9576,3008)(9586,3010)(9593,3011)
    (9597,3012)(9599,3012)(9600,3012)
\path(300,837)(301,837)(304,838)
    (308,840)(316,843)(327,846)
    (342,852)(362,859)(386,867)
    (415,877)(450,889)(490,903)
    (535,919)(585,937)(640,957)
    (700,979)(764,1002)(833,1027)
    (905,1053)(980,1081)(1058,1110)
    (1139,1140)(1221,1171)(1304,1202)
    (1388,1234)(1473,1267)(1557,1300)
    (1642,1333)(1725,1366)(1808,1399)
    (1889,1432)(1969,1465)(2048,1497)
    (2125,1530)(2200,1562)(2273,1594)
    (2345,1626)(2414,1657)(2482,1689)
    (2548,1720)(2613,1751)(2675,1782)
    (2736,1812)(2796,1843)(2854,1874)
    (2911,1904)(2966,1935)(3021,1966)
    (3074,1998)(3126,2029)(3178,2061)
    (3229,2094)(3279,2127)(3329,2160)
    (3378,2194)(3427,2228)(3476,2264)
    (3525,2299)(3570,2334)(3616,2368)
    (3661,2404)(3707,2440)(3752,2476)
    (3798,2514)(3844,2552)(3890,2590)
    (3936,2630)(3983,2670)(4029,2711)
    (4077,2752)(4124,2794)(4172,2837)
    (4220,2881)(4268,2925)(4317,2970)
    (4366,3015)(4415,3061)(4464,3107)
    (4514,3154)(4564,3202)(4614,3250)
    (4665,3298)(4715,3346)(4766,3395)
    (4817,3444)(4868,3493)(4919,3542)
    (4970,3592)(5022,3641)(5073,3690)
    (5124,3739)(5175,3788)(5226,3837)
    (5276,3885)(5327,3934)(5377,3981)
    (5427,4029)(5477,4076)(5527,4122)
    (5576,4168)(5625,4213)(5674,4258)
    (5723,4302)(5771,4345)(5818,4388)
    (5866,4430)(5913,4471)(5959,4512)
    (6006,4551)(6052,4590)(6098,4628)
    (6143,4666)(6188,4702)(6233,4738)
    (6278,4773)(6323,4807)(6367,4841)
    (6411,4874)(6456,4906)(6500,4937)
    (6549,4971)(6598,5004)(6648,5036)
    (6698,5068)(6748,5099)(6799,5129)
    (6850,5159)(6902,5187)(6955,5216)
    (7009,5244)(7064,5271)(7120,5298)
    (7177,5325)(7236,5351)(7296,5377)
    (7358,5403)(7421,5429)(7486,5455)
    (7553,5481)(7622,5506)(7693,5532)
    (7765,5558)(7840,5583)(7916,5609)
    (7994,5634)(8073,5660)(8154,5686)
    (8236,5711)(8319,5736)(8402,5761)
    (8486,5786)(8570,5810)(8654,5834)
    (8736,5858)(8817,5880)(8896,5902)
    (8973,5923)(9046,5943)(9117,5962)
    (9183,5980)(9245,5996)(9302,6011)
    (9354,6025)(9401,6037)(9442,6047)
    (9478,6057)(9509,6064)(9535,6071)
    (9555,6076)(9571,6080)(9583,6083)
    (9591,6085)(9596,6086)(9599,6087)(9600,6087)
\color{black}
\put(0,2920){\makebox(0,0)[lb]{$\scriptstyle 0$}}
\put(-250,2450){\makebox(0,0)[lb]{$\scriptstyle-\eps$}}
\put(3950,2150){\makebox(0,0)[lb]{$\scriptstyle +1$}}
\put(7700,2150){\makebox(0,0)[lb]{$\scriptstyle -1$}}
\end{picture}
}

%% file: app_D.tex
\chapter{Spin$^c$ Manifolds}\label{app:spinc}

In this appendix we give a summary of the constructions related to
so-called \spinc manifolds. First of all, we make some algebraic
remarks concerning the group $\Spinc$ and its representation
theory. As we assume some familiarity with the definition of
Clifford algebras, the Spin group, and their representation
theory, the presentation in Section \ref{spc} will be rather
sketchy, not containing proofs. We refer to the wide range of
literature, in particular Lawson \& Michelsohn \cite{LM}, Ch.~I,
or Berline et al. \cite{BGV}, Ch.~3, for a more detailed
treatment.

However, differences between \spinc and spin become more
intriguing when we consider the geometric framework in Section
\ref{spc:st}. Here, we shall go in more detail since an
understanding of the special nature of \spinc structures is an
important prerequisite for studying Seiberg-Witten theory. Once we
have established a suitable setting in which to define the
so-called \spinc Dirac operator, the discussion of the related
analytic properties proceeds in almost the same manner as for the
spin Dirac operator. Hence, in Section \ref{spc:do}, we will again
simply state the results, giving references for the proofs.

Section \ref{met:dep} contains some material related to the
question of how the \spinc Dirac operator depends on the metric.
This will be needed when we compare the structures of
Seiberg-Witten moduli spaces for different Riemannian metrics. It
is placed here in order not to interrupt the line of argument in
the main part of this thesis.

\section{The group \Spinc} \label{spc}

\noindent\textbf{The Clifford algebra.}\index{=@$\cl(V)$,
Clifford algebra} Let $(V,g)$ be a Euclidean vector space with
corresponding Clifford algebra $\cl(V)$. If $V=\R^n$, we shall
simply write $\cl_n$. Recall that $\cl(V)$ is the associative
real algebra generated by 1 and elements of $V$ subject to the
relations
\[
v\cdot w + w\cdot v= -2g(v,w)1,\quad v,w\in V.
\]
We denote the complexified Clifford algebra by
\index{=@$\clc(V)$, complex Clifford algebra}
\[
\clc(V):=\cl(V)\otimes\C.
\]
Then $\clc(V)$ can be interpreted as the Clifford algebra of the
complex vector space $V\otimes\C$ endowed with the complex
bilinear extension of $g$.

The representation theory of the complex Clifford algebra turns
out to be rather simple. Let $n=\dim V$.
\begin{itemize}
\item If $n=2k$, then there exists up to isomorphism exactly one
irreducible $\clc(V)$-module, denoted by $\gD$, which has
dimension $2^k$.
\item If $n=2k+1$, then there are two irreducible
$\clc(V)$-modules, $\gD^+$ and $\gD^-$, of dimension $2^k$.
\end{itemize}\index{=@$\gD$, irreducible Clifford
module} We recall how to distinguish the inequivalent
representations in the case that $n$ is odd. Let \index{=@$\go^c$,
complex volume element}
\begin{equation}\label{compl:vol}
\go^c := i^{[\frac{n+1}{2}]} e_1\cdots e_n
\end{equation}
be the {\em complex volume element} of the Clifford algebra.
Here, for a real number $a$, $[a]$ denotes the greatest integer
smaller than $a$. Then $\go^c$ is an involution (irrespective of
the parity of $n$), i.e., $(\go^c)^2=1$. Therefore, each
$\clc(V)$-module splits into $\pm 1$ eigenspaces of $\go^c$. If
$n$ is odd, then $\go^c$ is central and therefore, its eigenspaces
are $\clc(V)$-invariant. Then $\gD^\pm$ is the irreducible module
on which $\go^c$ acts as $\pm \id$.\\

\noindent\textbf{Spin and $\Spinc$.}\index{=@$\Spin(V)$, spin
group} Recall that the Spin group associated to $(V,g)$ is defined
by
\[
\Spin(V):=\bigsetdef{v_1\cdot\ldots\cdot v_m}{
\text{$m$ even},\,|v_i|=1}\subset \cl(V)^*\,,
\]
where $\cl(V)^*$ denotes the group of units in $\cl(V)$. It turns
out that $\Spin(V)$ is a compact, connected Lie group which is
simply connected if $\dim V\ge 3$. We shall always write
$\Spin_n:=\Spin(\R^n)$. The subspace $V$ of $\cl(V)$ is invariant
with respect to conjugation by elements of $\Spin(V)$. Moreover,
it turns out that for each $g\in\Spin(V)$, the endomorphism
\[
V\to V,\quad v\mapsto g v g^{-1}
\]
is, in fact, in $\SO(V)$ and that the Lie group homomorphism
obtained in this way, say,
\begin{equation}\label{xi0}
\xi_0:\Spin(V)\longrightarrow \SO(V),
\end{equation}
is surjective with $\ker \xi_0=\{\pm 1\}$. Therefore, $\xi_0$
gives a twofold covering of $\SO(V)$ which is universal if $\dim
V\ge 3$.
\begin{dfn}\index{=@$\Spinc(V)$, complex spin group} The {\em
complex spin group}, $\Spinc(V)$, is the group generated by
$\Spin(V)$ and $\U_1$ inside the group $\clc(V)^*$. If $V=\R^n$,
we write $\Spinc_n:=\Spinc(\R^n)$.
\end{dfn}
Since $\U_1$ lies in the center of $\clc(V)$ and $\Spin(V)\cap
\U_1 =\{\pm 1\}$, it is clear that
\[
\Spinc(V)=\Spin(V)\times_{\Z_2}\U_1\,,
\]
where $\Z_2=\{\pm 1\}$ acts diagonally. This action is free, hence
$\Spinc(V)$ inherits the structure of a compact, connected real
Lie group. The covering map $\xi_0:\Spin(V)\to\SO(V)$ gives rise
to an exact sequence of Lie groups
\begin{equation}\label{xi0c}
1\longrightarrow\U_1\longrightarrow\Spinc(V)
\overset{\xi_0^c}{\longrightarrow}\SO(V)\longrightarrow 1\,,
\end{equation}
where $\xi_0^c([g,z]):=\xi_0(g)$. Note that this is well-defined
since $\ker \xi_0=\Z_2$. Defining $\gz^c([g,z]):=z^2$, we obtain
another exact sequence
\begin{equation}\label{zeta}
1\longrightarrow\Spin(V)\longrightarrow\Spinc(V) \overset{
\gz^c}{\longrightarrow}\U_1\longrightarrow 1\,.
\end{equation}
The maps $\xi_0^c$ and $ \gz^c$ induce a two sheeted covering of
Lie groups
\begin{equation}\label{xi}
\xi:\Spinc(V)\to \SO(V)\times \U_1,\quad [g,z]\mapsto
(\xi_0(g),z^2)\,.
\end{equation}

\noindent\textbf{Spin representation.} Let us now turn to the
representation theory of the group \Spinc.
\begin{dfn}\label{spinrep}
Let $W$ be an irreducible $\clc(V)$-module. By restricting the
action of $\clc(V)$ we get a representation of $\Spinc(V)$ on
$W$. The representation obtained in this way is called a {\em spin
representation}.
\end{dfn}
Using the classification of irreducible $\clc$-modules and
analysing the restriction to $\Spinc$ yields the following:
\begin{itemize}
\item If $n$ is even, then the decomposition $\gD=\gD^+\oplus
\gD^-$ of the unique irreducible $\clc$-module is
\Spinc-invariant, inducing the irreducible {\em half spin
representations} of \Spinc;
\item if $n$ is odd, then the two non-isomorphic, irreducible
$\clc$-modules give rise to equivalent irreducible representations
of \Spinc.
\end{itemize}

\section{Spin$^c$ structures} \label{spc:st}
\index{spin$^c$ manifolds!spic$^c$
structure|(}\index{=@$P_{\SO}(g)$, principal $\SO$-bundle}

Let $(M,g)$ be an oriented, $n$-dimensional Riemannian manifold,
and let $P_{\SO}(g)$ denote its principal $\SO_n$-bundle of
oriented, orthonormal frames. Before we define the notion of a
\spinc structure, we first recall the definition of a spin
manifold.
\begin{dfn}\index{=@$(M,\eps)$, spin manifold} A {\em spin
structure} $\eps$ on $M$ is a principal $\Spin_n$-bundle
$P_{\Spin}(\eps)$ together with a bundle map
$\xi:P_{\Spin}(\eps)\to P_{\SO}$, which is Spin-equivariant with
respect to the two sheeted covering $\xi_0:\Spin_n\to\SO_n$.
Here, equivariance means that $\xi(pg)=\xi(p)\xi_0(g)$ for every
$p\in P_{\Spin}(\eps)$ and every $g\in\Spin_n$. The pair
$(M,\eps)$ is called a {\em spin manifold}.
\end{dfn}
Imitating the above definition, we now introduce the notion of a
\spinc structure:
\begin{dfn}\index{=@$(M,\gs)$, spin$^c$
manifold}\index{=@$P_{\Spinc}(\gs)$, $\Spinc$-bundle associated to
a \spinc structure} A {\em spin$^c$ structure} on $M$, denoted by
$\gs$, consists of a principal $\Spinc_n$-bundle $P_{\Spinc}(\gs)$
together with a bundle map $\xi^c:P_{\Spinc}(\gs)\to P_{\SO}$
which is $\Spinc_n$-equivariant with respect to the homomorphism
$\xi_0^c:\Spinc_n\to\SO_n$. The pair $(M,\gs)$ is called a {\em
spin$^c$ manifold}.
\end{dfn}
Another bundle is encoded in the definition of a \spinc manifold
$(M,\gs)$. Recall that to any principal bundle we can associate
new principal bundles and vector bundles via group homomorphisms
and representations of the structure group. Hence, via the map
$\gz^c:\Spinc_n\to \U_1$ of \eqref{zeta} we obtain the principal
$\U_1$-bundle
\[
P_{\U_1}(\gs):=P_{\Spinc}(\gs)\times_{ \gz^c}\U_1
\]
which is the quotient of $P_{\Spinc}(\gs)\times\U_1$ with respect
to $(p,z)\sim (pg,\gz^c(g^{-1})z)$ endowed with the right action
of $\U_1$ on the second factor. Equivalently, we may consider the
Hermitian line bundle
\[
L(\gs):=P_{\Spinc}(\gs)\times_{ \gz^c}\C.
\]
\begin{dfn}\label{canline}\index{=@$L(\gs)$, canonical line
bundle}\index{=@$c(\gs)$, canonical class}\index{spin$^c$
manifolds!spic$^c$ structure!canonical line bundle}\index{spin$^c$
manifolds!spic$^c$ structure!canonical class} The line bundle
$L(\gs)$ is called the {\em canonical line bundle} of the \spinc
structure $\gs$. Its topological first Chern class is called the
{\em canonical class} and will be denoted by $c(\gs)\in
H^2(M;\Z)$.
\end{dfn}
\begin{remark*}\quad
\begin{enumerate}
\item In the main part of this thesis, we shall also refer to the
image of $c(\gs)$ in $H^2(M;\R)$ as the canonical class of $\gs$.
As the meaning should be understood from the context, we can
avoid a notational distinction like ``$c^{\text{top}}(\gs)$" and
``$c^{\text{geom}}(\gs)$".
\item We note that a \spinc structure could equally well be
defined by the following data:
\begin{itemize}
\item A principal $\U_1$-bundle $P_{\U_1}(\gs)\to M$,
\item a principal $\Spinc_n$-bundle $P_{\Spinc}(\gs)\to M$,
\item a $\Spinc_n$-equivariant bundle map $\xi:P_{\Spinc}(\gs)\to
P_{\SO}\times P_{\U_1}(\gs)$.
\end{itemize}
Here, $P_{\SO}\times P_{\U_1}(\gs)$ denotes the fibre product of
the bundles $P_{\SO}\to M$ and $P_{\U_1}(\gs)\to M$. Note that
$\xi:P_{\Spinc}(\gs)\to P_{\SO}\times P_{\U_1}(\gs)$ is a twofold
covering.
\end{enumerate}
\end{remark*}

\begin{dfn}\index{=@$\text{Spin}^c(M)$, equivalence classes} Two
\spinc structures $\gs$ and $\gs'$ on $M$ are called {\em
equivalent} if there exists a bundle isomorphism
$\gF:P_{\Spinc}(\gs)\to P_{\Spinc}(\gs')$ inducing a commutative
diagram:
\begin{equation}\label{spinc:eq} \begindc[25]
\obj(3,1){$M$}[M]\obj(3,3){$P_{\SO}$}[PSO]
\obj(1,5){$P_{\Spin}(\gs)$}[PSp]
\obj(5,5){$P_{\Spin}(\gs')$}[PSpP] \mor{PSO}{M}{}
\mor{PSp}{PSO}{$\xi^c$} \mor{PSp}{PSpP}{$\gF$}
\mor{PSpP}{PSO}{${\xi^c}'$}[\atright,\solidarrow] \mor{PSp}{M}{}
\mor{PSpP}{M}{}
\enddc\end{equation}
The set of equivalence classes of \spinc structures shall be
denoted by spin$^c$($M$).
\end{dfn}

Before we study possible topological obstructions to the existence
of \spinc structures, let us first consider some examples.
\begin{example} \label{spin=>spinc}
\textit{Each spin structure induces a canonical \spinc structure.} \\
Given a spin manifold $(M,\eps)$, we can form a $\Spinc_n$-bundle
by letting
\[
P_{\Spinc}:=P_{\Spin}(\eps)\times_{\Z_2}\U_1\,,
\]
where $\U_1$ denotes the trivial bundle $M\times \U_1$, and
$\Z_2$ acts diagonally. The map $\gz=z^2:\U_1\to\U_1$ and the
bundle map $P_{\Spin}(\eps)\to P_{\SO}$ are $\Z_2$-invariant thus
giving a $\Spinc_n$-equivariant bundle map
\[
\xi:P_{\Spin}(\eps)\times_{\Z_2}\U_1\longrightarrow P_{\SO}\times
\U_1\,.
\]
Therefore, we obtain the so-called {\em canonical spin$^c$
structure}, $\gs(\eps)$, on $M$. Clearly, the canonical line
bundle $L(\gs(\eps))$ is the trivial bundle so that in particular,
$c(\gs(\eps))=0$.
\end{example}
We state another example which is important in four
dimensional Seiberg-Witten theory, especially when dealing with
Hermitian or K\"{a}hler manifolds:
\begin{example}
\textit{Every almost complex manifold has a canonical \spinc
structure}. Let $M$ be a $2k$-dimensional Riemannian manifold
which admits a compatible almost complex structure, i.e., an
orthogonal bundle map $J:TM\to TM$ with $J^2=-\id$. Then $(TM,J)$
carries the structure of a complex vector bundle of rank $k$ over
$M$ with an induced Hermitian metric. Therefore, the structure
group of $TM$ can be reduced to $\U_k$, i.e., we can construct a
principal $\U_k$-bundle $P_{\U_k}$ on $M$ such that $(TM,J)$ is
the vector bundle associated to $P_{\U_k}$ via the standard
representation of $\U_k$ on $\C^k$. The inclusion
$i:\U_k\hookrightarrow \SO_{2k}$ gives rise to a group
homomorphism $(i,\det):\U_k\to\SO_{2k}\times\U_1$. One checks
that the fundamental groups are related via
$(i,\det)_*\pi_1(\U_k)\subset \xi_*\pi_1(\Spinc_{2k})$ so that
there exists a unique lifting
\[
\begindc
\obj(1,1){$\U_k$}[Uk] \obj(4,1){$\SO_{2k}\times\U_1$}[SOxU]
\obj(4,3){$\Spinc_{2k}$}[Spc]
\mor{Uk}{SOxU}{$(i,\det)$}[\atright,\solidarrow]
\mor{Spc}{SOxU}{$\xi$}\mor{Uk}{Spc}{$j$}[\atleft,\dasharrow]
\enddc
\]
We now obtain a canonical \spinc structure $\gs_J$ by letting
\[
P_{\Spinc}(\gs_J):=P_{\U_k}\times_j \Spinc_{2k}.
\]
It turns out that the canonical line bundle of the \spinc
manifold $(M,\gs_J)$ is precisely the dual of $K$, the canonical
line bundle of the almost complex manifold $(M,J)$, i.e.,
\[
L(\gs_J)=\gL^{k,0}TM=K^*.
\]
\end{example}

\noindent\textbf{Principal bundles and \v{C}ech cohomology.} To
understand the topological obstructions to the existence of
\spinc structures, we briefly recall the interaction between the
local description of principal bundles and \v{C}ech cohomology.
For a more detailed exposition, the reader is referred to
Hirzebruch \cite{Hi}, Ch.~I. \index{=@$H^1(M;\Z)$, |v{C}ech
cohomology}

Let $\{U_{\ga}\}$ be a good open cover of a manifold $M$, i.e., a
covering by open sets such that all intersections are
contractible. Suppose $P\to M$ is a principal $G$-bundle, where
$G$ is a Lie group. Since $P$ can be trivialized over
$\{U_{\ga}\}$, we obtain a corresponding family of transition
functions $\{g_{\ga\gb}:U_{\ga}\cap U_{\gb}\to G\}$ fulfilling
the cocycle condition
\begin{equation}\label{coc:1}
g_{\ga\gb} g_{\gb\gamma} = g_{\ga\gamma}\quad\text{on } U_\ga\cap
U_\gb\cap U_\gamma\,.
\end{equation}
On the other hand, there is a well-known procedure to define a
principal $G$-bundle given such a family of transition functions.
Two cocycles $\{g_{\ga\gb}\}$ and $\{g'_{\ga,\gb}\}$ define
isomorphic $G$-bundles if and only if there exists a family
$\{\gF_\ga:U_\ga\to G\}$ such that
\begin{equation}\label{coc:2}
g_{\ga\gb}'=\gF_\ga g_{\ga\gb}\gF_\gb^{-1}\quad\text{on
}U_\ga\cap U_\gb\,.
\end{equation}
More generally, let $\mu:G\to H$ be a Lie group homomorphism and
$Q\to M$ a principal $H$-bundle with transition functions
$\{h_{\ga\gb}:U_{\ga}\cap U_{\gb}\to H\}$. A bundle map $\gF:G\to
H$ which is equivariant with respect to $\mu$ turns out to be the
same as a family $\{\gF_\ga:U_\ga\to H\}$ satisfying
\begin{equation}\label{coc:map}
h_{\ga\gb}=\gF_\ga\mu(g_{\ga\gb})\gF_\gb^{-1}\quad\text{on
}U_\ga\cap U_\gb\,.
\end{equation}
Let $\underline G$ be the sheaf of germs of differentiable
$G$-valued functions on $M$. Then formula \eqref{coc:1} is
exactly the condition for the \v{C}ech 1-chain $\{g_{\ga\gb}\}$
to define an element
\[
[g_{\ga\gb}]\in H^1(M;\underline G)\,.\footnote{If $G$ is
non-abelian, one can define $H^1(M;\underline G)$ in the same way
as for abelian $G$---with the difference that it will not be a
group but only a pointed set (with the trivial $G$-bundle as a
base point). The long exact cohomology sequence associated to an
exact sequence of sheaves (see \cite{Hi}, Sec.~I.2) then
terminates at the $H^2$ level if a non-abelian group is involved.
Note that this sequence is then an exact sequence of pointed
sets.}
\]
Furthermore, equation \eqref{coc:2} is equivalent to
$[g_{\ga\gb}]=[g_{\ga\gb}']\in H^1(M;\underline G)$. Hence, there
is a natural correspondence between the isomorphism classes of
principal $G$-bundles and the first cohomology $H^1(M;\underline
G)$.
\begin{example}\label{exp:seq}
Consider the exponential sequence
\[
0\to \Z\longrightarrow\underline\R \overset{\exp(2\pi
ix)}{\longrightarrow} \underline\U_1\to 0.
\]
As an exact sequence of sheaves it yields a long exact sequence
in cohomology. The sheaf of germs of differentiable, $\R$-valued
functions admits partitions of unity and therefore,
$H^k(M;\underline{\R})=0$ for every $k\ge 1$.\footnote{Note the
difference between $H^k(M;\underline\R)$ and the cohomology
groups $H^k(M;\R)$, associated to the sheaf $\R$ of locally
constant functions. The latter cohomology groups are isomorphic
to the deRham cohomology. Hence, in general, $H^k(M;\R)\neq 0$.}
This shows that for $k\ge 1$, the connecting homomorphism
\[
\gd^k:H^k(M;\underline \U_1)\to H^{k+1}(M;\Z)
\]
is an isomorphism. In particular, the set of isomorphism classes
of principal $\U_1$-bundles is isomorphic to $H^2(M;\Z)$. It
turns out (cf. Wells \cite{W}, Sec.~III.4) that, if $P\to M$ is a
$\U_1$-bundle with transition functions $\{\gl_{\ga\gb}:U_\ga\cap
U_\gb\to \U_1\}$, then
\[
\gd^1([\gl_{\ga\gb}])=c_1(P),
\]
the latter denoting first Chern class of $P$.\\
\end{example}

\noindent\textbf{Local description of spin and \spinc structures.}
In this paragraph we shall treat only spin structures since the
discussion for \spinc structures is completely analogous.

Let $\{g_{\ga\gb}:U_\ga\cap U_\gb\to\SO_n\}$ be a cocycle
defining the $\SO_n$ bundle of $M$. A spin structure consists of a
principal $\Spin_n$-bundle $P_{\Spin}\to M$, given by a cocycle
$\{h_{\ga\gb}:U_\ga\cap U_\gb\to \Spin_n\}$ and a bundle map
$\xi:P_{\Spin}\to P_{\SO}$ which is equivariant with respect to
the Lie group homomorphism $\xi_0:\Spin_n\to \SO_n$. According to
\eqref{coc:map}, the map $\xi$ is equivalently given by a family
$\{\xi_\ga:U_\ga\to \SO_n\}$ satisfying
\begin{equation}\label{spin:cov}
g_{\ga\gb}(x)=\xi_\ga(x)\xi_0\big(h_{\ga\gb}(x)\big)
\xi_\gb^{-1}(x), \quad x\in U_\ga\cap U_\gb.
\end{equation}
Since all $U_\ga$ are contractible and $\xi_0$ is a
covering\footnote{When discussing \spinc structures, the
corresponding map $\xi^c_0:\Spinc_n\to\SO_n$ is a $\U_1$-fibration
and thus also has the lifting property.} map, we can find maps
$\gF_\ga:U_\ga\to \Spin_n$ such that
\[
\xi_0\circ\gF_\ga=\xi_\ga.
\]
Letting $h_{\ga\gb}':=\gF_\ga h_{\ga\gb}\gF_\gb^{-1}$ we obtain a
family of transition function defining a $\Spin_n$-bundle
$P_{\Spin}'\to M$ and a bundle isomorphism $\gF:P_{\Spin}\to
P_{\Spin}'$. Hence, $P_{\Spin}'$ and $\xi':=\xi\circ \gF^{-1}$
define a spin structure which is equivalent to the original one.
However, equation \eqref{spin:cov} is simplified since
\[
g_{\ga\gb}=\xi_\ga\xi_0\big(h_{\ga\gb}\big) \xi_\gb^{-1}=
\xi_0(\gF_\ga)\xi_0(h_{\ga\gb})\xi_0(\gF_\gb^{-1})
=\xi_0(h_{\ga\gb}').
\]
Hence, modulo equivalence, a spin structure is always given by a
cocycle $\{h_{\ga\gb}:U_\ga\cap U_\gb\to \Spin_n\}$ lifting the
family $\{g_{\ga\gb}\}$ via $\xi_0$. Note that this implies that
the bundle map $\xi:P_{\Spin}\to P_{\SO}$ is given by the family
$\{\xi_\ga=\id:U_\ga\to \Spin_n\}$. Mutatis mutandis, the same
holds for \spinc structures. For brevity, transition functions
with the above property will be called \emph{fitting} cocycles.
\begin{remark*}
Notice, however, that isomorphic $\Spin_n$-bundles covering
$P_{\SO}$ can give rise to nonequivalent spin structures if the
corresponding diagram \eqref{spinc:eq} is not commutative. An
example for this phenomenon is given by Milnor in \cite{Mil:SS}
(see also Lawson \& Michelsohn \cite{LM}, II.1.14).\\
\end{remark*}

\noindent\textbf{The set spin$^c$($\boldsymbol{M}$).} The above
local description yields a possibility to analyse the set
spin$^c(M)$ of all possible \spinc structures on a manifold $M$.

\begin{prop}\label{spinc(M)}\index{spin$^c$ manifolds!spic$^c$
structure!canonical line bundle} Let $M$ be a manifold admitting
a \spinc structure. Then there exists a natural action
\[
\text{\rm spin}^c(M)\times H^1(M;\underline \U_1)\to\text{\rm
spin}^c(M),
\]
denoted by
\[
(\gs,L)\mapsto \gs\otimes L,
\]
which is free and transitive. Hence, up to fixing a \spinc
structure, {\rm spin$^c(M)$} is isomorphic\footnote{A set
carrying a free and transitive group action is usually called a
\emph{torsor}. Thus, spin$^c(M)$ is an $H^2(M;\Z)$ torsor.} to
$H^2(M;\Z)$, cf. Example \ref{exp:seq}.  Moreover,
\[
L(\gs\otimes L)=L(\gs)\otimes L^2\quad\text{or,
equivalently,}\quad c(\gs\otimes L)=c(\gs)+2c_1(L)\,,
\]
where $c_1(L)$ is the first Chern class of $L$.
\end{prop}

\begin{proof}
Let $\{U_\ga\}$ be a good open cover of $M$, and let
$\{g_{\ga\gb}\}$ define the principal $\SO_n$-bundle of $M$.
Suppose $\gs$ is a \spinc structure, and let $\{h_{\ga\gb}\}$ be a
fitting cocycle, i.e., $\xi_0^c\circ h_{\ga\gb}=g_{\ga\gb}$,
where $\xi_0^c$ is the group homomorphism of \eqref{xi0c}.
Moreover, let $L$ be a Hermitian line bundle and
$\{\gl_{\ga\gb}:U_\ga\cap U_\gb\to \U_1\}$ a set of transition
functions. Then $\{h_{\ga\gb}\gl_{\ga\gb}\}$ fulfills the cocycle
condition and, since $\U_1=\ker \xi_0^c$, also lifts the family
$\{g_{\ga\gb}\}$. Hence, it is a fitting cocycle for a \spinc
structure which we denote by $\gs\otimes L$. One readily checks
that isomorphic line bundles give equivalent \spinc structure.
Thus, we obtain a well-defined right action of $H^1(M,\U_1)$ on
spin$^c(M)$.

Given another \spinc structure $\gs'$, with fitting cocycle
$\{h_{\ga\gb}'\}$, we have
\[
\xi_0^c\circ h_{\ga\gb}=\xi_0^c\circ h_{\ga\gb}'=g_{\ga\gb}\,.
\]
Since $\ker \xi_0^c=\U_1$, there exists a family
$\{\gl_{\ga\gb}:U_\ga\cap U_\gb\to \U_1\}$ such that
\[
h_{\ga\gb}'=h_{\ga\gb}\gl_{\ga\gb}\,.
\]
Clearly, $\{\gl_{\ga\gb}\}$ fulfills the cocycle condition
\eqref{coc:1} hence defining an element in $H^1(M;\underline
\U_1)$. Therefore, the action is transitive.

Now suppose that $\gs=\gs\otimes L$ for some line bundle $L$. This
implies that the corresponding principal $\Spinc_n$-bundles are
isomorphic, i.e., if $\{h_{\ga\gb}\}$ and $\{\gl_{\ga\gb}\}$ are
transition functions corresponding to $P_{\Spinc}(\gs)$ and $L$,
we can find a family $\{\gF_\ga:U_\ga\to \Spinc_n\}$ such that
\[
h_{\ga\gb}\gl_{\ga\gb} = \gF_\ga h_{\ga\gb} \gF_\gb^{-1}\,.
\]
Since $\{h_{\ga\gb}\}$ is a fitting cocycle, the bundle map
$\xi^c:P_{\Spinc}(\gs)\to P_{\SO}$ is given by the family
$\{\xi^c_\ga:=\id:U_\ga\to \SO_n\}$. As the same holds for
$\{h_{\ga\gb}\gl_{\ga\gb}\}$, commutativity of the diagram
\eqref{spinc:eq} implies that
\[
\xi_0^c\circ \gF_\ga = 1,\quad \text{i.e.,}\quad\gF_\ga:U_\ga\to
\ker\xi_0^c=\U_1.
\]
Since $\U_1$ lies in the center of $\Spinc_n$, this shows that
necessarily,
\[
\gl_{\ga\gb} = \gF_\ga\gF_\gb^{-1}.
\]
Hence, $L$ is isomorphic to the trivial line bundle. This proves
that the action of $H^1(M;\underline \U_1)$ on spin$^c(M)$ is
free.
\end{proof}

We now want to understand the additional structure of spin$^c(M)$
in the case of spin manifolds.

\begin{example}
Let $M$ be a spin manifold and assume that $H^2(M;\Z)$ has no
2-torsion elements. Let $\eps$ and $\eps'$ be two spin structures
on $M$ and let $\gs(\eps)$ and $\gs(\eps')$ denote the
corresponding canonical \spinc structures as defined in Example
\ref{spin=>spinc}. According to Proposition \ref{spinc(M)} there
exists a Hermitian line bundle $L$ that fulfills
\[
\gs(\eps')=\gs(\eps)\otimes L\quad\text{ and }\quad
c(\gs(\eps'))=c(\gs(\eps))+2c_1(L).
\]
Since $c(\gs(\eps))=c(\gs(\eps'))=0$, this yields $2c_1(L)=0$.
Hence, according to our assumption $c_1(L)=0$. Therefore, $L$ is
isomorphic to the trivial line bundle, i.e., $\gs(\eps)$ is
equivalent to $\gs(\eps')$. This shows that all spin structures on
$M$ induce equivalent \spinc structures. We conclude that on a
spin manifold $M$ there is a canonical ``origin" of spin$^c(M)$,
whenever there are no 2-torsion elements in $H^2(M;\Z)$.\\
\end{example}

\noindent\textbf{Existence of \spinc structures.} The interplay
between the local description of principal bundles and \v{C}ech
cohomology lies at the heart of understanding possible
topological obstructions to the existence of \spinc structures.

The canonical group homomorphisms described in Section \ref{spc}
can be assembled in the following commutative diagram, which has
exact rows and columns.
\[
\begin{CD}
    @.   1    @.   1        @.            @.           \\
  @.     @VVV      @VVV                   @.         @.\\
  1 @>>> \Z_2 @>>> \Spin_n  @>{\xi_0}>>   \SO_n @>>> 1 \\
  @.     @VVV      @VVV                   @|         @.\\
  1 @>>> \U_1 @>>> \Spinc_n @>{\xi_0^c}>> \SO_n @>>> 1 \\
  @.     @VV{\gz}V @VV{ \gz^c}V              @.         @.\\
    @.   \U_1 @=   \U_1     @.                  @.     \\
  @.     @VVV      @VVV                   @.         @.\\
    @.   1    @.   1        @.                  @.
\end{CD}
\]
The corresponding commutative diagram in \v{C}ech cohomology reads
\begin{equation}\label{diag}\begindc[4]
\obj(81,0){$\vdots$}[Z4]\obj(11,10){$H^1(M;\underline \U_1)$}[A1]
\obj(34,10){$H^1(M;\underline \U_1)$}[A2]
\obj(81,10){$H^2(M;\underline \U_1)$}[A4]
\obj(11,20){$H^1(M;\underline \U_1)$}[B1]
\obj(34,20){$H^1(M;\underline \Spinc)$}[B2]
\obj(58,20){$H^1(M;\underline \SO)$}[B3]
\obj(81,20){$H^2(M;\underline \U_1)$}[B4]
\obj(11,30){$H^1(M;\Z_2)$}[C1] \obj(34,30){$H^1(M;\underline
\Spin)$}[C2] \obj(58,30){$H^1(M;\underline \SO)$}[C3]
\obj(81,30){$H^2(M;\Z_2)$}[C4] \obj(11,40){$\vdots$}[D1]
\mor(17,10)(28,10){}[\atright,\solidline]
\mor(17,11)(28,11){}[\atright,\solidline] \mor{B1}{B2}{}
\mor{B2}{B3}{} \mor{B3}{B4}{$\gd_{\xi_0^c}^1$} \mor{C1}{C2}{}
\mor{C2}{C3}{} \mor{C3}{C4}{$\gd_{\xi_0}^1$} \mor{D1}{C1}{}
\mor{C1}{B1}{} \mor{B1}{A1}{} \mor{C2}{B2}{} \mor{B2}{A2}{}
\mor(60,30)(60,20){}[\atright,\solidline]
\mor(59,30)(59,20){}[\atright,\solidline] \mor{C4}{B4}{}
\mor{B4}{A4}{} \mor{A4}{Z4}{}
\cmor((11,7)(12,3)(16,2)(31,2)(41,2)(45,3)(46,7)(46,12)(46,33))
\pup(48,13){$\gd_{\gz}^1$}[\solidline]
\cmor((46,33)(47,37)(51,38)(61,38)(76,38)(80,37)(81,33))
\pdown(0,0){}
\enddc\end{equation}
Here, the $\gd_{\dots}^1$ denote the various connecting
homomorphisms. According to the preceding considerations, we can
interpret a spin structure on $M$ as an element $[h_{\ga\gb}]\in
H^1(M;\underline \Spin_n)$ which is mapped to $[g_{\ga\gb}]$,
i.e.,
\[
\xi_0[h_{\ga\gb}]:=[\xi_0\circ h_{\ga\gb}]=[g_{\ga\gb}].
\]
It now follows from the exactness of diagram \eqref{diag} that
there exists a spin structure on $M$ if and only if
\begin{equation}\label{spin:exist}
[w_{\ga\gb\gamma}]:=\gd_{\xi_0}^1[g_{\ga\gb}]=0\in H^2(M;\Z_2)\,
\end{equation}
In the same way, we conclude that there
exists a \spinc structure on $M$ if and only if
\begin{equation}\label{spinc:exist1}
\gd_{\xi_0^c}^1[g_{\ga\gb}]=0\in H^2(M;\underline \U_1)\cong
H^3(M;\Z)\,,
\end{equation}
where the isomorphism is the one described in Example
\ref{exp:seq}. Since $\gd_{\xi_0^c}^1[g_{\ga\gb}]$ is the image of
$[w_{\ga\gb\gamma}]$ in $H^2(M;\underline \U_1)$, $M$ is \spinc if
and only if $[w_{\ga\gb\gamma}]$ is mapped to 0, i.e., if and only
if
\begin{equation} \label{spinc:exist2}
[w_{\ga\gb\gamma}]\in \Im\left(\gd_{\gz}^1:H^1(M;\underline
\U_1)\to H^2(M;\Z_2)\right).
\end{equation}
It is not difficult to verify that under the isomorphism
$H^1(M;\underline \U_1)\cong H^2(M;\Z)$, the connecting
homomorphism $\gd_{\gz}^1$ corresponds to mod 2 reduction
$H^2(M;\Z)\to H^2(M;\Z_2)$. Therefore, we can reformulate the
above in the following way: \textit{$M$ is \spinc if and only if
$[w_{\ga\gb\gamma}]$ is the mod 2 reduction of an integral class}.

The diagram also shows that a lift of $[g_{\ga\gb}]$ to a
$\Spinc_n$-bundle $[h_{\ga\gb}^c]\in H^1(M;\underline \Spinc)$
gives a $\U_1$-bundle via
\[
[\gl_{\ga\gb}]:=\gz^c[h_{\ga\gb}^c]=[\gz^c\circ h_{\ga\gb}^c]
\]
This $\U_1$-bundle is the canonical line bundle of the \spinc
structure (cf. Def.\ref{canline}). By commutativity of the
diagram \eqref{diag} we conclude that
\[
\gd_{\gz}^1[\gl_{\ga\gb}]=[w_{\ga\gb\gamma}],
\]
hence---in retrospect---condition \eqref{spinc:exist2} for the
existence of a \spinc structure $\gs$ is fulfilled by the
representative of $L(\gs)$ in $H^1(M;\underline \U_1)$.
\begin{remark*}
The class $[w_{\ga\gb\gamma}]$ which we have constructed above is
a characteristic class of $M$ depending only on the homotopy type.
It is the so-called {\em second Stiefel-Whitney class} $w_2(M)$.
The general topological construction of Stiefel-Whitney classes
can be found in Milnor \& Stasheff's book \cite{MS}. Moreover, it
can be proved directly (cf. Lawson \& Michelsohn \cite{LM},
Sec.~II.1) that the class $[w_{\ga\gb\gamma}]$ satisfies the
characterizing properties of the second Stiefel-Whitney class.
\end{remark*}
We summarize the above considerations in the following
proposition:
\begin{prop}
Let $M$ be an oriented Riemannian manifold. Then $M$ admits a
spin structure if and only if its second Stiefel-Whitney class
$w_2(M)\in H^2(M;\Z_2)$ vanishes. $M$ admits a \spinc structure if
and only if $w_2(M)$ is the mod 2 reduction of an integral class.
\end{prop}
\begin{prop}\label{3mfd=spinc}
Let $M$ be a connected, compact and oriented 3-manifold. Then $M$
admits a \spinc structure.
\end{prop}
\begin{proof}
We use the condition \eqref{spinc:exist1}. Since $M$ is
connected, $H_0(M;\Z)=\Z$. Then Poincar\'e duality shows that
$H^3(M;\Z)=\Z$ as well. In particular, $H^3(M;\Z)$ contains no
torsion elements. Therefore, the elements of $H^2(M;\Z_2)$ are
mapped to zero in $H^3(M;\Z)$.
\end{proof}
\begin{remark*}
It should be pointed out that by making use of more efficient
topological methods one can prove that every compact and oriented
3-manifold is not only \spinc but even spin. Moreover, a famous
result by Wu, Hirzebruch and Hopf \cite{HiHo:FE} guarantees that
every compact and oriented 4-manifold is spin$^c$.\\
\end{remark*}\index{spin$^c$ manifolds!spic$^c$ structure|)}

\noindent\textbf{Gauge transformations.}\index{gauge
transformations|(} Let $(M,\gs)$ be a \spinc manifold. An
automorphism of $(M,\gs)$ is a bundle automorphism
$\gF:P_{\Spinc}(\gs)\to P_{\Spinc}(\gs)$ such that the
corresponding diagram \eqref{spinc:eq} is commutative. In other
words, $\gF$ is given by a collection of smooth maps
$\{\gF_{\ga}:U_{\ga}\to \U_1\}$ satisfying
\[
h_{\ga\gb}=\gF_{\ga}h_{\ga\gb}\gF_{\gb}^{-1},
\]
where $\{h_{\ga\gb}\}$ is a fitting cocycle for $\gs$. Since
elements of $\U_1$ commute with elements of $\Spinc_n$, we
conclude that
\[
\gF_{\ga}|_{U_{\ga}\cap U_{\gb}}= \gF_{\gb}|_{U_{\ga}\cap
U_{\gb}}.
\]
Therefore, the family $\{\gF_{\ga}\}$ defines a smooth map, say,
$\gamma:M\to \U_1$.

\begin{dfn} \label{gg}\index{>@$\cG$, group
of gauge transformations} The automorphism group of a \spinc
manifold $(M,\gs)$ is called the {\em group of gauge
transformations}. It is denoted by $\cG$. According to the above,
\[
\cG=C^\infty(M,\U_1)=H^0(M;\underline \U_1).
\]
\end{dfn}

Consider the exponential sequence\footnote{Note the difference
compared to the sequence in Example \ref{exp:seq}. The modified
version is more suitable when we consider $\U_1$ as a Lie group
with Lie algebra $i\R$.}
\[
1\to 2\pi i\Z\longrightarrow \underline{i\R}
\xrightarrow{e^x}\underline\U_1\to 1,
\]
and let $\gd: H^0(M;\underline\U_1)\to H^1(M;2\pi i \Z)$ denote
the connecting homomorphism in the long exact cohomology sequence.
Since $2\pi i\Z$ is a subsheaf of the sheaf of locally constant
$i\R$-valued functions\footnote{Recall that $\underline{\R}$ is
the sheaf of differentiable functions on $M$; the sheaf of
locally constant $\R$-valued functions is simply denoted by
$\R$.} on $M$, there exists a natural map $H^1(M;2\pi i\Z)\to
H^1(M;i\R)$. According to the deRham Theorem, $H^1(M;\R)$ is
isomorphic to $H^1_{dR}(M;\R)$, the space of closed 1-forms
modulo exact 1-forms. We thus obtain a natural
map\index{=@$H^1_{dR}(M;\R)$, deRham
cohomology}\index{=@$H^1_{dR}(M;\Z)$}
\[
\gr:H^1(M;2\pi i\Z)\to H^1_{dR}(M;i\R).
\]
Using the notation $H^1_{dR}(M;2\pi i\Z):= \im\gr$, we have the
map
\[
\gr\circ\gd:H^0(M;\underline \U_1)\to H^1_{dR}(M;2\pi i\Z)
\]
assigning to a gauge transformation a cohomology class of an
imaginary valued, closed 1-form. Since $\underline{i\R}$ admits
partitions of unity, $\gd$ is surjective. This implies that the
above composition is also surjective. Moreover,
\[
\ker\gr\circ\gd=\ker \gd=\setdef{\exp(f)}{f\in
H^0(M;\underline{i\R})}
\]
which is essentially the identity component of the group of gauge
transformations (cf. also Proposition \ref{action:diff}). An
explicit formula for $\gr\circ\gd$ is
\begin{equation}\label{derham}
\gr\circ\gd:H^0(M;\underline \U_1)\to H^1_{dR}(M;2\pi i\Z),\quad
\gamma\mapsto [\gamma^{-1}d\gamma]\,.
\end{equation}
\begin{proof}
We use a description of the involved maps as it can be found, for
example, in Wells' book (cf. \cite{W}, Sec.~III.4). Let $\gamma\in
H^0(M;\underline \U_1)$ be a gauge transformation and let
$\{U_\ga\}$ be a good open cover of $M$. As $U_\ga$ is
contractible, $\gamma_\ga:=\gamma|_{U_\ga}$ can be lifted to
$f_\ga:U_\ga\to i\R$ via $\exp$. Since $\gamma_\ga=\gamma_\gb$ on
$U_\ga\cap U_\gb$,
\[
f_\gb - f_\ga: U_\ga\cap U_\gb \to 2\pi i\Z\subset i\R.
\]
This defines a $2\pi i\Z$ valued \v{C}ech 1-cocycle whose
cohomology class is $\gd\gamma$. On the other hand, each $f_\ga$
satisfies
\[
df_\ga= \exp(-f_\ga)d\exp(f_\ga)=\gamma_\ga^{-1}d\gamma_\ga\,.
\]
Therefore, the explicit description of the deRham isomorphism
shows that $\gr[f_\gb-f_\ga]=[\gamma^{-1}d\gamma]$. Thus,
$\gr\circ\gd(\gamma)=[\gamma^{-1}d\gamma]$.
\end{proof}
\index{gauge transformations|)}

\section{The spin$^c\,$ Dirac operator}\label{spc:do}

In this section we shall associate to each \spinc structure a
vector bundle over $M$, which turns out to have a rich geometrical
structure. This will give the background to introduce the \spinc
Dirac operator. \\

\noindent\textbf{Spinor bundles.}\index{spin$^c$ manifolds!spinor
bundle|(} Any \spinc manifold is endowed with a vector bundle
carrying reflecting much of the rich geometric structure of the
underlying manifold.

\begin{dfn}\label{spinorbdl}\index{=@$S(\gs)$, spinor bundle}
Let $(M,\gs)$ be an $n$-dimensional \spinc manifold. If
$\gr:\Spinc_n\to \GL(\gD)$ is a spin representation as in
Definition \ref{spinrep}, the vector bundle associated to
$P_{\Spinc}(\gs)$ via $\gr$, i.e.,
\[
S(\gs):=P_{\Spinc}(\gs)\times_{\gr}\gD,
\]
is called a {\em fundamental spinor bundle} on $M$. A section
$\psi$ of $S(\gs)$ is a {\em spinor field} or simply a {\em
spinor}.
\end{dfn}

The action of gauge transformations on $P_{\Spinc}(\gs)$ induces
an action on the space of spinor fields which is given by scalar
multiplication\footnote{Whether one lets $\cG$ act via $\gamma$
or $\gamma^{-1}$ is merely a matter of taste. Choosing the
inverse action for spinors has the advantage that $\cG$ must act
on gauge fields (cf. \eqref{gg:on:A} below) in the usual way.}
\begin{equation} \label{gg:on:spinor}\index{gauge
transformations!action on spinor fields} \cG\times
C^\infty(M,S(\gs)) \longrightarrow C^\infty(M,S(\gs)),\quad
(\gamma,\psi)\longmapsto \gamma^{-1}\psi\,.
\end{equation}

As a result of Proposition \ref{spinc(M)}, any other \spinc
structure on $M$ is equivalent to $\gs\otimes L$ for some
appropriate Hermitian line bundle $L\to M$. The corresponding
fundamental spinor bundles are related via
\begin{equation}
S(\gs\otimes L)=S(\gs)\otimes L\,.
\end{equation}
\begin{proof}
Let $\{h_{\ga\gb}:U_{\ga}\cap U_{\gb}\to \Spinc_n\}$ be a fitting
cocycle for $\gs$ and $\{\gl_{\ga\gb}:U_{\ga}\cap U_{\gb}\to
\U_1\}$ a family transition functions for $L$ respectively. Then
$P_{\Spinc}(\gs\otimes L)$ is defined by the cocycle
$\{h_{\ga\gb}\gl_{\ga\gb}\}$ and the associated vector bundle
$S(\gs\otimes L)$ is given by $\{\gr(h_{\ga\gb}\gl_{\ga\gb})\}$.
Since Clifford multiplication by scalars is just scalar
multiplication in $\gD$, we obtain
\[
\gr(h_{\ga\gb}\gl_{\ga\gb})=\gr(h_{\ga\gb})
\gr(\gl_{\ga\gb})=\gr(h_{\ga\gb})\gl_{\ga\gb}.
\]
This implies the asserted formula because the right hand side
defines a family of transition functions of $S(\gs)\otimes L$.
\end{proof}

\begin{example}
Suppose $(M,\eps)$ is a spin manifold with spinor bundle $S(\eps)$
by means of a spin representation of Spin$_n$. It is easy to
check that this bundle equals the spinor bundle associated to the
canonical \spinc structure $\gs(\eps)$ on $M$. In combination
with Proposition \ref{spinc(M)}, the above result shows that we
obtain any other spinor bundle $S(\gs)$ over $M$ by tensoring
$S(\eps)$ with a Hermitian line bundle.\\
\end{example}

\noindent\textbf{The Clifford bundle.}\index{spin$^c$
manifolds!Clifford bundle} Let $(M,g)$ be an oriented Riemannian
manifold. Then we can form the {\em Clifford bundle} $\cl(M;g)$
over $M$ whose fibres consist of the Clifford algebras of the
tangent spaces\footnote{We shall frequently identify $\cl(M;g)$
with the bundle of Clifford algebras associated to the cotangent
bundle, suppressing the explicit reference to the isomorphism
$TM\to T^*M$ induced by the metric.}, i.e.,
\[
\cl(M;g)_x=\cl(T_xM;g_x)\,.
\]
This yields a bundle of algebras which is naturally associated to
the principal $\SO_n$-bundle $P_{\SO}(g)$ by means of the
canonical representation
\[
\SO_n\to \Aut(\cl_n)
\]
which is given in the following way: Note that every $A\in \SO_n$
gives a linear map $\R^n\to \cl_n$, $v\mapsto Av$ which satisfies
\[
(Av)^2= -|Av|^21= -|v|^21.
\]
Hence by the universal property of the Clifford algebra, the map
$v\mapsto Av$ extends to an algebra automorphism of $\cl_n$.

On the other hand, each element of $\Spinc_n$ acts on $\clc_n$ by
conjugation. This leaves the real subalgebra $\cl_n$ fixed so that
there is a canonical representation
\begin{equation*}
\gr_c:\Spinc_n\to \Aut(\cl_n).
\end{equation*}
From the definition of $\xi_0^c:\Spinc_n\to \SO_n$ it is
immediate that the canonical representations of $\SO_n$ and
$\Spinc_n$ on $\cl_n$ are related in the following way:
\[
\begindc
\obj(3,2){$\Aut(\cl_n)$}[cl] \obj(1,1){$\SO_n$}[SO]
\obj(1,3){$\Spinc_n$}[Spc] \mor{SO}{cl}{}
\mor{Spc}{SO}{$\xi_0^c$}[\atright,\solidarrow]\mor{Spc}{cl}{$\gr_c$}
\enddc
\]
Therefore, if $(M,\gs)$ is a \spinc manifold, the Clifford bundle
coincides with the bundle of Clifford algebras associated to
$P_{\Spinc}(\gs)$ via $\gr_c$.

Let $\gD$ be a complex spinor module and $S(\gs)$ the
corresponding spinor bundle. For $g\in\Spinc_n$, $x\in\cl_n$ and
$\psi\in\gD$, we have
\[
\big(\gr_c(g)x\big)\cdot(g\cdot\psi)=(gxg^{-1})\cdot(g\cdot\psi)
=(gx)\cdot\psi.
\]
Since both, $S(\gs)$ and $\cl(M)$, are bundles are associated to
$P_{\Spinc}$, the action of $\cl_n$ on $\gD$ thus extends to a
global action, i.e., to a bundle homomorphism
\[
c:\cl(M;g)\longrightarrow \End(S(\gs)).
\]


\noindent\textbf{Hermitian structure on $\boldsymbol{S(\gs)}$.}
Since a spin representation $\gr:\Spinc_n\to \GL(\gD)$ is a
representation of a compact Lie group on a complex vector space,
there exists a $\Spinc_n$-invariant Hermitian metric
$\scalar{.}{.}$ on $\gD$. However, it turns out (cf. \cite{LM},
Sec.~I.5) that we can also achieve that
\begin{equation} \label{cl:metric:comp}
\langle c(x)\psi,\psi'\rangle
=-\langle\psi,c(x)\psi'\rangle,\quad x\in\R^n,\;\psi,\psi'\in\gD.
\end{equation}
Here, $c:\cl_n\to\End(\gD)$ denotes the irreducible complex
representation of $\cl_n$ which induces the spin representation.

If $(M,g)$ is an $n$-dimensional oriented Riemannian manifold
which admits a \spinc structure $\gs$, then the above metric on
$\gD$ extends to a Hermitian metric on the fundamental spinor
bundle $S(\gs)=P_{\Spinc}(\gs)\times_{\gr}\gD$satisfying
\eqref{cl:metric:comp} with respect to the global Clifford
multiplication of vector fields (or 1-forms) on
$S(\gs)$.\\

\noindent\textbf{Covariant derivatives on spinor
bundles.}\index{spin$^c$ manifolds!spin connection|(} We assume
some familiarity with the definition of connections on principal
bundles and the correspondence between them and covariant
derivatives on associated vector bundles. As a general reference
we refer to Kobayashi \& Nomizu \cite{KN}. Moreover, a
comprehensive presentation of all definitions and results we need
can be found in Lawson \& Michelsohn \cite{LM}, Sec.~II.4.

Let $(M,g)$ be an oriented Riemannian manifold with connection
1-form $\go=\go^g\in\gO^1(P_{\SO}(g))\otimes\cso_n$ associated to
the Levi-Civita covariant derivative $\nabla=\nabla^g$ on
$(TM,g)$. Here, $\cso_n$ denotes the Lie algebra of $\SO_n$,
i.e., the real vector space of skew adjoint $(n\times n)$-matrices
endowed with the canonical Lie bracket. Suppose that
$\{U_{\ga}\}$ is a good open cover of $M$ so that we may choose a
section $e_{\ga}=(e_1,\ldots,e_n)$ of $P_{\SO}(g)|_{U_{\ga}}$ for
each $\ga$. Then $\go$ is determined by
\[
\Tilde{\go}_{ij}:=(e_{\ga}^*\go)_{ij}= g(\nabla
e_i,e_j)\in\gO^1(U_{\ga})
\]
and the Levi-Civita covariant derivative over $U_{\ga}$ is then
locally given by
\[
\nabla=d+\sum_{i<j}\Tilde{\go}_{ij}J_{ij}.
\]
Here, $\{J_{ij}\}$ is the standard basis of $\cso_n$ which is
defined with respect to an orthonormal basis\footnote{Note that on
the one hand, $(e_1,\ldots,e_n)$ is an orthonormal basis of
$\R^n$ and on the other hand, we use the same notation for the
components of a local section $e_\ga$ of $P_{\SO}$. However, this
ambiguity should cause no confusion since the components of
$e_\ga(x)$ form an orthonormal basis for every $x\in M$.}
$(e_1,\ldots,e_n)$ of
$\R^n$ via $J_{ij}e_k=\gd_{ik}e_j-\gd_{jk}e_i$.\\

Let us now assume that in addition, $M$ admits a \spinc structure
$\gs$.
\begin{dfn}\index{=@$\cA(\gs)$, space of gauge fields} The space
of connections on the principal $U_1$-bundle $P_{\U_1}(\gs)$ is
denoted by $\cA(\gs)$. An element of $\cA(\gs)$ is also called a
{\em gauge field}.
\end{dfn}
$\cA(\gs)$ is an affine space modelled on $i\gO^1(M)$, where we
identify the Lie algebra $\cu_1$ of $\U_1$ with the purely
imaginary numbers $i\R$.

Let us fix $A\in\cA(\gs)$. Choosing local sections
$s_{\ga}:U_{\ga}\to P_{\U_1}(\gs)|_{U_{\ga}}$, we define the
imaginary valued 1-forms
\[
A_{\ga}:=s_{\ga}^*A\in i\gO^1(U_{\ga})
\]
Since $\U_1$ is abelian, the family $\{A_{\ga}\}$ satisfies
\[
A_{\gb}=\gl_{\ga\gb}^{-1}A_{\ga}\gl_{\ga\gb}
+\gl_{\ga\gb}^{-1}d\gl_{\ga\gb}=A_{\ga}+\gl_{\ga\gb}^{-1}d\gl_{\ga\gb}\,,
\]
where $\{\gl_{\ga\gb}:U_{\ga}\cap U_{\gb}\to \U_1\}$ is the
cocycle given by the sections $\{s_{\ga}\}$.

The fibre product $P_{\SO}\times P_{\U_1}(\gs)\to M$ is endowed
with a connection $\go\oplus A$ induced by the connection 1-forms
$\go$ and $A$. Let $\xi:P_{\Spinc}(\gs)\to P_{\SO}\times
P_{\U_1}(\gs)$ be the two sheeted covering map encoded in the
\spinc structure $\gs$. Since $\xi$ is equivariant with respect to
\begin{equation}\label{double:cov}
(\xi_0^c,\gz^c):\Spinc_n\to\SO_n\times\U_1,
\end{equation}
the product connection $\go\oplus A$ lifts to a connection on
$P_{\Spinc}(\gs)$ via
\[
\go^A:=\gF^{-1}\circ \xi^*(\go\times A)\in
\gO^1\big(P_{\Spinc}(\gs)\big)\otimes \cspinc_n,
\]
where
\[
\gF:=(\xi_0^c,\gz^c)_*:\cspinc_n\to\cso_n\oplus i\R
\]
is the Lie algebra isomorphism induced by \eqref{double:cov}.

Let us briefly recall an explicit description of $\gF$. Since
there are some subtleties involved, we refer to Lawson \&
Michelsohn \cite{LM}, Sec.~I.6, for more details. As
$\Spinc_n=\Spin_n\times_{\Z_2}\U_1$, there is a canonical
isomorphism
\begin{equation}\label{spinc=spinxiR}
\cspinc_n\cong\cspin_n\oplus i\R\,.
\end{equation}
The Lie algebra $\cspin_n$ can be identified with the subspace of
$\cl_n$ spanned by elements of the form $e_ie_j$ and endowed with
the Lie bracket induced by the commutator in $\cl_n$. Then the
differential of $\xi_0:\Spin_n\to\SO_n$ at the unit element is
given by the action of the basis in the following way
\[
(\xi_0)_*:\cspin_n\to\cso_n,\quad (\xi_0)_*(e_ie_j):= 2J_{ij}\,.
\]
The differential of $\gz:=z^2:\U_1\to\U_1$ at the unit element is
\[
\gz_*: i\R \to i\R,\quad ia \mapsto 2ia\,.
\]
Therefore, with respect to \eqref{spinc=spinxiR}, the isomorphism
$\gF$ is given by
\begin{equation}\label{spinc=so}
\gF(e_ie_j,ia)=(2J_{ij},2ia).
\end{equation}

\begin{dfn}
The connection $\go^A$ is called the {\em Clifford connection} on
$P_{\Spinc}(\gs)$ associated to $A$.
\end{dfn}

The section $(e_{\ga},s_{\ga}):U_{\ga}\to P_{\SO}\times
P_{\U_1}|_{U_{\ga}}$ can be lifted to a section
$t_{\ga}:U_{\ga}\to P_{\Spinc}|_{U_{\ga}}$. Then it follows from
\eqref{spinc=so} that
\[
t_{\ga}^*\go^A =\Big(\lfrac{1}{2}\sum_{i<j}\Tilde{\go}_{ij}e_ie_j,
\lfrac{1}{2}A_{\ga}\Big)\in\gO^1(U_{\ga})\otimes\cspinc_n.
\]
Here, we are using the local connection 1-forms we described
above.

\index{=@$\nabla^A$, covariant derivative} We now consider a spin
representation $\gr:\Spinc_n\to\U(\gD)$. The connection $\go^A$
induces a covariant derivative $\nabla^A$ on the fundamental
spinor bundle $S(\gs)$. It can locally be described by
\begin{equation}\label{cl:cov:loc}
\nabla^A\psi=d\psi+\lfrac{1}{2}\sum_{i<j}\Tilde{\go}_{ij}c(e_i)
c(e_j)\psi + \lfrac{1}{2}A_{\ga}\psi.
\end{equation}
Here, $\psi$ is a spinor and $c$ denotes Clifford multiplication.
Since the representation $\gr$ is unitary, $\nabla^A$ is
compatible with the canonical metric on $S(\gs)$, i.e.,
\[
\scalar{\nabla^A\psi}{\psi'} + \scalar{\psi}{\nabla^A\psi'} =
d\scalar{\psi}{\psi'},\quad \psi,\psi'\in C^{\infty}(S(\gs))\,.
\]
Furthermore, it can be established (cf. \cite{LM}, Sec.~II.4.11)
that $\nabla^A$ satisfies the following compatibility
rule\footnote{This condition can be reformulated by saying that
the Clifford multiplication $c$ is parallel with respect to
$\nabla^A$.}
\begin{equation} \label{cl:cov:comp}
\nabla^A(c(X)\psi)=c(X)\nabla^A\psi + c(\nabla^gX)\psi\,,
\end{equation}
where $X\in C^\infty(M,TM)$ and $\psi\in C^\infty(M,S(\gs))$.

The group of gauge transformations $\cG=C^{\infty}(M,\U_1)$ acts
on the set of covariant derivatives on $S(\gs)$ by
\[
(\gamma,\nabla)\longmapsto \gamma\cdot\nabla:=
\gamma^{-1}\nabla\gamma.
\]
Then we conclude from \eqref{cl:cov:loc} that locally,
\[
\begin{split}
(\gamma\cdot\nabla^A)\psi & = \gamma^{-1}d(\gamma\psi) +
\lfrac{1}{2}\sum_{i<j}\Tilde{\go}_{ij}c(e_i)c(e_j)\psi +
\lfrac{1}{2}A_{\ga}\psi \\
&=d\psi + \lfrac{1}{2}\sum_{i<j}\Tilde{\go}_{ij}c(e_i)c(e_j)\psi +
\lfrac{1}{2}\big(A_{\ga}+2\gamma^{-1}d\gamma\big)\psi \\
& = \nabla^{A+2\gamma^{-1}d\gamma}\psi\,.
\end{split}
\]
Therefore, the natural action of $\cG$ on the space of gauge
fields is given by
\begin{equation}\label{gg:on:A}\index{gauge
transformations!action on gauge fields} \cG\times \cA(\gs)\to
\cA(\gs),\quad\gamma\cdot A:= A+2\gamma^{-1}d\gamma\,.
\end{equation}\index{spin$^c$ manifolds!spinor
bundle|)}\index{spin$^c$ manifolds!spin connection|)}

\noindent\textbf{The spin$^c\,$ Dirac operator.}\index{spin$^c$
manifolds!Dirac operator|see{Dirac operator}}\index{Dirac
operator|(} On a spinor bundle, we now want to construct a
first-order differential operator whose square is a generalized
Laplacian. This construction has a more general background. We
thus briefly recall the structure which is needed to carry out the
construction in general.
\begin{dfn}
Let $E\to M$ be a Hermitian or Euclidean vector bundle over a
Riemannian manifold $(M,g)$. A {\em Dirac structure} on $E$ is
given by the following data:
\begin{itemize}
\item A covariant derivative $\nabla$ on $E$ which is compatible with the
metric,
\item a {\em Clifford structure} on $E$, i.e., a bundle map
$c:T^*M\to \End(E)$ which satisfies
\[
c(\ga)\circ c(\ga')+c(\ga')\circ c(\ga) = -2g(\ga,\ga')\id_E,
\]
and which is skew adjoint with respect to the metric on $E$,
\item the compatibility condition
\[
\nabla(c(\ga)e)=c(\ga)\nabla e + c(\nabla^g\ga)e,\quad
e\in C^{\infty}(M,E), \ga\in\gO^1(M)\,.
\]
\end{itemize}
If $E$ carries a Dirac structure, it is called a {\em Dirac
bundle} over $M$.
\end{dfn}
Equations \eqref{cl:metric:comp} and \eqref{cl:cov:comp} show that
the canonical metric on a fundamental spinor bundle over a \spinc
manifold $M$ together with the Clifford connection satisfies all
of the above conditions. We thus obtain:
\begin{prop}
A fundamental spinor bundle $S(\gs)$ over an oriented Riemannian
\spinc manifold $(M,\gs)$ is a Dirac bundle.
\end{prop}

A Clifford structure yields a bundle map $c:T^*M\otimes E\to E$.
Moreover, a covariant derivative on $E$ is a $\K$-linear map
$\nabla:C^{\infty}(E)\to C^{\infty}(T^*M\otimes E)$. Therefore,
\begin{equation}\label{dirac:def}
\cD:=c\circ\nabla:C^{\infty}(E)\longrightarrow C^{\infty}(E)
\end{equation}
defines a first-order differential operator. Whenever $E$ is a
Dirac bundle over $M$, the operator $\cD:=c\circ\nabla$ is called
a {\em geometric Dirac operator}.

\begin{dfn}\index{=@$\cD_A$, Dirac operator}
Suppose $(M,\gs)$ is a \spinc manifold. Let $A\in\cA(\gs)$ be a
connection on the $\U_1$-bundle $P_{\U_1}(\gs)$. The geometric
Dirac operator $\cD_A$, given by the Dirac structure
$(S(\gs),\nabla^A,c)$, is called the {\em spin$^c$ Dirac operator}
associated to $A$.
\end{dfn}
\begin{remark*}

If $(M,\eps)$ is a spin manifold, the \spinc Dirac operator
associated to the canonical \spinc structure $\gs(\eps)$ and the
flat connection on $L(\gs(\eps))=M\times\C$ is the well-known spin
Dirac operator. Hence, on a spin manifold, all \spinc Dirac
operators are twisted versions of the spin Dirac operator.\\
\end{remark*}

\noindent\textbf{Properties of $\boldsymbol{\cD_A}$.} Geometric
Dirac operators and, specifically, the \spinc Dirac operator have
very important analytical properties some of which we state now.
The corresponding proofs, which are not easy but standard
calculations, can be found in any textbook on spin geometry or
index theory.

\begin{prop}{\rm (cf. \cite{LM}, II.5.3)}.
Let $E$ be a Dirac bundle over a Riemannian manifold $M$ and let
$\cD$ be the geometric Dirac operator.
\begin{enumerate}
\item The principal symbol of $\cD^2$ satisfies
\[
\gs(\cD^2)_{\xi}=-|\xi|^2,\quad \xi\in T^*M\setminus \{0\},
\]
i.e., $\cD^2$ is a generalized Laplacian. In particular, $\cD$ is
an elliptic operator.
\item Suppose $M$ is oriented and compact. Then $\cD$ is formally
self-adjoint with respect to the $L^2$ scalar product on
$C^\infty(M,E)$.
\end{enumerate}
\end{prop}

\begin{prop}[Weitzenb\"{o}ck Formula]\label{weitzenboeck}{\rm
(cf. \cite{LM}, Thm.~D.12)}. Suppose $(M,\gs)$ is a compact,
oriented Riemannian \spinc manifold. Let $\cD_A$ be the \spinc
Dirac operator associated to a gauge field $A\in\cA(\gs)$. Then
\[
\cD_A^2=(\nabla^A)^*\nabla^A+\lfrac{1}{4}s_g+\lfrac{1}{2}c(F_A).
\]
Here, $s_g$ denotes the scalar curvature of $M$ and $F_A$ is the
connection 2-form of $A$.
\end{prop}\index{Dirac operator!Weitzenbock Formula}

\begin{remark*}
Observe that $F_A$ can be interpreted as an imaginary valued
2-form on $M$ since $\U_1$ is abelian. Hence, the expression
$c(F_A)$ is well-defined. Recall that Clifford multiplication by
$k$-forms is defined via the isomorphism of vector spaces
$\gL^\bullet V\cong \cl(V)$.
\end{remark*}
Using the local description of the Clifford connection
\eqref{cl:cov:loc}, one straightforwardly establishes the
following.

\begin{lemma}\label{A+a}
Suppose $M$ is an oriented Riemannian manifold equipped with a
\spinc structure $\gs$. Let $A\in\cA(\gs)$, and let $a\in i
\gO^1(M)$ be an imaginary valued 1-form. Then
\[
\cD_{A+a}=\cD_A+\lfrac{1}{2}c(a).
\]
\end{lemma}
\begin{remark*}\index{Dirac operator!Atiyah-Singer index Theorem}
If $M$ is an even dimensional oriented Riemannian manifold which
admits a \spinc structure, then the fundamental spinor bundle
splits into the eigenbundles of the complex volume element
$\go^c\in\clc(M)$. Therefore, $\cD_A$ decomposes into the elliptic
operators
\[
\cD_A^\pm:C^\infty(M,S^\pm(\gs))\to C^\infty(M,S^\mp(\gs)).
\]
If $M$ is compact, then $\cD_A^\pm$ are Fredholm operators. The
famous Atiyah-Singer index Theorem relates $\ind_\C(\cD_A^+)$ to
an integral over characteristic classes of $M$. However, as we are
mainly interested in the three dimensional case, we will not go
in more detail and refer to the literature for a further
discussion.
\end{remark*}\index{Dirac operator|)}

\section{Dependence on the metric}\label{met:dep}
\index{Dirac operator!dependence on the metric|(}

At a first glimpse the notion of a \spinc structure seems to
depend on the metric $g$ on $M$, which is encoded in the bundle
$P_{\SO}(g)$. However, it turns out that this is not the case.

The first observation is that for every Riemannian metric $g$ on
$M$ the inclusion $P_{\SO}(g)\subset P_{\GL^+}$ is an equivariant
homotopy equivalence, where the homotopy inverse is defined via
the Gram-Schmidt orthogonalization process. Hence, for any fixed
principal $\U_1$-bundle $P_{\U_1}$, the equivariant two sheeted
coverings of the fibre product $P_{\SO}(g)\times P_{\U_1}$ are in
natural one-to-one correspondence with the equivariant twofold
coverings of $P_{\GL^+}\times P_{\U_1}$. Therefore, interpreting
a \spinc structure $\gs$ as a choice of principal $\U_1$-bundle
$P_{\U_1}(\gs)$ together with a two sheeted covering of
$P_{\GL^+}\times P_{\U_1}$ yields a possibility to define $\gs$
independently of any Riemannian metric.

Unfortunately, there is no possibility to go on in this way and
construct in a metric independent way a bundle which corresponds
to the spinor bundle $S(\gs)$. The deeper reason for this is that
there exists no representation of $\GL^+_n$ stemming from some
generalization of the spinor representation of $\SO_n$ (cf.
Lawson \& Michelsohn \cite{LM}, II.5.23). We thus have to find a
procedure to identify the spinor bundles $S(\gs;g)$ and
$S(\gs;h)$ associated to different metrics $g$ and $h$. The
material presented here is partly taken from S. Maier's article
\cite{Mai:Met} which includes an excellent summary of the results
due to Bourguignon and Gauduchon in \cite{BouGau:Met}.

To begin with, we take a brief look on how to compare data on $TM$
and $T^*M$ for different metrics. Let $k:P_{\SO}(g)\to
P_{\SO}(h)$ denote the $\SO$-equivariant bundle map induced by
$P_{\SO}(g)\subset P_{\GL^+}\to P_{\SO}(h)$. Note that we can
alternatively describe $k$ in the following way: Let $H:TM \to
TM$ be the unique positive bundle endomorphism defined by
$h(.,.)=g(H.,.)$. Then $H$ is symmetric with respect to $g$, and
we have the relation $k=H^{-1/2}$. The map $k$ also gives an
operation on $T^*M$ via $k(\ga)= \ga \circ k^{-1}$.

We need to compare the Hodge-star-operators $*_g$ and $*_h$
associated to $g$ and $h$ respectively. Since for all
$\ga,\gb\in\gO^j(M)$ we have
\[
\begin{split}
\ga\wedge (k*_g k^{-1}\gb) &= \big( k^{-1}\ga\wedge*_g k^{-1}\gb
\big) \circ k^{-1}\\ &= g\big(k^{-1}\ga, k^{-1}\gb\big) dv_g\circ
k^{-1} = h(\ga,\gb)dv_h,
\end{split}
\]
the result is
\begin{equation}\label{hodge:met}
*_h= k\circ *_g \circ k^{-1}.
\end{equation}
Observe that we have used $dv_h = dv_g\circ k$ which is a
consequence of the fact that $k$ maps an orthonormal frame of
$(TM,g)$ to an orthonormal frame of $(TM,h)$.

In general, $k$ need not give rise to an isometry of Hilbert
spaces, $L^2(M,T^*M;g)\to L^2(M,T^*M;h)$, because $dv_g\neq dv_h$.
We therefore let\index{=@$\Hat k$}
\[
\Hat k:=f^{-1}\cdot k,
\]
where $f$ is defined via $dv_h=f^2dv_g$, that is, $f^2=\det k$. We
then obtain
\[
\int_M h(\Hat k\ga,\Hat k\gb)dv_h = \int_M f^{-2} h(k\ga,k\gb)
f^2 dv_g = \int_M g(\ga,\gb)dv_g.
\]
As a result, $\Hat k$ is a Hilbert space isometry
$L^2(M,T^*M;g)\to L^2(M,T^*M;h)$. This gives a possibility to
establish a relation between $d^{*_h}$ and $d^{*_g}$:
\begin{equation}\label{dast:met}
d^{*_h} = \Hat k^2\circ d^{*_g}\circ\Hat k^{-2}.
\end{equation}
\begin{proof}
Suppose $\ga\in \gO^j(M)$. Then for each $\gb\in\gO^{j-1}(M)$, the
following holds:
\[
\big(\ga,d\gb)_{L^2(h)} = \big(\Hat k^{-2} \ga,d\gb\big)_{L^2(g)}
= \big(d^{*_g}\Hat k^{-2}\ga,\gb\big)_{L^2(g)} = \big(\Hat
k^2d^{*_g}\Hat k^{-2}\ga,\gb\big)_{L^2(h)}\qedhere
\]
\end{proof}
Note that in contrast to \eqref{hodge:met} formula
\eqref{dast:met} contains derivatives of $f$ so that explicit
computations are much more involved.

We will now study the relation between $S(\gs;g)$ and $S(\gs;h)$.
For this let $\gk:P_{\Spinc}(\gs;g)\to P_{\Spinc}(\gs;h)$ be the
$\Spinc$-equivariant bundle map induced by lifting $k\times
\id:P_{\SO}(g)\times P_{\U_1}(\gs) \to P_{\SO}(h)\times
P_{\U_1}(\gs)$ to the corresponding twofold coverings. $\gk$
extends to an isometry $\gk:S(\gs;g)\to S(\gs;h)$ of Hermitian
vector bundles. The following is easily established.
\[
\gk\big(c^g(\ga)\psi\big)= c^h(k(\ga))\gk(\psi).
\]
As before, in order to obtain an isometry of Hilbert spaces
$L^2(M,S;g)\to L^2(M,S;h)$ we have to define\index{=@$\Hat\gk$}
\[
\Hat \gk := f^{-1}\cdot \gk.
\]
Then $\Hat \gk$ provides a suitable instrument for pulling back
the \spinc Dirac operator on $S(\gs;h)$ to $S(\gs;g)$. Let $A$ be
a connection on $P_{\U_1}(\gs)$. Then via the lift of the
corresponding Levi-Civita connection, $A$ gives rise to a
covariant derivative $\nabla^{A;h}$ on $S(\gs;h)$. Let $\cD_A^h$
denote the associated \spinc Dirac operators on $S(\gs;h)$. Then
\begin{equation}\label{Dirac:met}
\Hat \gk^{-1}\circ \cD_A^h \circ \Hat \gk
\end{equation}
defines a first-order elliptic operator on $S(\gs;g)$ which is
formally self-adjoint with respect to the $L^2(g)$-metric.
Fortunately, we shall not need a more explicit description of this
operator. For more information concerning these points we refer
to Maier \cite{Mai:Met} and Bourguignon \& Gauduchon
\cite{BouGau:Met}. \index{Dirac operator!dependence on the
metric|)}

\cleardoublepage